\documentclass[11pt,openright]{bookkk}
\usepackage{amsmath,amsthm,amsfonts,amssymb,amscd}
\usepackage[latin1]{inputenc}
\usepackage{subfigure}
\usepackage{color}
\usepackage{pstricks}
\usepackage{psfrag}
\usepackage{url}
\usepackage{lscape}

\usepackage{mathrsfs}
\usepackage{frcursive}

\usepackage{makeidx}

\usepackage[active]{srcltx}

\usepackage{float}

\usepackage{geometry}
\usepackage{prerex}
\usepackage{verbatim}

\renewcommand{\figurename}{{\small Figure}}

\input xy
\xyoption{all}

\theoremstyle{plain}
\newtheorem{thm}{Theorem}[section]
\newtheorem*{THM}{Theorem}
\newtheorem{prop}[thm]{Proposition}

\newtheorem{lemma}[thm]{Lemma}
\newtheorem{cor}[thm]{Corollary}

\newtheorem{conjecture}{Conjecture}

\theoremstyle{definition}

\newtheorem{problem}{Problem}

\newtheorem{example}[thm]{Example}

\theoremstyle{remark}
\newtheorem{remark}[thm]{Remark}

\newcommand{\eq}[1][r]
   {\ar@<-3pt>@{-}[#1]
    \ar@<-1pt>@{}[#1]|<{}="gauche"
    \ar@<+0pt>@{}[#1]|-{}="milieu"
    \ar@<+1pt>@{}[#1]|>{}="droite"
    \ar@/^2pt/@{-}"gauche";"milieu"
    \ar@/_2pt/@{-}"milieu";"droite"}

\def\dd{\begin{center} $\S$ \end{center}}
\def\hot{{h.o.t.}}
\def\defi[#1]{{\bf{#1}}}
\def\A2W{\mathcal A_2(\mathcal W)}
\def \l  #1 {{ {\bf L}{\mbox{i}}_#1}}

\def\ro[#1]{{\textcolor{red}{#1}}}

\definecolor{pipo}{rgb}{0,0.6,0}
\definecolor{blo}{rgb}{0,0,1}
\definecolor{mag}{cmyk}{0.1,1,0.1,0.1}
\definecolor{ro}{rgb}{0,0.6,0}

\makeindex

\date{}

\begin{document}

\title{\bf AN INVITATION TO WEB GEOMETRY \\ \vskip 0.65cm  From Abel's addition Theorem to  the \\  algebraization of codimension one webs}
\vskip1.7cm 

\author{\large Jorge Vit\'{o}rio Pereira and Luc Pirio}

\maketitle

\pagenumbering{roman}
\setcounter{page}{5}

\thispagestyle{empty}

~

\vskip 11.8cm

{\flushright{   Para  Dayse e Jorge, \\ os meus Pastores. \\ J.V.P. \\}}

\bigskip \bigskip

{\flushright{  Pour Min' \\ L.P. \\}}

%\newpage

\chapter*{Preface}
\addcontentsline{toc}{chapter}{Preface}
\thispagestyle{empty}

The first purpose of this text was  to serve as  supporting material
for  a  mini-course on web geometry delivered at the 27th Brazilian Mathematical Colloquium which
 took place at IMPA in the last week of July 2009. But, in its almost 210 pages there is much
more than one can possibly cover in five lectures of one hour each. The abundance of material is due to
the second purpose of this text:  convey  some of the beauty of web geometry
  and to provide an account, as self-contained as possible, of some of the exciting advancements
the field has  witnessed in the last few years.
\smallskip

We have tried to write a text which is  not very demanding in terms of pre-requisites. It is true that at some points
familiarity with the basic language of algebraic/complex  geometry is welcome but, except at very few passages, not
much more is needed. An effort has been made to explain, even if sometimes superficially, every single {\it unusual}
concept appearing in the text.
\smallskip

At an early stage of this project we decided to use  the third instead of the first person. In retrospect, it is
hard to understand why two authors, none of them particularly comfortable with the English language,
took this decision. Today, the only explanation that comes to mind is a subconscious attempt to expire the sins
of two bad writers. We apologize for the awkwardness of the prose and hope that those more familiar with English
than us will find some amusement with the clumsiness of it.
\smallskip

This text would take much longer to come to light without the invitation of M\'{a}rcio Gomes Soares to
submit a mini-course proposal
to the 27th Brazilian Mathematical Colloquium. Besides Soares, we would like to thank
Hernan Falla Luza and Paulo Sad, whom catched a number of misprints and  mistakes appearing in preliminary versions.
We are also indebted to Annie Bruter  for her help in translating to English a draft of the introduction
originally written in French. Even more importantly, Jorge wants to thank Dayse and Luc wants to thank Mina for all the patience and
unconditional support gracefully given  during the writing of this book.
\medskip

%\begin{flushright}
%Rio de Janeiro and Rennes, May 2009.\\
%Jorge Vit\'{o}rio Pereira and Luc Pirio\\
%\end{flushright}

\vfill

\begin{minipage}[b]{0.4\linewidth}
\begin{center}
Jorge Vit\'{o}rio Pereira \\
IMPA  \\
  \verb"jvp@impa.br"
\end{center}
\end{minipage}
\hspace{0.5cm}
\begin{minipage}[b]{0.5\linewidth}
\begin{center}
   Luc Pirio     \\
  IRMAR -- UMR 6625 du CNRS  \\
   \verb"luc.pirio@univ-rennes1.fr"
\end{center}
\end{minipage}

\tableofcontents

\clearpage

\newpage
\pagenumbering{arabic}

%%%%%%%%%%%%%%%%%%%%%%%%%%%%%%%%%%%%%%%
%  Introduction
%  label = Chapter:intro
%  First version by JVP
%  last modification: 3/nov/2008
%  Remarks:
%%%%%%%%%%%%%%%%%%%%%%%%%%%%%%%%%%%%%%%

\chapter*{Introduction}
\markboth{\MakeUppercase{Introduction}}{\MakeUppercase{Introduction}}

\addcontentsline{toc}{chapter}{Introduction}

It seems impossible to grasp the ins and outs of a mathematical field without
setting it back in its historical context. An attempt, certainly incomplete and biased,  is made  in the next
few pages. 

At the end of the Introduction, one finds a description of the contents of this
text, and suggestions on how to use it.

\section*{Historical Notes}

If the {\it birth} of web geometry can be ascribed to the middle of the 1930s in Hambourg (see below),
some precedents can be found as early as the middle of the XIXth century.
The concepts and problems of web geometry springs from  two different fields of the XIXth century mathematics :
 projective differential geometry and  nomography.
\medskip

It is mainly from the first that web geometry comes from. At that time,
projective differential geometry mainly consisted of the study of
 projective properties of curves and surfaces in $\mathbb R^3$, that is
  of their differential properties that are invariant up to homographies.
\medskip

\thispagestyle{empty}

Gaussian geometry, which had appeared before, studied the properties of (curves and) surfaces in ordinary euclidian space that are invariant up to isometric transformations. Gauss and other mathematicians  pointed out how useful the first and second fundamental forms are
 for the study of surfaces. They also brought to light the relevance of derived concepts, such as
the  principal, asymptotic and  conjugated  directions.
When considering the integral curves of these  tangent direction fields, the mathematicians of the time were considering what they called
  2-nets of ``lines'' on surfaces, that is the data of 2 families of curves, or in more modern terms, 2-webs.
It is when they endeavored to generalize these constructions
to the projective differential geometry  that some {\it 3-nets}
projectively attached to surfaces in $\mathbb R^3$
quite naturally made their appearance (for instance, Darboux  introduced a 3-web called after him in \cite{darboux80}; see also Section  \ref{S:Webs attached to projective surfaces} in this book). \medskip

These webs   were useful back then because they encoded  properties of the surfaces  under study.
Thomsen's paper \cite{thomsen27}
is a good illustration of this fact. In this article, Thomsen shows that a surface area in $\mathbb R^3$ is isothermally asymptotic\footnote{\,
Geometers of the XIXth century  had established   a very rich ``bestiary''  of surfaces in
  $\mathbb R^3$. The {\it  isothermally asymptotic surfaces}  (or ``{\it F-surfaces}'') formed one of the classes in their classification
 (see \cite{ferapontov} for a  modern definition.)}
if and only if its Darboux 3-web is  hexagonal\footnote{\, Thomsen's result applies to real surfaces in $\mathbb R^3$ thus his statement is different  as one takes place at a neighborhood of an elliptic point or a hyperbolic point of the considered surface.}.
At that time, the study of 3-webs on surfaces from the point of view of projective differential geometry was on the agenda. \medskip

Thomsen's result has this particular feature of  characterizing  the geometric-differential property of being isothermally asymptotic by a closedness property of more topological nature that is (or not) verified by a configuration traced on the surface itself. It is this feature which struck some mathematicians and led to the study of webs  at the beginning of the 1930s. %\medskip

\dd

The second source of web  geometry is nomography. This discipline, nowadays practically extinct, belonged to the field of applied mathematics in the 1900s. It was established as an autonomous mathematical discipline by M. d'Ocagne. It consisted in a method of ``graphical calculus''  which allowed engineers to calculate rather fast. To explain  its principle (which to-day appears rather na\"ive), let  $F(a_1,a_2,a_3)=0$
be  a law linking three physical variables. Is there a quick and accurate way to determine one variable say $a_i$ from the other two: $a_j$ and $a_k$?  To solve this problem, people used nomograms. A nomogram  is a graphic which represents  curves  according to  values of the variables $a_1,a_2$ and $a_3$.  For instance, to find the value of $a_1$ in function of values $\alpha_2$ and $\alpha_3$ of the variables $a_2$ and $a_3$ (respectively), one has to find the intersection point of
the curves  $a_2=\alpha_2$ and $a_3=\alpha_3$.
 Through (or near) this point goes a curve $ a_1=\alpha_1$, and $\alpha_1$  is the sought value.

\begin{figure}[H]
\begin{center}
%\hspace{-0.5cm}
{\fbox{\scalebox{0.45}{{\includegraphics*[60,235][575,560]{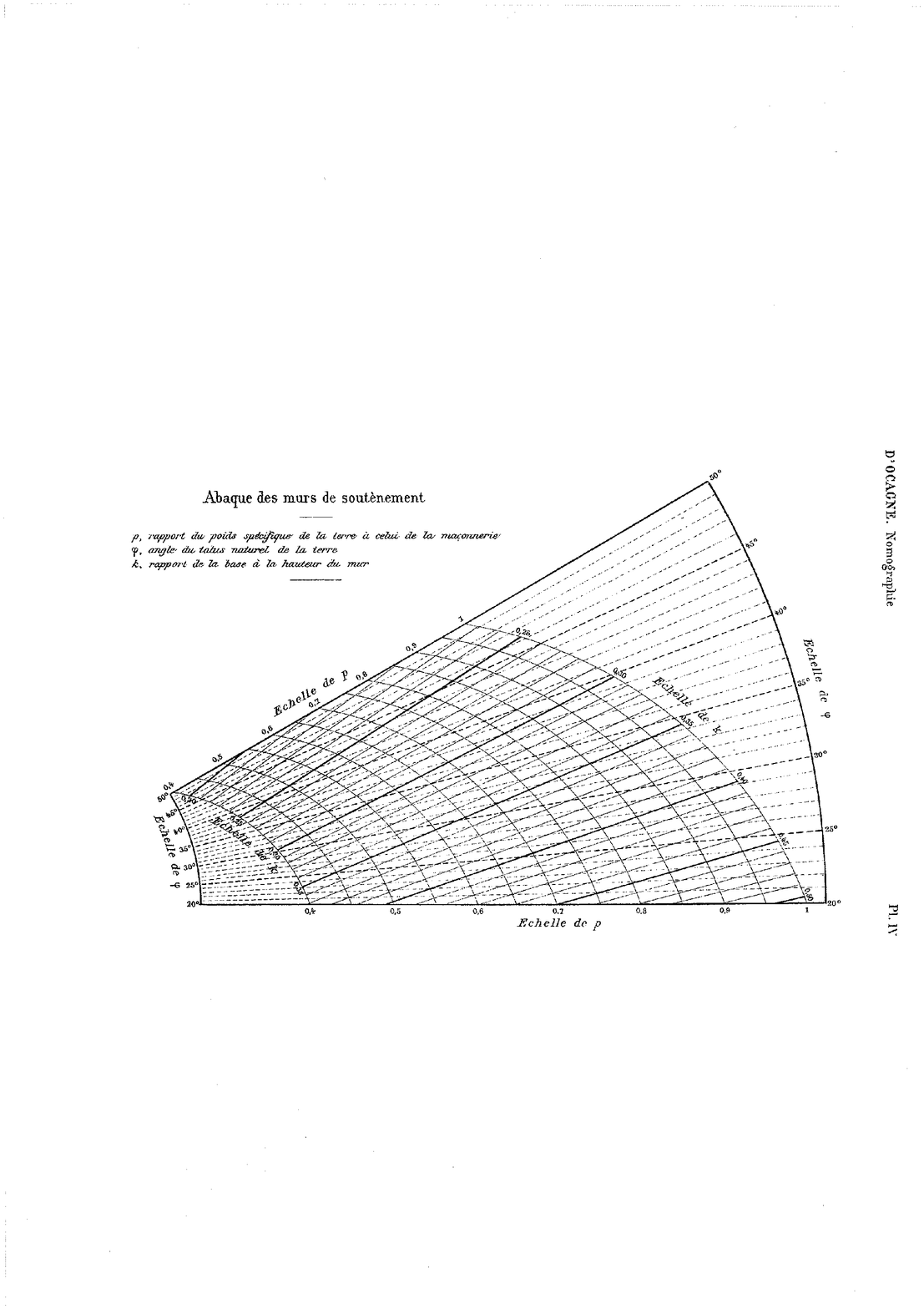}}}}}
\caption{A nomogram from a book by M. D'Ocagne.}
\end{center}
\end{figure}

%\vspace{-0.35cm}

What nowadays seems to be far from actual mathematics was once an important part of the  mathematical culture.
It was probably after considering some results of nomography that Hilbert formulated the thirteenth of the famous 23 problems that he
stated at the International  Congress of Mathematics  of 1900.\medskip

The main disadvantage of nomography was the problem of its readability. Of course, the nomograms where the curves coincided with
 (pieces of) lines   were easier to use. Hence the problem to know whether it is possible to linearize the curves of a given nomogram.
 Or equivalently, whether it is possible to linearize a 3-web of curves on the plane.
 For more precisions on the links between nomography and web geometry, the interested reader can consult \cite{Aczel}.

\subsection*{Birth of web geometry: Spring of 1927 in Naples}

Thomsen's paper \cite{thomsen27} is considered as the  birth of web geometry. According to Blaschke (see the beginning of the foreword in \cite{BB}) this paper is the result of their Spring walks on Posillipo  hill, at the vicinity of Naples, in 1927. Even if it concerns the study of some surfaces in $\mathbb R^3$, it shows clearly that a plane configuration  made of three families of curves  ({\it i.e.} a 3-web) admits local analytic invariants. It seems that the  equivalence between the vanishing of the curvature of a 3-web (which is a condition of analytic nature) and the hexagonality condition (which is a property seemingly of topological nature)\footnote{See  Theorem  \ref{T:hexagonal} farther in this book.}
struck these two mathematicians and led them (with others) to study the matter.

\begin{figure}[H]
\begin{center}\includegraphics[height=7cm,width=7.8cm]{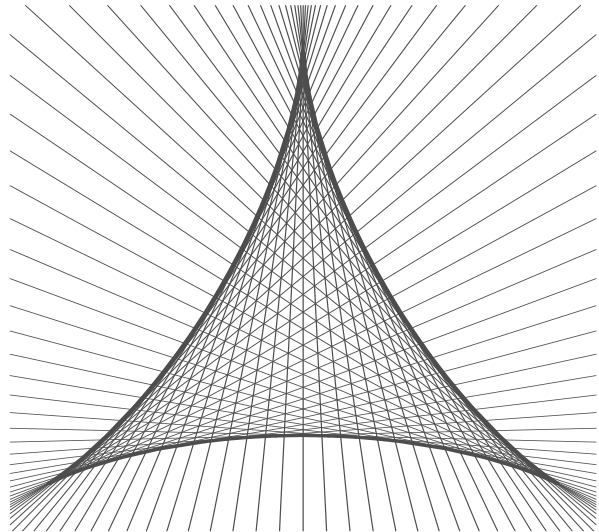}\end{center}
\caption{A $3$-web with vanishing curvature.}
\end{figure}

\bigskip

\subsection*{Early developments: Hamburg school (1927-1938)}

A short time after Thomsen's paper \cite{thomsen27} was published,  a group led by Blaschke was set up in Hamburg to do research on webs. 
 Blaschke and his coworkers\footnote{Bol, Chern, Mayrhoffer, Podehl, Walberer
were active members of this group. K\"ahler, Zariski, Reidemester
and others  also worked on this subject but in a occasional way.}
found many results which established web geometry as a discipline. It is a remarkable fact that a
series of more than 60 papers were  published in a variety of  journals between 1927 and 1938, mainly by members of the Hamburg school of web geometry,
under the common label of ``Topologie Fragen der Differentialgeometrie''\begin{footnote}{In English, ``Topological questions of differential geometry''.}\end{footnote}.

Their work focused on three main directions:
\begin{enumerate}
\item[--] The study of webs from the   differential geometry viewpoint, through the analytical invariants  which can be associated to them;
\item[--] The study of the relations between webs and abstract geometric configurations linked to the algebraic theory of (quasi-)groups;
\item[--] The interpretation of web geometry as a relative of projective algebraic  geometry, notably via the notion of abelian relation.
\end{enumerate}

\medskip

This book focus on the latter direction of study, and  will not expand on the two former ones, due to  lack of space
and of competence as well. It deals with the links between webs and algebraic geometry, which have their
 origins in results obtained by Blaschke, Bol and Howe.
At the beginning, these results were mainly about planar webs.  Firstly, Blaschke came up with an  interpretation of a theorem by Graf and Sauer \cite{grafsauer24}
in the framework of web geometry.  This theorem says that a linear 3-web carrying an abelian relation is constituted by the tangents to a plane algebraic  curve of class 3. Later, as soon as 1932, Blaschke and Howe \cite{blaschkehowe32} generalized this theorem to the case of linear $k$-webs carrying at least one complete abelian relation, thus bringing to light the usefulness of the notion of abelian relation. Bol's result giving the explicit bound  $\frac{1}{2}(k-1)(k-2)$ on the dimension of the space of abelian relations of a planar $k$-web appeared shortly afterwards in \cite{bol32} and allowed to define the rank of a web. Using this formalism, Howe noticed that Lie's result about the surfaces of double translation can be understood in the framework of web geometry as the striking fact that a planar 4-web of rank 3 is algebraizable. The relationship between the planar webs of maximal rank and Abel's Theorem was reported the following year by Blaschke in \cite{blaschke332}, which brought up the final definition of the notion of algebraic web. In the same paper, Blaschke expounded  the generalization of Lie's Theorem to the 5-webs of maximal rank, a result which was later proved by Bol to be incorrect. Surprisingly, he also exhibited   Bol's 5-web $\mathcal B_5$  as an example of non algebraizable 5-web of rank 5, while it is of maximal rank 6 (see below).

\medskip

In 1933, Blaschke set about studying webs in dimension three. In \cite{blaschke33}, he established a bound
$\pi(3,k)$ on the rank of a $k$-web
of hypersurfaces in  $\mathbb C^3$. One year later, Bol gave in
\cite{bol34}  one of the most important results obtained at the time: for $k\geq 6$,  a $k$-web of hypersurfaces
on  $\mathbb C^3$ of maximal rank
$\pi(3,k)$ is algebraizable.  This success certainly played a role in Blaschke's attempt in \cite{blaschke332,blaschke333} to obtain algebraization results for  planar webs of maximal rank. Only in 1936 it was made clear  that the result he was looking for were unattainable. In \cite{bol36}, Bol realized that ${\mathcal B_5}$  carries one more abelian relation, related to Abel's five terms equation for the dilogarithm; hence it is an instance of a 5-web of rank 6, which is not algebraizable.

\medskip

In the year of 1936,  Chern defended his PhD dissertation on webs, written under Blaschke's direction. He then published two papers. The 60th issue of the ``Topologische Fragen der Differentialgeometrie'' series \cite{chern1935} is of special interest here. Generalizing Blaschke's result, he obtains a
bound  on the rank of a web of codimension one in arbitrary dimension\footnote{See Theorem \ref{T:cota} in this book.} which  now bears his name.
\medskip

 Thus, in 1936, most of  the notions studied in this book
 had been brought to light. A general survey of the state the art then  can be found in the third part of the book \cite{BB}, to which the reader is asked to refer. \medskip

Finally, this very year is when Blaschke shifted his interest from  web geometry to  integral geometry. Few members of the Hamburg school
 worked again on webs, with the notable exceptions of Blaschke and Chern, but this time in a different way (see below).

\bigskip

\subsection*{Web~geometry~in~mid~XXth century~(1938-1960)}
Blaschke strongly supported exchanges between mathematicians. From 1927 to 1960 he travelled a lot and had the opportunity to give lectures about web geometry in numerous countries (for instance in Romania, Greece, Spain, Italy, the United States, India, Japan), thus inspiring people with of a  variety of nationalities and backgrounds to do research on web geometry.\footnote{For instance, it is in attending  to some conferences given by Blaschke at Pekin in 1933  that Chern became interested in web geometry and decided to go to study at Hamburg.}

\medskip

It seems that Blaschke went to Italy many times during this period. As a by-product,  an Italian school of web geometry developed at that time. Bompiani, Terracini and Buzano were its most prominent contributors. Their work was chiefly about the links between geometry of planar webs and the projective differential geometry of surfaces.
In the 1950s and 1960s a second Italian school of web geometry appeared, probably thanks to Bompiani's influence. He, Vaona and Villa (among others) published papers on the projective  deformation  of planar 3-webs, but with no major outcome.

\medskip

It also must be mentioned the work of the Romanian mathematicians Pantazi and Mih\u{a}ileanu. During the 1930s and 1940s, they obtained
 interesting results on how to determine the rank of planar webs. These results were published as short notes in Romanian journals (see \cite{Pantazi,Mihaileanu}) and were then forgotten.

 \medskip

The war and later Blaschke's political stance during the war (see \cite[p. 423]{segal}) put an end to his influence for some time. When things went back to normal, he gave lectures again, on webs among other things. Although he didn't obtain new results, these lectures induced new researches once more, for instance  Dou in Barcelona \cite{dou1,dou2,dou3}  and Ozk\"{a}n in Turkey \cite{ozkan}.

\medskip

\subsection*{Russian school (from 1965 onwards)}

More than at the Hamburg school, it is at the Moscow school of
differential geometry that the Russian school of web geometry,  led first by Akivis, and then by Akivis and Goldberg, seems to have its origin.
Under the influence of the work of \'Elie Cartan, a Russian school of differential geometry  developed in USSR at the instigation of Finikov from the 1940s onwards. Projective differential geometry was studied in full generality  and involved the study of some nets (which could be called webs  but only in a weak sense) projectively attached to (analytic) projective subvarieties. It is probably this fact which led to the study of webs for their own sake   in arbitray dimension and/or  codimension, from the 1960s onwards. Akivis was joined by Goldberg quite early. They explored several directions in web geometry, published many papers and had many students.

\medskip

The work of this school led  chiefly dealt with
the differential geometry of webs and with the interactions between webs and the theory of quasigroups. The links
with algebraic geometry were not their major concern. Their results had little influence in the West for two main reasons: (1) their papers were in Russian, hence they were not distributed in the West; (2) the method they used was the {\it Cartan-Laptev method}\footnote{Cartan-Laptev method is  a reinterpretation/generalization of the methods of the mobile frame and of equivalence of \'Elie Cartan, by the Russian geometer G. Laptev.}, which  non specialists do not understand easily.

\medskip

The reader who wishes to get an outline of the methods and results of this school may consult the books \cite{Akivisbook} and \cite{Goldbergbook}.

\medskip

\subsection*{Chern's and Griffith's work (1977-1980)}

Throughout his professional life Chern kept being interested in webs, particularly in the notion of web of maximal rank, as shown in
\cite{Chern},
\cite{ChernB2} and \cite{Chern92}. This point can be illustrated  by quoting the last lines of \cite{Chern92}:
\begin{quote}
 {\it
Due to my background I like algebraic
  manipulation, as Griffiths once observed. Local differential geometry calls for such
  works. But good local theorems are difficult to come by. The problem
  on maximal rank webs discussed above\footnote{He is referring to the classification  of webs of maximal rank.}  is clearly an important
  problem, and will receive my attention. \vspace{0.15cm}\\
${}^{}$\hspace{0.15cm}My mathematical education
  goes on.}
\end{quote}

In 1978, he resumed working on webs of maximal rank jointly with Griffiths. In the long paper \cite{Jbr}, they set about demonstrating that a $k$-web of codimension one and of maximal rank $\pi(n,k)$
is algebraizable when $n>2$ and $k\geq 2n$. Their proof is not complete (cf. \cite{Jbr2}) and it is necessary to make an extra non-natural  assumption  to ensure the validity of the result. They also got a sharp bound for the rank of webs of codimension two in \cite{CGpisa}.
 Griffiths's interest on the subject probably came from the links between web geometry and algebraization results like the converse of Abel's Theorem discussed
 in Chapter \ref{Chapter:4}. Although  he  published no other paper on the subject, he kept being interested in webs since he discussed them in the opening lecture he gave for the bicentennial of Abel's birth in Oslo (transcribed in \cite{ABEL}). \medskip

Although it contains a non-trivial mistake, the paper \cite{Jbr} has been quite influential in web geometry. It has popularized the subject  and led the Russian school to pay attention to the notions of abelian relations and rank. It is probably from  \cite{Jbr}
that Tr\'{e}preau has taken up Bol's method to obtain a proof of the result originally aimed at by Chern and Griffiths. The present book would not exist if  \cite{Jbr} had not been
written. The readers should read it,  as it contains a masterfully written introductory part putting  things in perspective, and offers  different
proofs of many of the results included here.

\bigskip

\subsection*{Recent developments (since 1980)}

A number of new results in web geometry have been obtained in the last twenty years.
Here only, and certainly not all, results related to rank, abelian relations and maximal rank webs will be mentioned.  \medskip

The abelian relations of Bol's web all come (after analytic prolongation) from its dilogarithmic abelian relation,
which thus appears more fundamental than the other relations. In 1982, in \cite{gelfandmacpherson}, Gelfand and MacPherson found a geometric interpretation of this relation. In it
 Bol's web appears as  defined on the space of projective configurations of 5 points of $\mathbb R\mathbb P^2$. In \cite{damiano}, Damiano considers, for $n\geq 2$  a curvilinear $(n+3)$-web $D_n$
naturally defined on the space of projective configurations of $n+3$  points in $\mathbb R\mathbb P^n$.  He shows that this web is of maximal rank and gives a geometric interpretation of the ``main abelian relation'' of $D_n$, thus obtaining a family of exceptional webs which generalizes Bol's web.

\medskip

From 1980 to 2000, Goldberg studied the webs of codimension strictly bigger than one from the point of view of their rank. He obtained many  results, most of which are expounded in \cite{Goldbergbook}. More recently, he started studying planar webs under the same viewpoint  in collaboration with Lychagin.

\medskip

At the beginning of the 1990s, H\'{e}naut started studying webs in the complex analytical realm. He published about 15 papers on the matter. His research is
mainly about  rank and abelian relations, and  is concerned with  webs of arbitrary codimension  as well as  planar webs.\footnote{For an outline of the results  he obtained before 2000, see  \cite{Henaut0}.} The papers
\cite{Henaut1993,Henaut1994,Henaut2004} have to be mentioned,  as related with the topic of this book.
At the time when he started working, the field attracted  little attention. Without any doubt his tenacity played a major role to popularize web geometry in France, and in other countries as well. With Nakai he co-organized the conference {\it G\'{e}om\'{e}trie des tissus et \'{e}quations diff\'{e}rentielles}, held at the CIRM in 2003, which was attended by researchers from all over the world and was for some mathematicians (particularly for  the first author of this book) an opportunity to have
their first contact with  web geometry.

\smallskip

In 2001, the second author \cite{PTese,Piriopoly} and Robert \cite{Robert}
 independently showed that the Spence-Kummer 9-web associate to the trilogarithm  is an example of exceptional  web. Within a short time were published several papers \cite{crasluc,PT,MPP,CDQL} bringing to light a myriad  of exceptional webs.  \medskip

\smallskip

In 2005,   Tr\'{e}preau provided a proof of the result which Chern and Griffiths aimed at in \cite{Jbr}, {\it i.e.} the algebraization  of maximal rank $k$-webs on  $(\mathbb C^n,0)$, when  $k\geq 2n$ and $n>2$.

\smallskip

It seems to us that nowadays the study of webs is undergoing a revival, as testifies  the Bourbaki seminar \cite{Bou} devoted to the results hitherto mentioned. Mathematicians with the most diverse backgrounds now publish papers on the matter. A few recent articles, but not all of them, are mentioned in this book. The
readers are invited to consult the  literature in order to get a better acquaintance with the advancement of the researches.

\newpage

\section*{Contents of the chapters}

The table of contents tells rather precisely what the book is about. The following descriptions give additional information.

\bigskip

\defi[Chapter \ref{Chapter:intro}] is  introductory, and describes the basic notions of web geometry.
The content of this chapter is for the main part quite  well known, except for the notion of
duality for global webs on projective spaces $\mathbb P^n$, which appears to be new when $n>2$.
A short survey of this notion is presented in Section \ref{subsection:Projective duality}. The first
 two sections, more specifically Section \ref{S:basicdef} and Section \ref{S:3webs},
are of rather elementary nature and might  be read by an undergraduate student.

\medskip

\defi[Chapter \ref{Chapter:AR}] is about the notions of abelian relation and rank. It offers an outline of Abel's method to determine the abelian relations of a given planar web. It also gives a description of the abelian relations of planar webs admitting an infinitesimal symmetry. The most important results in this chapter are Chern's bound on the rank (Theorem \ref{T:cota})  and the normal form for the  conormals of a web of maximal rank (Proposition \ref{P:normal}). This last result is demonstrated through a geometric approach based on classical concepts and results  from projective algebraic  geometry  which are described in detail.

\medskip

\defi[Chapter \ref{Chapter:3}] is devoted to Abel's notorious addition Theorem. It first deals with the case of smooth projective curves, then tackles the general case after introducing the notion of abelian differentials.
Section \ref{S:CC} gives a rather precise description of the Castelnuovo curves, hence of some algebraic webs of maximal rank.
Section \ref{S:beyond} expounds new results: an (easy) variant of Abel's Theorem (Proposition \ref{P:wabel}), which is combined with Chern's bound on the rank so as to obtain  bounds on the genus of curves included in abelian varieties
({\it cf.} Theorem \ref{T:abelabel}).

 \medskip

\defi[Chapter~\ref{Chapter:4}] is where the converse  to Abel's Theorem is demonstrated. Its proof is given through
a reduction to the plane case which is then treated  using
 a classical argument that can be traced back to Darboux.  Then  a presentation of some algebraization results
 follows. Important concepts as Poincar\'{e}'s and  canonical maps for webs are discussed in this chapter.
 Our only contribution is of formal nature and is situated in Section \ref{S:AIdual*},
 where we endeavor to work as intrinsically as possible.

 \medskip

\defi[Chapter \ref{Chapter:Trepreau}] is entirely devoted to
Tr\'{e}preau's algebraization result.   The proof that is  presented is essentially the same as the original one \cite{Trepreau}.
The only ``novelty'' in this chapter is Section
\ref{S:CC5}, where  a geometric interpretation of the proof is given.
As in the preceding chapter, an effort was made to formulate some of the results and theirs proofs as intrinsically as possible.

\medskip

\defi[Chapter \ref{Chapter:6}]
takes up the case of planar webs of maximal rank,  more specifically of exceptional planar webs.
Classical criteria which characterize linearizable webs on the one hand, and maximal rank webs on the other hand are explained.
Then the existence of exceptional planar $k$-webs,  for arbitrary $k \ge 5$, is established through the study of webs admitting infitesimal
automorphisms. The classification of the so called CDQL webs on compact complex surfaces obtained  recently by the authors is  also reviewed.
The chapter  ends with a brief discussion about all the  examples of  planar exceptional webs we are aware of.

%\newpage

\bigskip

\section*{How to use this book}
%There are numerous ways to use this book. Without any effort  a few examples come to mind:  feed up a fire, prop up a rickety table,
%impress friends,  or even study  its mathematical content. For those most interest in the latter, a few lines
%are dropped below.
%\medskip
The logical organization of this book is rather simple: the readers with enough time to  spare can read it from cover
to cover. 
\smallskip 

Those mostly interest in  Bol-Tr\'{e}preau's algebraization Theorem,  may find useful the  graph below
which suggests  a minimal route towards it.
\medskip

\begin{tabular}{c}
\qquad 
\setcounter{diagheight}{50}
%\begin{center}
\begin{chart}
\reqhalfcourse 25,45:{Chapter 1}{Basic definitions}{Section 1.1}
%\reqhalfcourse 35,45:{Section 1.1}{Basic definitions}{}
\reqhalfcourse 25,35:{Chapter 2}{Abelian relations}{}
   \prereq 25,45,25,35:
\reqhalfcourse 15,25:{Chapter 3}{Abel's  theorem}{Sections 3.1, 3.2 }
   \prereq 25,35,15,25:
\reqhalfcourse 35,25:{Chapter 3}{Algebraic webs of \\ maximal rank}{Section 3.3}
   \prereq 25,35,35,25:
\reqhalfcourse 15,15:{Chapter 4}{A converse of \\ Abel's theorem}{Sections 4.1, 4.2}
   \prereq 15,25,15,15:
\reqhalfcourse 40,15:{Chapter 4}{Algebraization of \\ smooth $2n$-webs}{Section 4.3}
   \prereq 15,25,15,15:
   \prereq 15,15,40,15:
\reqhalfcourse 25,05:{Chapter 5}{Algebraization of \\ codimension one webs}{}
   \prereq 15,15,25,05:
   \prereq 40,15,25,05:
\prereq 35,25,25,05:
\end{chart}
%\end{center}
\end{tabular}
\newpage

Those  anxious to learn  more about  exceptional webs
might prefer to use instead  the following graph  as a reading guide.

\bigskip

\begin{tabular}{c}
\qquad 
\setcounter{diagheight}{50}
\begin{chart}
\reqhalfcourse 15,45:{Chapter 1}{Basic Definitions}{Section 1.1}
\reqhalfcourse 35,45:{Chapter 1}{Planar $3$-webs}{Section 1.2}
\prereq 15,45,35,45:
\prereq 35,45,25,05:
%\prereq 15,45,25,25:
\reqhalfcourse 10,35:{Chapter 2}{Determining the \\ abelian relations}{Section 2.1}
\prereq 15,45,10,35:
\reqhalfcourse 40,35:{Chapter 2}{Bounds for \\ the rank}{Section 2.2}
   \prereq 15,45,40,35:
%\prereq 40,35,25,25:
\reqhalfcourse 25,25:{Chapter 3}{Abel's  theorem}{Sections 3.1, 3.2 }
\reqhalfcourse 12,15:{Chapter 4}{A converse of \\ Abel's theorem}{Sections 4.1, 4.2}
   \prereq 25,25,12,15:
\reqhalfcourse 38,15:{Chapter 4}{Algebraization of \\ smooth $2n$-webs}{Section 4.3}
   \prereq 25,25,38,15:
   \prereq 12,15,38,15:
\reqhalfcourse 25,05:{Chapter 6}{Exceptional webs}{}
   \prereq 25,25,25,05:
   \prereq 40,15,25,05:
\prereq 40,35,25,05:
\prereq 10,35,25,05:
\end{chart}
\end{tabular}

\renewcommand{\thefootnote}{\fnsymbol{footnote}}

\chapter*{Conventions}
\markboth{\MakeUppercase{Conventions}}{\MakeUppercase{Conventions}}
\addcontentsline{toc}{chapter}{Conventions}
\thispagestyle{empty}

All the \defi[definitions]\index{Definition}, including this one,   are presented in bold case and have a corresponding entry
at the remissive index.

\medskip

Unless stated otherwise all the geometric entities like curves, surfaces,  varieties and manifolds considered
in this text are reduced and complex holomorphic. Curves, surfaces and varieties may be singular, and may have
several irreducible components. The manifolds are smooth connected  varieties.

Web geometry lies on the interface of local differential geometry and projective algebraic geometry.
Throughout the text, the reader will be confronted with both local non-algebraic subvarieties of
the projective space as well as with global, and hence algebraic and compact, projective subvarieties. A
 projective curve, surface, variety, or manifold will mean a compact curve, surface, variety, or manifold contained in some  projective space. Beware
that some authors use the term projective to qualify any subvariety, compact or not, algebraic or not, of a given projective space.

\medskip

Throughout there will be references to points $x \in (\mathbb C^n,0)$ and properties
of germs at the point $x$. The point $x$ has to be understood as a point at a sufficiently small neighborhood of the origin  and the
property as a property of some representative of the germ defined in this very same sufficiently small neighborhood.

\medskip

If $n$ is a positive integer, $\underline n$ will stand for the set $\{ 1, \ldots, n\}$.
For any $q\in \mathbb N$,  $\mathbb C_q[x_1, \ldots, x_n]$ will stand for the vector space of degree $q$ homogeneous polynomials in $x_1, \ldots, x_n$.
The span of a subset $S$ of a projective space or of a vector space will be denoted by $\langle S \rangle$.

%%%%%%%%%%%%%%%%%%%%%%%%%%%%%%%%%%%%%%%
%  Introduction
%  label = Chapter:intro
%  First version by JVP
%  last modification: 23/mars/2009
%  Remarks:
%%%%%%%%%%%%%%%%%%%%%%%%%%%%%%%%%%%%%%%

\chapter{Local and global webs}\label{Chapter:intro}

In its classical form web geometry studies  local configurations of finitely many smooth foliations in
general position. In Section \ref{S:basicdef}  the basic definitions of our subject are laid down and the  algebraic webs
are introduced. These are among the most important examples  of the whole theory.

\medskip

Germs of webs defined by few foliations in  general position are far from being  interesting. Basic results
from  differential calculus imply that the theory is locally trivial.
As soon as the number of foliations surpasses the dimension of the ambient manifold
this is no longer true. The discovery in the last years of the 1920 decade of the
curvature  for $3$-webs on surfaces is considered as the birth of web geometry. In Section \ref{S:3webs}
 this curvature form is discussed and an early emblematic result of theory that characterizes
its vanishing is presented.

\medskip

Although the emphasis of the theory is local the most emblematic examples are indeed globally
defined on  projective manifolds. In Section \ref{S:fancydef}   the  basic definitions are extended to encompass both germs of singular
as well as global webs. Certainly more demanding  than the previous sections,  Section \ref{S:fancydef} should be
read in parallel with Section \ref{S:examples} where the algebraic webs are revisited from a global viewpoint and
 is discussed how one can associated webs to linear systems on surfaces.

\section{Basic definitions}\label{S:basicdef}

\subsection{Germs of smooth webs}\label{S:gsw}

A \defi[germ of smooth codimension one  $k$-web]\index{Web!smooth} $$\mathcal
W=\mathcal F_1\boxtimes
\cdots \boxtimes \mathcal F_k$$ on $(\mathbb
C^n,0)$ %\index[not]{$\boxtimes$}
is a collection of $k$ germs of smooth codimension one
holomorphic foliations such that their tangent spaces at the origin
are in \defi[general position], that is,  for any number $m$ of these foliations, $m\le
n$, the corresponding  tangent spaces at the origin have  intersection of codimension $m$.

Usually  the foliations $\mathcal F_i$ are presented by germs of holomorphic $1$-forms
$\omega_i \in \Omega^1{(\mathbb C^n,0)}$, non-zero
at $0 \in \mathbb C^n$ and satisfying Frobenius integrability condition $\omega_i \wedge d \omega_i=0$.
To present a germ of smooth web and keep track of  its defining $1$-forms two alternative notations will
be used: $\mathcal W=\mathcal W(\omega_1 \cdot \omega_2 \cdot \cdots \cdot \omega_k)$ or $\mathcal W=\mathcal W(\omega_1, \ldots, \omega_k)$.
While the latter is self-explanatory the former presents $\mathcal W$ as
an object defined by an element of $\mathrm{Sym}^k \Omega^1{(\mathbb C^n,0)}$. Notice that the general position assumption
translates into
\[
\big(\omega_{i_1} \wedge \cdots \wedge \omega_{i_m}\big) (0) \neq 0
\]
where $\{ i_1, \ldots, i_m \}$ is any subset of $\underline k$ of cardinality ${m\le \min\{k,n\}}$.

Since the foliations $\mathcal F_i$ are smooth they can be defined by level sets of submersions $u_i: (\mathbb C^n,0) \to \mathbb C$.
When profitable to present the web in terms of its defining  submersions   $\mathcal W=\mathcal W(u_1, \ldots, u_k)$
will be used.

\smallskip

The germs of \defi[quasi-smooth] \index{Web!quasi-smooth} webs on $(\mathbb C^n,0)$ are defined by
replacing the general position hypothesis on the tangent
spaces at zero by the weaker condition of pairwise tranversality. Explicitly, a germ of quasi-smooth $k$-web $\mathcal W=\mathcal F_1
\boxtimes \cdots \boxtimes \mathcal F_k$  on $(\mathbb C^n,0)$ is a collection of smooth foliations such that $T_0 \mathcal F_i \neq T_0 \mathcal
F_j$ whenever $i$ and $j$ are distinct elements of $\underline k$.

\medskip

There are similar definitions for webs of arbitrary ( and even
mixed ) codimensions.  Although extremely rich,  the  theory of webs of arbitrary codimension will not be discussed in
this book.

It is also interesting to study webs in different categories. For instance one can paraphrase the definitions
above to obtain differentiable, formal, algebraic, \ldots webs. This text, unless stated otherwise,  will stick
to the holomorphic category.

\subsection{Equivalence and first examples}\label{S:eq1st}
\index{Web!equivalence}
Local web geometry is ultimately interested in the classification of germs of smooth webs up to
the natural action of $\mathrm{Diff}(\mathbb C^n,0)$ -- the group of germs of biholomorphisms of $(\mathbb C^n,0)$.
If $\varphi \in \mathrm{Diff}(\mathbb C^n,0)$ is a germ of biholomorphism then the natural action just referred to is given by
\[
\varphi^*  \mathcal W(\omega_1 \cdots \omega_k)   = \mathcal W(\varphi^*( \omega_1  \cdots \omega_k ) ) \, .
\]

\smallskip

The germs of $k$-webs $\mathcal W(\omega_1 \cdots \omega_k) $ and $\mathcal W'(\omega_1' \cdots \omega_k')$ will be considered
\defi[biholomorphically
equivalent] if
\[
\varphi^* (\omega_1 \cdots \omega_k) = u \, \cdot \left( \omega_1' \cdots \omega_k' \right) \,
\]
for some  germ of biholomorphism $\varphi$ and some germ of invertible function $u \in \mathcal O_{(\mathbb C^n,0)}^*$.
In other words, there exists a permutation $\sigma \in \mathfrak{S}_k$ -- the symmetric group on $k$ elements -- \,  such that the germs of
2-forms $\varphi^* \omega_i \wedge \omega_{\sigma(i)}'$ are identically zero for
every $i \in \underline k$.

\begin{figure}[ht]
\begin{center}\includegraphics[height=3.0cm,width=4.5cm]{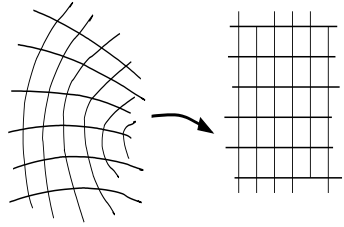}\end{center}
\caption{There is only one smooth $2$-web.}
\end{figure}

\smallskip

Clearly the bihomolorphic equivalence defines an equivalence relation on the set of smooth $k$-webs on $(\mathbb C^n,0)$.
When the dimension of the space is greater than or equal to the number of defining foliations, that is when $n \ge k$, there
is just one equivalence class. Indeed, if one considers a smooth $k$-web defined by $k$ submersions $u_i:(\mathbb C^n,0)  \to
\mathbb C$ then the map $U : (\mathbb C^n,0) \to \mathbb C^k$, $U = ( u_1, \ldots u_k)$ is a submersion thanks to the general
position hypothesis. The constant rank Theorem ensures the existence of  a biholomorphism $\varphi : (\mathbb C^n,0) \to (\mathbb C^n,0)$
taking  the function $u_i$ to the coordinate function $x_i$ for every $i \in \underline k$. Symbolically,  $\varphi^* u_i = x_i$.

\medskip

When the number of defining foliations exceeds the dimension of the space by at least two ($ k \ge n +2$)  then one can
see the existence of a multitude of equivalence class by the following considerations.

For a $k$-web $\mathcal W= \mathcal F_1 \boxtimes \cdots \boxtimes \mathcal F_k$,
the tangent spaces of the foliations $\mathcal F_i$  at the origin determine a collection of
$k$ unordered points in $\mathbb P T^*_0 (\mathbb C^n,0) = \mathbb P^{n-1}$. The set of isomorphism classes of
$k$ unordered points in general position in a projective space $\mathbb P^{n-1}$  is the quotient
of the open subset  $U$ of $(\mathbb P^{n-1})^k$  parametrizing  $k$ distinct points in general position by the action
\[
\big( (\sigma , g) , (x_1, \ldots,x_k )\big)  \mapsto  \left( g ( x_{\sigma(1)}) , \ldots, g( x_{\sigma(k)} ) \right)
\]
of the group $G=\mathfrak S_k \times \mathrm{PGL}(n, \mathbb C)$.

When $k \le n+1$ the action of $G$ on $U$ is transitive and
 there is exactly one isomorphism class. When $k \ge n+2$ the action is locally free (the stabilizer of
any point  in $U$ is finite)    and in particular
the  set of isomorphism classes of $k$ unordered points in $\mathbb P^{n-1}$ has  dimension $(k- n - 1)(n-1)$.

If $\mathcal W$ and $\mathcal W'= \varphi^* \mathcal W$ are two biholomorphically equivalent $k$-webs on $(\mathbb C^n,0)$ then
their tangent spaces at the origin  determine two sets of $k$ points on $\mathbb P^{n-1}$ which are
isomorphic through $[d\varphi(0)]$, the projective automorphism determined by the projectivization of the linear map $d\varphi(0)$.
It is then clear that for $k \ge n+2$ there are many non equivalent germs of smooth $k$-webs on $(\mathbb C^n,0)$.

\dd

It is tempting to  infer from the discussion  above that there is only one  equivalence class of smooth $(n+1)$-webs  on $(\mathbb C^n,0)$
using the following fallacious argument: (a) to a $(n+1)$-webs on $(\mathbb C^n,0)$ one can associate $n+1$ sections of $\mathbb  P T^* (\mathbb C^n,0)$;
(b) since there is only one isomorphism class of unordered $(n+1)$ points in general position in $\mathbb P^{n-1}$  these sections can be send, through
an biholomorphism of $\mathbb P T^* (\mathbb C^n,0)$, to the constant sections $[dx_1], \ldots, [dx_n], [dx_1 + \cdots + dx_n]$; (c) therefore (a) and (b)
implies that every smooth $(n+1)$-web is equivalent to the web $\mathcal W(dx_1,\ldots, dx_n, dx_1 + \cdots + dx_n)$.

While (a) and (b) are sound,  the conclusion (c)  is completely unjustified. The point is that the automorphism used in (b) is not necessarily
induced by a biholomorphism $\varphi \in \mathrm{Diff}(\mathbb C^n,0)$. To wit,  every biholomorphism
$\Phi: \mathbb P T^* (\mathbb C^n,0) \to \mathbb P T^* (\mathbb C^n,0)$ can be written in the form
\[
\Phi(x,v) = \big( \varphi(x), [A(x) \cdot v ] \big)
\]
where $\varphi \in \mathrm{Diff}(\mathbb C^n,0)$ and $A \in \mathrm{GL}(n, \mathcal O_{(\mathbb C^n,0)})$. But for
very few of them $[A(x) \cdot v] = [d \varphi (x) \cdot v]$.

It will be shown in  Section \ref{S:3webs} that not every $3$-web on $(\mathbb C^2,0)$ is equivalent to the parallel $3$-web
$\mathcal W(dx,dy , dx  + dy)$.

\subsection{Algebraic webs}\label{S:defalg}
\index{Web!algebraic|(}
Given a projective  curve $C\subset \mathbb P^n$ of degree $d$ and a hyperplane $H_0 \in \check{\mathbb P}^n$ intersecting
$C$ transversely, there is a natural germ of quasi-smooth $d$-web $\mathcal W_C(H_0)$ on $(\check{\mathbb P}^n, H_0)$ defined by the submersions
${p_1, \ldots, p_d : (\check{\mathbb P}^n,H_0) \to C}$ which describe the intersections of $H \in (\check{\mathbb P}^n, H_0)$ with $C$.
Explicitly, if one writes the restriction of $C$ to a sufficiently small neighborhood of $H_0 \subset \mathbb P^n$ as $C_1 \cup \cdots \cup C_d$,
where the curves $C_i$ are pairwise disjoint curves, then the functions $p_i$ are defined as $p_i(H) = H \cap C_i$. The corresponding
$d$-web is $\mathcal W_C(H_0) = \mathcal W(p_1, \ldots, p_d)$.
The $d$-webs of the form $\mathcal W_C(H_0)$ for some reduced projective curve $C$ and some transverse hyperplane
$H_0$ are classically called \defi[algebraic $d$-webs].

From the definition of $p_i$ it is clear that the inclusion
\[
p_i^{-1}(p_i(H)) \subset \{ H' \in \check{\mathbb P}^n \, \vert \, p_i(H) \in H'   \}
\]
holds true for every $H \in (\check{\mathbb P}^n,H_0)$ and every $i \in \underline d$. In other words the  fiber of $p_i$ through a point $H \in (\check{\mathbb P^n},H_0)$
is contained in the  set of hyperplanes
that contain the point $p_i(H) \in C_i \subset C \subset \mathbb P^n$. Consequently the fibers of the submersions $p_i$ are
(pieces of) hyperplanes.

It is clear from the definition of $\mathcal W_C ( H_0)$ that when $C$ is a reducible curve with irreducible components
$C_1, \ldots, C_m$  then
\[
\mathcal W_C( H_0)  = \mathcal W_{C_1}(H_0) \boxtimes \cdots \boxtimes \mathcal W_{C_m}(H_0) \, .
\]

%\smallskip
\begin{figure}[H]
\begin{center}
\begin{tabular}{cc}
\includegraphics[width=4.3cm,height=4.3cm]{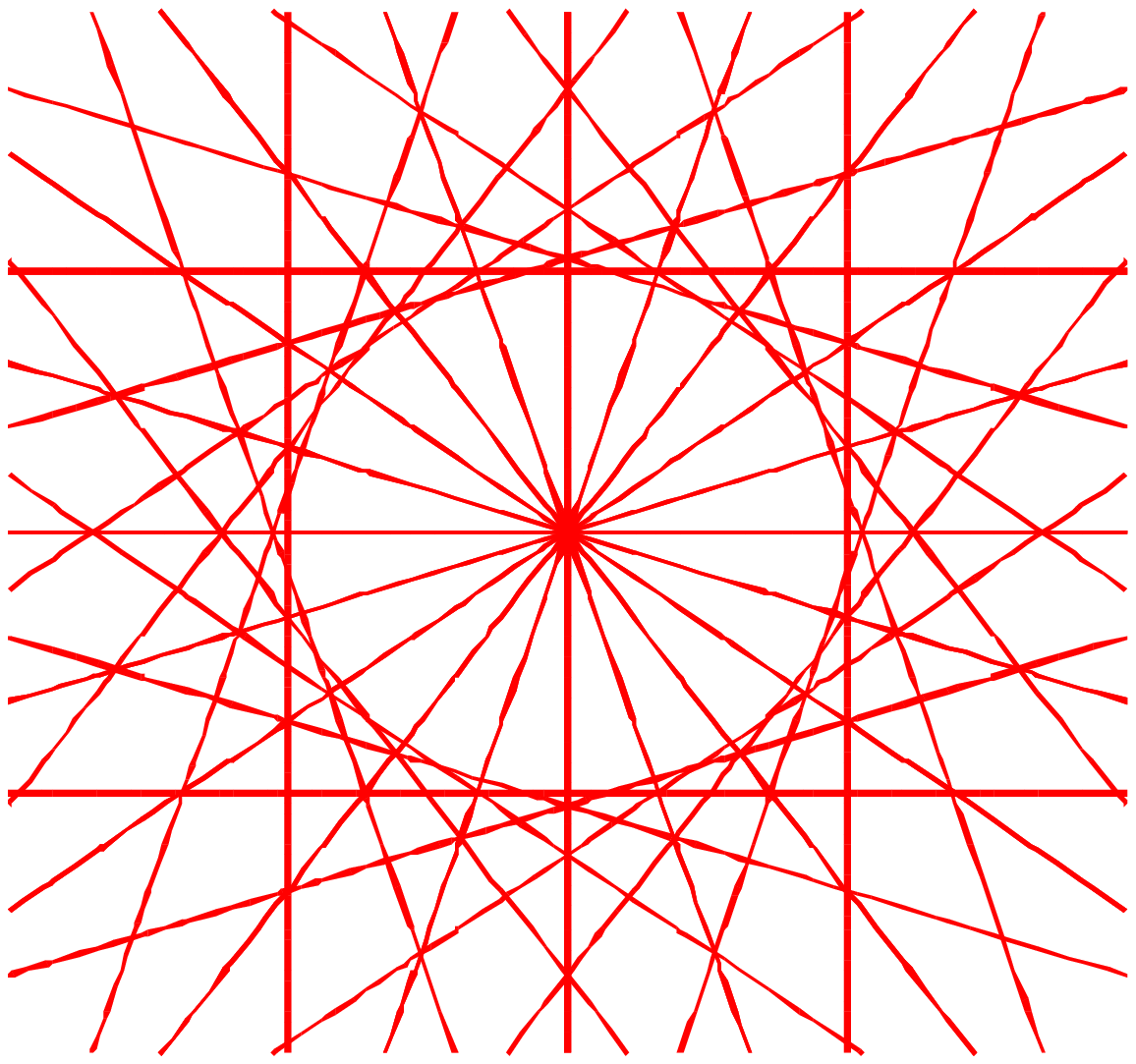}
\qquad &\qquad 
\includegraphics[width=4.3cm,height=4.3cm]{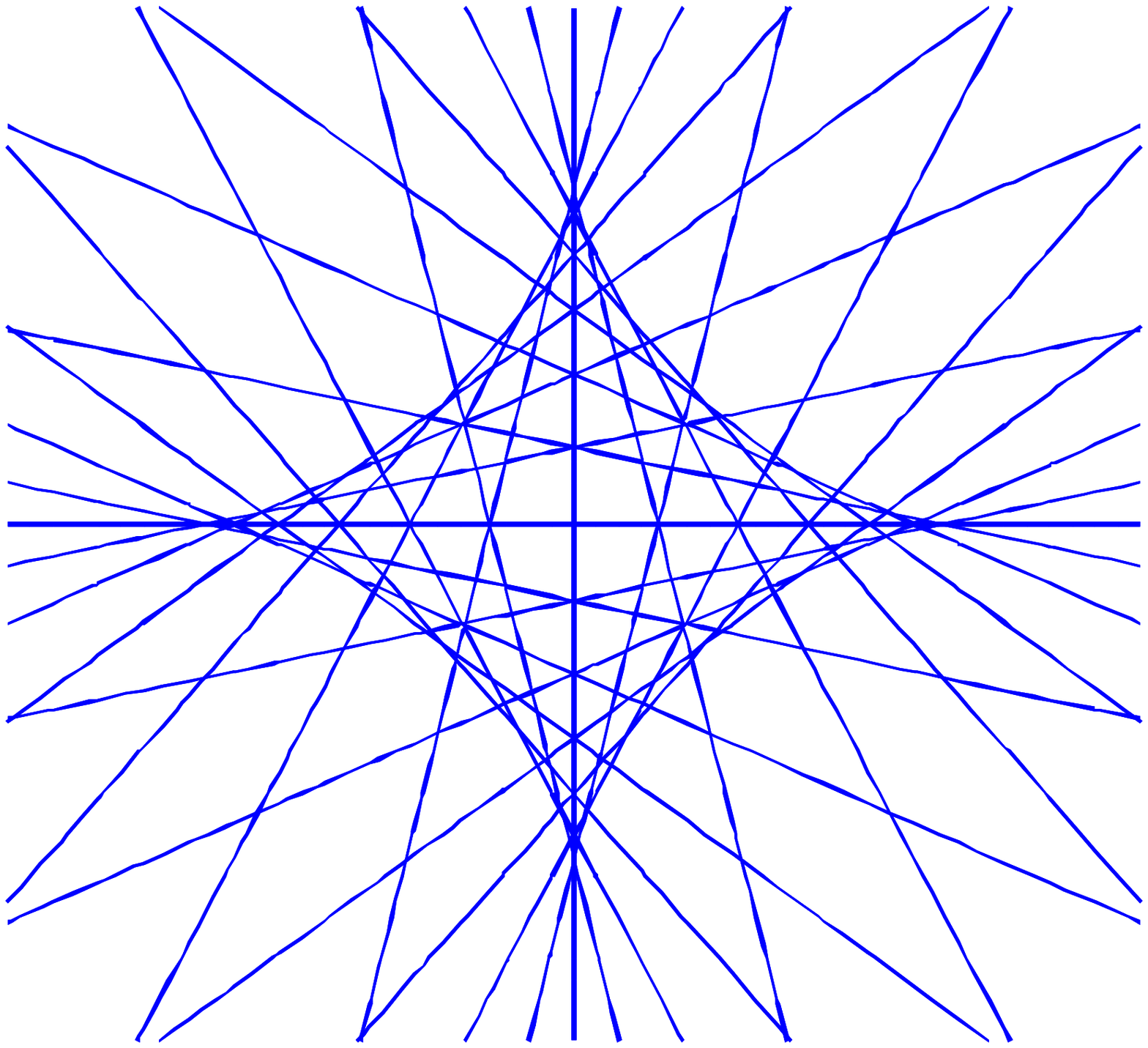}
\end{tabular}
\caption[Two algebraic webs] { On the left $\mathcal W_C$
is pictured for  a planar reduced cubic curve $C$ formed by a line and a
conic. On the right  $\mathcal W_C$ is drawn for a planar rational quartic
 $C$.}
 \end{center}
\end{figure}

Moreover, it has not  been really used that $C$ is a projective curve.
Indeed, if it is agreed to define \defi[linear webs] \index{Web!linear} as  the ones  for which all leaves are (pieces of) hyperplanes
then  the construction just presented
establishes an equivalence between linear quasi-smooth $k$-webs on $(\check{\mathbb P}^n,H_0)$,
 and $k$ germs
of curves in $\mathbb P^n$ intersecting $H_0$ transversely in $k$ distinct points.

\begin{figure}[ht]
\label{F:AbelInverseTheorem00}
\begin{center}
\psfrag{P}[][]{$ \mathbb P^2$}
\psfrag{Pd}[][]{$ \check{\mathbb P}^2$}
\psfrag{F1}[][]{\textcolor{pipo}{$ \mathcal F_1 $}}
\psfrag{CF1}[][]{
\textcolor{pipo}{$ C_{1} $}}
\psfrag{F2}[][]{\textcolor{mag}{$ \mathcal F_2 $}}
\psfrag{CF2}[][]{
\textcolor{mag}{$ C_{2} $}}
\psfrag{F3}[][]{\textcolor{blo}{$ \mathcal F_3 $}}
\psfrag{CF3}[][]{
\textcolor{blo}{$ C_{3} $}}
\psfrag{produ}[][]{\begin{tabular}{c}
 projective \\duality $\mathcal D$\end{tabular}}
\resizebox{2.9in}{1.5in}{
\includegraphics{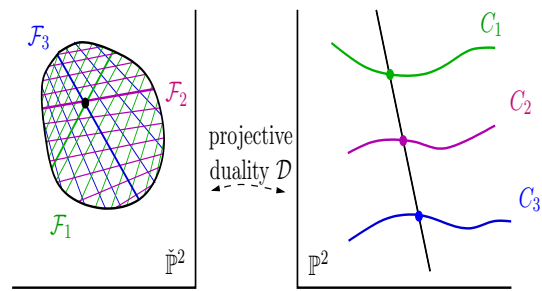} }
\end{center}
\caption[Projective duality]{Projective duality for a  linear 3-web  in dimension two.}
\end{figure}

Back to the case where $C$ is projective, if no irreducible component of $C$ is a line then
there is the following alternative description of  $\mathcal W_C(H_0)$.
Let $\check C$ be the dual hypersurface of $C$, that is, $\check C \subset \check{ \mathbb P^n}$ is the closure of the union
of hyperplanes $H \in \check{\mathbb P}^n$ containing a tangent line of $C$ at some smooth point $p \in C_{{sm}}$. Symbolically,
\[
\check C = \overline{\bigcup_{p \in C_{sm}} \bigcup_{\substack{   H \in \check{\mathbb P}^n     \\  T_p C \subset H }}\, H \,} \, .
\]
The leaves of $\mathcal W_C(H_0)$ through $H_0$ are the hyperplanes passing through it and tangent to $\check C$ at some point $p \in \check C$.

A similar interpretation holds true when $C$ does contain  lines among its irreducible components. The differences are: the dual of
a line is no longer a hypersurface but a $\mathbb P^{n-2}$ linearly embedded in $\check{ \mathbb P}^n$; and  the leaf of a $1$-web dual to
a line, through a point $H_0 \in \check{\mathbb P}^n$ is the hyperplane in $\check{ \mathbb P}^n$ containing both $H_0$ and the dual $\mathbb P^{n-2}$.
\index{Web!algebraic|)}

\section{Planar $3$-webs}\label{S:3webs}

This section  presents one of the founding stone of web geometry: the characterization
of hexagonal planar $3$-webs through  their holonomy and curvature.  The exposition here follows closely \cite{Nakai1}.
The hexagonality  of algebraic $3$-webs is also worked out in detail, see Section \ref{S:algcub}. For a more leisure  account the
reader can consult {\cite[Lecture 18]{taba}}.

\subsection{Holonomy and hexagonal webs}
\index{Web!holonomy|(}
Let $\mathcal W= \mathcal F_1 \boxtimes \mathcal F_2 \boxtimes \mathcal F_3$ be a germ of smooth
$3$-web on $(\mathbb C^2,0)$. Denote by $L_1, L_2, L_3$ the leaves through $0$ of $\mathcal F_1,\mathcal F_2,\mathcal F_3$
respectively.

\smallskip

If $x=x_1 \in L_1$ is a point sufficiently close to the origin then, thanks to the persistence of transversal intersections
under small deformations,  the leaf of $\mathcal F_3$ through it intersects $L_2$ in a unique  point $x_2$ close to the origin. Moreover
the map that associates to $x=x_1 \in L_1$ the point $x_2 \in L_2$ is a germ of biholomorphic  map $h_{12}:(L_1,0) \to (L_2,0)$.

Analogously there exists a bihomolorphism $h_{23}:(L_2,0)\to (L_3,0)$ that associates to $x_2\in L_2$
the point $x_3 \in L_3$ defined by the  intersection  of $L_3$ with the leaf of $\mathcal F_1$ through $x_2 \in L_2$.

Proceeding in this way one can construct a sequence of points $x_1 \in L_1$, $x_2 \in L_2$, $x_3 \in L_3$, $x_4 \in L_1$, $x_5 \in L_2$,
$x_6 \in L_3$, $x_7 \in L_1$. The function  that associates to the initial point $x=x_1$ the end
point $x_7$ is the germ of bihomolomorphism $h:(L_1,0) \to (L_1,0)$ given by the composition
\[
h_{31} \circ h_{23} \circ h_{12} \circ  h_{31} \circ h_{23} \circ h_{12} .
\]

The reader is invited to verify the following properties of the biholomorphism $h$.
\begin{enumerate}
\item[(a)] If one does the same construction but with the  roles of the foliations $\mathcal F_1, \mathcal F_2, \mathcal F_3$
replaced by the foliations $\mathcal F_{\sigma(1)},\mathcal F_{\sigma(2)},\mathcal F_{\sigma(3)}$ -- $\sigma$ being a permutation
of $\{1,2,3\}$ --   then the resulting biholomorphism is conjugated to $h^{\mathrm{sign}(\sigma)}$;
\item[(b)] If $\varphi:(\mathbb C^2,0) \to (\mathbb C^2,0)$ is a biholomorphism and $\overline{ \mathcal W}= \varphi^* \mathcal W$ then
the corresponding biholomorphism $\overline h:(\overline L_1,0)\to (\overline L_1,0)$ for the leaf $\overline{L}_1=\varphi^{-1}(L_1)$ of $\overline{ \mathcal F_1} =  \varphi^* \mathcal F_1$
is equal to $\varphi^{-1} \circ h \circ \varphi$.
\end{enumerate}

It follows from the two properties above that the conjugacy class in $\mathrm{Diff}(\mathbb C,0)$ of the
 group generated by $h$ is intrinsically attached to $\mathcal W$.
This class is by definition the \defi[holonomy of $\mathcal W$ at $0$]. It will  be convenient to say that
$h$ is the holonomy of $\mathcal W$ at  $0$ instead of repeatedly refereing to the conjugacy
 class of the group generated by it. Hopefully no confusion will arise from this abuse of terminology.

\begin{figure}[ht]
\begin{center}
\psfrag{L1}[][]{$ L_1 $}
\psfrag{L2}[][]{$ L_2 $}
\psfrag{L3}[][]{$ L_3 $}
\includegraphics[height=4.2cm,width=4.5cm]{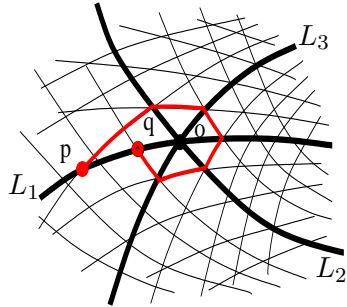}\end{center}
\caption{The holonomy of a planar $3$-web \label{F:hexagonal}}
\end{figure}

To get a better grasp of the definition of the holonomy of $\mathcal W$ and to prepare
the ground for what is to come, a family of examples parametrized by $\mathbb C$ is  presented below
in the form of a lemma.

\begin{lemma}\label{L:Wk}
If $k \in \mathbb C$ is a complex number and $\mathcal W_k = \mathcal W(x, y, x+y +xy(x-y)(k + \hot))$ then
 the holonomy of $\mathcal W_k$
is generated by a germ of biholomorphism $h_k : (\mathbb C,0) \to (\mathbb C,0)$
which  has as  first coefficients in its series expansion
\[
h_k(x) = x  + 4k x^3 + \hot \, .
\]
\end{lemma}
\begin{proof}
Let $x_1 = (x,0) \in (L_1,0)$. To compute $x_2$ notice that $f_k(x,y) = x+y +  xy(x-y)(k+\hot)$ is equal
to $x$ when evaluated on both $x_1 = (x,0)$ and $(0,x)$. In other words the leaf of
$\mathcal F_3$, the foliation determined by $f_k$, through $x_1$ cuts the leaf of $\mathcal F_2$,
the foliation determined by $y$, in $x_2= (0,x)$.

From the definition of $x_3$ it is clear that  its second coordinate is equal to $x$.
To determine its first coordinate one has to solve the implicit equation $f_k(t,x)=0$. A straight forward computation
yields
$t = -x -2kx^3 + \hot
$ and consequently $x_3 = (-x-2kx^3,x)$ up to higher order terms.

Proceeding in this way one finds
\begin{align*}
x_4&=\,(-x-2kx^3,0)  &&x_5=(0,-x-2kx^3) \\
x_6&=\, (x+4kx^3,-x-2kx^3)&& x_7 =(x + 4kx^3 , 0)
\end{align*}
up to higher order terms. The details are left to the reader.
\end{proof}

\medskip

In what concerns their  holonomy the simplest smooth $3$-webs on $(\mathbb C^2,0)$ are the \defi[hexagonal webs]. \index{Web!hexagonal}
 By definition these
are the ones which can be represented in a neighborhood $U$ of the origin by three pairwise transversal smooth foliations whose
germification at any point $x \in U$ is a germ of $3$-web with trivial holonomy. Using the convention about germs
spelled out at the end of  Section \ref{S:basicdef} the hexagonal webs on $(\mathbb C^2,0)$ are the ones with trivial holonomy
at every point
$x \in (\mathbb C^2,0)$. The guiding example  is
$\mathcal W(x,y,x+y)$, see Figure \ref{F:proofhex} below for a proof of its hexagonality.

\begin{figure}[ht]
\begin{center}\includegraphics[height=4cm,width=4cm,angle=90]{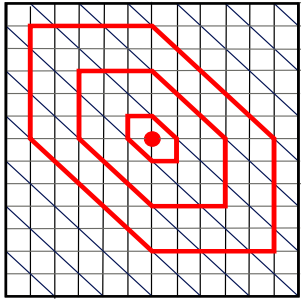}\end{center}
\caption{The web $\mathcal W(x,y,x+y)$ is hexagonal.} \label{F:proofhex}
\end{figure}

Beware that hexagonality  is much stronger  than asking the holonomy to be  trivial only at the origin. It is an
instructive exercise to produce an example of $3$-web having trivial holonomy at zero but non-trivial holonomy at a generic
$x  \in (\mathbb C^2,0)$.

\index{Web!holonomy|)}

\subsection{Curvature for planar $3$-webs}
\index{Web!curvature|(}
Suppose now that a $3$-web $\mathcal W$ on $(\mathbb C^2,0)$ is presented by its defining $1$-forms,
that is, $\mathcal W= \mathcal W(\omega_1, \omega_2, \omega_3)$.

\smallskip

There exist invertible functions $u_1, u_2, u_3  \in \mathcal O^*(\mathbb C^2,0)$ for which
\begin{equation}\label{E:normal11}
      u_1 \,\omega_1 + u_2 \,\omega_2 + u_3 \,\omega_3 = 0 \, .
\end{equation}
For instance, if $\delta_{ij}$ are the holomorphic functions defined by the
relation $\delta_{ij} dx \wedge dy = \omega_i \wedge \omega_j$ for $i,j = 1,2,3$ then
\[
\delta_{23}\, \omega_1 + \delta_{31}\, \omega_2 + \delta_{12}\, \omega_3 = 0 \, .
\]
Although the triple $(\delta_{23}, \delta_{31} , \delta_{12})$ is not the unique
solution to equation (\ref{E:normal11}), any other will differ from it by the multiplication
by an invertible  function. In other words, the most general solution of (\ref{E:normal11}) is
$(u_1,u_2,u_3) = u \cdot (\delta_{23}, \delta_{31} , \delta_{12})$ where $u \in \mathcal O^*(\mathbb C^2,0)$ is
 arbitrary.

\smallskip

\begin{lemma}\label{L:eta}
Let  $\alpha_1, \alpha_2, \alpha_3 \in \Omega^1(\mathbb C^2,0)$  be three $1$-forms with pairwise wedge product nowhere zero. If
$\alpha_1+ \alpha_2 + \alpha_3 =0$ then there exists a unique $1$-form $\eta$ such that
\[
d \alpha_i = \eta \wedge \alpha_i \, \quad \text{ for  every } i \in \{ 1, 2, 3 \} \, .
\]
\end{lemma}
\begin{proof}
Because the ambient space has dimension two, and the $1$-forms $\alpha_i$
are nowhere zero, there exist $1$-forms $\gamma_i$ satisfying
\[
 d\alpha_i  = \gamma_i \wedge \alpha_i.
\]
Notice that the $1$-forms $\gamma_i$ can be replaced by $\gamma_i + a_i \alpha_i$, with $a_i \in \mathcal O(\mathbb C^2,0)$ arbitrary, without changing
the identity above.

The difference  $\gamma_1 - \gamma_2$ is again a  $1$-form. As such, it  can be written
as $a_1 \alpha_1 + a_2 \alpha_2$ with $a_1, a_2 \in \mathcal O{(\mathbb C^2,0)}$. Therefore
\[
 \gamma_1 - a_1 \alpha_1 = \gamma_2 - a_2 \alpha_2 \, .
\]
If $\eta = \gamma_1 - a_1 \alpha_1 = \gamma_2 - a_2 \alpha_2 $ then  it clearly satisfies $d\alpha_1 = \eta \wedge \alpha_1$  and $d\alpha_2 = \eta \wedge \alpha_2$.
Moreover, since $\alpha_3 = - \alpha_1 - \alpha_2$, it also satisfies $d\alpha_3 = \eta \wedge \alpha_3$.
This establishes the existence of $\eta$. For the uniqueness, notice that two distinct solutions $\eta$ and $\eta'$ would verify $(\eta - \eta') \wedge \alpha _i= 0$ for
$i=1,2,3$. Any  two of these identities are sufficient to ensure that  $\eta= \eta'$.
\end{proof}

The lemma above applied  to $( \delta_{23} \omega_1 , \delta_{31} \omega_2 , \delta_{12} \omega_3)$ yields the existence of $\eta \in \Omega^1(\mathbb C^2,0)$
for which $d ( \delta_{jk}\omega_i )= \eta \wedge ( \delta_{jk}\omega_i )$ for any cyclic   permutation  $(i, j, k )$ of  $(1,2,3)$. Presenting $\mathcal W$ through
three others $1$-forms -- say $\omega_1'=a_1 \omega_1, \omega_2'=a_2 \omega_1$, and  $\omega_3'=a_3 \omega_3$ -- one sees that the corresponding $\eta'$ relates to $\eta$
through the  equation
\[
\eta - \eta' = d \log( a_1 a_2 a_3) \, .
\]
In particular the $1$-form $\eta$ depends on the presentation of $\mathcal W$ but in such a way that its differential does not.
The $2$-form $d \eta$ is, by definition, the \defi[curvature of $\mathcal W$] and will be  denoted by $K(\mathcal W)$.

It seems appropriate to borrow  terminology from the XIX century theory of invariants and say that
 the $2$-form $K(\mathcal W)$  is a covariant of the web $\mathcal W$ since
\[
 K(\varphi^* \mathcal W) = \varphi ^* K ( \mathcal W) \,
\]
for every germ of biholomorphism $\varphi \in \mathrm{Diff}(\mathbb C^2,0)$.

\begin{lemma}\label{L:Wk2}
If $\mathcal W= \mathcal W ( x , y , f)$ where  $f \in \mathcal O(\mathbb C^2,0)$ then
\[
K(\mathcal W) = \frac{\partial ^2}{\partial x \partial y} \big( \log (f_x / f_y) \big) dx \wedge dy \, .
\]
In particular, if   $\mathcal W_k = \mathcal W(x, y, x+y +xy(x-y)(k + \hot))$
then
$$
K(\mathcal W_k) = 4k \big( dx \wedge dy\big)
$$ up to higher order terms.
\end{lemma}
\begin{proof}
Because $(- f_x dx) + (- f_y dy) + (df) = 0$, the $1$-form $\eta$ is
\[ \partial_x (\log f_y)\,
 dx + {\partial_y (\log f_x)}\,
 dy \, .
\]
Hence $K(\mathcal W)$ is as claimed.

Specializing to  $\mathcal W_k= \mathcal W_k(x,y, x+y + kxy(x-y))$ it follows that
\[
K(\mathcal W_k) = \frac{\partial ^2}{\partial x \partial y} \log \left(\frac{1+(2xy-y^2)(k + \hot)}{1-(2xy-x^2)(k+ \hot)} \right) dx \wedge dy\, .
\]
The second claim follows from the evaluation of the above expression at zero.
\end{proof}

\subsubsection{Structure of planar hexagonal  $3$-webs}
The next result can be considered as the founding stone of web geometry.
It seems fair to say that it awakened the interest of Blaschke and his coworkers on the subject  in the early 1930's.
\index{Web!hexagonal|(}

\begin{thm}\label{T:hexagonal}
Let  $\mathcal W=\mathcal F_1\boxtimes \mathcal F_2 \boxtimes \mathcal F_3$ be a smooth $3$-web on $(\mathbb C^2,0)$. The following assertions are equivalent:
\begin{enumerate}
\item[{\bf (a).}] the web $\mathcal W$ is hexagonal;
\item[{\bf (b).}] the $2$-form $K(\mathcal W)$ vanishes identically;
\item[{\bf (c).}] there exists closed $1$-forms $\eta_i$ defining $\mathcal
F_i$,  $i=1,2,3$, such that
$
 \eta_1 + \eta_2 + \eta_3 = 0 \, ;
$
\item[{\bf (d).}] the web $\mathcal W$ is equivalent to $\mathcal W( x,y, x+y)$.
\end{enumerate}
\end{thm}

Besides Lemma \ref{L:Wk} and Lemma \ref{L:Wk2}, the proof of Theorem \ref{T:hexagonal} will also
make use of the following.

\begin{lemma}\label{L:normal}
Every germ of smooth $3$-web $\mathcal W$ on $(\mathbb C^2,0)$ is equivalent to $\mathcal W(x,y,f)$, where $f \in \mathcal O(\mathbb C^2,0)$ is
of the form
\[
f(x,y) = x + y + xy(x-y)(k+  \hot)
\]
for a suitable $k \in \mathbb C$.
\end{lemma}
\begin{proof}
As already pointed out  in Section \ref{S:eq1st} every smooth $2$-web on $(\mathbb C^2,0)$ is equivalent to $\mathcal W(x,y)$. Therefore
it can be assumed that $\mathcal W =\mathcal W(x,y,g)$ with
$g \in \mathcal O{(\mathbb C^2,0)}$ such that $g(0)=0$. The smoothness assumption on $\mathcal W$ translates into $dx \wedge dg (0) \neq 0 $ and
$dy \wedge dg (0) \neq 0$ or, equivalently, both $g_x(0)$ and $g_y(0)$ are non-zero complex numbers.

After pulling back $\mathcal W$ by $\varphi_1(x,y) = ( g_x(0) x , g_y(0) y ) $ one can  assume  that  $\mathcal W$ still takes the form
$\mathcal W(x,y,g)$ but now with the function $g$ having $x+y$ as its linear term.

Let $a(t) = g(t,0)$ and $b(t) = g(0,t)$. Clearly both $a$ and $b$ are germs of biholomorphisms of $(\mathbb C,0)$.
Let $\varphi(x,y) = (a^{-1}(x) , b^{-1}(y))$
and set $h(x,y) = \varphi^* g ( x,y) = g(a^{-1}(x), b^{-1}(y))$. Notice that $\varphi^* \mathcal W(x,y,g) = \mathcal W(x,y,h)$ and that $h$ still
has linear term equal to $x+y$. Moreover $h(0,t)= h(t,0) = h(t,t)/2 = t$ up to higher order terms.

Because the germ $\alpha(t) = h(t,t)$ has derivative at zero of modulus distinct from one it follows from Poincar\'{e} Linearization Theorem
\index{Poincar\'{e}!Linearization Theorem}\cite[Chapter 3, \S 25.B]{Arnold} the existence
of a germ of biholomorphism $\phi \in \mathrm{Diff}(\mathbb C,0)$ conjugating $\alpha$ to its linear part. More succinctly,
\[
\phi^{-1} \circ \alpha \circ \phi (t) = 2t \, .
\]

After setting $\varphi(x,y) = (\phi(x), \phi(y))$ and $f = \phi ^{-1} \circ h \circ \phi$ one promptly verifies the identities
\[
f(t,0) = f(0,t) = \frac{f(t,t)}{2} = t .
\]
To conclude the proof it suffices to analyze  the implications of the above identities
to   the series expansion  $f(x,y) = \sum a_{ij} x^i y^j$. The reader is invited to fill in the details.
\end{proof}

\subsubsection{Proof of Theorem \ref{T:hexagonal}}
To prove that  (a) implies (b) start by applying Lemma \ref{L:normal} to see that   $\mathcal W$, at any point $p \in (\mathbb C^2,0)$,
is equivalent to 
$$\mathcal W\big(x,y,x+y +  xy(x-y)(k + \hot)\big) \quad \text{ with } k \in \mathbb C.$$
Because $\mathcal W$ is hexagonal the holonomy  at an arbitrary $p \in (\mathbb C^2,0)$  is the identity.
Lemma \ref{L:Wk} implies  $k=0$. Lemma \ref{L:Wk2}, in its turn, allows one to deduce that
 $K(\mathcal W)$ is also zero at an arbitrary point of $(\mathbb C^2,0)$, thus proving that  (a) implies (b).

 \smallskip

Suppose now that (b) holds true, and assume   ${\mathcal W= \mathcal W(\omega_1, \omega_2, \omega_3)}$ with the $1$-forms
$\omega_i$ satisfying $\omega_1 + \omega_2 + \omega_3=0$.  Let  $\eta$ be the unique $1$-form given by Lemma \ref{L:eta}.
Because $K(\mathcal W)=0$, the $1$-form $\eta$ is closed. If
\[
\eta_i = \exp\left( - \int \eta \right) \omega_i
\]
then
\begin{align*}
d \eta_i = - \eta \wedge \exp\left( - \int \eta \right) \omega_i + \exp\left( - \int \eta \right)  d \omega_i = 0 \,
\end{align*}
because $d\omega_i = \eta \wedge \omega_i.$ Moreover
\[
\eta_1 + \eta_2 + \eta_3 = \exp\left( - \int \eta \right) ( \omega_1 + \omega_2 + \omega_3 ) = 0 \, .
\]
This proves that (b) implies (c).

\smallskip

Now assuming the validity of  (c), one can define
\[ f_i(x) = \int_0^x \eta_i \quad \text{ for } i \in \underline{3} \, .
\]  Notice that  $\varphi(x,y) = (f_1(x,y),f_2(x,y))$ is a biholomorphism, and clearly $\varphi^* \mathcal W( x,y,x+y) = \mathcal W(f_1, f_2, f_3)$  since
$\eta_1 + \eta_2 = - \eta_3$. Thus $\mathcal W$ is equivalent to $\mathcal W(x,y,x+y)$.

The missing implication,  (d) implies (a), has already been established  in Figure \ref{F:proofhex}.
\qed

\subsection{Germs of hexagonal webs on the plane}\label{S:bol}

Having Theorem \ref{T:hexagonal}  at hand it is natural to enquire about germs of smooth $k$-webs on $(\mathbb C^2,0)$, $k > 3$,
for which every $3$-subweb is hexagonal. The $k$-webs having this property will  also be called \defi[hexagonal].

\index{Web!hexagonal|)}
\smallskip

The simplest examples of  hexagonal $k$-webs are the  \defi[parallel] \index{Web!parallel} \mbox{$k$-webs}. These  webs $\mathcal W$ are the
 superposition of $k$ pencils of parallel lines. They all can be written
explicitly as $$\mathcal W( \lambda_1 x - \mu_1 y, \cdots, \lambda_k x - \mu_k y)$$ where the pairs $\lambda_i,\mu_i \in \mathbb C^2 \setminus \{ 0\}$
represent the slopes $(\mu_i : \lambda_i) \in \mathbb P( \mathbb C^2) = \mathbb P^1$ of the pencils.

More generally, if $\mathcal L_1, \ldots, \mathcal L_k$ are $k$ pairwise distinct pencils of lines on $\mathbb C^2$
such that no line joining two base points passes through the origin then $\mathcal W= \mathcal L_1 \boxtimes \cdots \boxtimes \mathcal L_k$,
seen as a germ at the origin, is also a  smooth hexagonal   $k$-web.

A less evident family of examples was found by Bol \index{Bol!$5$-web} and are the germs of  $5$-webs defined as follows. Let $\mathcal L_1, \ldots, \mathcal L_4$ be
four pencil of lines satisfying the same conditions as above,  plus the extra condition that no three among the four
base points of the pencils  are colinear. The $5$-web obtained from the superposition of $\mathcal L_1 \boxtimes \cdots \boxtimes \mathcal L_4$
with the pencil of conics through the four base points is a   smooth  hexagonal   $5$-web on $(\mathbb C^2,0)$. According to
the relative position of the base points with respect to the origin, one obtains in this a way a two-dimensional family of non-equivalent  germs
of smooth $5$-webs on $(\mathbb C^2,0)$. Any of these germs will be called \defi[Bol's $5$-web] $\mathcal B_5$.\index{Bol!$5$-web}
The abuse of terminology is justified by the fact they are all germifications of the very same global singular $5$-web (~a concept to
be introduced in Section \ref{S:global}~)  defined on $\mathbb P^2$.

\begin{figure}[H]
\begin{center}\includegraphics[height=4.2cm,width=4.5cm]{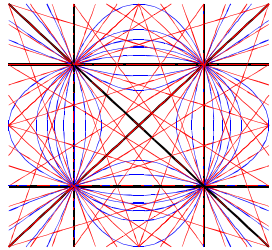}\end{center}\vspace{-1cm}
\caption{Bol's $5$-web} \label{F:B5}
\end{figure}

Anyone taking the endeavor of finding a  smooth hexagonal $k$-web  on $(\mathbb C^2,0)$ not equivalent to any of the previous
examples is doomed to failure. Indeed Bol proved the following

\begin{thm}
If $\mathcal W$ is a smooth hexagonal $k$-web, $k \ge 3$, then $\mathcal W$ is
 equivalent to the superposition of $k$ pencil of lines or $k=5$ and $\mathcal W$ is equivalent
to  $\mathcal B_5$.
\end{thm}

A proof will not be  presented here. For a recent exposition, with a fairly detailed sketch of proof, see \cite{RobertBol}.

\index{Web!curvature|)}
\subsection{Algebraic  planar $3$-webs are hexagonal}\label{S:algcub}
\index{Web!algebraic|(}
\index{Web!hexagonal|(}

\begin{prop}\label{P:cubic}
If  $C \subset \mathbb P^2$ is a reduced cubic and  $\ell_0 \subset \mathbb P^2$ is a line intersecting $C$
transversely then  the $3$-web $\mathcal W_C(\ell_0)$ is hexagonal.
\end{prop}

The simplest instance of the  proposition above is when $C$ is the union of three distinct concurrent lines. In this particular
case it can be promptly verified that $\mathcal W_C(\ell_0)$ is the $3$-web
$\mathcal W(x,y,x-y)$ in a suitable affine coordinate system $(x:y:1) \in \check{\mathbb P}^2$ without further ado.
\smallskip

In the next simplest instance, $C$ is still the union of three distinct lines but they are no longer concurrent.
Then, $\mathcal W_C(\ell_0)$ is the $3$-web $\mathcal W(x,y,(x-1)/(y-1) )$ in suitable affine coordinates.
The most straight forward way to verify the hexagonality of $\mathcal W(x,y, (x-1)/(y-1))$ consists in observing that
the closed differential forms $\eta_1 = - d \log(x-1), \eta_2= d \log(y-1) $ and $\eta_3 = d \log\big( (x-1)/ (y-1) \big)$
define the same foliations as  the submersions $x,y$ and $(x-1)/(y-1)$ and  satisfy $\eta_1+ \eta_2+ \eta_3=0$.
Therefore the $3$-web under scrutiny is hexagonal thanks to the equivalence between items (a) and (c) in Theorem \ref{T:hexagonal}.
\smallskip

To deal with the other cubics, one could still try to make explicit  three  submersions defining the web and work his way to determine
a relation between closed $1$-forms defining the very same foliations. Once the submersions are determined the second step is
rather straight-forward since the proof of Theorem \ref{T:hexagonal} gives an algorithmic way to perform it. Besides having many
particular cases to treat, the lack
of rational parametrizations for  smooth cubics would lead one
to compute with Weierstrass $\wp$-functions or similar transcendental  objects, adding a considerable amount of difficulty
to such task.  Perhaps the most elementary way to prove the hexagonality of algebraic planar $3$-webs relies on the following
Theorem of Chasles.

\begin{thm}
Let $X_1, X_2 \subset \mathbb P^2$ be  two plane cubics  meeting in exactly nine distinct points . If $X \subset \mathbb P^2$ is any cubic
containing at least eight of these nine points then  it automatically contains all the nine points.
\end{thm}\index{Chasles' Theorem}
\begin{proof}
Aiming  at a contradiction, suppose that $X$ does not contain  $X_1 \cap X_2$.
Let $F_1$, $F_2$ be homogenous cubic polynomials defining $X_1$, $X_2$ respectively and $G$ be the one defining $X$.
Since there are nine points in the intersection of $X_1$ and $X_2$
then, according to Bezout's Theorem, the curves $X_1$ and $X_2$ must intersect transversely. In particular both curves
are smooth at the intersection points.  After replacing $X_1$ by the generic
member of the pencil $\{ \lambda F_1 + \mu F_2 =0\}$ one can assume, thanks to Bertini's Theorem \cite[page 137]{GH}, that $X_1$ is a smooth
cubic. Consequently $X_1$ is a smooth elliptic curve.

Consider now the rational function $h: X_1 \to \mathbb P^1$ defined
as
\[
h = \left(\frac{G}{F_2}\right)\Big|_{ X_1} \, .
\]
Since $X$ passes through $8$ points of $X_2 \cap X_1$ it follows that $h$ has only one zero: the unique point of $X\cap X_1$ that
does not belong to $X_1 \cap X_2$. Moreover, the transversality of $X_1$ and $X_2$ ensures that this zero is indeed a simple zero.
It follows that $h$ is an isomorphism. Since elliptic curves are not isomorphic to rational curves one arrives at a contradiction
that settles the Theorem.
\end{proof}

\medskip

The proof just presented cannot be qualified as elementary since it makes use
of Bertini's and Bezout's Theorems and some  basic facts of differential topology.
For an elementary proof and a comprehensive account on Chasles' Theorem including its distinguished lineage
and recent -- rather non-elementary -- developments the reader is urged to  consult \cite{CayBac}.

\subsubsection{Proof of Proposition \ref{P:cubic}} To deduce the hexagonality of $\mathcal W_C(\ell_0)$ from Chasles' Theorem start by observing that the leaf of the foliation $\mathcal F_i$
through $\ell_0 \in \check{\mathbb P}^2$, denoted by $L_i$, corresponds to  lines through the point $p_i=p_i(\ell_0)$. To choose a point $x_1 \in L_1$
is therefore the same as choosing a line through $p_1 \in C_1 \subset \check{\mathbb P}^2$. If such line is sufficiently close to $\ell_0$ then it
cuts $C_3$ in a unique point still denoted by $x_1$. In this way the leaf $L_1$ of $\mathcal F_1$ can be identified with the curve $C_3$.
It will also be useful to identify through the same process  $L_2$, the leaf of $\mathcal F_2$ through $\ell_0 \in \check{\mathbb P}^2$,
with $C_1$ and  $L_3$ with $C_2$.

\begin{figure}[ht]
\begin{center}\includegraphics[height=4.2cm,width=4.5cm]{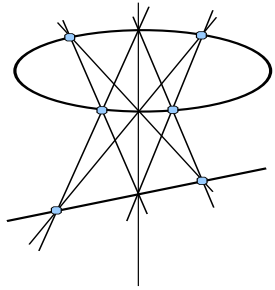}\end{center}\vspace{-0.7cm}
\caption{A cubic with two irreducible components.}
\end{figure}

Now, follow the leaf of $\mathcal F_3$ through $x_1 \in C_3$ until it meets $L_2$ corresponds
to consider the line $\overline{x_1 p_2}$ and  intersect it with $C_1$. The intersection point $x_2 = \overline{x_1 p_2} \cap C_1  \in C_1$ corresponds
to a point  in $L_2$.

Similarly the sequence of points $x_3, x_4, \ldots, x_7$ appearing in the definition of the holonomy of $\mathcal W_C(\ell_0)$
can be synthetically obtained as follows:
\begin{align*}
x_3 = \overline{x_2 p_3} \cap C_2 \in L_3 \simeq C_2 \, , \\
x_4 = \overline{x_3 p_1} \cap C_3 \in L_1 \simeq C_3 \, ,\\
x_5 = \overline{x_4 p_2} \cap C_1 \in L_2 \simeq C_1 \, ,\\
x_6 = \overline{x_5 p_3} \cap C_2 \in L_3 \simeq C_2 \, ,\\
x_7 = \overline{x_6 p_1} \cap C_3 \in L_1 \simeq C_3 \, .
\end{align*}
Of course, all the identifications $L_i \simeq C_j$ above, have to be understood as identifications of germs of curves.

\begin{figure}[ht]
\begin{center}
\psfrag{1}[][]{$ x_2 $}
\psfrag{2}[][]{$ x_3 $}
\psfrag{3}[][]{$ x_4 $}
\psfrag{4}[][]{$ x_5 $}
\psfrag{5}[][]{$ x_6 $}
\psfrag{6}[][]{$ x_7 $}
\psfrag{0}[][]{$ x_1 $}
\includegraphics[height=4.2cm,width=4.5cm]{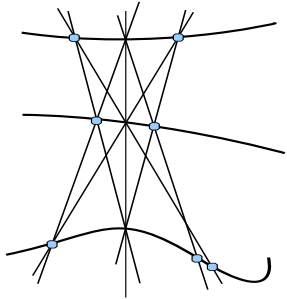}\end{center}
\caption{This is not a cubic.}
\end{figure}

Notice that the line $\overline{x_1 x_2}$ is the same as $\overline {x_1 p_2}$. Therefore it contains the three points $x_1,x_2,p_2$. The line
$\overline{x_3 x_4}$ in its turn contains the points $x_3, x_4, p_1$ and the line $\overline{x_5 x_6}$ contains the points $x_5,x_6, p_3$. Thus the
reduced cubic $X_1 = \overline{x_0 x_1} \cup \overline{x_2 x_3} \cup \overline{x_4 x_5}$ intersects the cubic $X_2 = C$ in exactly nine
distinct points namely $p_1,p_2, p_3, x_1,x_2,x_3,x_4,x_5,x_6$. The same reasoning shows that the reduced cubic
$X=\overline{x_2 x_3} \cup \overline{x_4 x_5} \cup \overline{x_6 x_7}$ intersects $X_2=C$ in the nine points $p_1,p_2, p_3, x_2,x_3,x_4,x_5,x_6,x_7$.
Thus $X_1 \cap X_2 \cap X$  contains at least eight points. Chasles' Theorem implies that this eight is indeed a nine and consequently
the points $x_1$ and $x_7$ must coincide. This is sufficient to prove that the holonomy of $\mathcal W_C(\ell_0)$ is the identity. \qed

\dd

Later on this text the hexagonality of the planar algebraic $3$-webs will be established again using Abel's addition Theorem.
Although apparently unrelated both approaches are intimately intertwined. Abel's addition Theorem can be read as a result
about the group structure of the Jacobian \index{Jacobian} of projective curves while Chasles' Theorem turns out to be equivalent to the existence
of an abelian group structure  for plane cubics where aligned points sum up to zero. \index{Chasles' Theorem}

\index{Web!algebraic|)}
\index{Web!hexagonal|)}

\section{Singular and global webs}\label{S:fancydef}

\subsection{Germs of singular webs I}\label{S:singwebs1}

\index{Foliation!singular holomorphic}

It is customary to say that a germ of singular holomorphic foliation is an equivalence
class $[\omega]$ of germs of  holomorphic $1$-forms in $\Omega^1{(\mathbb C^n,0)}$ modulo
multiplication by elements of $\mathcal O^*({\mathbb C^n},0)$ such that any representative $\omega$ is
integrable (~$\omega \wedge d \omega =0$~) and with \defi[singular  set]
$\mathrm{sing}(\omega)=\{ p \in (\mathbb C^n,0) \, ; \, \omega(p) =0 \}$ of codimension at
least two.

\smallskip

\index{Web!singular|(}

An analogous definition can be made for codimension one webs. A \defi[germ of singular codimension
one $k$-web] on $(\mathbb C^n,0)$ is an equivalence class $[\omega]$ of germs of  $k$-symmetric $1$-forms,
that is sections of $\mathrm{Sym}^k \Omega^1{(\mathbb C^n,0)}$,
modulo multiplication by $\mathcal O^*({\mathbb C^n},0)$ such that  a suitable  representative $\omega$
defined in a connected neighborhood $U$ of the origin satisfies the following conditions:
\begin{enumerate}
\item[{\bf (a).}] the zero set of $\omega$ has codimension at least two;
\item[{\bf (b).}] $\omega$, seen as a  homogeneous polynomial of
degree $k$ in the ring $\mathcal O({\mathbb C^n,0})[dx_1, \ldots, dx_n]$, is square-free;
\item[{\bf (c).}] {\bf (Brill's condition)} for a generic $p \in U$, $\omega(p)$ is a product of $k$
 linear forms; \index{Brill's condition|(}
\item[{\bf (d).}] {\bf (Frobenius' condition)}  for a generic $p \in U$, the germ of  $\omega$ at $p$ is the product of $k$
germs of integrable $1$-forms.
\end{enumerate}

Both conditions (c) and (d) are automatic for germs of webs on $(\mathbb C^2,0)$ and non-trivial for germs
on $(\mathbb C^n,0)$ when $n\ge 3$. Notice also that condition (d) implies condition (c). Nevertheless  the two  conditions
are stated independently  because  condition (c) is of purely algebraic nature (~depends only
on the value of $\omega$ at $p$~) while  condition (d) involves the exterior differential and therefore depends
not only on the value of  $\omega(p)$ but also on the local behavior of $\omega$ near $p$.

\smallskip

A germ of singular web will be called \defi[generically smooth]\index{Web!generically smooth} if the condition below is also satisfied:
\begin{enumerate}
\item[{\bf (e).}] {\bf (Generic position)} for a generic $p \in U$,  any $m\le n$ linear forms  $\alpha_1, \ldots, \alpha_m$
 dividing $\omega(p)$ satisfy
 \[
 \alpha_1 \wedge \cdots \wedge \alpha_m \neq 0 \, .
 \]
\end{enumerate}
One can rephrase condition $(e)$ by saying that for a generic $p \in U$ the germ of $\omega$ at $p$ defines a smooth web.

Notice that conditions (b) and (c) together imply that every germ of singular web is {\it generically quasi-smooth}.\index{Web!generically quasi-smooth}

\dd
\index{Foliation!singular holomorphic}
It is interesting to compare the above definition with the following: a germ of singular codimension
$q$ foliation on $(\mathbb C^n,0)$ is an equivalence class $\mathcal F= [\omega]$ of germs of $q$-forms modulo
multiplication by elements of $\mathcal O^*_{(\mathbb C^n,0)}$
satisfying (~as above $\omega$ is a representative of $\mathcal F$ defined on $U \subset \mathbb C^n$~)
\begin{enumerate}\index{Pl\"{u}cker's condition|(}
\item[{\bf (a).}] the zero set of $\omega$ has codimension at least two;
\item[{\bf (c).}] {\bf (Pl\"{u}cker's condition)} for a generic $p \in U$, $\omega(p)$ is a wedge product of $k$
 linear forms $\alpha_1, \ldots, \alpha_k$;
\item[{\bf (d).}] {\bf (Frobenius' condition)}\index{Frobenius' condition|(}  for a generic $p \in U$, the germ of  $\omega$ at $p$ is the product of $k$
germs of  $1$-forms $\alpha_1, \ldots, \alpha_k$ and each one of them satisfies  $d\alpha_i \wedge \omega =0$.
\end{enumerate}

Notice that the absence of condition (b) is due to the antisymmetric character of $\Omega^q{(\mathbb C^n,0)}$.
Although apparently similar  conditions (c),  (d) for codimension $q$ foliations and $k$-webs
have rather distinct features.

\smallskip

It is a classical result of Pl\"{u}cker that the $q$-form $\omega(p)$ satisfies condition (c) ( foliation version )
if only if{\begin{footnote}{Here, and through out,   $i_v $  denotes  interior product.}\end{footnote}}
\[
 \big( i_v \omega(p) \big) \wedge \omega(p) =0 \quad \text{ for every } \quad v \in \bigwedge^{q-1} T_p \mathbb C^n \, .
\]
 Moreover,  varying $v \in T_p \mathbb C^n$, the above formulas
are the  well known Pl\"{u}cker quadrics and generate the homogeneous ideal defining the locus of completely decomposable
$q$-forms  in $\bigwedge^q T^*_p\mathbb  C^n$, see for instance \cite[pages 209--211]{GH} or \cite[Chapter 3,Theorem 1.5]{GKZ}. \index{Pl\"{u}cker!quadrics}

Less well known are Brill's equations \index{Brill's equations}describing the locus of completely decomposable $q$-symmetric $1$-forms,
for a modern exposition see \cite[chapter 4, section 2]{GKZ}. They
differ from Pl\"{u}cker equations in a number of ways: they cannot be so easily described since their definition depends on
some concepts of representation theory; they are not quadratic equations, for $q$-symmetric $1$-forms they are of
degree $q+1$; the ideal generated by Brill's equation is not reduced in general, already for    $3$-symmetric $1$-forms on $\mathbb C^4$
Brill's equations do not generate the ideal of the locus of completely decomposable forms, see for instance \cite[proposition 2.5]{dalbec}.
\index{Brill's condition|)}\index{Pl\"{u}cker's condition|)}
\smallskip

For condition (d) the situation is even worse. While for alternate $q$-forms the
integrability condition can be written, see \cite[Propositions 1.2.1 and 1.2.2]{airton}, as \label{page:integral}
\[
 ( i_v d \omega ) \wedge \omega =0 \quad \text{ for every } \quad v \in \bigwedge^{q-1} T_p \mathbb C^n \, ,
\]
the integrability condition for $q$-symmetric $1$-forms   have not  been treated   in the
literature yet.
\index{Frobenius' condition|)}

\subsection{Germs of singular webs II}\label{S:singII}
\index{Web!singular}
There is an alternative definition for germs of singular webs that is in a certain sense more geometric. The idea
is to consider the (germ of) web as a meromorphic section of the projectivization of cotangent bundle.
This is a classical point of view in the theory of differential equations
which  has been recently explored in web geometry by Cavalier-Lehmann, see \cite{CL2}.
Beware that the  terminology here adopted does not always coincides with the one used in \cite{CL2}.

\subsubsection{The contact distribution}
\index{Contact distribution|(}
Let $\mathbb P = \mathbb PT^* (\mathbb C^n,0)$  be the projectivization of the cotangent bundle of $(\mathbb  C^n,0)$
and $\pi : \mathbb P \to (\mathbb C^n,0)$ the natural projection\begin{footnote}{The convention
adopted in this text is that over a point $p$ the fiber $\pi^{-1}(p)$ parametrizes the one-dimensional subspaces of $T^*_p (\mathbb C^n,0)$.
Beware
that some authors consider $\pi^{-1}(p)$ as a parametrization   of the one-dimensional quotients of $T^*_p (\mathbb C^n,0)$.}\end{footnote}.
On $\mathbb P$ there is a canonical codimension one distribution, the so called contact distribution $\mathcal D$. Its description in terms
of a system of  coordinates $x_1, \ldots, x_n$  of $(\mathbb C^n,0)$ goes as
follows:  if $y_i = \partial_i$ are interpreted as coordinates\begin{footnote}{In case of confusion, notice that the coordinate functions on a vector space $E$ can
be chosen  to be elements of $E^*$, that is,  linear forms on $E$.}\end{footnote}
of the total space of  $T^* (\mathbb C^n,0)$ then the lift from $\mathbb P$ to $T^* (\mathbb C^n,0)$ of the
contact distribution, is the kernel of the  $1$-form
\begin{equation}\label{E:cform}
  \alpha = \sum_{i=1}^n y_i dx_i.
\end{equation}

The usual way to define  $\mathcal D$ in more intrinsic terms goes as follows. Recall that the tautological line-bundle $\mathcal O_{\mathbb P}(-1)$
is the rank one sub-bundle of $\pi^* T^* (\mathbb C^n,0)$ determined over a point $p=(x,[y]) \in \mathbb P$ by the direction parametrized by it. Its dual
$\mathcal O_{\mathbb P}(1)$ is the therefore a quotient of $\pi^* T (\mathbb C^n,0)$. The distribution on $\mathbb P$  induced by the kernel of the composition
\[
\xymatrix{T  \mathbb P \ar[r]^-{ d \pi} &\pi^* T ( \mathbb C^n,0) \ar[r] & \mathcal O_{\mathbb P}(1)  } \,
\]
is nothing more than the contact distribution $\mathcal D$. Notice that the composition is given by the interior
product of local section of $T \mathbb P$ with a twisted $1$-form $\alpha \in H^0( \mathbb P, \Omega^1_{\mathbb P} \otimes \mathcal O_{\mathbb P}(1))$,
which in local coordinates coincides with the $1$-form  (\ref{E:cform}).
This $1$-form is the so called \defi[contact form] of $\mathbb P$. \index{Contact form}
\index{Contact distribution|)}

\subsubsection{Webs as  closures of meromorphic multi-sections}
Let now $W \subset \mathbb P$ be a subvariety not necessarily irreducible but of pure dimension $d$. Suppose also that
$W$ satisfies the following conditions
\begin{enumerate}
\item[(a)] the image under $\pi$ of every irreducible component of $W$ has dimension $n$;
\item[(b)] the generic fiber  of $\pi$ intersects $W$ in $k$ distinct smooth points, and at these
 the differential $d \pi_{|W} : T_p W \to T_{\pi(p)} (\mathbb C^n,0)$ is surjective; and
\item[(c)] the restriction of the contact form $\alpha$ to the smooth part of every irreducible component of $W$
is integrable.
\end{enumerate}

One can then define a germ of web as the subvarieties $W$ of $\mathbb P$ as above. This definition is equivalent to the one laid down
in Section \ref{S:singwebs1}. Indeed given a singular $k$-web  $[\omega]$ in the sense of \S \ref{S:singwebs1}
one can consider the closure of its ``graph"  in $\mathbb P$. More precisely, over a generic point $p \in (\mathbb C^n,0)$
the  ``graph" of $[\omega(p)]$ is formed by the points in $\mathbb P T^*_p (\mathbb C^n,0)$ corresponding to the
factors of $\omega(p)$.  In this way one defines a locally closed subvariety of $\mathbb P$ with closure satisfying
conditions (a), (b) and (c) above.

Reciprocally the restriction of the contact form $\alpha$ to a subvariety $W \subset \mathbb P$ satisfying (a), (b) and (c) above
induces a codimension one foliation $\mathcal F$ on the smooth part of $W$. Moreover, over regular values of $\pi$ the direct image of $\mathcal F$
can be identified with the superposition of $k$ foliations. Since the symmetric product of the $k$ distinct $1$-forms defining these foliations
is invariant under the monodromy of $\pi$, one ends up with a germ of section of $\mathrm{Sym}^k \Omega^1_{(\mathbb C^n,0)}$ inducing $\pi_* \mathcal F$.
After cleaning up eventual codimension one components of the zero set one obtains  a $k$-symmetric $1$-form $\omega$  satisfying the conditions
(a), (b), (c) and (d) of Section \ref{S:singwebs1}.

\index{Web!singular|)}

\subsection{Global webs}\label{S:global}\index{Web!global|(}

Although this text is ultimately interested in the classification of germs of smooth webs of
maximal rank, a concept to be introduced in Chapter \ref{Chapter:AR},  most of the relevant examples
are globally defined on projective manifolds. It is therefore natural to lay down the definitions
of a global web and related concepts.

A global $k$-web $\mathcal W$ on a  manifold $X$ is given by an open covering $\mathcal U= \{ U_i \}$ of $X$
and $k$-symmetric $1$-forms $\omega_i \in \mathrm{Sym}^k \Omega^1_X(U_i)$ subject to the conditions:
\begin{enumerate}
\item for each non-empty intersection $U_i\cap U_j$  of elements of $\mathcal U$ there exists a non-vanishing
function $g_{ij} \in \mathcal O^*(U_i\cap U_j)$ such that $\omega_i = g_{ij} \omega_j$;
\item for every $U_i \in \mathcal U$ and every $x \in U_i$ the germification of $\omega_i$ at $x$ satisfies the conditions
(a), (b), (c) and (d) of Section \ref{S:singwebs1}, in other words, the germ of $\omega_i$ at $x$ is a representative
of a germ of a singular web.
\end{enumerate}

The transition functions $g_{ij}$ determine a line-bundle $\mathcal N$ over $X$ and the $k$-symmetric $1$-forms
$\{\omega_i \}$ patch together to form a section of $\mathrm{Sym}^k \Omega^1_X \otimes \mathcal N$, that is,
$\omega=\{ \omega_i\}$ can be interpreted as an element of $H^0(X,\mathrm{Sym}^k\Omega^1_X \otimes \mathcal N)$.
The line-bundle $\mathcal N$ will be called the \defi[normal bundle] of $\mathcal W$. \index{Web!normal bundle}
Two global sections $\omega, \omega' \in H^0(X, \mathrm{Sym}^k\Omega^1_X \otimes \mathcal N)$ determine the same
web if and only if they differ by the multiplication by an element $g \in H^0(X, \mathcal O^*_X)$.

If  $X$ is compact, or more generally if the only global sections of $\mathcal O^*_X$ are the non-zero constants, then a global $k$-web is nothing
more than an element of ${\mathbb P H^0(X,\mathrm{Sym}^k\Omega^1_X \otimes \mathcal N)}$, for a suitable line-bundle $\mathcal N \in \mathrm{Pic}(X)$,
with germification of any representative at any point of $X$ satisfying  conditions (a), (b), (c) and (d) of Section \ref{S:singwebs1}.

When $X$ is a variety for which every line-bundle has  non-zero meromorphic sections one can alternatively
define global $k$-webs as equivalence classes $[\omega]$ of meromorphic $k$-symmetric $1$-forms modulo
multiplication by meromorphic functions such that at a generic point $x \in X$ the germification of any representative
$\omega$ satisfies the very same conditions refereed to above. The transition to the previous
definition is made by observing that a meromorphic $k$-symmetric $1$-form $\omega$ can be interpreted
as a global holomorphic section of $\mathrm{Sym}^k \Omega^1_X \otimes \mathcal O_X( (\omega)_{\infty}- (\omega)_0)$
where $(\omega)_0$, respectively $(\omega)_{\infty}$, stands for the zero divisor, respectively polar divisor, of $\omega$.
\index{Web!global|)}

\smallskip

{A $k$-web $\mathcal W\in {\mathbb P H^0(X,\mathrm{Sym}^{k}\Omega^1_X \otimes \mathcal N)}$
is \defi[decomposable]  if there are global webs $\mathcal W_1, \mathcal W_2$ on $X$
sharing no common subwebs  such that $\mathcal W$ is the superposition of $\mathcal W_1$ and $\mathcal W_2$, that is
 $\mathcal W=\mathcal W_1\boxtimes \mathcal W_2$. A $k$-web $\mathcal W$ will be called  \defi[completely  decomposable]
 if one can write ${\mathcal W=\mathcal F_1\boxtimes \cdots \boxtimes \mathcal F_k}$ for $k$  global
 foliations $\mathcal F_1,\ldots,\mathcal F_k$ on $X$. Remark
 that the restriction of a web  at a sufficiently small  neighborhood of a generic $x\in X$
 is completely decomposable.
\index{Web!decomposable}
\index{Web!completely decomposable}

\subsubsection{Monodromy}\index{Monodromy|(}

Thanks to condition \ref{S:singwebs1}.(b) the germ of a global $k$-web $\mathcal W$ at a generic point $x \in X$
is completely decomposable. Moreover the set of points $x \in X$ where $\mathcal W_x$ is not completely decomposable
is a closed analytic subset of $X$. If $U$ is the complement of this subset then for arbitrary $x_0 \in U$ it is
possible to write $\mathcal W_{x_0}$, the germ of $\mathcal W$ at $x_0$, as $\mathcal F_1 \boxtimes \cdots \boxtimes \mathcal F_k$.
Notice that $\mathcal W$ does not have to be quasi-smooth at $x_0 \in U$. It may happen that $T_{x_0} \mathcal F_i = T_{x_0} \mathcal F_j$ for some $i\neq j$.

Analytic continuation of this decomposition along paths $\gamma$ contained in $U$ determines an anti-homomorphism\begin{footnote}{As usual the fundamental group acts on the right and thus $\rho_{\mathcal W}$ is not a homomorphism but instead an anti-homomorphism, that is $$\rho_{\mathcal W}( \gamma_1 \cdot \gamma_2) = \rho_{\mathcal W}( \gamma_2) \cdot \rho_{\mathcal W}(\gamma_1)$$ for arbitrary $\gamma_1, \gamma_2 \in \pi_1(U,x_0)$ . }\end{footnote}
\[
\rho_{\mathcal W} : \pi_1(U, x_0) \longrightarrow \mathfrak S_k
\]
from the fundamental group of $U$ to the permutation group on $k$ letters $\mathfrak{S}_k$.
Because distinct choices of base points yield conjugated anti-homomorphisms, it is harmless
to identify all these anti-homomorphisms and call them
the monodromy (anti)-representation of $\mathcal W$. The image of $\rho_{\mathcal W} \subset \mathfrak S_k$ is,
by definition, the \defi[monodromy group] of $\mathcal W$. \index{Web!monodromy group}

The reader is invited to verify the validity of the following proposition.

\begin{prop}\label{P:rho}
The following assertions hold:
\begin{enumerate}
\item[(a)] If $\mathcal W$ is not completely decomposable then there exists $\gamma \in \pi_1(U,x_0)$ such that $\rho_{\mathcal W}(\gamma)$ is
a non-trivial permutation;
\item[(b)] Every irreducible component of the complement of $U$ has codimension one.
\end{enumerate}
\end{prop}

Proposition \ref{P:rho} makes clear that $\rho_{\mathcal W}$ measures the obstruction to $\mathcal W$ be completely decomposable.

\medskip

Alternatively one can also define a global $k$-web on $X$ as a closed subvariety $W \subset \mathbb PTX$ satisfying the natural global analogues of
conditions (a), (b) and (c) of Section \ref{S:singII}. In this alternative take the monodromy is  nothing more than the usual monodromy
of the projection $\pi_{|W}:  W \to X$.\index{Monodromy|)}

\subsection{Discriminant}\index{Discriminant|(}

The discriminant locus $\Delta(\mathcal W)$ of a $k$-web $\mathcal W$ on a complex manifold $X$ is composed by the set of points
where the germ of $\mathcal W$ is not quasi-smooth. Thinking $\mathcal W$ as a subvariety $W \subset \mathbb P (TX)$, the discriminant is precisely the
image under the natural projection $\pi_{|W}: W \to X$ of the union of the singular points of $\mathcal W$ with the critical set of the restriction of $\pi_{|W}$ to the smooth locus of $\mathcal W$.

From its very definition it is clear that $\Delta(\mathcal W)$ is a closed analytic subset with complement contained in the subset $U$ used in the definition
of the monodromy representation. Therefore the monodromy representation can be thought as a anti-homomorphism from $\pi_1(X \setminus \Delta(\mathcal W))$ to
$\mathfrak{S}_k$.

For  webs $\mathcal W$ on  surfaces there are simple expressions for their discriminants inherited from the
classical invariant theory of binary forms.\index{Discriminant|)}

\subsubsection{The resultant and tangencies between webs on surfaces}\index{Resultant|(}

Recall that for two homogeneous polynomials in two variables, also known as binary forms, \index{Binary forms}
\[
P = \sum_{i=0}^m p_i x^i y^{m-i} \quad \text{ and } \quad Q = \sum_{i=0}^n q_i x^i y^{n-i}
\]
the resultant $R[P,Q]$ of $P$ and $Q$ is given by the determinant of the  Sylvester matrix \index{Sylvester matrix}
\[
\left[\begin{array}{cccccc}
p_{m} & \dots & \dots & p_{0} &  & \\
& \ddots &  & &\ddots & \\
&  & p_{m} & \dots & \dots & p_{0}  \\
q_{n} & \dots & q_{0} &  & &\\
& \ddots &  & \ddots & &\\
&  & q_{n} & \dots & q_{0}  &\\
\end{array}\right] \, .
\]
This is the  $(m+n) \times (m+n)$-matrix formed from the coefficients of $P$ and $Q$ as schematically presented above with
$n= \deg(Q)$ rows builded from the coefficients of $P$ and $m=\deg(P)$ rows coming from the coefficients of $Q$.

If $g(x,y) = (\alpha x + \beta y, \gamma x + \delta y)$  is a linear automorphism of $\mathbb C^2$
 and $\lambda, \mu \in \mathbb C^*$ then the resultant obeys the transformation rules
\begin{equation}\label{E:TransRes}
\begin{array}{lcl}
R[\lambda P, \mu Q] &=& \lambda^{\deg(Q)}\mu^{\deg(P)} R[P,Q] \, ,\\[0.2cm]
R[g^* P , g^* Q] &=& \det(Dg)^{\deg(P)\cdot \deg(Q)} R[P,Q] \, .
\end{array}
\end{equation}
Moreover $R[P,Q]$ vanishes if and only if $P$ and $Q$ share a common root.

If, for $\ell =1,2$,   $\mathcal W_\ell = [ \omega_{\ell} ] \in P H^0 (S, \mathrm{Sym}^{k_{\ell}} \Omega^1_S \otimes \mathcal N_{\ell})$ is a $k_{\ell}$-web
on a surface $S$ then the local defining $1$-forms $\omega_{\ell,i} \in \mathrm{Sym}^{k_{\ell}} \Omega^1_S(U_i)$ can be interpreted as binary forms
in the variables $dx,dy$ with coefficients in $\mathcal O_S(U_i)$. The resultant $R[\omega_{1,i},\omega_{2,i}]$ is then an element of $\mathcal O_S(U_i)$ with
zero locus coinciding with the tangencies between ${\mathcal W_1}_{|U_i}$  and ${\mathcal W_2}_{|U_i}$.
The transformation rules (\ref{E:TransRes}) imply that
the collection $\{R[\omega_{1,i}, \omega_{2,i}]\}$ patch together to form a global holomorphic section of the line-bundle $K_S^{\otimes k_1 \cdot k_2} \otimes \mathcal N_1^{\otimes k_2} \otimes \mathcal N_2^{\otimes k_1}.$ This section is different from the zero section if and only $\mathcal W_1$ and $\mathcal W_2$ do
not share a common subweb since the resultant vanishes only when its parameters share common roots.

If $\mathrm{tang}(\mathcal W_1, \mathcal W_2)$ is defined as the divisor locally given by the resultant of the defining $k_\ell$-symmetric $1$-forms
then the discussion just made can be summarized in the following proposition.

\begin{prop}
Let $\mathcal W_1$ be a $k_1$-web and $\mathcal W_2$ a $k_2$-web  with respective normal bundles $\mathcal N_1$ and $\mathcal N_2$,
both defined  on the same surface $S$. If they do not share a common subweb
then the identity
\[
\mathcal O_S( \mathrm{tang}(\mathcal W_1, \mathcal W_2) ) = K_S^{\otimes k_1 \cdot k_2} \otimes \mathcal N_1^{\otimes k_2} \otimes \mathcal N_2^{\otimes k_1}
\]
holds true in the Picard group of $S$.
\end{prop}
\index{Resultant|)}
\subsubsection{The discriminant of webs on a surface}\index{Discriminant|(}

By definition, the \defi[discriminant] $\Delta(P)$ of binary form $P$ is
\[
\Delta(P) = \frac{R[ P, \partial_x P] }{n^n p_n } = \frac{R[ P, \partial_y P ] }{n^n p_0 }  \, .
\]
Notice that $\Delta(P)$ vanishes if and only if both
$P$ and $\partial_x P$ share a common root, that is, $P$ has a root with multiplicity greater than one.

The discriminant obeys rules analogous to the ones obeyed by the resultant. Namely
\begin{equation*}
\begin{array}{lcl}
\Delta(\lambda P) &=& \lambda^{2(\deg(P)-1)}\Delta(P) \, ,\\[0.2cm]
\Delta(g^* P) &=& \det(Dg)^{\deg(P){(\deg(P)-1)}} \Delta(P) \, .
\end{array}
\end{equation*}

If $\mathcal W$ is a $k$-web, $k \ge2$, on a surface $S$  then the discriminant divisor of $\mathcal W$ is,
by definition, the divisor locally defined by $\Delta(\omega_i)$ where as before $\omega_i \in \mathrm{Sym}^k \Omega^1_S(U_i)$ locally defines
$\mathcal W$. Notice that the support of the discriminant divisor coincides with the discriminant set of $\mathcal W$ previously defined.

The discussion about the tangency of two webs adapts verbatim to yield the proposition below.

\begin{prop}\label{P:degdisc}
If $\mathcal W$ is a $k$-web  with normal-bundle $\mathcal N$ defined on a surface $S$ then
\[
\mathcal O_S(\Delta(\mathcal W)) = K_S^{ \otimes k(k-1)} \otimes \mathcal N^{\otimes 2(k -1)} \, .
\]
\end{prop}

\subsubsection{Discriminants of real webs}
\index{Web!real trace|(}

Due to obvious technical constraints all the pictures of planar webs are drawn over the real plane. In particular
the  webs portrayed  ought to be defined by  real analytic $k$-symmetric $1$-forms $\omega$  on some open subset
$U$ of $\mathbb R^2$. Most of the time  these $1$-forms will be polynomial $1$-forms and hence globally defined on $\mathbb R^2$.

The sign of the discriminant of $\omega$ at a given point $p \in U$ gives clues about the number of real leaves of $\mathcal W = [ \omega]$
through $p$. For a $2$-web $\mathcal W$  induced by $\omega = a dx^2 + b dx dy + c dy^2$ the sign of $\Delta= \Delta(\omega) = b^2 - 4ac$ tells all one
may want to know about the number of real leaves: when $\Delta(p)>0$ there are two real leaves through $p$, and when $\Delta(p)<0$ there are
no real leaves through $p$.

For $3$-webs the situation is as good as for $2$-webs. According to wheter the sign of $\Delta$ is positive or negative at a given point $p$
the $3$-web has one or three real leaves through $p$.

For $k$-webs with $k\ge 4$ the sign of $\Delta$ at $p$ does not determine the number of real leaves of $\mathcal W$ through it but does tell
that, see \cite{RWDN},
\begin{itemize}
\item when $k$ is odd, the number of real leaves through $p$ is congruent to $1$ or $3$ modulo $4$ according as $\Delta(p)>0$ or $\Delta(p)<0$, and;
\item when $k$ is even, the number of real leaves through $p$ is congruent to $0$ or $2$ modulo $4$ according as  $\Delta(p)<0$ or $\Delta(p)>0$.
\end{itemize}

It is tempting to claim that a planar $k$-web $\mathcal W$ on $\mathbb C^2$ defined
by a real $k$-symmetric $1$-form with only one leaf through each point of a given domain $U \subset \mathbb R^2$
is nothing more than an analytic foliation on $U$. Although trivially true if the discriminant of $\mathcal W$ does not
intersect $U$ this claim is far from being true in general.

Perhaps the simplest example comes from a variation on the  classical Tait-Kneser Theorem presented in \cite{Taba,GhysP}.

Let  $f \in \mathbb R[x]$ be a polynomial in one real variable of degree $k$. For  fixed $n < k$ and $t \in \mathbb R$ define the
$n$-th osculating polynomial $g_t$  of $f$ as the polynomial of degree at most $n$ whose graph osculates the graph of $f$ at $(t,f(t))$ up to order $n$.
From its definition follows that $g_t(x)$ is nothing than the truncation of the Taylor series of $f$ centered at $t$ at order $n+1$, that is
\[
g_t(x) = \sum_{i=0}^n \frac{f^{(i)}(t) } {i!}  (x-t)^i \, .
\]

Notice that for a fixed $ t \in \mathbb C$ for which $f^{(n)}(t) \neq 0$ the graph of $g_t$,  $G_t = \{ y = g_t(x) \}$, is an irreducible plane curve
of degree $n$. Moreover varying $t \in \mathbb C$ one obtains a family of degree $n$ curves which corresponds to a degree $k$ curve $\Gamma_f$ on the
space of degree $n$ curves. The degree $n$ curves through a generic point $p \in \mathbb C^2$ determine a hyperplane $H$ in the space of degree $n$ curves.
Because $H$ intersects $\Gamma_f$ in $k$ points, through a generic $p \in \mathbb C^2$ passes  $k$ distinct curves of the family $\{ G_t\}$.
Therefore this family of curves determines a $k$-web $\mathcal W_f$ on $\mathbb C^2$.

To obtain a polynomial $k$-symmetric $1$-form defining $\mathcal W_f$ it suffices to eliminate $t$ from the pair of equations
\[
 \begin{array}{lcl} y - g_t(x) &=& 0  \, ,\\ dy - \partial_x g_t(x) dx &=& 0 \, . \end{array}
\]
Such task can be performed by considering the resultant of $y- g_t(x)$ and $dy - \partial_x g_t(x) dx$ seen
as degree $k$ polynomials in the variable $t$ with coefficients in $\mathbb C[x,y,dx,dy]$.

To investigate the real trace of $k$-web $\mathcal W_f$ the following variant of the classical Tait-Kneser Theorem \cite{GhysP,Taba}
will be useful. \index{Tait-Kneser Theorem}

\begin{thm}\label{T:KN}
If $n$ is even and $f^{(n+1)}(t) \neq 0$ for every real number $t$ in an interval $(a,b)$  then
the curves $G_a$ and $G_b$ do not intersect in $\mathbb R^2$.
\end{thm}
\begin{proof}
On the one hand if  $a<b$ are real numbers for which $G_a$ and $G_b$ intersects in $\mathbb R^2$ then there exists
a real number $x_0 \in \mathbb R$ such that $g_a(x_0)- g_b(x_0)=0$.

On the other hand the fundamental theorem of calculus implies
\begin{align*}
g_a(x_0) - g_b(x_0) =  & \int_a^b \frac{\partial g_t}{\partial t}(x_0) dt = \int_a^b \left(\sum_{i=0}^n \frac{f^{(i+1)}(t)}{i!} (x_0 -t)^i - \sum_{i=1}^n \frac{f^{(i)}(t)}{(i-1)!} (x_0-t)^{i-1} \right) dt \\  =  & \int_a^b \frac{f^{(n+1)}(t)}{n!}(x_0 -t)^n dt \neq 0 \, .
\end{align*}

%\begin{align*}
%&g_a(x_0) - g_b(x_0) = \int_a^b \frac{\partial g_t}{\partial t}(x_0) dt = \\
%&= \int_a^b \left(\sum_{i=0}^n \frac{f^{(i+1)}(t)}{i!} (x_0 -t)^i - \sum_{i=1}^n \frac{f^{(i)}(t)}{(i-1)!} (x_0-t)^{i-1} \right) dt =
%\\& \int_a^b \frac{f^{(n+1)}(t)}{n!}(x_0 -t)^n dt \neq 0 \, .
%\end{align*}
This contradiction concludes the proof.
\end{proof}

\begin{figure}[ht]
\begin{center}\includegraphics[height=4.2cm,width=4.5cm]{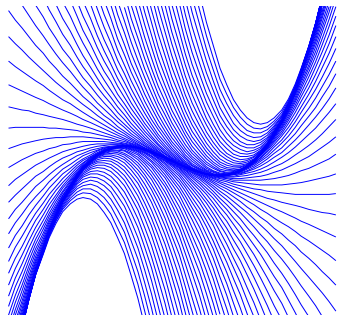}\end{center}
\caption[The real trace of $\mathcal W_f$]{The real trace of the $3$-web $\mathcal W_f$ for  $f=x(2x-1)(2x+1)$.}
\end{figure}

 If  $f \in \mathbb R[x]$ is a function of odd degree $k$ then the real trace of the $k$-web $\mathcal W_f$  defined
by the  graph of the  $(k-1)$-th osculating functions of $f$ is a continuous foliation in a neighborhood of $\Gamma_f$, the real graph
of $f$,  which is not differentiable since every point of  $\Gamma_f$ is tangent to some leaf of the foliation without being itself
a leaf. To summarize: the real trace of a holomorphic ( or even polynomial ) web can be a  non-differentiable, although continuous,  foliation.
\index{Foliation!non-differentiable}
\index{Web!real trace|)}
\index{Discriminant|)}

\section{Examples}\label{S:examples}
\subsection{Global webs on  projective spaces}

Let $\mathcal W=[\omega] \in \mathbb PH^0(\mathbb P^n, \mathrm{Sym}^k \Omega^1_{\mathbb P^n}\otimes \mathcal N)$ be a
$k$-web on $\mathbb P^n$.  The \defi[degree of $\mathcal W$] \index{Web!degree} is defined as the number of tangencies, counted with multiplicities,
 of $\mathcal W$ with a line not everywhere tangent to $\mathcal W$. More precisely, if $i: \mathbb P^1 \to \mathbb P^n$ is linear
 embedding of $\mathbb P^1$ into $\mathbb P^n$ then the points of tangency of the image line with $\mathcal W$ correspond to
 the zeros of $[i^* \omega] \in \mathbb P H^0(\mathbb P^1, \mathrm{Sym}^k \Omega^1_{\mathbb P^1} \otimes i^* \mathcal N )$. Notice that
 $i^* \omega$ vanishes identically if and only the image of $i$ is everywhere tangent to $\mathcal W$.

Recall that every line-bundle $\mathcal L$ on $\mathbb P^n$ is an integral multiple of $\mathcal O_{\mathbb P^n}(1)$ and consequently
one can write $\mathcal L = \mathcal O_{\mathbb P^n}(\deg(\mathcal L))$. Because the embedding $i$ is linear, the identity
 $i^* \mathcal O_{\mathbb P^n}(1) = \mathcal O_{\mathbb P^1}(1)$ holds true. Putting these two facts together
  with the identity $\mathrm{Sym}^k \Omega^1_{\mathbb P^1} = \mathcal O_{\mathbb P^1}(-2k)$ yields
 \[
 \deg( \mathcal W) = \deg(\mathcal N) - 2k \, .
 \]

\subsubsection{Characteristic numbers of projective webs}

Let $X \subset \mathbb P^n$ be an irreducible projective subvariety. The
projectivized conormal variety \index{Conormal variety} of $X$, \defi[conormal variety] of $X$ for short, is the
unique closed subvariety $\mathrm{Con}(X)$ of $\mathbb P T^* \mathbb P^n$ satisfying:
\begin{enumerate}
\item ${\pi(\mathrm{Con}(X)) = X}$, where $\pi: \mathbb P T^* \mathbb P^n \to \mathbb P^n$ is the natural projection;
\item the fiber $\pi^{-1}(x) \cap \mathrm{Con}(X)$ over any smooth point $x$ of $X$ is $\mathbb P T^*_x X \subset \mathbb P T^* \mathbb P^n$.
\end{enumerate}
The conormal variety of $X \subset \mathbb P^n$ can  succinctly be defined as
\[
\mathrm{Con}(X) = \overline{\mathbb P N^* X_{sm}} \, ,
\]
with $X_{sm}$ denoting the smooth part of $X$ and $N^* X_{sm}$ its conormal bundle.

For example, the conormal of a point $x \in \mathbb P^n$ is all the fiber $\pi^{-1}(x) = \mathbb P T^*_x \mathbb P^n$. More generally the
conormal variety of a linearly embedded $\mathbb P^i \subset \mathbb P^n$ is a trivial $\mathbb P^{n-i-1}$ bundle over $\mathbb P^i$.

If $W \subset \mathbb P T^* \mathbb P^n$ is the natural lift of $\mathcal W$ then the \defi[characteristic numbers] \index{Web!characteristic numbers} of  $\mathcal W$ on $\mathbb P^n$ are, by definition, the $n$ integers
\[
d_i ( \mathcal W ) =  W \cdot \mathrm{Con}(\mathbb P^i) \, ,
\]
with $i$ ranging from $0$ to $n-1$, and where    $A \mathbf{\cdot} B $   stands for the intersection product of $A$ and $B$.

Notice that $d_0(\mathcal W)$ counts the number of leaves of $\mathcal W$ through a generic point of $\mathbb P^n$, that is $\mathcal W$ is a $d_0(\mathcal W)$-web . The integer $d_1(\mathcal W)$ counts the number of points over a generic line $\ell$ where the web has a leaf with tangent space
containing $\ell$. Therefore $d_1(\mathcal W)$ is nothing more than the previously defined degree of $\mathcal W$.

\subsection{Algebraic webs revisited}\label{S:awr}

It is seems fair to say that the simplest $k$-webs on projective spaces are the ones of degree zero. Perhaps the best way to describe them
is through  projective duality.

Let $\check{\mathbb P}^n$ denote the projective space parametrizing hyperplanes in $\mathbb P^n$ and $\mathcal I \subset \mathbb P^n \times \check{\mathbb P}^n$
be the incidence variety, that is
\[
\mathcal I = \{ ( p , H ) \in \mathbb P^n \times \check{\mathbb P}^n \, | \, p \in H \} .
\]
The natural projections from $\mathcal I$ to $\mathbb P^n$ and $\check{\mathbb P}^n$ will be respectively denoted by $\pi$ and $\check \pi$.

\begin{prop}
The incidence variety $\mathcal I$ is naturally  isomorphic to $\mathbb P T^* \mathbb P^n$ and also to $\mathbb P T^* \check{\mathbb P}^n$. Moreover,
under these isomorphisms the natural projections $\pi$ and $\check \pi$ from $\mathcal I$ to $\mathbb P^n$ and $\check{\mathbb P}^n$,   coincide with
the projections from $\mathbb P T^* \mathbb P^n$ to  $\mathbb P^n$ and from $\mathbb P T^* \check{\mathbb P}^n$ to  $\check{\mathbb P^n}$ respectively.
\end{prop}
\begin{proof}
If one identifies $\mathbb P^n \times \check {\mathbb P}^n$ with $\mathbb P(V) \times \mathbb P(V^*)$ where $V$ is  a vector space $V$ of dimension $n+1$ then
the incidence variety can be identified with the projectivization of the locus defined on $V \times V^*$ through the vanishing of the natural pairing.
Combining this with the natural isomorphism between $V$ and $V^{**}$ the proposition follows.
For details  see \cite[Chapter 1, Section 3.A]{GKZ}
\end{proof}

Using this identification of $\mathcal I$ with $\mathbb P T^* \check {\mathbb P}^n$ one defines for every projective
curve $C \subset {\mathbb P}^n$  its dual web $\mathcal W_C$ as the one defined by variety ${\pi}^{-1}(C) \subset \mathbb P T^* \check{\mathbb P^n}$ seen
as a multi-section of $\check{\pi}: \mathbb P T ^* \check{\mathbb P}^n \to \check{\mathbb P}^n$.
At once one verifies that the germification of this global web at a generic point $H_0 \in \check{\mathbb P}^n$
coincides with the germ of  web $\mathcal W_C(H_0)$ defined in Section \ref{S:defalg}.

\begin{prop}\label{P:lowtech}
If $C \subset {\mathbb P}^n$ is a  projective curve of degree $k$ then $\mathcal W_C$ is a $k$-web of degree zero on $\check{\mathbb P}^n$.
Reciprocally,
if $\mathcal W$ is a $k$-web of degree zero on $\check{\mathbb P}^n$ then
there exists $C \subset {\mathbb P}^n$, a projective  curve of degree $k$, such that
$\mathcal W = \mathcal W_C$.
\end{prop}
\begin{proof}
If $C\subset {\mathbb P}^n$ is a  projective curve then all the leaves of $\mathcal W_C$ are hyperplanes. Therefore a line
tangent to $\mathcal W_C$  at a quasi-smooth  point is automatically contained in the leaf through that point. This is sufficient to
prove that $\mathcal W_C$ has degree $0$. Alternatively one can compute directly
\[
d_0(\mathcal W_C) = \pi^{-1} (C) \cdot \mathrm{Con}(\mathbb P^0) =0
\]
since $\pi^{-1} (C)$ does not intersect the conormal variety of any point $p= \mathbb P^0$ outside the dual variety $\check{C}$.

The proof of the reciprocal is similar and the reader is invited to work it out.
\end{proof}

\subsubsection{The discriminant of $\mathcal W_C$}

If $C$ is a smooth projective curve then the discriminant of $\mathcal W_C$  is nothing more than the set
of hyperplanes tangent to $C$ at some point. Succinctly,
\[ \Delta(\mathcal W_C) = \check C  \quad \text{when } C \text{ is smooth.}
\]
 For an arbitrary curve $C$ the discriminant of $\mathcal W_C$
will also contain the hyperplanes on $\check{\mathbb P}^n$ corresponding to singular points of $C$, and the
projective subspaces of codimension two dual to the lines contained in $C$.

Because for a plane curve of degree $k$, the normal bundle of the web  $\mathcal W_C$ is $\mathcal O_{\mathbb P^2}(2k)$, one has
\begin{align*}
\deg ( \Delta(\mathcal W_C) ) &= \deg( K_{\mathbb P^2}^{ k(k-1)} \otimes \mathcal O_{\mathbb P^2}(4k(k-1)))  \\ &= -3k(k-1) + 4k(k-1) = k (k-1)  \, ,
\end{align*}
according to Proposition \ref{P:degdisc}. In particular, one recovers the classical Pl\"{u}cker formula \index{Pl\"{u}cker formula} for the degree of
the dual of smooth curves:
\[
C \, \text{ smooth  and   }  \deg(C) = k \implies \deg(\check C) = k (k-1) \, .
\]

If $C$ has singularities then the lines dual to the singular points will also be part of $\Delta(\mathcal W_C)$. The multiplicity
with which this line appears  in the discriminant will vary according to the analytical type of the singularity. For instance
the lines dual to ordinary nodes will appear with multiplicity two, while the lines dual to ordinary cusps will appear with
multiplicity three. In particular for a degree $k$ curve with at $n$ ordinary nodes and $c$ ordinary cusps as singularities one obtains
another instance of Pl\"{u}cker formula
\[
\deg (\check C) = k(k-1) - 2n - 3 c \, .
\]

\subsubsection{The monodromy of $\mathcal W_C$}
Recall that a subgroup $G \subset \mathfrak S_k$  of the $k$-th symmetric group is $2$-transitive
if for any pair of pairs $(a,b), (c,d) \in \underline k^2$ there exists  $g \in G$ such that $g(a)=c$ and $g(b)=d$. To describe the monodromy of $\mathcal W_C$ for irreducible curves $C$
the simple lemma below will be useful.

\begin{lemma}
Let $G \subset \mathfrak S_k$ be a subgroup. If $G$ is $2$-transitive and contains a transposition then $G$ is the full
symmetric group.
\end{lemma}
\begin{proof}
It is harmless to assume that $G$ contains the transposition $(1 \, 2)$. Since $G$ is $2$-transitive for every
pair $(a,b) \in \underline k$ there exists  $g \in G$ such that $g(a)=1$ and $g(b)=2$. Therefore the transposition
\[
(a \, b ) = g^{-1} (1 \, 2 ) g
\]
belongs to $G$. Consequently, every transposition in $\mathfrak S_k$ belongs to $G$. Since $\mathfrak S_k$ is generated by transpositions
the lemma is proved.
\end{proof}

\begin{prop}
If $C$ is an irreducible projective curve on $\mathbb P^n$ of degree $k$
then the monodromy group of $\mathcal W_C$ is the full symmetric group.
\end{prop}
\begin{proof}
It is harmless to assume that $n>2$. Indeed, if $C\subset \mathbb P^2$ then just embed $\mathbb P^2$ linearly in $\mathbb P^3$ to obtain
a projective curve $C'\subset \mathbb P^3$. Notice that $\mathcal W_{C'}$ is the pull-back of $\mathcal W_C$
under the linear projection dual to the embedding, and that  both webs $\mathcal W_C$ and $\mathcal W_{C'}$  have
isomorphic monodromy groups.

The irreducibility of $C$ implies that $\mathcal W_C$ is indecomposable and consequently its monodromy group is
$1$-transitive. Let $p \in C \subset \mathbb P^n$  be a generic point and consider the hyperplane $H_p$ in $\check{\mathbb P}^n$
determined by it. Since $n>2$,  the restriction of $\mathcal W_C$ at $H_p$ is still an algebraic web. If $C'$ is the curve in $\mathbb P^{n-1}$ image
of the projection from $\mathbb P^n$ to $\mathbb P^{n-1}$ centered at $p$ then $(\mathcal W_C)_{|H_p}$  is projectively equivalent to
$\mathcal W_{C'}$. Since  the projection of irreducible curves are irreducible it follows that the monodromy of $\mathcal W_{C'}$ is also
transitive. This suffices to show that the monodromy of $\mathcal W_C$ is $2$-transitive.
Indeed, given $a\in \underline k$ (~$\underline k$ is
now identified with the set of leaves of $\mathcal W_C$ through a generic point and $\check {\mathbb P}^n$ ) by the transitivity of the monodromy
group, one can send $a$ to an arbitrary $c \in \underline k$. If one now considers the restriction of $\mathcal W_C$ to the
leaf corresponding to $c$ then the monodromy group of the restricted web is again transitive,
but now on the set  $\underline k - \{ c \}$ of cardinality $k-1$,
and  for an arbitrary pair  $b,d \in \underline k - \{ c \} $ there exists an element  that sends $b$ to $d$ while fixing $c$.

It remains to show that there exists a transposition on the monodromy group of $\mathcal W_C$. To do that one it suffices to
consider the case of algebraic webs on $\mathbb P^2$ after restricting to a suitable intersection of leaves. Suppose now that
$C$ is an irreducible plane curve and let $\ell$ be a simple tangent line of $C$, that is, $\ell$ is a tangent line of $C$ at a
smooth, non-inflection,  point $p \in C$ and  $\ell$ intersects $C$ transversely on the complement of $p$. In affine coordinates $(x,y)$
where $\ell = \{ y=0\}$ and $p$ is the origin the curve $C$ can be expressed  as the zero locus of $y-x^2 + \hot$. The intersection of
$C$ with the line $y= \epsilon$ is therefore of the form $(\sqrt{\epsilon} + \hot , \epsilon)$. Notice that the intersections
are exchanged when $\epsilon$ gives a turn around $0$. In the dual plane this reads as the existence of a transposition for the dual web.
\end{proof}

\subsubsection{Smoothness of $\mathcal W_C$}

\begin{prop}\label{P:gp1}
Let $C$ be an irreducible non-degenerate projective curve in $\mathbb P^n$.
If $H \in \check{\mathbb P}^n$ is a generic hyperplane then $\mathcal W_C(H)$ is a germ of
smooth web.
\end{prop}

By duality, the proposition is clearly equivalent to the so called {\it uniform position principle} \index{Uniform position principle}  for curves.
The proof  presented here follows closely \cite[pages 109--113]{ACGH}.

\begin{prop}\label{P:gp2}
If $C\subset \mathbb P^n$, $n\ge 2$, is an irreducible non-degenerate projective curve of degree $d\ge n$ then a generic  hyperplane $H$
intersects $C$ at $d$ distinct points.  Moreover, any $n$ among these $d$ points  span $H$.
\end{prop}

The restriction on the degree of $C$ is not really a hypothesis. Every non-degenerate curve on $\mathbb P^n$ have
degree at least $n$ as will be shown in Proposition \ref{P:ratdeg} of Chapter \ref{Chapter:AR}.

\begin{proof}[Proof of Propositions \ref{P:gp1} and \ref{P:gp2}]
Let $U=\check{\mathbb P}^n - \Delta(\mathcal W_C)$ and
 ${I \subset C^{n}  \times U}$ be the locally closed variety
defined by the relation
\[
(p_1, \ldots, p_{n} , H ) \in I \Longleftrightarrow p_1, \ldots,p_n \text{ are distinct points in } H \cap C .
\]
Because the monodromy group of $\mathcal W_C$ is the full symmetric group the variety $I$ is irreducible and in particular
connected. Moreover the natural projection to $U$ is surjective and has finite fibers. Therefore $I$ has dimension $n=\dim U$.

Let now $J \subset I$ be the closed subset defined by
 \[
(p_1, \ldots, p_{n} , H ) \in J \Longleftrightarrow p_1, \ldots,p_n \text{ are contained in  a } \mathbb P^{n-2}  .
\]
Since $C$ is non-degenerated,  one can choose $n$ distinct points on it which span a $\mathbb P^{n-1}$. Thus
$J$ is a proper subset of $I$. The irreducibility of $I$ implies $\dim J < \dim I = n$. Therefore the image of the projection to $U$
is a proper subset, with complement parametrizing hyperplanes intersecting $C$ with the wanted property.
\end{proof}

\subsection{Projective duality}\label{subsection:Projective duality}
Given a global $k$-web $\mathcal W$ on $\mathbb P^n$,  and its natural lift $W$ to $\mathbb P T^* \mathbb P^n \simeq \mathcal I$ it
is natural to enquire which sort of object $W$ induces on $\check{\mathbb P}^n$ through the projection $\check \pi$.

\smallskip

To answer such question, assume for a moment that $W \subset \mathcal I$ is irreducible, or equivalently that the monodromy
of $\mathcal W$ is transitive.

\smallskip

If the map $\check{\pi}_{|W}: W \to \check{\mathbb P}^n$ is surjective, then there exists a web $\check{\mathcal W}$ on $\check {\mathbb P^n}$ with lift
to $\check{\mathcal I} = \mathcal I$ equal to $W$.
The order of $\check{\mathcal W}$ is precisely the degree of $\check{\pi}_{|W}$, that is
$
d_0(\check{\mathcal W}) = d_{n-1}(\mathcal W) \, .
$

\medskip

In the two dimensional
case the degree of $\check{\pi}_{|W}$ is nothing more than the degree of $\mathcal W$. But beware that this
is no longer true when the  dimension is at least three.
To determine the degree of $\check{\mathcal W}$ notice that $\check{\mathcal W}$ is tangent to a line $\ell$ at a point $p$ if and only if one of the
tangent spaces of $\check{\mathcal W}$ at $p$ contains the line $\ell$. Therefore the number of tangencies of $\check{\mathcal W}$
and $\ell$ is the intersection
of $W$ with the conormal variety of $\ell \subset \mathbb P T^* \check{\mathbb P}^n$. In other words
$d_1(\check{\mathcal W}) = d_{n-2}(\mathcal W)$.

Arguing similarly, it follows that for $i$ ranging from $0$ to $n-1$ the  identity $d_i(\check{\mathcal W}) = d_{n-i-1}(\mathcal W)$
holds true.

\medskip

In order to deal with the case where $\check{\pi}_{|W}: W \to \check{\mathbb P^n}$ is not surjective it is convenient to
extend the definition of characteristic numbers to pairs $(X, \mathcal W)$, where $X \subset \mathbb P^n$ is an irreducible
projective variety and $\mathcal W$ is an irreducible web\begin{footnote}{When $X$ is a singular
variety, a web on $X$ is a web on its    smooth locus  which extends to a global web on any of its desingularizations.}\end{footnote} of codimension one on $X$. To repeat the same definition
as before all that is needed is a definition of the lift of $(X,\mathcal W)$ to $\mathbb P T^* \mathbb P ^n$. Mimicking the
definition of conormal variety for subvarieties of $\mathbb P^n$, define $\mathrm{Con}(X, \mathcal W)$, the conormal variety of the pair $(X, \mathcal W)$,
as the  closed subvariety of $\mathbb P T^* \mathbb P ^n$ characterized by the following conditions:
\begin{enumerate}
\item[(a)] $\mathrm{Con}(X, \mathcal W)$ is irreducible;
\item[(b)] $\pi(\mathrm{Con}(X, \mathcal W)) = X$;
\item[(c)] For a generic $x \in X$ the fiber $\pi^{-1}(x) \cap\mathrm{Con}(X, \mathcal W)$ is a union of
linear subspaces corresponding to the projectivization of the conormal bundles in $\mathbb P^n$  of the leaves of $\mathcal W$
through $x$.
\end{enumerate}

\smallskip

For every pair  $(X,\mathcal W)$, there exists a unique pair $\mathcal D(X,\mathcal W)$ on $\check{\mathbb P}^n$ with conormal
variety in $\mathbb P T^* \check{\mathbb P^n}$ equal to the conormal variety of $(X, \mathcal W)$, if the following conventions are
adopted.
\begin{enumerate}
\item[--] an irreducible codimension one web $\mathcal W$ on $\mathbb P^n$ is identified with the pair
$(\mathbb P^n,\mathcal W)$;
\item[--]  on an irreducible projective curve $C$ there is only one irreducible web, the $1$-web
$\mathcal P$ which has  as leaves the points of $C$, and;
\item[--] a projective curve $C$ is identified with the pair $(C,\mathcal P)$.
\end{enumerate}

In this terminology, Proposition \ref{P:lowtech} reads as
\[
\mathcal D( C) = \mathcal W_C \quad \text{ and } \quad \mathcal D(\mathcal W_C) = C.
\]
Notice that the pairs $(X,\mathcal W)$ come with naturally attached characteristic numbers
\[
d_i ( X, \mathcal W) = \mathrm{Con}(X,\mathcal W) \cdot \mathrm{Con} ( \mathbb P^i )  \, ,
\]
and these generalize the characteristic numbers for a web $\mathcal W$ on $\mathbb P^n$ as
previously defined.

\medskip

\begin{example}
Let $X$ be an irreducible
subvariety of codimension $q \ge 1$ on $\mathbb P^n$ and $\mathcal W$ and irreducible $k$-web
on $X$. Since $X$ has codimension $q$, for $0\le i \le q-1$ a generic $\mathbb P^i$ linearly embedded  in $\mathbb P^n$
does not intersect $X$. Therefore $d_i(X,\mathcal W)= 0$ for $i= 0 , \ldots, q-1$. A generic $\mathbb P^q$ will intersect
$X$ in $\deg(X)$ smooth points and over each one of these  points $\mathrm{Con}(\mathbb P^q)$ will intersect  $\mathrm{Con}(X,\mathcal W)$
in $k$  points. Therefore
$
d_q(X,\mathcal W)  = \deg( X) \cdot k \, .
$
\end{example}

With these definitions at hand it is possible to prove the following Biduality Theorem. Details
will appear elsewhere.

\begin{thm}
For any pair $(X,\mathcal W)$, where $X \subset \mathbb P^n$ is an irreducible projective subvariety
and $\mathcal W$ is an irreducible codimension one web on $X$, the identity
\[
\mathcal D \mathcal D (X, \mathcal W) = (X, \mathcal W)
\]
holds true. Moreover, for $i = 0 ,\ldots , n-1$,  the characteristic numbers of $(X,\mathcal W)$ and $\mathcal D(X,\mathcal W)$
satisfy
\[
d_i( X, \mathcal W) = d_{n-1-i}(\mathcal D(X,\mathcal W)) \, .
\]
\end{thm}

Finally to deal with arbitrary pairs $(X,\mathcal W)$, where $X$ is not necessarily irreducible nor $\mathcal W$ has
necessarily transitive monodromy, one  writes $(X,\mathcal W)$ as the superposition of irreducible pairs and
applies $\mathcal D$ to each factor. Everything generalizes smoothly.

\subsection{Webs attached to projective surfaces}\label{S:Webs attached to projective surfaces}
One particularly rich source of examples of webs on surfaces is the classical projective differential
geometry widely practiced   until the early beginning of the $\mathrm{XX^{th}}$ century. The simplest
example is perhaps the asymptotic webs on surfaces on $\mathbb P^3$ that are now described.

\subsubsection{Asymptotic webs}\index{Web!asymptotic|(}

Let $S \subset \mathbb P^3$ be a germ of smooth surface. As such it admits a parametrization
$[\varphi] : ( \mathbb C^2,0) \to \mathbb P^3$, projectivization of a map
$\varphi: (\mathbb C^2,0) \to \mathbb C^4 \setminus \{ 0 \}$.

The tangent plane of $S$ at the point $[\varphi(p)]$ is the determined by the vector subspace of $T_{\varphi(p)} \mathbb C^4$
generated by
\[
\varphi(p),  \frac{\partial \varphi } {\partial x} (p), \frac{\partial \varphi } {\partial y}(p) \, .
\]
A germ of smooth curve $C$ on $S$ admits a parametrization of the form $\varphi\circ \gamma(t)$ where  $\gamma:(\mathbb C,0) \to (\mathbb C^2,0)$
is an immersion. Its osculating plane at $\varphi \circ \gamma (t)$ is determined by the vector space generated by $(\varphi \circ \gamma) (t),
(\varphi \circ \gamma)' (t),(\varphi \circ \gamma)'' (t)$.
Although the vectors $(\varphi\circ \gamma)'(t)$ and $(\varphi \circ \gamma)'' (t)$ do depend on the choice of the parametrization $\varphi \circ \gamma$ of $C$,
the same is not true for the vector space generated by them and $(\varphi \circ \gamma) (t)$.

A curve $C$ is an \defi[asymptotic curve] \index{Asymptotic curve} of $S$ if, at every point $p$ of $C$,
its osculating plane is contained in the tangent
space of $S$. Since $(\varphi \circ \gamma)'(t)$ always belong to the tangent space of $S$ at $(\varphi \circ \gamma) (t)$,
the determinant ( where each entry represents a distinct row )
\begin{equation}\label{E:detas}
\det \left( (\varphi \circ \gamma)'' (t), \varphi ( \gamma (t)), \frac{\partial \varphi } {\partial x} (\gamma(t)) ,
\frac{\partial \varphi } {\partial y} (\gamma(t))\right)
\end{equation}
vanishes identically when $\gamma$ parametrizes an asymptotic curve. But
\[
 (\varphi \circ \gamma) '' (t)  = D^2 \varphi(\gamma(t)) \cdot \gamma'(t) \cdot \gamma'(t) + D \varphi(\gamma(t)) \cdot \gamma''(t)
\]
and the image of $D\varphi(\gamma(t))$ is always contained in the vector space generated by the last three rows of the above matrix. Hence
the vanishing of (\ref{E:detas}) is equivalent to the vanishing of
$$
\det \left( D^2 \varphi(\gamma(t)) \cdot \gamma'(t) \cdot \gamma'(t), \varphi ( \gamma (t)), \frac{\partial \varphi } {\partial x} (\gamma(t)) , \frac{\partial \varphi } {\partial y} (\gamma(t))\right) \, .
$$
This last expression can be rewritten as
\[
\gamma^*  ( a dx^2 + 2 b dx dy + c dy^2 )
\]
where
\begin{align*}
 a = &\,\det \left(\varphi_{xx} , \varphi, \varphi_x, \varphi_y \right)   \\
b =&\, \det \left(\varphi_{xy} , \varphi, \varphi_x, \varphi_y \right) \\
 \mbox{and}
\quad
c = &\,\det \left(\varphi_{yy} , \varphi, \varphi_x, \varphi_y \right) .
\end{align*}

It may happen that $a,b,c$ are all identically zero. It is well-know that this
is the case if and only if $S$ is contained in a hyperplane of $\mathbb P^3$. It may also
happen that although non-zero the $2$-symmetric differential form $a dx^2 + 2b dx dy + c dy ^2$
is proportional to the square of a differential $1$-form. This is the case, if and only if, the surface $S$
is developable. Recall that a surface is \defi[developable] \index{Developable surface} if it is contained in a plane,  a cone or
  the tangent surface of a curve.

In general for non-developable surfaces what one gets is a $2$-symmetric differential form that induces an (eventually singular)
$2$-web on $S$: the \defi[asymptotic web] of $S$.

The simplest example  is the asymptotic web of a smooth quadric $Q$ on $\mathbb P^3$. Since it is isomorphic to $\mathbb P^1 \times \mathbb P^1$
and under these isomorphism the fibers of both natural projections to $\mathbb P^1$ are lines on $\mathbb P^3$, it is clear
that the asymptotic web of $Q$ is formed by these two families of lines.
\subsubsection{Asymptotic webs -- Alternative take}

When  $S \subset \mathbb P^3$ is a smooth projective surface  the definition of the asymptotic web of $S$
is amenable to a more intrinsic formulation. Suppose that $S$ is cutted out  by an irreducible homogenous
polynomial $F \in \mathbb C[x_0,  \ldots, x_3]$. The Hessian matrix of $F$,
\[
\mathrm{Hess}(F) = \left(  \frac{\partial^2 F}{ \partial x_i \partial x_j} \right) \, ,
\]
when evaluated at the tangent vectors of $S$ gives rise to a morphism
\[
\mathrm{Sym}^2 TS \longrightarrow NS
\]
where $NS \simeq \mathcal O_S(\deg(F))$ is the normal bundle of $S\subset \mathbb P^3$.
This morphism is  usually called the (projective) \defi[second fundamental form] \index{Second fundamental form} of $S$. Dualizing it,
and tensoring the result by $NS$ one obtains a holomorphic section of $\mathrm{Sym}^2 \Omega^1_S \otimes NS$.

When $S$ is not developable ( which under the smoothness and projectiveness assumption on $S$ is equivalent to
$S$ not being a plane) this section, after factoring eventual codimension one components of its zero set, defines
a singular $2$-web on $S$. Its discriminant coincides, set theoretically, with the locus on $S$ defined by the
vanishing of $\mathrm{Hess}(F)$.\index{Web!asymptotic|)}

\subsubsection{The general philosophy}

One can abstract from the definition of asymptotic web the following procedure:
\begin{enumerate}
\item Take a linear system\begin{footnote}{Recall that a linear system is the projectivization $|V|$ of a finite dimensional
vector subspace $V \subset H^0(S,\mathcal L)$, where $\mathcal L$ is a line-bundle on $S$.
In the case $S$ is a surface germ, a linear system is nothing more than the projectivization
of a finite dimensional vector space of germs of functions.}\end{footnote} $|V|$ on a surface $S$;
\item Consider the elements of $|V|$ with abnormal singularities at a generic point $p$ of $S$;
\item If there are only finitely many abnormal elements of $V$ for a given generic point $p$ consider
the web with tangents at $p$ determined by the tangent cone of these elements.
\end{enumerate}

This kind of construction abounds in classical projective differential geometry.

\subsubsection{Darboux $3$-web}
Let $S \subset \mathbb P^3$ be a surface and consider the restriction to $S$ of the linear
system of quadrics $|\mathcal O_{\mathbb P^3}(2)|$.

If $S$ is generic enough then at a generic point $p \in S$ there are exactly three quadrics whose
restriction at $S$ is a  curve with first non-zero jet at $p$ of the form
\[
\ell_i(x,y)^3 \, \quad i = 1,2,3
\]
where $(x,y)$ are local coordinates of $S$ centered at $p$ and the $\ell_i$ are linear forms. These
three quadrics are the \defi[quadrics of Darboux] \index{Darboux!quadrics} of $S$  at $p$. For more details see \cite[pages 141--144]{Lane}.

In this way one defines  a $3$-web with tangents at $p$ given by $\ell_1, \ell_2, \ell_3$. This is
the \defi[Darboux $3$-web] of $S$. \index{Darboux!$3$-web}

\subsubsection{Segre $5$-web}\index{Segre!$5$-web|(} \label{S:Segre}

Let now $S$ be a surface on $\mathbb P^5$ and consider the restriction to $S$ of the linear system of
hyperplanes $| \mathcal O_{\mathbb P^5}(1) |$.

For a generic point $p$ in a generic surface $S$ there are exactly five hyperplanes which intersect
$S$ along a curve which has a tacnode\begin{footnote}{An {\it ordinary} tacnode is a singularity of curve with exactly two branches, both of them smooth, having
an ordinary tangency. Here tacnode refer to a curve cut out by a  power series of the form
\[
\ell(x,y)^2  + \ell(x,y) P_2(x,y) + h.o.t.
\] where $\ell$ is a linear form and $P_d$ is a homogeneous form of degree $2$.     }\end{footnote} singularity at $p$. The
five directions determined by these tacnodes are \defi[Segre's principal directions]. \index{Segre!principal directions}
By definition \defi[Segre's $5$-web]  is defined as the $5$-web determined pointwise by Segre's principal directions in the case where they are
distinct at a generic point of $S$.

There are surfaces such that through every point there are infinitely many principal directions. For instance the developable surfaces
--- planes, cones and tangent of curves --- do have this property and so do the degenerated surfaces, that is surfaces contained in a proper hyperplane
of $\mathbb P^5$. A remarkable theorem of Corrado Segre says that besides these examples,  the only surfaces in $\mathbb P^5$ \index{Segre!characterization of Veronese surface}
 with infinitely many principal directions through a generic  point are the ones contained in the Veronese surface obtained through the embedding of $
 \mathbb P^2$ into $\mathbb P^5$ given by the linear system $|\mathcal O_{\mathbb P^2}(2)|$.

 If ${\varphi: (\mathbb C^2,0) \to \mathbb C^6}$ is a parametrization of the surface $S$, then  Segre's $5$-web of $S$ is induced by
  the \mbox{$5$-symmetric} differential form
\[
\omega_{\varphi}  = \det \left(
 \begin{array}{c}
 \varphi \\
 \varphi_x \\ \varphi_y\\
 \varphi_{xx} dx + \varphi_{xy} dy \\ \varphi_{xy} dx + \varphi_{yy} dy
\\
\varphi_{xxx} dx^3 + 3\, \varphi_{xxy} dx^2dy +
3\, \varphi_{xyy} dxdy^2
  \varphi_{yyy} dy^3
\end{array}
  \right) \, .
\]

 It can be verified that once the parametrization $\varphi$ is changed by one of the
 form
 \[
 \hat \varphi(x,y) = \lambda(x,y) \cdot \varphi ( \psi(x,y) )
 \]
where $\lambda$ in a unit in $\mathcal O_{(\mathbb C^2,0) } $ and $\psi \in \mathrm{Diff}(\mathbb C^2,0)$ is a
germ of biholomorphism, then one has
\[
\omega_{\hat \varphi} =  \lambda^6 \cdot \det(D\psi)^2 \cdot \psi^* \omega_{\varphi}.
\]
This identity implies that the collection $\omega_{\varphi}$, with $\varphi$ ranging over germs of parametrizations of
$S$, defines a holomorphic section of
\[
\mathrm{Sym}^5 \Omega^1_S \otimes \mathcal O_S(6) \otimes K_S^{ \otimes 2} \, .
\]

A nice example is given by the cubic surface $S$ obtained as  the
image of the rational map from  $\mathbb P^2$ to  $\mathbb P^5$ determined by the linear system of
cubics passing through four  points $p_1, p_2, p_3,p_4 \in \mathbb P^2$ in general position.
At a  generic point  $p \in \mathbb P^2$, the five cubics in the linear system with a tacnode at $p$ are:
the union of conic through $p_1,p_2,p_3,p_4$ and $p$ with its tangent line at $p$; and for every $i \in \underline 4$, the
union of the line  $\overline{p p_i}$ with the conic through $p$ and all the $p_j$ with $j \neq i$ which is moreover tangent to $\overline{p p_i}$
at $p$. This geometric description makes  evident the fact that Segre's $5$-web of $S$ is nothing more than Bol's $5$-web $\mathcal B_5$ presented in Section \ref{S:bol}.
\index{Bol!$5$-web}
\index{Segre!$5$-web|)}

\chapter{Abelian relations}\label{Chapter:AR}
\thispagestyle{empty}

\index{Abelian relation|(}A central concept in this text is the one of \defi[abelian relation] for a
germ of quasi-smooth  web $\mathcal W$. Roughly speaking, abelian relations are additive functional equations among the first integrals of the defining foliations of
$\mathcal W$. More precisely, if $\mathcal W=[\omega_1 \cdots \omega_k ]$ is a germ
of quasi-smooth $k$-web  on $(\mathbb C^n,0)$ then an abelian relation of $\mathcal W$
is a $k$-uple of germs of $1$-forms  $\eta_1, \ldots, \eta_k$ satisfying the following three
conditions:
\begin{enumerate}
\item[(a)] for every $i \in \underline k$, the $1$-form $\eta_i$ is closed, that is, $d\eta_i=0$;
\item[(b)] for every $i \in \underline k$, the $1$-form $\eta_i$ defines  $\mathcal F_i$,  that is  $ \omega_i \wedge \eta_i=0$;
\item[(c)] the $1$-forms $\eta_i$ sum up to zero, that is $
\sum_{i=1}^k \eta_i = 0 .
$
\end{enumerate}

Notice that a  primitive of $\eta_i$  exists, since  $\eta_i$ is closed. Such primitive  is a first integral
of $\mathcal F_i$, because $\omega_i \wedge \eta_i=0$. In particular, if $\mathcal F_i$ is defined through
a submersion $u_i:(\mathbb C^n,0) \to (\mathbb C,0)$ then
\[
\int_0^z \eta_i = g_i(u_i(z)) \, ,
\]
for some germs of holomorphic functions $g_i:(\mathbb C,0) \to (\mathbb C,0)$.
Condition (c)  translates into
\[
\sum_{i=1}^k g_i\circ u_i= 0
\]
which is the functional equation among the first integrals of $\mathcal W$ mentioned at the beginning of the discussion.
An abelian relation $\sum_{i=1}^k \eta_i=0$  is
  \defi[non-trivial] \index{Abelian relation!non-trivial} if at least one of the $\eta_i$  is not identically zero.
  If none of the $1$-forms $\eta_i$ is identically zero then the abelian relation is called \defi[complete]. \index{Abelian relation!complete}

\smallskip

With the concept of abelian relation at hand Theorem \ref{T:hexagonal} can be rephrased
as the  following equivalences for a germ of smooth $3$-web on $(\mathbb C^2,0)$.
\[
{\mathcal W \text{ is hexagonal}} \iff {K(\mathcal W)=0} \iff
\begin{tabular}{c}
  ${\mathcal W}$ has a
non-trivial\\ abelian relation.
\end{tabular}
\]
To some extent, the main results of this text can be thought as generalizations of  this equivalence to
arbitrary webs of codimension one.

\medskip

It is clear from the definition of abelian relation that for a given germ of
quasi-smooth $k$-web $\mathcal W$ the set of all abelian relations of $\mathcal W$
forms a $\mathbb C$-vector space, the \defi[space of abelian relations of $\mathcal W$], which will  be denote by $\mathcal A(\mathcal W)$.
If $\mathcal W= \mathcal W(\omega_1, \ldots, \omega_k)$ then, one can write
\[
\mathcal A(\mathcal W) = \left \{ ( \eta_1, \ldots, \eta_k ) \in \big( \Omega^1(\mathbb C^n,0) \big)^k
\left|
\begin{tabular}{c}
$d \eta_i = 0$ \\
$ \omega_i \wedge \eta_i =0 $\\
$\sum_{i=1}^k \eta_i =0$ \\
\end{tabular} \right. \right\} \, .
\]
 Notice  that a germ of  diffeomorphism $\varphi$ establishing an equivalence between two germs of  webs $\mathcal W$ and $\mathcal W'$, induces
a natural isomorphism between their spaces of abelian relations.

One of the  main goals  of this chapter is to prove that $\mathcal A(\mathcal W)$ is indeed
a finite dimensional vector space and that its dimension --  the \defi[rank of $\mathcal W$], denoted by $\mathrm{rank}(\mathcal W)$ --
is bounded by Castelnuovo's number
\[
\pi (n,k)  =
\sum_{j=1}^{\infty} \max (0 , k - j(n-1) - 1) \, ,
\]
when $\mathcal W$ is a germ of smooth $k$-web on $(\mathbb C^n,0)$.

\smallskip
This bound, by the way, was proved by Bol when $n=2$, and was generalized by S.-S.~Chern in his PhD thesis under the direction of Blaschke.
Before embarking in its  proof, carried out in Section \ref{S:BCbound},
the  determination of the space of abelian relations for planar webs in some particular
cases is discussed in Section \ref{S:geral}.

\smallskip

Of tantalizing importance for what is to come later in Chapter \ref{Chapter:Trepreau}, is the content of Section \ref{S:constraints}. There Castelnuovo's
results on  the   geometry of point sets   on projective space are proved, and from them are deduced constraints on the
geometry of webs attaining Chern's bound.

\dd

The space of abelian relations of a global web  is no longer a vector space but a local system defined
on an open subset containing the complement of the discriminant of the web. This can be inferred from
the results by Pantazi-H\'{e}naut expound  in Section \ref{S:Pantazi-Henaut} of Chapter \ref{Chapter:6}.
For an elementary and simple argument see \cite[Th\'{e}or\`{e}me 1.2.2]{PTese}. Note that both approaches mentioned  above
deal a priori with webs on surfaces, but there is no real difficult to deduce from them the general case.

\index{Abelian relation|)}

\section{Determining the abelian relations}\label{S:geral}\index{Abel's method|(}

If  $\mathcal W$ is a quasi-smooth $k$-web on $(\mathbb C^n,0)$ and
$\varphi:(\mathbb C^2,0)\rightarrow (\mathbb C^n,0)$  is a   generic holomorphic immersion then
$\varphi^* \mathcal W$ is a smooth $k$-web on $(\mathbb C^2,0)$, and $\varphi$ induces naturally an injection of $\mathcal A(\mathcal W)$ into $\mathcal A(\varphi^*W)$.
Thus,  the specialization to the two-dimensional case,  in vogue  up to the end of this section, is not  seriously restrictive.

\subsection{Abel's method} Before the awaking of web geometry,  Abel already studied functional
equations of the  form
\[
\sum_{i=1}^k g_i \circ u_i = 0
\]
for given  functions $u_i$ depending on two variables. In his first published paper \cite{Abel}, \index{Abel} he devised
a method to determine the functions $g_i$  satisfying  this functional equation.  Abel's method
will not be presented in its full generality but instead  the particular case where
all but one  of the functions $u_i$ are homogenous polynomials of degree one will be carefully scrutinized following \cite{PT}.
Under these additional assumptions, Abel's method is remarkably simplified but still leads to interesting examples of functional equations and
webs.  For a comprehensive account and modern exposition of Abel's method  in the context of web geometry, the reader can consult \cite{Piriopoly}.

\medskip

For $i \in \underline  k$,  let $u_i(x,y) = a_i x + b_i y $
where $a_i,b_i \in \mathbb C$ are complex numbers satisfying $a_i b_j - a_j b_i \neq 0$ whenever $i \neq j$.
These conditions imply the smoothness of the $k$-web  $\mathcal W( u_1,  \ldots, u_{k} )$.
It will be  convenient to consider  the vector fields $v_i = b_i \partial_x - a_i \partial_y$ which define the very same foliation
as the submersions $u_i$.

Let $u:(\mathbb C^2,0) \to (\mathbb C,0)$ be a germ of holomorphic  submersion  satisfying $v_i (u)\neq 0$ for
every $i \in \underline k$, and consider the smooth  $(k+1)$-web $\mathcal W = \mathcal W(u_1,\ldots, u_k, u)$ on $(\mathbb C^2,0)$.
To determine the rank of $\mathcal W$, it suffices  to look for holomorphic solutions $g, g_1, \ldots, g_k:(\mathbb C,0)\to(\mathbb C,0)$  of the equation
\[
g \circ u = \sum_{i=1}^k g_i \circ u_i \, .
\]
For that sake,
apply the derivation $v_1$ to both sides of the equation above
to obtain
\[
 g'(u) \cdot v_1(u) = \sum_{i=2}^k  {g_i}'(u_i) \cdot v_1(u_i) \, .
\]
Notice that $u_1$ no longer appears in the right hand-side.

Apply now the derivation $v_2$ to this new equation. Use the commutativity of $v_1$ and $v_2$, that is  $[v_1,v_2]=0$, to get
\[
g''(u) \cdot v_2(u) \cdot v_1(u) +g'(u) \cdot v_2(v_1(u)) = 
 \sum_{i=3}^k  {g_i}''(u_i) \cdot v_2(u_i) \cdot v_1(u_i) +  (g_i)'(u_i) \cdot v_2(v_1(u_i)) \, .
\]

%\begin{multline*}
%g''(u) \cdot v_2(u) \cdot v_1(u) +g'(u) \cdot v_2(v_1(u)) = \\
 %\sum_{i=3}^k  {g_i}''(u_i) \cdot v_2(u_i) \cdot v_1(u_i) +  (g_i)'(u_i) \cdot v_2(v_1(u_i)) \, .
%\end{multline*}

Iterating this procedure one arrives at a equation of the form
\[
 \left( \prod_{i=1}^k v_i(u) \right)    g^{(k)}(u) + \cdots + v_k(v_{k-1}(\cdots v_1(u)) ) g'(u) = 0\,
\]
which after dividing by the coefficient of $g^{(k)}(u)$ can be written as
\[
g^{(k)} ( u)  = \sum_{i=1}^{k-1} h_i g^{(i)} (u)
\]
where the $h_i$ are germs of meromorphic functions.

Let  $v= u_y \partial_x -  u_x \partial_y$ be
the hamiltonian vector field of $u$. If for some $i$ the function $v(h_i)$ is not identically zero then
one can apply the derivation $v$ to the above equation in order to reduce the order of it. Otherwise the functions
$h_i$ are functions of $u$ only and not of $(x,y)$, that is $h_i = h_i (u)$.

 Eventually
one arrives at a linear  differential equation of the form
\begin{equation}\label{E:tt}
g^{(\ell)} ( u)  = \sum_{i=1}^{\ell-1} h_i(u) g^{(i)} (u)
\end{equation}
with $\ell \le k$. Thus the possibilities for $g$ are reduced to a finite dimensional vector space: the space of
solutions of (\ref{E:tt}).

After discarding the constant solutions of (\ref{E:tt})  one notices at this point that
\begin{equation*}
\mathrm{rank} \, \mathcal W \le \mathrm{rank} \, \mathcal W(u_1, \ldots, u_k)  + k-1 \,
\end{equation*}
when $u_1, \ldots, u_k$ are linear homogeneous polynomials. Beware that this is no longer true, if the linear polynomials
are replaced by arbitrary submersions. The point being that the hamiltonian vector fields $v_i$ no longer commute. One can still work his way out to deduce that an equation as (\ref{E:tt}) will still
hold true, as  done in  \cite{Piriopoly}, but it will be no longer true that $\ell$ is bounded by $k$.

\medskip

\begin{example}\label{E:llll}
Let $u_1, \ldots, u_{k} \in \mathbb C[x,y]$ be homogeneous linear polynomials. Suppose that they
are pairwise linearly independent and let
$\mathcal W = \mathcal W(u_1, \ldots, u_k)$ be the induced $k$-web.
Then
\[
\mathrm{rank} (\mathcal W ) = \frac { (k-1)(k-2) } {2}
\]
\end{example}
\begin{proof}
The proof goes by induction. For $k =2$  there is no abelian relation.  Suppose the result holds
for $k \ge 2$. That is every parallel $k$-web has rank $(k-1)(k-2)/2$.
Looking for solutions of
\[
g \circ u_{k+1} = \sum_{i=1}^k g_i \circ u_i
\]
following the strategy explained above one arrives at the equation
$
g^{(k)} ( u_{k+1} ) =  0 .
$
Thus $g$ must be a polynomial in $\mathbb C[t]$ of degree at most $(k-1)$. Imposing that $g(0)=0$ leaves a vector
space of dimension $k-2$ to chose $g$ from. Hence the rank of $\mathcal W_{k+1}= \mathcal W(u_1, \ldots, u_{k+1})$
is bounded by $(k-1)(k-2)/2 + (k-2)$.

But for every positive integer $ j \le  {k-1}$, a dimension count shows that
$({u_{k+1})^{j} = \sum_{i=1}^k \lambda_{i,j } \cdot (u_i)^{j}}$  for suitable $\lambda_{i,j} \in \mathbb C$.
It follows  that the rank of $\mathcal W_{k+1}$ is $k(k-1)/2$ as wanted.
\end{proof}

In the particular case under analysis one can make
use of the following lemma.

\begin{lemma}\label{L:ss}
Suppose as above that the functions $u_i$ are linear homogenous and, still as above, let
$v_i$ be the hamiltonian vector field of $u_i$.
The following assertions are equivalent:
\begin{enumerate}
\item[(a)] the function $g(u)$ is of the form $\sum_{i=1}^k g_i (u_i)$;
\item[(b)] the identity $v_1v_2\cdots v_k(g(u))=0$ holds true.
\end{enumerate}
\end{lemma}
\begin{proof}
Clearly (a) implies (b). The converse will be proved by induction. For $k=1$ the result is evident. By induction hypothesis,
\[
v_k(g(u)) =  \sum_{i=1}^{k-1} h_i(u_i)\,.
\]
If $H_i$ is a primitive of $h_i$ then $v_k(H_i(u_i)) = h_i(u_i) \cdot v_k(u_i)$. Because $ v_k(u_i)$ is a non-zero constant when $i < k$,
one can write
\[
v_k\big( g(u ) - \sum_{i=1}^{k-1} v_k(u_i)^{-1}H_i(u_i) \big) =0 \, .
\]
To conclude it suffices to apply the basis of the induction.
\end{proof}

\begin{prop}\label{P:TTTT}
Let $\mathcal W$ be as in Example \ref{E:llll}.
Let also  ${u:(\mathbb C^2,0) \to (\mathbb C,0)}$ be a submersion and $\mathcal F$  be the induced foliation.
If the $(k+1)$-web $\mathcal W \boxtimes \mathcal F$ is smooth then
\[
\mathrm{rank} (\mathcal W  \boxtimes \mathcal F)  \le \frac{ k(k-1) } { 2 }  \, .
\]
Moreover equality holds if and only if $\ell = k$ in equation (\ref{E:tt}).
\end{prop}
\begin{proof}
First notice that equation (\ref{E:tt}) has been derived by first developing formally
\[
v_1 v_2 \cdots v_k \big( g( u ) \big) = \sum_{i=1}^k f_i(x,y) g^{(i)}(u) \, ,
\]
then  dividing by $f_k$, setting $h_i = f_i / f_k$ and  deriving repeatedly with respect to the hamiltonian vector field of $u$
 until arriving at an  equation depending only on $u$. Therefore, Lemma \ref{L:ss} implies that any
 solution $g$ of
 \[
 g(u) = \sum_{i=1}^k g_i (u_i)
 \]
will  also be a solution of (\ref{E:tt}). But if $\mathrm{rank} (\mathcal W  \boxtimes \mathcal F)  = \frac{ k(k-1) } { 2 }$
then equation (\ref{E:tt}) has to have at least $k-1$ non-constant solutions vanishing at zero. Consequently $k=\ell$.

Reciprocally if $k=\ell$ then there will be $k-1$  non-constant solutions vanishing at zero for $v_1 v_2 \cdots v_k ( g( u ) )=0$.
Lemma \ref{L:ss} implies that $\mathrm{rank} (\mathcal W  \boxtimes \mathcal F)  = \mathrm{rank}(\mathcal W) + (k-1)$.
\end{proof}

Webs as $\mathcal W \boxtimes \mathcal F$ of the proposition above obtained from the superposition of a parallel web and one non-linear foliation  will be called \defi[quasi-parallel] webs.  \index{Web!quasi-parallel}

\begin{example}\label{E:exabel}\rm
let $u : (\mathbb C^2,0) \to \mathbb C$ be a submersion of the form $u(x,y) = a(x) + b(y)$.
Assume that the quasi-parallel  $5$-web  ${\mathcal W(x,y,x-y,x+y, u(x,y) )}$ is smooth, that is
$a_xb_y(a_x^2-b_y^2)(0) \neq 0$. A straightforward computation shows that
\[
v_1v_2v_3v_4 \big( g(u) \big) = g''''(u)a_x b_y(a_x^2-b_y^2) +  
3g'''(u) a_x b_y(a_{xx}-b_{yy}) + g''(u)(b_ya_{xxx}-a_x b_{yyy}).
\]
%\begin{multline*}
%v_1v_2v_3v_4 \big( g(u) \big) = g''''(u)a_x b_y(a_x^2-b_y^2) +  \\
%3g'''(u) a_x b_y(a_{xx}-b_{yy}) + g''(u)(b_ya_{xxx}-a_x b_{yyy}).
%\end{multline*}

Proposition \ref{P:TTTT} implies that $\mathcal W$ has rank equal to $6$ if and only if
\begin{align*}
  v_u \left(\frac{a_{xx}-b_{yy}}{a_x^2-b_y^2}\right)=0 \qquad \mbox{and} \qquad
  v_u \left(\frac{a_y b_{xxx}-a_x b_{yyy}}{a_x b_y(a_x^2-b_y^2)}\right)=0.
\end{align*}

The simplest functions $u(x,y)=a(x) + b(y)$ satisfying  this system of partial differential equations are
 $ x^2 + y^2$,   $x^2 - y^2$,   $ \exp{x} + \exp{y}$, $ \log( \sin x \sin y)$ and
  $ \log(\tanh x \tanh y)$. But these are not all. There is a  continuous family of solutions that can be written
  with the help of theta functions of elliptic curves.

  The $5$-webs of the form ${\mathcal W(x,y,x-y,x+y, a(x) + b(y) )}$
  with rank equal to $6$, have been completely classified  in \cite{PT}  through  a careful analysis of the above
  system of PDEs. \end{example}

If nothing else, this example shows  how involved can be the search for  webs of high rank, even in considerably simple cases.

\index{Abel's method|)}
\subsection{Webs with infinitesimal automorphisms}\label{S:infinitesimalggg}

Let $\mathcal F$ be a germ of smooth  foliation on $(\mathbb C^2,0)$
induced by a germ of $1$-form $\omega$. A germ of  vector field $v$ is an \defi[infinitesimal
automorphism  of $\mathcal F$] \index{Foliation!infinitesimal automorphism}if the foliation $\mathcal F$
is preserved by the local flow of $v$. In algebraic terms:
$ L_v \omega \wedge \omega = 0$ where $L_v = i_vd + di_v$ is the \defi[Lie derivative] \index{Lie derivative} with respect
to the vector field $v$. Those not familiar with the Lie derivative can
 find its basic algebraic properties in \cite[Chapter IV]{Godbillon}.

When the infinitesimal automorphism $v$ is transverse to $\mathcal F$, that is  $\omega(v)\neq0$,
then an elementary  computation (see  \cite[Corollary 2]{Percy} or \cite[Chapter III Section 2]{CerveauMattei} ) shows that the
$1$-form
\[
   \alpha = \frac{\omega}{i_v \omega}
\]
 is   closed and satisfies $L_v \alpha =0$. By definition, the integral
\[
u(z) = \int^z_0 \alpha
\]
is the \defi[canonical first integral] \index{Canonical first integral} of ${\mathcal F}$ with respect to $v$.
Clearly $u(0)=0$ and $L_v(u)= 1$. In particular the latter equality implies that $u$ is a germ of submersion.

The canonical first integral admits a nice {\it physical} interpretation: its value at $z$
measures the time which the local flow of $v$ takes to move the leaf through zero to the leaf through $z$.

Now let $\mathcal W= \mathcal W( \omega_1, \ldots, \omega_k)$ be a germ of smooth $k$-web on $(\mathbb C^2,0)$
 and let
 $v$ be an  \index{Web!infitesimal automorphism} \defi[infinitesimal
automorphism of $\mathcal W$], in the sense that $v$ is an infinitesimal automorphism of all the foliations defining $\mathcal W$.

By hypothesis, one has   $L_v\, \omega_i \wedge \omega_i=0$ for every $i \in \underline k$. Because $L_v$  commutes with
$d$, it induces a linear map
\begin{eqnarray}
\label{themap!}
L_v : \mathcal A(\mathcal W) & \to & \mathcal A(\mathcal W) \\
(\eta_1,\ldots ,\eta_k) & \mapsto& (L_v\eta_1 ,\ldots, L_v \eta_k)
\, .  \nonumber
\end{eqnarray}
A simple analysis of
the $L_v$-invariants subspaces of $ {\mathcal A}(\mathcal W)$ will provide valuable information
about the abelian relations of webs admitting infinitesimal automorphisms.

%-----------------------------------------------------------------------------------------------
%  SECTION 3
%-----------------------------------------------------------------------------------------------

\subsubsection{Description of ${\mathcal A}({\mathcal W})$ }
Suppose that ${\mathcal W}=\mathcal F_1 \boxtimes \cdots \boxtimes \mathcal F_k $
is a smooth  $k$-web on $({\mathbb C}^2,0)$ which admits
an infinitesimal automorphism $v$, regular and transverse to all the foliations ${\mathcal F}_i$.

Let $i\in \underline k$  be fixed.
Set  ${\mathcal A}_i({\mathcal W})$
as the vector subspace of $\Omega^1(\mathbb C^2,0)$ spanned
by the $i$-th components $\eta_i$ of
abelian relations ${(\eta_1,\ldots,\eta_k)\in {\mathcal A}({\mathcal W})}$. \index{Abelian relation!$i$-th component}
In other words, if $p_i: \Omega^1(\mathbb C^2,0)^k \to \Omega^1(\mathbb C^2,0)$ is the projection to the $i$-th factor
then
\[
\mathcal A_i(\mathcal W) = p_i(\mathcal A(\mathcal W)) \, .
\]

If $u_i=\int \alpha_i$ is the canonical first integral of ${\mathcal F}_i$
with respect to $v$, then for $\eta_i \in {\mathcal A}_i({\mathcal W})$,
there exists a  germ $f_i \in \mathbb C\{t\}$
for which $\eta_i=f_i(u_i)\,du_i$.

Assume now that ${\mathcal A}_i({\mathcal W})$ is not the zero vector space  and
let \[
\big\{\eta_i^\nu=f_\nu(u_i)\,du_i \, | \, \nu \in \underline {n_i} \}
\]
be a basis of it, consequently  $n_i = \dim {\mathcal A}_i({\mathcal W})$.
Since $L_v :  \mathcal A_i(\mathcal W) \to  \mathcal A_i(\mathcal W)$
is a linear map, there exist complex constants $c_{\nu \mu}$ such that
\begin{equation}
\label{baz}\quad \quad
L_v(\eta_{i}^\nu)=\sum_{\mu=1}^{n_i}
c_{\nu\mu}\, \eta_{i}^\mu
\, , \quad \quad \, \nu \in \underline{n_i} \, .
\end{equation}

But for any $\nu \in \underline{n_i}$, the  identity below holds true,
\begin{align*}
L_v(\eta_i^\nu)&=L_v\big(f_\nu(u_i)\,du_i\big) \\
&= v\big(
f_\nu(u_i)\big)du_i+f_\nu(u_i)\,L_v\big(du_i\big)=f_\nu'(u_i)\,{du_i} \, .
\end{align*}
Thus the relations (\ref{baz}) are equivalent to the following
\begin{eqnarray}
\label{diffeq}
f_\nu'
=\sum_{\mu=1}^{n_i} c_{\nu \mu}\, {f_\mu} \, , \quad \quad \, \nu \in \underline{n_i} \, .
\end{eqnarray}

Now let $\lambda_1,\ldots,\lambda_\tau \in {\mathbb C}$ be the
eigenvalues of the map $L_v$ acting on ${\mathcal A}({\mathcal W})$
corresponding to  Jordan blocks of respective dimensions
 $\sigma_1,\ldots,\sigma_\tau $. The system of linear differential equations
 (\ref{diffeq}) provides the following description of $ {\mathcal A}({\mathcal W})$.

\begin{prop} \label{description}
The abelian relations of $ {\mathcal W}$ are of the form
 $$ P_1(u_1)\,e^{\lambda_i\,u_1}\,du_1+\cdots+ P_k(u_k)\,e^{\lambda_i\,u_k}\,du_{k}=0
$$
where $P_1, \ldots, P_k$ are  polynomials of degree less   or equal to $ \sigma_i$.
\end{prop}

Proposition \ref{description} suggests an approach  to effectively determine $\mathcal{A}(\mathcal{W})$.
Once  the possible non-zero eigenvalues of the map (\ref{themap!}) are restricted to a finite set then
 the abelian relations can be found by simple linear algebra.

To restrict the possible eigenvalues first notice that  $0$ is an eigenvalue of (\ref{themap!})
if  and only if  for every germ of vector field $w$ the Wronskian determinant \index{Wronskian determinant}
\begin{equation}
\label{wronskian0}
\det\left(\begin{array}{ccc}
u_{1} & \cdots & u_{k} \\
w(u_{1}) & \cdots & w(u_{k}) \\
\vdots & \ddots & \vdots \\
w^{k-1}(u_{1}) & \cdots & w^{k-1}(u_{k})
\end{array}\right)
\end{equation}
is identically zero.
In fact, if this is the case
then there are two possibilities: the functions $u_{1},\ldots,u_{k}$ are $\mathbb{C}$-linearly dependent or all the orbits of $w$
are cutted out by some element of the linear system
generated by $u_{1},\ldots,u_{k}$, see \cite[Theorem~4]{Pereira}.
In particular if $w$ is a vector field of the form
$w=\mu x\frac{\partial}{\partial x}+y\frac{\partial}{\partial y}$, with $\mu\in\mathbb{C}\setminus\mathbb{Q}$, then the leaves of $w$
accumulate at $0$ and are  cutted out by no holomorphic function.
Therefore the vanishing of (\ref{wronskian0}) implies the existence of an abelian relation of the form
$$\sum c_{i}u_{i}=0,$$
where the $c_{i}$'s are complex constants.

To determine the possible  complex numbers $\lambda$ which are eigenvalues of the map (\ref{themap!}) first notice that these
corresponds to a functional equation of the form
$ c_1\,e^{\lambda\,u_1}+\cdots+ c_k\,e^{\lambda\,u_k}={cst.}$
where, as before, the $c_i$'s are complex constants.
In the same spirit of what has just been made for the zero eigenvalue case consider the holomorphic function given by
\begin{equation}
\label{wronskian-lambda}
\det\left(\begin{array}{ccc}
\exp(\lambda u_{1}) & \cdots & \exp(\lambda u_{k}) \\
w(\exp(\lambda u_{1})) & \cdots & w(\exp(\lambda u_{k})) \\
\vdots & \ddots & \vdots \\
w^{k-1}(\exp(\lambda u_{1})) & \cdots & w^{k-1}(\exp(\lambda u_{k}))
\end{array}\right)
\end{equation}
for an arbitrary germ of vector field $w$.

The Wronskian determinant \index{Wronskian determinant} (\ref{wronskian-lambda}) is of the form
$$\exp(\lambda(u_{1}+\cdots+u_{k}))\lambda^{k-1}P_{w}(\lambda)\, ,$$
where $P_{w}$ is a polynomial in $\lambda$, of degree at most
$\frac{(k-1)(k-2)}{2}$, with germs of holomorphic functions as coefficients. The common constant roots of these polynomials, when $w$ varies, are exactly the eigenvalues of the map (\ref{themap!}).

\begin{example}
The $k$-web $\mathcal{W}$ induced by the functions
$f_{i}(x,y)=y+x^{i}$, $i=1,\ldots,k$, has no abelian relations.
\end{example}
\begin{proof}
Notice that the vector field $v=\frac{\partial}{\partial y}$ is an infinitesimal automorphism of $\mathcal{W}$
and $v(df_{i})=1$, for every $i \in \underline k$. It follows that $u_{i}=f_{i}$ are the canonical first integrals of $\mathcal{W}$.
On the other hand  for the vector field $w=\frac{\partial}{\partial x}$,
$P_{w}(\lambda)|_{x=y=0}=(-1)^{k-1}\prod_{n=1}^{k-1}n!$. Consequently, the only candidate for an eigenvalue of the map (\ref{themap!}) is $\lambda=0$. Because  the functions $f_{i}$ are linearly independent over $\mathbb{C}$ the web $\mathcal W$ carries no
abelian relations at all.
\end{proof}

The next example determines the abelian relations of one of the $5$-webs of rank $6$
discussed in Example \ref{E:exabel}.

\begin{example}
The radial vector field $R=x\frac{\partial}{\partial x}+y\frac{\partial}{\partial y}$ is an infinitesimal automorphism of
the $5$-web $\mathcal W= \mathcal W ( x,y,x+y,x-y,x^2+y^2)$.
 The canonical first integrals are $u_{1}=\log x$, $u_{2}= \log y$, $u_{3}= \log (x+y)$, $u_{4}= \log(x-y)$ and  $u_{5}=\frac{1}{2}\log(x^2 + y^2)$.

If $w=x\frac{\partial}{\partial x} - y\frac{\partial}{\partial y}$ then  $P_{w}$ is a complex multiple of
\[
x^7y^7\lambda(\lambda-1)^{2}(\lambda-2)^2(\lambda-4)(\lambda-6) .
\]
According to Proposition \ref{description}, it suffices to look for abelian relations of the form
$\sum_{i=1}^{5}P_{i\lambda}(\log f_{i})f_{i}^{\lambda}\frac{df_{i}}{f_{i}}=0$, for $\lambda=0,1,2,4,6$, where $P_{i\lambda}$ are polynomials and
$f_{i}=\exp(u_{i})$.

Looking first for abelian relations where the polynomials $P_{i \lambda}$ are constant polynomials one has to
find linear dependencies between the linear polynomials $f_{1},\ldots,f_{4}$ and the degree $\lambda$ polynomials $f_{1}^{\lambda},\ldots,f_{5}^{\lambda}$ for $\lambda=2,4,6$.

For $\lambda=1$ there are two linearly independent abelian relations
$$f_{1}+f_{2}-f_{3}=0,\qquad f_{1}-f_{2}-f_{4}=0.$$
For $\lambda=2$ there are another  two:
$$f_{1}^{2}+f_{2}^{2}-f_{5}=0,\qquad 2f_{1}^{2}+2f_{2}^{2}-f_{3}^{2}-f_{4}^{2}=0.$$
Finally, there is one abelian relation for each  $\lambda \in \{ 4,6 \}$:
$$5f_{1}^{4}+5f_{2}^{4}+f_{3}^{4}+f_{4}^{4}-6f_{5}^{4}=0,\qquad 8f_{1}^{6}+8f_{2}^{6}+f_{3}^{6}+f_{4}^{6}-10f_{5}^{6}=0.$$
According to Proposition \ref{P:TTTT}, $\mathrm{rank}(\mathcal W)\le 6$. Hence
the abelian relations above generate $\mathcal A(\mathcal W)$.
\end{example}
%
%\begin{example}
%\rm Let $\mathcal{W}$ be the germ at a generic point $(x_{0},y_{0})$ of $\mathbb{C}^{2}$ of the
%algebraic web induced by $k$ homogeneous linear forms $f_{i}=a_{i}x+b_{i}y$. Clearly the radial vector field $R=x\frac{\partial}{\partial x}+y\frac{\partial}{\partial y}$ is an infinitesimal automorphism of $\mathcal{W}$. The canonical first integrals $u_{i}$ are the functions $$u_{i}(x,y)=\log\left(\frac{ f_{i}(x,y)}{f_{i}(x_{0},y_{0})}\right).$$ Taking for instance $Y=\frac{\partial}{\partial x}$, a straightforward computation shows that $P_{Y}(\lambda)$ is a constant multiple of $\prod_{i=1}^{k-1}(\lambda-i)^{k-i}$, so that the characteristic exponents are among $1,2,\ldots,k-1$.
%%
%We note that there are $k-i+1$ linearly independent abelian relations of the form
%$c_{1}f_{1}^{i}+\cdots +c_{k}f_{k}^{i}=0$ in the vector space of homogeneous polynomials of degree $i$, for $i=1,\ldots,k-1$. The total number is then $\frac{(k-1)(k-2)}{2}=\mathrm{rk}(\mathcal{W})$.
%\end{example}

%%%%%%%%%%%%%%%%%%%%%%%%%%%%%%%%%%%%%%%%%%%%%%%%%%%%%%%%%%%%%%%%%%%%%%%%%%%%%%%%%%%%%%%%%%%%%%%

\section{Bounds for the rank}\label{S:BCbound}
\index{Rank!bound|(}
Let $\mathcal W = \mathcal F_1 \boxtimes \cdots \boxtimes \mathcal F_k= \mathcal W( \omega_1 \cdots \omega_k )$ be a germ of quasi-smooth $k$-web on $(\mathbb C^n,0)$.

For every  positive integer $j$, define $\mathcal L^j(\mathcal W)$ as the vector subspace of the $\mathbb C$-vector space
$\mathrm{Sym}^j ( \Omega^1_0(\mathbb C^n,0) )$ generated by $\{ \omega_i^j(0) ; i \in \underline k \}$,
the $j$-th symmetric powers of the differential forms $\omega_i(0)$.
Set $${\ell^j(\mathcal W) = \dim \mathcal L^j(\mathcal W)}.$$

Equivalently, one can  define $\ell^j(\mathcal W)$ in terms of the linear parts of the submersions $u_i:(\mathbb C^n,0) \to (\mathbb C,0)$ defining $\mathcal W$.
If   $h_i$ is the linear part at the origin of $h_i$  then
\[
\ell^j(\mathcal W) = \dim  \left( \mathbb C h_1^j + \cdots + \mathbb C h_k^j \right) \, .
\]

\subsection{Bounds for  $\ell^j(\mathcal W)$}

The integers  $\ell^j(\mathcal W)$ are  bounded from above by the dimension of the space of degree $j$ homogeneous
polynomials in $n$ variables, that is
\begin{equation}\label{E:UB}
\ell^j ( \mathcal W) \le \min \left(k,\binom{n+j -1}{n-1} \right)\, .
\end{equation}
A good lower bound  is more delicate to obtain.
For smooth webs there is the following proposition.

\begin{prop}\label{P:cast}
If $\mathcal W$ is a germ of smooth  $k$-web on $(\mathbb C^n,0)$ then
\[
\ell^j ( \mathcal W) \ge \min( k , j(n-1) + 1 ) \, .
\]
\end{prop}

The key point is next lemma which translates questions about the dimension of vector spaces generated by powers of linear forms
to questions about the codimension of space of hypersurfaces containing finite sets of points.

\begin{lemma}\label{L:translate}
Let $h_1, \ldots, h_k \in \mathbb C_1[x_1, \ldots, x_n]$ be linear forms
and let $\mathcal P= \{ [h_1], \ldots, [h_k] \} $ be the corresponding set of  points of $\mathbb P^{n-1} = \mathbb P \mathbb C_1 [x_1, \ldots, x_n]$.
If $V(j) \subset | \mathcal O _{\mathbb P^{n-1}} (j) |$ is the linear system of degree $j$ hypersurfaces through $\mathcal P$ then
\[
 \dim ( \mathbb C h_1^j + \cdots + \mathbb C h_k ^j ) = \dim | \mathcal O_{\mathbb P^{n-1}}(j) | - \dim  V( j)  \, .
\]
\end{lemma}
\begin{proof}
Set $n_j$ equal to $h^0 ( \mathbb P^{n-1}, \mathcal O_{\mathbb P^{n-1}}(j)) -1 $, and consider the $j$-th Veronese embedding \index{Veronese embedding}
\begin{align*}
\nu_ j :\;   \mathbb P^{n-1} & \longrightarrow  \mathbb P^{n_j} \\
 [h]\; & \longmapsto  [h^j] \, .
\end{align*}
On the one hand the projective dimension of $\mathbb C h_1^j + \cdots + \mathbb C h_k ^j$ is equal to the dimension
of the linear span of the image  of $\mathcal P$. On the other hand the codimension of this linear span is equal to
the dimension  of the linear system of hyperplanes containing  it. But the pull-back under $\nu_j$ of these  hyperplanes are
exactly the elements of $|V(j)|$, the  degree $j$ hypersurfaces in $\mathbb P^{n-1}$ containing $\mathcal P$. The lemma
follows.
\end{proof}

\begin{proof}[Proof of Proposition \ref{P:cast}]
Let, as above, $h_i$ be the linear terms of the submersions defining $\mathcal W$ and let  $\mathcal P \subset \mathbb P^{n-1}$ be the
set of  $k$ points in general position determined by the
linear forms $h_1, \ldots, h_k$. According to the above lemma all that is needed to prove is that $\mathcal P$ imposes
$m = \min( k , j(n-1) + 1 )$ independent conditions on the space of degree $j$ hypersurfaces in $\mathbb P^{n-1}$. For
that sake, it suffices to show that for a subset $\mathcal Q \subset \mathcal P$ of cardinality $m$ one can
construct for each $q \in \mathcal Q$
a degree $j$ hypersurface that passes through  all the points of $\mathcal Q$ but $q$.

The set
  $\mathcal Q - \{ q \}$ can be written as a disjoint union of $j$ subsets of cardinality at most $(n-1)$. Any of these subsets
can be supposed  contained in a hyperplane that does not contain $q$.
Thus there exists a union of $j$ hyperplanes that contains $\mathcal Q - \{ q \}$
and avoids $q$.
\end{proof}

\begin{remark}\label{R:ind}
\rm Notice that the proof of Proposition \ref{P:cast} shows that when $k \le j(n-1) + 1$ any $k$ points in
general position impose $k$ independent conditions on the linear system of degree $j$ hypersurfaces on $\mathbb P^{n-1}$.
\end{remark}

For an essentially equivalent proof of Proposition \ref{P:cast},
but with a more analytic  flavor,  see \cite[Lemme 2.1]{Trepreau}.

\begin{cor}\label{C:cast}
If $\mathcal W$ is a germ of smooth  $k$-web on $(\mathbb C^2,0)$ then
\[
\ell^j ( \mathcal W) = \min( k , j + 1)  \, .
\]
\end{cor}
\begin{proof}
One has just to observe that the space of homogenous polynomials of degree $j$ in two variables
has dimension $j+1$ and use  Proposition \ref{P:cast}.
\end{proof}

For arbitrary webs, without any  further restriction on the relative position of the tangent spaces at the origin besides pairwise transversality,
it is not possible to improve the bound beyond the specialization of the above to $n=2$. That is for  an arbitrary quasi-smooth $k$-web
\begin{equation}\label{E:qsg}
\ell^j ( \mathcal W) \ge \min( k , j + 1 ) \, .
\end{equation}

\medskip

\begin{remark}\rm
 Recently Cavalier and Lehmann have drawn  special attention  to $k$-webs on $(\mathbb C^n,0)$
for which the upper bounds (\ref{E:UB})  are sharp, see \cite{CL1}.
These have been labeled by them \defi[ordinary webs]. \index{Web!ordinary}
\end{remark}

\subsection{Bounds for the rank}
There is a \defi[natural decreasing  filtration] $F^\bullet \mathcal A(\mathcal W)$ \index{Abelian relation!filtration} on the vector space $\mathcal A(\mathcal W)$.
The first term is, of course, $F^0 \mathcal A(\mathcal W) = \mathcal A(\mathcal W)$ and for $j\ge 0$ the $j$-th
piece of the filtration is defined as
\[
F^j \mathcal A(\mathcal W) = \ker  \left\{ \mathcal A(\mathcal W)\longrightarrow
 \left( \frac{ \Omega^1(\mathbb C^n ,0)}{\mathfrak m^j \cdot \Omega^1(\mathbb C^n
 ,0)} \right)^k \right\}\, ,
\]
with $\mathfrak m$ being the maximal ideal of $\mathcal O(\mathbb C^n,0)$.

\begin{lemma}\label{L:besta}
If $\mathcal W$ is a germ of  quasi-smooth $k$-web on $(\mathbb C^n,0)$ then
\[
\dim \frac{F^j \mathcal A(\mathcal W)}{F^{j+1} \mathcal A(\mathcal W)} \le \max(0,k - \ell^{j+1}(\mathcal W))\, .
\]
\end{lemma}
\begin{proof}
Let, as above,  $h_1, \ldots, h_k$ be the linear terms at the origin of the submersions defining $\mathcal W$.
Consider the linear map
\begin{align*}
 \varphi :\quad  \mathbb C^k \quad  &\longrightarrow \mathbb C_{j+1}[x_1, \ldots, x_n] \\
(c_1, \ldots, c_k) &\longmapsto \sum c_i (h_i)^{j+1} \, .
\end{align*}
From the definition of the space $\mathcal L^j(\mathcal W)$, it is clear that the image of
$\varphi$ coincides with it. In particular,
\[
\dim \ker \varphi = \max(0,k - \ell^{j+1}(\mathcal W)) \, .
\]

If $(\eta_1, \ldots, \eta_k)$ is an abelian relation in $F^j\mathcal A(\mathcal W)$ then for
suitable complex numbers $\mu_1,\ldots, \mu_k $, the following identity holds true
\[
  (\eta_1,\ldots, \eta_k)  = (\mu_1 (h_1)^{j} dh_1, \ldots, \mu_k ({h_k})^j dh_k)  \mod F^{j+1}\mathcal A(\mathcal W) \, .
\]
Consider the linear map taking $(\eta_1, \ldots, \eta_k) \in F^j\mathcal A(\mathcal W)$ to the $k$-uple of complex numbers $(\mu_1, \ldots, \mu_k) \in \mathbb C^k$.
Since  $\sum \eta_i = 0$ it follows that this map induces an injection of $F^j\mathcal A(\mathcal W) / F^{j+1} \mathcal A (\mathcal W)$
into the kernel of $\varphi$. The lemma follows.
\end{proof}

\begin{cor}
If $\mathcal W$ is a quasi-smooth $k$-web then for $j\ge k-2$
\[
F^{j} \mathcal A (\mathcal W) =  0 \, .
\]
\end{cor}
\begin{proof}
For  $j\ge k-2$,  equation (\ref{E:qsg}) reads as  $\ell^{j+1} (\mathcal W) = k$. Therefore Lemma \ref{L:besta}
implies
\[
F^{j} \mathcal A(\mathcal W ) = F^{j+1} \mathcal A(\mathcal W) \, .
\]
Thus an element of $F^{k-2} \mathcal A(\mathcal W)$ has a zero of infinite order at the origin.
Since it is a $k$-uple of holomorphic $1$-forms it has to be identically zero.
\end{proof}

With what have been done so far,  Bol's, respectively  Chern's, bound for the rank of smooth $k$-webs on
$(\mathbb C^2,0)$, respectively $(\mathbb C^n,0)$, can be easily proved. \index{Bol!bound} \index{Chern bound}

\begin{thm}\label{T:cota}
If $\mathcal W$ is a germ of  quasi-smooth $k$-web on $(\mathbb C^n,0)$  then
\[
\mathrm{rank}(\mathcal  W )\le \sum_{j=0}^{k-3} \max( 0 , k - \ell^{j+1}(\mathcal W) )\, .
\]
Moreover, if $\mathcal W$ is smooth then
\[
\mathrm{rank}(\mathcal  W )\le \pi(n,k) = \sum_{j=0}^{k-3} \max( 0 , k - (j+1)(n-1) - 1 )\, .
\]
\end{thm}
\begin{proof}
It follows from  the corollary above that $\mathcal A (\mathcal W)$ is isomorphic as a vector space
to
\[
 \bigoplus_{j=0}^{k-3} \frac{ F^j \mathcal A (\mathcal W)}{F^{j+1}\mathcal A (\mathcal W)} \, .
\]
If $\mathcal W$ is quasi-smooth the result follows promptly from Lemma \ref{L:besta}. If moreover
$\mathcal W$ is smooth one can invoke Proposition \ref{P:cast} to conclude.
\end{proof}
\index{Rank!bound|)}

The number $\pi(n,k)$ appearing in the bound for the rank of  smooth webs is
\defi[Castelnuovo number]. \index{Castelnuovo number} It is the  bound for the arithmetical genus of non-degenerate
irreducible curves in $\mathbb P^{n}$ according to a  classical result by Castelnuovo. In Chapter \ref{Chapter:3}
Castelnuovo result will be recovered from Theorem \ref{T:cota} combined with Abel's addition Theorem.

\begin{remark}\label{R:closed}\rm
Following \cite{EHCAMG}, let   $m = \left\lfloor\frac{k-1}{n-1}\right\rfloor$ and $\epsilon$ be the remainder of the division of
$k-1$ by $n-1$. Thus  $k -1 = m (n-1) + \epsilon$ with $0\le \epsilon \le n-2$.
Using this notation Castelnuovo's numbers can be expressed as
\[
\pi(n,k) = \binom{m}{2} ( n- 1 ) + m \epsilon \, .
\]
In this way one obtains a family of closed formulas for the bound of the rank of a $k$-web on $(\mathbb C^n,0)$
according to  the residue $\epsilon$ of $k-1$  modulo $n-1$.
\end{remark}

\begin{remark}\label{R:closed2}\rm
Alternatively, one can set  $\rho = \left\lfloor\frac{k-n-1}{n-1}\right\rfloor$
and $\epsilon$ equal to the remainder of the division of $k-n-1$ by $n-1$.
Hence  $k -n-1 = \rho (n-1) + \epsilon$ with $0\le \epsilon \le n-2$.
Castelnuovo's numbers admit the following alternative presentation
\[
\pi(n,k) = (\epsilon+1)\binom{\rho+2}{2} +(n-2-\epsilon)
\binom{\rho+1}{2}
 \, .
\]
The two distinct presentations are given here because the former is the usual one found in the literature, while the latter seems
to be better adapted to some constructions that will be carried out in Section \ref{S:PMPW} of Chapter \ref{Chapter:4}.
\end{remark}

Notice that for smooth webs the  bound for the rank  is attained  if and only if the partial bounds
provided by the combination of Proposition \ref{P:cast} with Lemma \ref{L:besta} are also attained.
For further use, this remark is stated below as a corollary.

\begin{cor}\label{C:bbb}
Let  $\mathcal W$ be a germ of  smooth $k$-web on $(\mathbb C^n,0)$. If $\mathrm{rank}(\mathcal W) = \pi(n,k)$ then
\[
\dim \frac{F^j \mathcal A(\mathcal W)}{F^{j+1} \mathcal A(\mathcal W)} = \max\big( 0, k - (j+1)(n-1) - 1 \big)
\]
for every $j \ge 0$.
\end{cor}

\subsection{Webs of maximal rank}

One of the central problems in web geometry, and the central  theme  of this text, is the characterization
of germs of smooth $k$-webs on $(\mathbb C^n,0)$  for which $\mathrm{rank}(\mathcal W) = \pi(n,k)$. They are
called   \defi[webs of maximal rank]. \index{Web!maximal rank}

\smallskip

Theorem \ref{T:cota} recovers, and generalizes to  arbitrary planar webs, the bound provided by Proposition \ref{P:TTTT} for germs of planar quasi-parallel webs. In particular the planar parallel webs are examples of webs of maximal rank, see Example \ref{E:llll}. The $5$-webs mentioned in Example  \ref{E:exabel} are also of maximal rank.

\smallskip

In dimension greater than two webs of maximal rank are harder to come by. In contrast with the planar case, not every parallel web
is of maximal rank. In the next Section the parallel webs of maximal rank will be characterized in Proposition \ref{P:exmax} and constraints on the distribution of
conormals of maximal rank webs will be established.

\section{Conormals of    webs of maximal rank}\label{S:constraints}

\index{Web!maximal rank|(}

If $\mathcal W$ is a germ of smooth $k$-web of maximal rank then the lower bounds
for $\ell^j(\mathcal W)$ given by Proposition \ref{P:cast} are attained, that is,
\[
\ell^j(\mathcal W) = \min ( k , j(n-1) +1)
\]
holds true for every positive integer $j$.

\medskip

When the ambient space has dimension two ( $n=2$ ) these equalities do not impose any restriction on the
web as Corollary \ref{C:cast} testifies. When $n$ is at least three then the equalities above impose rather
strong restrictions of the distributions of conormals \index{Web!conormal} of the web $\mathcal W$. Indeed, in the next few pages
the corresponding equality for $j=2$ -- that is,  $\ell^2(\mathcal W) = \min (k , 2(n-1) +1 )$ -- will be exploited and
the following proposition will be proved.

\begin{prop}\label{P:Cast}
Let $\mathcal W$ be a germ of smooth $k$-web on $(\mathbb C^n,0)$. Suppose that $n \ge 3$ and $k \ge  2n + 1 $.
If  $\ell^2(\mathcal W) = 2n -1$ then there exists a non-degenerate rational normal
curve $\Gamma$ in $\mathbb P(T_0^* (\mathbb C^n,0))$ containing the conormals of the web $\mathcal W$
at the origin.
\end{prop}

\subsubsection{A particular case }

For the sake of clarity the case $n=3$ of Proposition \ref{P:Cast} will be here presented.
It is harmless to assume that $k= 2n = 6$ even if the hypothesis for $n=3$ reads $k\ge 7$.

Let $h_1, \ldots, h_6 \in \mathbb C_1[x_1,x_2,x_3]$
be six linear forms in general position and let $\mathcal L^2$ be the vector space contained in $\mathbb C_2[x_1,x_2,x_3]$
generated by theirs squares.

If $\dim \mathcal L^2 = 5$, since $\mathbb C_2[x_1,x_2,x_3]$ has dimension six, there is a hyperplane $H$ through $0 \in \mathbb C_2[x_1,x_2,x_3]$
containing $\mathcal L^2$. If one now interprets the linear forms $h_i$ as points in $\mathbb P ^2 = \mathbb P \mathbb C_1[x_1,x_2,x_3]$
and consider the Veronese embedding of this $\mathbb P^2$ into $\mathbb P^5 = \mathbb P \mathbb C_2[x_1,x_2,x_3]$, as in the proof of Lemma \ref{L:translate}, then the
pull-back of $[H]$ to $\mathbb P^2$ is a conic containing $[h_1], \ldots, [h_6]$. This is the sought rational normal curve.

\subsubsection{Dimension shift and reduction to Castelnuovo Lemma}

As suggested by its statement all the action in the proof of Proposition \ref{P:Cast} will take place in $\mathbb P T_0^* (\mathbb C^n,0) = \mathbb P^{n-1}$.  To avoid carrying over a $-1$ throughout instead of working with a $k$-web on $(\mathbb C^n,0)$ it is convenient to consider
a $k$-web on $(\mathbb C^{n+1},0)$. Of course with this shift on the dimension  the hypotheses of
Proposition \ref{P:Cast} now read as
\[
 k \ge 2n +3  \quad \text{ and }  \quad \ell^2(\mathcal W) = 2n + 1.
\]

If $h_1, \ldots, h_k \in \mathbb C[x_0, \ldots, x_n]$ are the linear forms defining the tangent space of the leaves of $\mathcal W$ through the origin then
according to Lemma \ref{L:translate} the number of conditions imposed on quadrics of $\mathbb P^n$ by the corresponding set of points $\mathcal P= \{ [h_1] , \ldots, [h_k] \} \subset \mathbb P^n$ is exactly $\ell^2(\mathcal W)$. Therefore Proposition \ref{P:Cast} is equivalent to the famous Castelnuovo Lemma.

\begin{prop}[{\bf Castelnuovo Lemma}]\label{P:Castvero}\index{Castelnuovo Lemma}
Let  $\mathcal P \subset \mathbb P^n$  be a set of $k$ points in general position. Suppose that $n \ge 2$ and $k \ge 2n + 3$.
If $\mathcal P$ imposes only  $2n + 1$ conditions on the linear system of quadrics $|\mathcal O_{\mathbb P^n}(2)|$ then $\mathcal P$ is contained in
a rational normal curve $\Gamma$ of degree $n$.
\end{prop}

Before dealing with the proof of Proposition \ref{P:Castvero} itself, which will follow  \cite[Chapter III]{EHCAMG}, some  basic properties of
rational normal curves will be reviewed.

\subsection{Rational normal curves}\index{Rational normal curves|(}

The  rational normal curves on a projective space $\mathbb P^n$ are the ones
that admit a parametrization of the form
\begin{align*}
 \varphi : \quad \mathbb P^1\;  &\longrightarrow\quad  \mathbb P^n \\
 [s:t] &\longmapsto [ a_0(s:t): \cdots : a_n(s:t) ] \, .
\end{align*}
where $a_0, \ldots, a_n$ form a basis of the space $\mathbb C_n[s,t]$ of  binary forms of degree $n$.
In other words, a  rational normal curve $\Gamma$ on $\mathbb P^n$ is the image of  an embedding of $\mathbb P^1$
into $\mathbb P^n$ given by the linear system  $| \mathcal O_{\mathbb P^1} (n)|$.

Since the hyperplanes in $\mathbb P^n$ are in one to one correspondence with the non-zero elements of $\mathbb C_n[s,t]$ modulo multiplication by $\mathbb C^*$,
  the intersection of $\Gamma$ with an hyperplane $H$ consists of at most $n$ points. If the hyperplane is generic then the
intersection has exactly $n$ points, that is $\Gamma$ has degree $n$. It turns out that this  is the minimal degree among the non-degenerated
curves in $\mathbb P^n$. Moreover  the rational normal curve is the unique irreducible non-degenerated curve of degree $n$, see Proposition \ref{P:ratdeg} below.

Notice that  any $k$ distinct points on a rational normal curve $\Gamma \subset \mathbb P^n$
are automatically in general position with respect to the linear system of hyperplanes. Indeed, if a subset of cardinality $a \le n$
is contained in a $\mathbb P^{a-2}$, then by choosing other $n-a+1$ points and considering a hyperplane containing all these $n+1$ points
one arrives at a contradiction since a rational normal curve $\Gamma$ intersects every hyperplane in at most $\deg \Gamma= n$ points.

Parallel webs \index{Web!parallel} defined by points on a rational normal curve are the simplest examples
of webs of maximal rank on $(\mathbb C^n,0)$, with $n \ge 3$.   More precisely,

\begin{prop}\label{P:exmax}
Let $n\ge 2$ and $k \ge 2n+3 $ be integers. Let also $h_1, \ldots, h_k \in \mathbb C_1[x_0,x_1, \ldots, x_n]$ be pairwise distinct linear forms
and $\mathcal W=\mathcal W(h_1, \ldots, h_k)$ be the corresponding  parallel $k$-web on $(\mathbb C^{n+1},0)$.
Then  $\mathcal P= \{ [h_1], \ldots, [h_k] \} $, the corresponding set of  points of $\mathbb P^{n} $, lies in a
rational normal curve $\Gamma$ if  and only if   $\mathcal W$ is smooth and of maximal rank.
\end{prop}
\begin{proof}
If $\mathcal W$ is smooth and of maximal rank then Proposition \ref{P:Cast} implies the result.

Reciprocally, if $\mathcal P$ is contained in a rational normal curve then $\mathcal W$ is smooth because the points $[h_i]$ are in general position, see discussion
preceding the statement of the Proposition.
To prove that $\mathcal W$ is of maximal rank notice that the kernel of the restriction map
\[
H^0(\mathbb P^{n}, \mathcal O_{\mathbb P^{n}}(j)) \longrightarrow H^0(\Gamma, \mathcal O_{\Gamma}(j)) \simeq H^0 ( \mathbb P^1, \mathcal O_{\mathbb P^1}( j n ))
\]
has codimension at least $h^0 ( \mathbb P^1, \mathcal O_{\mathbb P^1}( j n ) ) = jn + 1$. Therefore, by Lemma \ref{L:translate},
\[
\ell^j( \mathcal W)  = \dim ( \mathbb C h_1^j + \cdots + \mathbb C h_k ^j ) \le jn + 1 \, .
\]
Hence Proposition \ref{P:cast}  implies  $\ell^j(\mathcal W) = \min(k, jn +1)$.

Because  $\mathcal W$ is a parallel web, $F^j  \mathcal A(\mathcal W) / F^{j+1} \mathcal A(\mathcal W)$ not just
embeds into  the kernel of the map $\mathbb C^k \to \mathcal L ^j( \mathcal W)$, but is indeed isomorphic to it. Therefore
\[
\mathrm{rank}(\mathcal  W ) =  \sum_{j=0}^{k-3} \max( 0 , k - (j+1)n - 1 )
\]
as wanted.
\end{proof}

\subsubsection{Steiner's synthetic construction}
\index{Rational normal curve!Steiner construction}
\index{Steiner construction}

The rational normal curves admit a nice geometric description: the so called \defi[Steiner construction].
Let $p_1, \ldots, p_{n+3} \in \mathbb P^n$ be $n + 3$  points in general position.
For each $i$ ranging from $1$ to $n$, let $\Pi_i$ be the $\mathbb P^{n-2}$
spanned by $ p_1, \ldots, \hat{p_i} , \ldots , p_n$. The hyperplanes
containing $\Pi_i$ form a family $H_i(s:t)$ with $(s:t) \in \mathbb P^1$.
One can choose the parametrizations in order to have
\[
p_{n+1} = \bigcap_{i=1}^n H_i(0:1), \,   p_{n+2} = \bigcap_{i=1}^n H_i(1:0) \,  \text{ and } \,  p_{n+3} = \bigcap_{i=1}^n H_i(1:1) \, .
\]

\begin{prop}\label{P:Steiner}
The set
\[
 \Gamma = \bigcup_{[s:t] \in \mathbb P^1} \left( \bigcap_{i=1}^n H_i(s:t) \right)
\]
is the unique rational normal curve through the points $p_1, \ldots, p_{n+3}$.
\end{prop}
\begin{proof}
Because the points are  in general position  the expression under
parentheses defines for each $(s:t) \in \mathbb P^1$  a unique point of $\mathbb P^n$.
Consequently $\Gamma$ is a curve parametrized by $\mathbb P^1$.

Clearly it contains $p_{n+1}, p_{n+2}$ and $p_{n+3}$. To see that it contains $p_1, \ldots, p_n$ notice that
$p_i \in H_j(s:t)$ for every $[s:t] \in \mathbb P^1$ when $j \neq i$ and there exists an $(s_i: t_i)$ such that
$p_i \in H_i(s_i:t_i)$.
It remains to show that the  linear system defining $\Gamma$ is $|\mathcal O_{\mathbb P^1}(n)|$.

Using an  automorphism of  $\mathbb P^n$ the points  $p_1, \ldots, p_{n+1}$
can normalized as \[p_i = [0: \ldots : 0:\underbrace{1}_{{ \text{\tiny{i-th entry}}}}:0: \ldots : 0]  \, \quad i \in \underline{n+1} .\]
 If $p_{n+2} = [a_0: a_1: \ldots : a_n ] $ and $p_{n+3} = [b_0: b_1: \ldots : b_n ] $ then it is  a simple matter to verify that
 \[
 \begin{array}{lcl}
 \mathbb P^1 &\longrightarrow& \mathbb P^n \\
 (s:t) &\mapsto& \big[ ({a_1^{-1} s - b_1^{-1}t })^{-1}: \ldots : ({a_{n+1}^{-1} s - b_{n+1}^{-1} t })^{-1}    \big] \,
 \end{array}
 \]
 is a parametrization of $\Gamma$ in the normalization above.
Multiplying all the entries by $\prod_{i} ({a_i^{-1} s - b_i^{-1}t })$ one ends up with $n+1$ binary forms  of degree $n$. Since
$p_1, \ldots, p_{n+3}$ are not contained in any hyperplane these must generate the space of binary forms  of degree $n$.
\end{proof}\index{Rational normal curves|)}

\subsection{Proof of Castelnuovo Lemma}\label{S:Castproof}\index{Castelnuovo Lemma|(}
A variant of the synthetic construction presented above allows to
construct rank three quadrics -- these are quadrics which in a suitable system of coordinates are
cut out by a polynomial of the form $x_0^2 + x_1^2 + x_2^2$ -- containing $\Gamma$. This construction will
be the key to prove Proposition \ref{P:Castvero}.

Let $\mathcal P=\{ p_1, \ldots, p_{2n +2} , p_{2n+3} \} $  be a set of $2n + 3$ distinct points of  $\mathbb P^n$ in general position,
and  $\Lambda \simeq \mathbb P^{n-2}$ be the linear span of $p_1, \ldots, p_{n-1}$.

\begin{lemma}
If $\mathcal P$ imposes at most $2n+1$ independent conditions on the space of quadrics then
there are at least $n-1$ linearly independent quadrics containing $\Lambda \cup \mathcal P$.
\end{lemma}
\begin{proof}
Let $F_0$ be a linear form ( unique up to multiplication by $\mathbb C^*$ ) vanishing at the span of $p_1, \ldots, p_n$ and $G_0$ be the one vanishing on $p_1, \ldots,p_{n-1},p_{n+1}$.
Any quadric containing $\Lambda$ can be written in the form $F_0 G - G_0 F$ for suitable linear forms $F , G \in \mathbb C_1[x_0, \ldots, x_n]$.
Such pair $(F,G)$  is not unique, but is well defined modulo the addition of a multiple of $(G_0,F_0)$. Hence the vector space of quadrics containing $\Lambda$ has dimension $2n+1$.

Further imposing that the quadrics  contain the $n+2$ points $p_n, p_{n+1}, \ldots, p_{2n+1}$ one sees that there are
at least $2n+1 - (n+2)$, that is  $n-1$, linearly independent quadrics containing $\Lambda \cup \{p_1, \ldots , p_{2n+1} \}$. By hypothesis the space
of quadrics containing $\{p_1, \ldots , p_{2n+1} \}$ coincides with the space of quadrics containing $\mathcal P$.
\end{proof}

Keeping the notation from the lemma above one can write down the $n-1$ linearly independent quadrics $Q_1, \ldots, Q_{n-1}$ containing $\Lambda \cup \mathcal P$ in the form
\[
Q_i = \det \left( \begin{array}{cc}F_0 & F_i \\ G_0 & G_i \end{array} \right) = F_0 G_i - G_0 F_i
\]
for suitable linear forms $F_i, G_i \in \mathbb C_1[x_0, \ldots, x_n]$, $i \in \underline{n-1}$.

It is possible to recover a rational normal curve $\Gamma$ from the quadrics just constructed. It will turn out that the variety $X$ defined through the
determinantal formula below
\begin{equation}\label{E:quadrics}
X=  \left\{ p \in \mathbb P^n \, \Big| \; \mathrm{rank} \left( \begin{array}{ccc} F_0(p) & \ldots & F_{n-1}(p) \\ G_0(p) & \ldots & G_{n-1}(p) \end{array} \right) \le 1 \right\} \;
\end{equation}
is a rational normal curve.
To prove it, a couple of preliminary results  is needed.

\begin{lemma}\label{L:claimB}
For any pair $(\lambda, \mu ) \in \mathbb C^2$  distinct from $(0,0)$ the linear forms ${\{ \lambda F_i + \mu G_i  \, ; \, i = 0 , \ldots, n-1\}}$ are linearly independent.
\end{lemma}
\begin{proof}
Because the quadrics $Q_1, \ldots, Q_{n-1}$ are linearly independent for any ${\alpha = (\alpha_1, \ldots, \alpha_{n-1})}$ distinct from $(0,\ldots ,0)$
 the quadric $Q_{\alpha} = \sum_{i = 1}^{n-1} \alpha_i Q_i $ cut out by
 \begin{equation}\label{E:mat}
 \displaystyle{  \det \left( \begin{array}{cc}F_0 & \displaystyle{\sum_{i=1}^{n-1} \alpha_i F_i} \\ G_0 & \displaystyle{\sum_{i=1}^{n-1} \alpha_i G_i} \end{array} \right)}
 \end{equation}
is non-zero and still contains $\mathcal P$. If the linear forms $\lambda F_i + \mu G_i$, $i = 0 , \ldots, n-1$ are linearly dependent, then there exists
$(\alpha_0, \ldots, \alpha_n) \in \mathbb C^{n+1} \setminus \{ 0\}$ such that
\[
\alpha_0 ( \lambda F_0 + \mu G_0 ) = \sum_{i=1}^{n-1} \alpha_i \left( \lambda  F_i + \mu G_i \right) \, .
\]
Consequently the  matrix $Q_{\alpha}$ appearing in Equation (\ref{E:mat}) has rank one. Thus the quadric $Q_{\alpha}$ has rank at most two. Since $\mathcal P$
is not contained in the union of two hyperplanes, the lemma follows.
 \end{proof}

\begin{lemma}
The restriction of any linear combination of the linear forms $F_1, \ldots, F_{n-1}$ to $\Lambda$ is non-zero.
Consequently, it can be assumed that for every $i=1,\ldots,n-1$, the linear form $F_i$ satisfies
\[
p_1, \ldots, p_{i-1}, \widehat{p_i}, p_{i+1}, \ldots, p_{n-1} \in \{ F_i = 0 \}  \, .
\]
 \end{lemma}
\begin{proof}
Since $F_0(p_n) = 0$, $G_0(p_n)\neq 0$  and the quadrics $Q_i$ contain $p_n$, the linear form $F_i$ must vanish on $p_n$ for every $i\ge1$.
If some linear combination of $F_1, \ldots, F_{n-1}$ vanishes on $\Lambda$ then it would have to be a complex multiple of $F_0$ because
the span of $\Lambda$ and $p_n$ is the hyperplane cut out by $F_0$. This contradicts the linear independence of $F_0, \ldots, F_{n-1}$ established
in the previous lemma and proves the first claim. The second claim  follows immediately from plain linear algebra.
\end{proof}

Castelnuovo Lemma follows from the following proposition.

\begin{prop}
The variety $X$ is the unique  rational normal curve through $p_1, \ldots, p_{k}$.
\end{prop}
\begin{proof}
If, for $i = 0, \ldots, n-1$, $H_{n-i}(s:t)$ is the pencil of hyperplanes  $\{ s F_i + t G_i = 0\}$ then
 $X$ can be described as below
\[
X = \bigcup_{(s:t) \in \mathbb P^1} \, \left(  \bigcap_{i=1}^{n} H_i(s:t)   \right) \, .
\]
But this has exactly the same form as the presentation of a rational normal curve through Steiner's construction, see Propositon \ref{P:Steiner}.
Consequently, $X$ is a rational normal curve.

By construction, when $l>n$,  $Q_{\alpha}(p_l) =0 $ but $(F_0(p_l),G_0(p_l)) \neq 0$. Therefore
\[
\det \left( \begin{array}{ccc} F_{i}(p_l) & \ldots & F_{j}(p_l) \\ G_{i}(p_l) & \ldots & G_{j}(p_l) \end{array} \right) = 0
\]
for every pair $i,j $ and every $l > n$.  Thus  $X$ contains the points $p_{n+1}, \ldots, p_{k}$.

The careful reader probably noticed that  the inequality  $k \ge 2n +3$ have not been used so far, only the weaker $k \ge 2n+1$ played a role.
To prove that   $p_1, \ldots, p_n$ belong to $X$ the stronger inequality enters the stage. Observe  that the quadric $Q_{ij} = F_iG_j - F_jG_i$
contains the $k-2 \ge 2n+1$ points  $\mathcal P - \{ p_i,  p_j \}$. Remark \ref{R:ind} implies that these points impose at least $2n+1$ conditions
on the space quadrics. But, by hypothesis, the same holds true for $\mathcal P$. Thus $Q_{ij}$ also contains  $p_i$ and $p_j$.
\end{proof}
\index{Castelnuovo Lemma|)}
\subsection{Normal forms for webs of maximal rank}

For a quasi-smooth $k$-web $\mathcal W= [ \omega_1, \ldots, \omega_k ] $ on $(\mathbb C^n,0)$ it is natural to consider $\ell^j(\mathcal W)$ not just as
an integer but as a germ of integer-valued function defined on $(\mathbb C^n,0)$. The value at $x$ is given
by the dimension of the span of $\{ \omega_i(x)^j \, ; \, i \in \underline k \}$ in $\mathrm{Sym}^j \Omega^1_x (\mathbb C^n,0)$.

A priori this function does not need to be continuous but just lower-semicontinuous.
Nevertheless, as the reader can easily verify, when $\mathcal W$ is smooth and $F^{j-1} \mathcal A(\mathcal W)/ F^j\mathcal A(\mathcal W)$ has maximal dimension then $\ell^j(\mathcal W)$ is constant.  To easer further reference this fact is stated below as a lemma.

\begin{lemma}\label{L:constante}
Let $\mathcal W$ be a smooth $k$-web. If
\[
 \dim \frac{F^{j-1} \mathcal A(\mathcal W)}{ F^j\mathcal A(\mathcal W)} = \min\big(0 , k - j(n-1) - 1 \big)
\]
then  the integer-valued function $\ell^j(\mathcal W) : (\mathbb C^n,0) \to \mathbb N $ is constant and equal to ${j(n-1) + 1}$.
\end{lemma}

Combined with Castelnuovo Lemma, or rather with  Proposition \ref{P:Cast}, the Lemma above yields the following normal forms for
webs of maximal rank up to second order.

\begin{prop}\label{P:normal}
Let $\mathcal W = \mathcal F_1 \boxtimes \cdots \boxtimes \mathcal F_k$ be a germ of smooth $k$-web on $(\mathbb C^n,0)$. Suppose that $n \ge 3$ and $k \ge  2n + 1$.
If
\[
 \dim   \frac{F^1 \mathcal A(\mathcal W)}{F^2 \mathcal A(\mathcal W)} = k- 2n + 1 
\]
then there exist  a coframe $\varpi =(\varpi_0,\ldots,\varpi_{n-1})$ on $(\mathbb C^n,0)$
 and $k$ germs of holomorphic functions
$\theta_1, \ldots, \theta_k$ such that for every $i\in \underline{k}$
\[
 \mathcal F_i = \bigg[ \sum_{q=0}^{n-1} (\theta_i)^q \varpi_q \bigg]  .
\]
\end{prop}
\begin{proof}
According to Lemma \ref{L:constante} the  function $\ell^2(\mathcal W)$ is constant and equal to $2(n-1) + 1$.
Proposition \ref{P:Cast} implies the existence, for every $x \in (\mathbb C^n,0)$, of a rational normal curve in $\mathbb P T^*_x (\mathbb C^n,0)$
containing the conormals of the defining foliations of the web. Therefore it is possible to choose holomorphic $1$-forms
$\varpi_0, \ldots, \varpi_{n-1} \in \Omega^1(\mathbb C^n,0)$ such that at every $x \in (\mathbb C^n,0)$ the rational normal curve
given by Proposition \ref{P:Cast} is parameterized by
\[
t \longmapsto \sum_{q=0}^{n-1} t^q \varpi_q(x) \, .
\]
This parametrization can be chosen  in such a way that none of the foliations $\mathcal F_i$ have  conormal  corresponding
to $t= \infty$.
Thus, for every $i \in \underline k$, the foliation $\mathcal F_i$  will be induced by  $\sum_{q=0}^{n-1} (\theta_i)^q \varpi_q$ for a suitable
germ  of holomorphic function $\theta_i$.
\end{proof}

\subsection{A generalization of Castelnuovo Lemma}

\index{Rational normal curves|(}Proposition \ref{P:normal} can be seen as the starting of the proof of the algebraization of webs
of maximal rank to be presented in Chapter \ref{Chapter:Trepreau}. As it has been made clear above, Proposition \ref{P:normal}
 is an easy consequence of Castelnuovo Lemma. Loosely phrased Castelnuovo Lemma says that if
sufficiently many points in general position impose  the minimal number of conditions on the space of quadrics
then they must lie on particularly simple curves: the rational normal curves. A testimony of the {\it
simplicity} of rational normal curves is the following proposition.

\begin{prop}\label{P:ratdeg}
If $C$ is a non-degenerate irreducible projective curve in $\mathbb P^n$ then $\deg C \ge n$. Moreover, if the equality holds then
$C$ is a rational normal curve.
\end{prop}
\begin{proof}
If $C$ is non-degenerate then there exists $n$ points in $C$ that are in general position, otherwise $C$ would
be contained in a hyperplane. Intersecting $C$ with the hyperplane $H$ determined by $n$ of these points shows that
the degree of $C$ is at least $n$.

To prove the second part let $p_1, \ldots, p_{n-1}$ be $n-1$ general points of $C$  and let $\Sigma$ be the $\mathbb P^{n-2}$
determined by them. By hypothesis each generic hyperplane containing $\Sigma$ intersects $C$ in exactly one point away from $\Sigma$.
Therefore there is an injective map from the set of hyperplanes containing $\Sigma$, nothing else than a $\mathbb P^1$, to $C$. Thus $C$
is rational. Therefore $C$ is parametrized by $n+1$ homogenous binary forms of degree equal to $\deg C=n$. Since $C$ is non-degenerated
these $n+1$ binary forms must generate the space of degree $n$ binary forms. In other words, $C$ is a rational normal curve.
\end{proof}
\index{Rational normal curves|)}

It is  natural to enquire what can be said about sufficiently many points imposing a number of conditions on the
space of quadrics close to minimal. For instance one can ask if they   lie on {\it simple} varieties.

Of course to be more precise the meaning of {\it simple varieties} must be spelled out. One possibility  is to look for
non-degenerate irreducible varieties of minimal degree. For that sake it is important to generalize Proposition \ref{P:ratdeg}
for irreducible non-degenerated varieties  of $\mathbb P^n$ of arbitrary dimension. The first part of the statement
generalizes promptly as shown below.

\begin{prop}
If $X$ is a non-degenerate irreducible subvariety of  $\mathbb P^n$ then $\deg X \ge \mathrm{codim} \,  X  + 1$.
\end{prop}
\begin{proof}
Take $m+1=\mathrm{codim}(X)+ 1$ generic points on $X$. Because $X$ is non-degenerate they span a
$\mathbb P^m$ intersecting $X$ in at least $m + 1$ points.  To conclude, it remains to verify that
for a generic choice of $m+1$ points there are no positive dimensional component in the corresponding intersection
$\mathbb P^m \cap X$.

For that sake let $k$ be the dimension of the intersection of $X$ with a generic $\mathbb P^m$.
If $p_1, \ldots, p_{m-k+1} \in X$ are  $m-k+1$ generic points  then
their linear span $\Sigma \simeq \mathbb P^{m-k}$ intersects $X$ in a finite number of points.
If  $\Lambda$ is  the set of all the projective spaces $\mathbb P^{m-k+1}$ contained in $\mathbb P^n$
and containing $\Sigma$ then $\Lambda \simeq \mathbb P^{n-m + k}$.

On  the one hand $\dim X = n-m$, while on the other hand
\[
X - \Sigma = \bigcup_{\mathbb P^{n-m + k} \in \Lambda} ( \mathbb P^{m-k+1} - \Sigma) \cap X \, .
\]
implies that $\dim X = n-m + k$. Thus $k=0$, that is, $X$ intersects a generic $\mathbb P^m$ in a
finite number of points.
\end{proof}

The second part also does generalize but the generalization, which can be traced back at least to  Bertini,  is by no means evident.
\index{Varieties of minimal degree|(}
\begin{thm}\label{T:VMD}
If $V$ is an  irreducible non-degenerated projective subvariety of $\mathbb P^n$  with  $\deg V = \mathrm{codim} V  + 1$
then
\begin{enumerate}
\item $V$ is $\mathbb P^n$, or;
\item $V$ is a rational normal scroll, or;
\item $V$ is a cone over the Veronese surface $\nu_2(\mathbb P^2) \subset \mathbb P^5$, or ;
\item $V$ is a  hyperquadric.
\end{enumerate}
\end{thm}

The proof of this theorem would take the exposition  to far afield, and therefore will not be presented here. For a
modern exposition see for instance \cite{EH100}.

\index{Rational normal scroll|(}
A \defi[rational normal scroll] of dimension $m$ in $\mathbb P^n$  is characterized, up to an automorphism of $\mathbb P^n$, by
$m$ positive integers $a_1, \ldots, a_m$ summing up to $n-m+1$ and can be described as follows. Decompose
$\mathbb C^{n+1}$ as $\oplus \mathbb C^{a_i +1}$ and consider parametrizations $\varphi_i: \mathbb P^1 \to \mathbb P^{a_i} \subset \mathbb P^n$
of rational normal curves on the corresponding projective subspaces\footnote{
When $a_i=0$, the linear subspace $\mathbb P^{a_i}$ is nothing but  a point $p_i\in \mathbb P^n$. In this case, the following convention is adopted: the {\it rational normal curve} in $\mathbb P^{a_i}$ is not a curve, but the point  $p_i$.}. If $\Sigma(p)$ is the $\mathbb P^{m-1}$
spanned by $\varphi_1(p), \ldots, \varphi_m(p)$ then  the associated rational normal scroll is
\[
S_{a_1, \ldots, a_m} = \bigcup_{p \in \mathbb P^1}  \Sigma(p) \, .
\]
Notice that the rational normal curves are rational normal scrolls, with $m=1$ according to the definition above.

It is natural to consider the rational normal scroll as higher-dimensional analogues of rational normal curves.
The analogies between rational normal curves and scrolls do not reduce to similar definitions,  and to both being varieties of minimal degree.
They encompass  many other aspects. For instance, the rational normal scrolls of
dimension $m$ in $\mathbb P^n$ admit determinantal presentations, similar to (\ref{E:quadrics}) used for rational normal curves.
More precisely, if $F_0, \ldots, F_{n-m}, G_0, \ldots, G_{n-m}$ are linear forms such that for any $(\lambda, \mu) \neq (0,0)$, the linear forms
 ${\{ \lambda F_i + \mu G_i\}_{i=0,\ldots,n-m}}$ are linearly independent (compare with Lemma \ref{L:claimB}) then
\[
X=  \left\{ \mathrm{rank} \left( \begin{array}{ccc} F_0 & \ldots & F_{n-m} \\ G_0 & \ldots & G_{n-m} \end{array} \right) \le 1 \right\}
\]
is a rational normal scroll of dimension $m$. Moreover, any rational normal scroll can be presented in this way. \index{Rational normal scroll|)}

\medskip

Another testimony of the similarity between rational normal curves and scrolls, is the following generalization of  Castelnuovo Lemma.

\index{Castelnuovo Lemma!generalized}
\begin{prop}[{\bf Generalized Castelnuovo Lemma}]
Let  $\mathcal P \subset \mathbb P^n$  be a set of $k$ points in general position. Suppose that $n \ge 2$ and $k \ge 2n + 1 + 2m$.
If $\mathcal P$ imposes   $2n + m$ conditions on the linear system of quadrics then $\mathcal P$ is contained in
a rational normal scroll of dimension $m$.
\end{prop}

The proof of Castelnuovo Lemma presented in Section \ref{S:Castproof} is the specialization to $m=1$ of the Eisenbud-Harris proof of
the generalized Castelnuovo Lemma. Having at hand the determinantal presentation of a rational normal scroll given above the
reader should not have difficulties to recover the original  proof as found in \cite[pages 103-106]{EHCAMG}.
\index{Varieties of minimal degree|)}

\medskip

While Castelnuovo Lemma is essential in the proof of Tr\'{e}preau's algebraization theorem,
 to be carry out in Chapter \ref{Chapter:Trepreau}, the implications of the generalized Castelnuovo Lemma  to web geometry, if any, remain to be unfolded.

\index{Web!maximal rank|)} 
%%%%%%%%%%%%%%%%%%%%%%%%%%%%%%%%%%%%%%%
%  Introduction
%  label = Chapter:intro
%  First version by JVP
%  last modification: 3/nov/2008
%  Remarks:
%%%%%%%%%%%%%%%%%%%%%%%%%%%%%%%%%%%%%%%

\chapter{Abel's addition Theorem}\label{Chapter:3}
\thispagestyle{empty}

So far,  not many examples of abelian relations for webs  appeared in this text. Besides
the abelian relations for  hexagonal $3$-webs,   the polynomial abelian relations for parallel webs (~see Example \ref{E:llll}~), and  the abelian relations for the planar quasi-parallel webs discussed in  \ref{E:exabel},
which are by the way also polynomial, no other example was studied.

\smallskip

The main result of this chapter, Abel's addition Theorem,  repairs this unpleasant state of affairs. It implies, and is
essentially equivalent to,    an
injection of the space of global  abelian differentials -- also known as Rosenlicht's or regular differentials -- of a reduced
projective curve $C$ into the space of abelian relations of the dual web $\mathcal W_C$.

\medskip

The exposition that follows renounces conciseness in favor of clarity. First the result is proved for smooth
projective curves avoiding the technical difficulties inherent to the singular case.
Only then the  case of an arbitrary reduced projective curve is dealt with.

\medskip

The readers familiar with Castelnuvo's bound for the arithmetical genus of irreducible  non-degenerate projective
curves will promptly realize that  Castelnuovo curves (~briefly described in Section \ref{S:CC}~) give rise to webs of maximal rank. Those that are not,
will get acquainted with Castelnuovo's bound since it can be seen as a joint corollary of the bound for the rank
proved in Chapter \ref{Chapter:AR}, and  Abel's addition Theorem. While most of the arguments laid down in Chapter
\ref{Chapter:AR} to bound $\mathcal A(\mathcal W)$ can be found in the modern textbooks dealing with Castelnuovo
Theory, the same cannot be said about the use of Abel's addition Theorem. These days, the textbook proof of Castelnuovo's bound
makes use of some basic results about the cohomology of projective varieties besides Castelnuovo Lemma.

\smallskip

Not deprived of weaknesses when compared to the modern approach, the path to Castelnuovo's bound through
Abel's addition Theorem  has its own strong points. For instance, it  is more or less simple to obtain bounds
for reduced curves on other projective varieties as explained in Section \ref{S:beyond}.

\section{Abel's Theorem I: smooth curves}

\subsection{Trace  under  ramified coverings}\label{S:Traceunderramifiedcoverings}
\index{Trace!under ramified coverings|(}

Let $X$ and $Y$ be two smooth  connected  complex curves. A holomorphic map  $f:X \to Y$ is a  \defi[finite ramified covering] \index{Finite ramified covering}
if it is surjective and proper. The \defi[degree] of a finite ramified covering  \index{Finite ramified covering!degree} is defined as
the cardinality of the pre-image of any of its regular values.

\medskip

For $X$ and $Y$ as above, let $f:X \to Y$ be a finite ramified covering of degree $k$.
For a regular value  $q \in Y$   of $f$  and a meromorphic $1$-form $\omega$ defined at a neighborhood of
$f^{-1}(y)$,   define the \defi[trace of $\omega$ at $q$]
as the germ of  meromorphic $1$-form
\[
\mathrm{tr}_{f,q}(\omega)=\sum_{i=1}^k g_i^*(\omega) \, \in\,  \Omega^1(Y,q) \, ,
\]
where   $g_1,\ldots,g_k:(Y,q) \to X$ are the local inverses of $f$ at $q$.

\begin{prop}\label{P:pp}
Let $X$ and $Y$ be two smooth, compact and  connected  complex curves. If $f:X \to Y$ is a finite ramified covering,
$\omega$ is a meromorphic $1$-form globally defined on $X$, and $q$ is a regular value of $f$  then
$\mathrm{tr}_{f,q}(\omega)$ extends to a unique meromorphic $1$-form $\mathrm{tr}_{f}(\omega)$, which does not depend on $q$, and is
 globally defined on $Y$. Moreover, if $\omega$ is holomorphic then $\mathrm{tr}_f(\omega)$ is also
holomorphic.
\end{prop}

The meromorphic $1$-form $\mathrm{tr}_f(\omega)$ globally defined  on $Y$ is, by definition,  the \defi[trace of $\omega$ relative to $f$].

\begin{proof}[Proof of Proposition \ref{P:pp}]
For  $q$ varying among the regular values of $f$, the meromorphic $1$-forms $\mathrm{tr}_{f,q}$ patch together
to a  meromorphic $1$-form $\eta$ defined on the whole  complement of the
 critical values of $f$. Furthermore, if  $\omega$ is holomorphic then the same will
be  true for $\eta$.

Now, if $q \in Y$ is a critical value of $f$ then some point $p$  in the fiber $f^{-1}(q)$ is a
critical point. Although it is not possible to consider a local inverse $g:(Y,q) \to (X,p)$, the
map $f:(X,p) \to (Y,q)$ is, in  suitable coordinates, the monomial map $f(x) = x^n = y$
for some positive integer $n$.  Because $X$ is compact, the set of critical values of $f$ is discrete.
Therefore it suffices to consider the trace of the monomial maps from the disc $\mathbb D$ to itself to
prove that $\eta$ extends through the critical set of $f$.

For a point distinct from the origin, there are exactly $n$ local inverses for $f(x) = x^n$: the  distinct branches of $\sqrt[n]{x}$.
One passes   from one to another, via multiplication by powers of  $\xi_n$, a primitive $n$-th root of the unity. Hence
\[
\mathrm{tr}_f( x^m dx ) = \sum \xi_n^{m+1} x^m dx \, .
\]
Consequently,
\begin{equation}\label{E:tracef}
\mathrm{tr}_f( x^m dx ) =  \begin{cases}
y^{ \frac{m+1}{n}-1} dy   \;    \; \mbox{ if }   m+1= 0 \; {\rm mod}\; n \, ,
 \\
\quad \;  0             \hspace{1.1cm}            \mbox{  otherwise}     \, .  \end{cases}
\end{equation}
Therefore the trace of $x^m dx$ is meromorphic  at the origin. It follows that $\eta$ extends to the whole $Y$.
Moreover, when the $1$-form $\omega$ is holomorphic,  $m\ge0$ and, according to Equation (\ref{E:tracef}),  the trace of $x^m dx$ is also holomorphic.
\end{proof}

Beware that there are meromorphic $1$-forms with holomorphic trace as one can promptly infer from (\ref{E:tracef}).

\begin{remark}\rm
The algebraically inclined reader familiar with K\"{a}hler differentials and field extensions, might prefer
to define the trace as follows.
If $X$ and $Y$ are algebraic curves and $f : X \to Y $ is a finite ramified covering, then
there is an   induced  finite  field extension  $f^*:\mathbb C(Y)\rightarrow \mathbb C(X)$ of the corresponding
function fields.
In this case, a rational function  $\phi\in \mathbb C(X)$ has trace $\mathrm{tr}_f(\phi)\in \mathbb C(Y)$
equal to  the trace of the  endomorphism $\psi\mapsto \phi\psi$ of
the finite dimensional \mbox{$\mathbb C(Y)$-vector} space $\mathbb C(X)$.
Let $t\in \mathbb C(Y)$ be such that $dt$ generates   $\Omega_{\mathbb C(Y)}$ as a ${\mathbb C(Y)}$-module. Therefore  $ f^*(dt)=df^*(t)$ generates the ${\mathbb C(X)}$-module
of K\"ahler differentials on $X$. Hence,  $\omega= \varphi \,df^*(t)$ with $\varphi$ meromorphic on $X$. The trace of $\omega$ relative to $f$  is
algebraically defined as $\mathrm{tr}_f(\varphi) dt$.
\end{remark}
\index{Trace!under ramified coverings|)}
\subsection{Trace  relative to the family of hyperplanes}
\index{Trace!relative to hyperplanes|(}
Let now $C \subset \mathbb P^n$ be a smooth and irreducible projective curve of degree $k$, $H_0$ a
hyperplane intersecting $C$ transversely, and $\omega$ be a meromorphic $1$-form defined on a neighborhood
$V \subset C$ of $H_0 \cap C$.
Consider the germs of  holomorphic
maps ${p_1,\ldots,p_k: (\check{\mathbb P}^n,H_0) \rightarrow V \subset C}$ verifying
$$ H\cdot C=
%\sum_{i=1}^k p_i(H)
 p_1(H) +\cdots +p_k(H) \,
$$
for every $H\in (\check{\mathbb P}^n,H_0)$.

The \defi[trace of  $\omega$ at $H_0$ relative to the family of hyperplanes], denoted by $\mathrm{Tr}_{H_0}(\omega)$, is defined through the
formula
\begin{equation*}
 \mathrm{Tr}_{H_0}(\omega) =\sum_{i=1}^k p_i^*(\omega)\,.
\end{equation*}
It is clearly a germ of meromorphic differential 1-form.  As  the trace relative to a ramified covering, it extends
meromorphically to the whole projective space $\check{\mathbb P}^n$ as  proved in the next section. This is essentially the content
of Abel's addition Theorem.

\index{Trace!relative to hyperplanes|)}

\subsection{Abel's Theorem for smooth curves}
\index{Abel's Theorem!smooth curves|(}

The next result is the version for smooth curves
of what is called by  web geometers Abel's addition Theorem, or just Abel's Theorem.
The readers are warned that authors with other
backgrounds might call   a  different, but essentially equivalent, statement by the same name.
For a thorough discussion about the original version(s)  of Abel's  theorem
see \cite{kleiman}.

\begin{thm}[Abel's addition Theorem for smooth curves]\label{T:abel2}
If $\omega$ is a meromorphic $1$-form on a smooth projective curve $C\subset \mathbb P^n$, then the germ $\mathrm{Tr}_{H_0} (\omega)$ extends to
a unique meromorphic $1$-form $\mathrm{Tr}(\omega)$
globally defined on $\check{\mathbb P}^n$ which does not depend on $H_0$. Moreover, $\omega$ is a  holomorphic $1$-form on $C$  if and only if  $\mathrm{Tr}(\omega)=0$.
\end{thm}

To prove Theorem \ref{T:abel2}, let  $\check{U}_C\subset \check{\mathbb P}^n$ be   the Zariski open subset
formed by the hyperplanes $H\subset \mathbb P^n$ which intersect $C$ at $k=\deg C$ distinct points.
In other words, $\check{U}_C$ is the complement in $\check{\mathbb P}^n$ of the discriminant of
the dual web $\mathcal W_C$. The construction of $\mathrm{Tr}_{H_0}(\omega)$ made above can be done
for any hyperplane  $H \in \check{U}_C$. The results   patch together to define
a meromorphic 1-form $\mathrm{Tr}(\omega)$ on $\check{U}_C$.

To extend  $\mathrm{Tr}(\omega)$ through the discriminant of $\mathcal W_C$, it will be used  a relation between the trace under a
 ramified covering and the trace relative to the hyperplanes.
To draw this relation, let  $\ell$ be a line on $\check{\mathbb P}^n$.
It corresponds to a pencil of hyperplanes in $\mathbb P^n$ which has base locus equal to  $\Pi= \check \ell \subset \mathbb P^n$, the
$\mathbb P^{n-2}$ dual to $\ell$.
It will be convenient, although not strictly necessary,  to assume that $\Pi$ does not intersect $C$.
Define
\[
\pi_{\ell}: C \to \ell \simeq \mathbb P^1
\]
as the morphism that associates to a point $x \in C$ the hyperplane in $\ell$ containing it. Clearly it is
a ramified covering, thus  the trace $\mathrm{tr}_{\pi_{\ell}}(\omega)$ makes sense for any meromorphic $1$-form
on $C$.

\begin{lemma}\label{L:tt}
For every meromorphic $1$-form $\omega$  on $C$, the trace of $\omega$ under $\pi_{\ell}$ coincides with the
pull-back to $\ell$ of  the trace of $\omega$ relative to the family of hyperplanes, that is
\[
\mathrm{tr}_{\pi_\ell}(\omega) = i^* \mathrm{Tr}(\omega)\, ,
\]
where $i:\ell \to \check{\mathbb P}^n$ is the natural inclusion.
\end{lemma}
\begin{proof}
Consider the composition $\varphi = i \circ \pi_{\ell}: C\rightarrow \check{\mathbb P}^n$. The image of a point  $q \in C$ is the hyperplane $H$
containing both the point  $q$ and the linear space $\Pi= \check \ell$.

The hyperplane  $H$ intersects
$C$ at $q$, and  at  other $k-1$ points of $C$ with  multiplicities taken into account. All these other points  are
also mapped to $H$ by $\varphi$. Thus,  the  functions  $p_i \circ i : \ell \cap \check{U} \to C$
are local inverses of $\varphi$ at any $H \in \check{U}_C$. Hence
\[
\mathrm{tr}_{ \pi_\ell}(\omega) = \mathrm{tr}_{i \circ \pi_\ell}(\omega) = \sum_{i=1}^k (p_i \circ i)^* \omega = i^*  \mathrm{Tr}(\omega) \,
\]
at the generic point of $\ell$.
The lemma follows.
\end{proof}

Back to the proof of Abel's Theorem, recall that $\mathrm{Tr}(\omega)$ is defined all over the Zariski open set $\check{U}_C$. If  it does not
extend  meromorphically to the whole $\check{\mathbb P}^n$,  then its pull-back to a generic line ${\ell \subset \check{\mathbb P}^n}$ has an essential singularity at one of the
points of $\ell \cap \Delta(\mathcal W_C)$. Lemma \ref{L:tt} implies the existence of an essential singularity for $\mathrm{tr}_{\pi_{\ell}}(\omega)$.
But this cannot be the case according to Proposition \ref{P:pp}.

\smallskip

To prove that  $\omega$  holomorphic implies   $\mathrm{Tr}(\omega) =0 $, start by noticing
that  there are no non-zero holomorphic differential forms on $\check{\mathbb P}^n$.
If $\omega$ is holomorphic and $\mathrm{Tr}(\omega)$ is non-zero then $\mathrm{Tr}(\omega)$  has non-empty polar set.
Therefore $\mathrm{Tr}(\omega)$ pulls-back to a generic line ${\ell \subset \check{\mathbb P}^n}$ as a
 meromorphic, but not holomorphic, differential. As above, Lemma \ref{L:tt} and Proposition \ref{P:pp} lead to a contradiction.

\smallskip

It remains to establish the converse implication.  To
prove  the contrapositive, suppose    $\omega$ is not holomorphic. If  $x \in C$  is a pole of $\omega$ then the generic
hyperplane $H \subset \mathbb P^n$ through $x$ intersects $C$ transversely and avoids all the other poles
of $\omega$. Thus, in a neighborhood of $H$ in $\check{\mathbb P}^n$, the trace of $\omega$ is the sum
of the pull-back by a holomorphic map of a meromorphic, but not holomorphic,  $1$-form   with other $\deg(C)-1$ holomorphic $1$-forms.
Hence $\mathrm{Tr}(\omega)$ has non-empty polar set and, in particular, is not zero.
\qed
\index{Abel's Theorem!smooth curves|)}

\subsection{Abelian relations for algebraic webs}\label{S:ARAW}

Theorem \ref{T:abel2} can be interpreted in terms of webs/abelian relations
instead of projective curves/holomorphic $1$-forms.  More precisely,

\begin{thm}\label{T:curvesvswebs}
If  $C$ is a smooth projective curve of degree $k$ and $H_0$ is a hyperplane intersecting it transversely, then the space of holomorphic $1$-forms on $C$
injects into the space of abelian relations of the dual web $\mathcal W_C(H_0)$.
\end{thm}
\begin{proof}
Let  $H_0$  be a hyperplane intersecting  $C$ transversely in $k$ points, and
$p_i:(\check{\mathbb P}^n,H_0) \to C$ be germs of holomorphic functions
 such that $H\cdot C=\sum_i p_i(H)$ for all $H\in (\check{\mathbb P}^n,H_0)$.
 Recall from Chapter~\ref{Chapter:intro} that the $k$-web $\mathcal W_C(H_0)$ is  defined
 by the submersions $p_1, \ldots, p_k$. That is, $\mathcal W_C(H_0) = \mathcal W(p_1, \ldots, p_k)$.

\smallskip

If $\omega$ is a holomorphic $1$-form on $C$ then it is automatically closed, for  dimensional reasons. Since the
exterior differential commutes with pull-backs, the $1$-forms $p_i^* \omega$ are also closed. Moreover,
the \mbox{$1$-form} $p_i^* \omega$ defines the very same foliation as the submersion $p_i$. Abel's addition theorem, in its turn, implies  that
$$
 \mathrm{Tr}(\omega)  = p_1^*(\omega)+\cdots + p_k^*(\omega)=0  $$
holds identically on $(\check{\mathbb P}^n, H_0)$.
Therefore $(p_1^* \omega, \ldots, p_k^* \omega)$ is an abelian relation of $\mathcal W_C(H_0)$.
It follows that the injective linear map
\begin{align*}
 H^0(C,\Omega^1_C) &\longrightarrow \Omega^1(\check{\mathbb P}^n, H_0)^k \\
\omega \;\;&\longmapsto (p_1^* \omega, \ldots, p_k^*\omega)
\end{align*}
factors through $\mathcal A(\mathcal W_C)\subset \Omega^1(\mathbb P^n, H_0)^k$.
\end{proof}

\index{Genus}
Recall that $g(C)$ -- the \defi[genus] of $C$ -- coincides
with $h^0(C,\Omega^1_C)$, the dimension of the vector space of holomorphic $1$-forms on $C$.

\begin{cor}\label{C:BNC}
If $C$ is a smooth projective curve  then, for any hyperplane $H_0$   intersecting it transversely,
\[
 \mathrm{rank} (\mathcal W_C(H_0))  \ge g(C) \, .
\]
\end{cor}

In Chapter \ref{Chapter:4}  it will be seen that this lower bound is in fact an equality.

\subsection{Castelnuovo's bound}
\index{Castelnuovo's bound|(}

Corollary \ref{C:BNC} read backwards yields the celebrated Castelnuovo's bound for the genus
of projective curves. More precisely,

\begin{thm}[Castelnuovo's bound]\label{T:castelnuovobound}
If  $C$ is a smooth connected non-degenerate  projective curve on $\mathbb P^n$ of degree $k$ then
\[
g(C) \le \pi(n,k)\, .
\]
\end{thm}
\begin{proof}
For a generic $H_0$, the web  $\mathcal W_C(H_0)$ is smooth according to Proposition \ref{P:gp1}.
 Chern's bound on the rank of smooth webs, see Theorem \ref{T:cota}, combined with Corollary \ref{C:BNC} implies the result.
\end{proof}

It is instructive to compare this proof of Castelnuovo's bound, with the usual textbook proof. The first step of
both proofs, relies on the bounds for the number of conditions imposed by points on the complete linear systems
of hypersurfaces on the relevant projective space. While the former proof  uses Abel's addition theorem to conclude, the latter
instead appeals to Riemann-Roch Theorem. For thorough discussion on this matter see \cite{Jbr}.

\index{Castelnuovo's bound|)}

\section{Abel's Theorem II: arbitrary curves}

When studying germs of smooth algebraic webs $\mathcal W_C(H_0)$, it is hard to tell whether the
curve $C$ is smooth or not.  At first  sight the web only exhibits  properties of $C$ valid at a neighborhood
of the  transversal hyperplane $H_0$. For  smooth curves,  it has just been explained how the holomorphic differentials
give rise to abelian relations for the dual web.
It is them natural to enquire:
\begin{enumerate}
\item[(a)] are there another abelian relations for algebraic webs  ?
\item[(b)] what kind of {\it differentials} on a singular curve give rise to abelian relations for the dual web?
\end{enumerate}

Question (a) will be treated in Chapter \ref{Chapter:4}, while question (b) will be the subject of the present section.
Before dwelling with it, some conventions about singular curves are settled below.

\subsubsection{Conventions}

Given a curve $X$,  the \defi[desingularization]  of $X$ will denoted by $\nu=\nu_X:\overline X \to X$.

A \defi[meromorphic $1$-form] on $X$ is nothing more than a meromorphic $1$-form $\omega$ on the smooth part of $X$
such that $\nu^* \omega$, its pull-back to $\overline X$, extends to the whole $\overline X$ as a meromorphic $1$-form.
The sheaf of meromorphic differentials on a curve $X$, singular or not, will be denoted by $\mathcal M_X$.

 For  $X$ an arbitrary curve and  $Y$ a smooth irreducible curve, a morphism  $f: X \to Y$ will
be called  a \defi[finite ramified covering], \index{Finite ramified covering!singular domain}
 if the restriction of $\overline f =  f \circ\nu_{X}$ to each of the irreducible
components of $\overline X$ is a finite ramified covering as defined in Section \ref{S:Traceunderramifiedcoverings}.

\subsection{Residues and traces}

\index{Residue|(}
Assume that $p$ is a smooth point of a curve $X$, and  $x$ is
a local holomorphic coordinate on $X$ centered at it.
If $\omega$ is a germ of meromorphic differential at $p$ then
\[
\omega=\sum_{i=i_0}^{\infty} a_i x^i dx \,
\]
for some $i_0 \in \mathbb Z$ and suitable complex numbers $a_i$.
The \defi[residue of $\omega$ at $p$] is the  complex number
\[
{\mathrm Res}_p \big( \omega\big) =\left\{  \begin{array}{cc}
                                    a_{-1}  & \text{ if } i_0 \le -1;\\
                                    0    & \text{ otherwise. }
                                  \end{array} \right.
\]
It is a simple matter to verify that this definition does not depend on the local coordinate $x$. One possibility is to notice that the residue
can  be  determined through the integral formula  $${\mathrm Res}_p ( \omega ) =\frac{1}{2i\pi}\int_{\gamma } \omega \, ,  $$
for any  sufficiently small  (positively oriented)  loop $\gamma$ around $p$.

\medskip

The following properties can be easily verified:
\begin{enumerate}
\item ${\mathrm Res}_p: \mathcal M_{X,p} \rightarrow \mathbb C$ is $\mathbb C$-linear;
\item ${\mathrm Res}_p (\omega)=0$ when $\omega$ is holomorphic at $p $ ;
\item  ${\mathrm Res}_p(f^n df)=0$ for all $f\in \mathcal O_{X,p}$ when  $n\neq -1$;
\item ${\mathrm Res}_p(f^{-1} df)=\nu_p(f)$ for all meromorphic germs $f$ at $p$ (where $\nu_p$ is the valuation associated to $p$).
 \end{enumerate}

\subsubsection{Residues at singular points}

Assume now that $X$ is singular and  let $\omega$ be a meromorphic differential $1$-form on it.
Let $\nu:\overline{X}\rightarrow X$  be the desingularization of $X$.  The residue  of $\omega$ at a  singular point
$p \in X_{sing}$  is defined as
\begin{equation}
 \label{E:defRESsing}
{\mathrm Res}_p(\omega)=\sum_{q\in\nu^{-1}(p)} {\mathrm Res}_{q}\big(  \nu^*(\omega)\big)\,.
\end{equation}
It is completely determined by the germ of $\omega$ at $p$.

Given a ramified covering $f:X \to Y$ between an eventually singular curve $X$ and
a smooth and irreducible curve $Y$  then
the trace  under $f$ of any meromorphic $1$-form $\omega$ on $X$  is defined by the relation
\[
 \mathrm{tr}_f ( \omega ) =  \, \mathrm{tr}_{\overline f} ( \nu_X^* \omega) \, .
\]

\begin{prop} \label{P:restra}
Let $X$ and $Y$  be curves with $Y$ smooth and irreducible. If  $f:X\rightarrow Y$ is  a  ramified covering then for every $\omega\in H^0(X,\mathcal M_X)$ and every $ p\in Y$,
\begin{equation*}
 \quad  {\mathrm Res}_p\big(  \mathrm{tr}_f(\omega) \big)=\sum_{q\in f^{-1}(p)}
 {\mathrm Res}_q\big(  \omega\big)
.
\end{equation*}
%for all $y\in Y$.
\end{prop}
\begin{proof}
Let $\nu_X : \overline{X} \to X$  be the normalizations of $X$.  Let
also  $\overline{f}:\overline{X} \rightarrow {Y}$ be the natural lifting of $f$, that is  $\overline f= f \circ  \nu_X$.

Since $\overline{f}^{-1}(p)=(f\circ \nu_X)^{-1}(p)$ and because of definition (\ref{E:defRESsing}),
one verifies  that the proposition holds for $f:X\rightarrow Y$ if it holds for $\overline{f}:\overline{X} \rightarrow Y$.
It is therefore harmless to assume smoothness for both $X$ and $Y$.

Let $q_1, \ldots, q_m$ be the pre-images of $p$  under $f$.
For $i \in \underline m$, let  $f_i:(X,q_i)\rightarrow (Y,p)$ be  the germ of analytic morphism induced by $f$,
and  $\omega_i\in \mathcal M_{X,x_i}$ be the germ of  $\omega$ at $p_i$. Clearly,  $\mathrm{tr}_f(\omega)
=\sum_{i=1}^m \mathrm{tr}_{f_i}(\omega_i)$ as germs at $p$.  Thus
$${\mathrm Res}_p \big(  \mathrm{tr}_f(\omega) \big)=\sum_{i=1}^m  {\mathrm Res}_p \big( \mathrm{tr}_{f_i}(\omega_i)\big) $$ by the additivity of the residue.  In suitable coordinates,  each of the functions $f_i$ can be written
as  $f_i(z_i) = z_i^{n_i}$, for suitable $n_i \in \mathbb Z$. A quick inspection of (\ref{E:tracef}) leads  to
 the identity ${{\mathrm Res}_p\big[\mathrm{tr}_{f_i}(\omega_i)\big]={\mathrm Res}_{q_i}\big[ \omega_i\big]}$ for every $i \in \underline m$.
The proposition follows.
\end{proof}
\index{Residue|)}

\subsection{Abelian differentials}
\index{Abelian differential|(}

Let $X$ be a curve and $\nu: \overline X \to X$ be its desingularization.
An \defi[abelian differential] $\omega$ on  $X$ is
a meromorphic $1$-form on $X$ which satisfies
\[
 {\mathrm Res}_p \big(  f\omega  \big)=
\sum_{q\in \nu^{-1}(p)} {\mathrm Res}_{q}
\big[  \nu^*(f\omega ) \big]=
0
\]
for every $p \in X$ and every $f \in \mathcal O_{ X, p}$.

\medskip

For any open subset $U\subset X$, let  $\omega_X(U)$ be
 the set of abelian differentials on $U$.
Of course $\omega_X(U)$ inherits from $\mathcal M_X(U)$    a structure of $\mathcal O_X(U)$-module.
Indeed even more is true, the subsheaf $\omega_X $ of $\mathcal M_X$ is coherent.
Later, a proof that $\omega_X$ is coherent will be presented under the assumption that
$X$ is contained in a smooth surface. The general case will not be treated in this text.
In the algebraic category  the result  goes back to Rosenlicht. For a
treatment of the analytic case  see  \cite{barlet}.

\begin{remark}
It is  possible
 to characterize abelian differentials  in terms of currents. Indeed, a meromorphic $1$-form $\omega$
 on  a curve $X$  is abelian,  if and only if the current $[\omega]$
 defined by it  is $\overline{\partial}$-closed. That is, if
$$\langle \overline{\partial}[\omega], \theta\rangle: =
\langle[\omega],  \overline{\partial}\theta\rangle= \int_X \omega \wedge \overline \partial \theta = 0 \, ,$$
for every smooth complex-valued function $\theta$ with compact support on $X$.
The  characterization of abelian  differentials in terms of  currents  generalizes promptly,
and allows to define a notion of abelian differential $k$-forms  in arbitrary  dimension.
 The interested reader can consult the references  \cite{barlet,henkinpassare}.
\end{remark}

For an arbitrary   projective  curve $X$, the coherence of $\omega_X$ implies that $H^0(X, \omega_X)$ is a finite dimensional vector space.
Its dimension is, by definition, the \index{Arithmetic genus} \index{Genus!arithmetic} \defi[arithmetic genus] $g_a(X)$ of $X$.
The \index{Geometric genus} \index{Genus!geometric} \defi[geometric genus] $g(X)$ of $X$, in its turn,  is
defined as the dimension of $H^0(\overline X, \Omega^1_{\overline X})$, where $\overline X$ is the  desingularization of $X$.

When  $X$ is smooth,  the sheaf $\omega_X$ is nothing more than
 $\Omega^1_X$, hence $g(X)$ coincides with $g_a(X)$. For singular curves, the equality between  $g(X)$ and $g_a(X)$ is the exception rather than the rule.

 \medskip

The sheaf  $\omega_X$ is also called the \defi[dualizing sheaf] of $X$.  \index{Dualizing sheaf}
The terminology steams from  Serre's duality for projective curves: \index{Serre's dualiaty} {\it for any coherent sheaf $\mathscr F$ on
a  projective curve $X$, there are  natural isomorphisms between
$H^i(X, \mathscr F)$ and $ H^{1-i}(X,{\mathscr F}^*\otimes  \omega_X)^*$ for $i=0,1$.}

\medskip

When $X$ is not just  projective but also connected,
Serre's duality for projective curves   is essentially equivalent to  \defi[Riemann-Roch Theorem]:
\index{Riemann-Roch Theorem} {\it for any line bundle $\mathcal L$ on $X$,
the identity\begin{footnote}{In the formula, $\chi(X,\mathcal L)$ stands for the Euler-characteristic
of the line-bundle $\mathcal L$ which, by definition, is $h^0(X,\mathcal L) - h^1(X,\mathcal L)$. }\end{footnote}  $\chi(X,\mathcal L)  =deg(\mathcal L)-g_a(X)+1$
holds true. }\smallskip

Applying Riemann-Roch Theorem to the dualizing sheaf itself, one obtains the \defi[genus formula] for irreducible projective curves
\begin{equation}\label{E:genusformula}
\deg(\omega_X)  = 2g_a(X) -2 \, .
\end{equation}

\medskip

Before proceeding toward  the proof of Abel's Theorem for arbitrary curves, a couple of examples will be considered
in order to clarify the concept of abelian differential.

\begin{example}
Let  $X=\{(x,y)\in \mathbb C^2\, , \, y^2 - x^3 =0 \}$. Clearly, it is a curve with the origin of $\mathbb C^2$ as its unique singular point.
The stalk  $\omega_{X_0}$  is a free $\mathcal O_{X,0}$-module  generated by
$ y^{-1}{dx}  \,.
$
Indeed, the normalization of $X$ is given by
\begin{align*}
 \nu :\;  \mathbb C &\longrightarrow X \\
 t &\longmapsto  (t^2,t ^3) \, .
\end{align*}
Therefore $
  \nu^* \mathcal O_{X,0} =   \nu^* \mathcal O_{\mathbb C^2,0} = \nu^* \mathbb C \{ x, y\}  = \mathbb C \{ t^2, t^3 \}  .
$
If $\omega $ is a meromorphic differential on $X$  then $\nu^* \omega = \sum_{i=-k}^{\infty} a_i t^i dt$.
Moreover, if  $\omega$ is abelian then not only
$a_{-1}= {\mathrm Res}_0 ( \nu^* \omega)$ must be zero but also $a_{ 2 n + 3 m -1} =  {\mathrm Res}_0 \big( \nu^* (x^n y^m \omega) \big)$ for any
pair of positive integers $(n,m)$. Hence every germ at $0$ of abelian differential on $X$ can be written as
\[
 \frac{dt}{t^2} \left( a_{-2} + a_0 t^2 + a_1 t^3 + \cdots \right) \in  \frac{dt}{t^2} \cdot \mathbb C \{ t^2,t^3\} \, .
\]
\end{example}

\smallskip

In a similar vein, a family of rational projective curves contained
in projective spaces of dimension $n \ge 2$ is considered below.

\smallskip

\begin{example}
\label{Ex:unicursalSINGULARcurve}
 Fix $n\geq 2$ and let $C$ be the rational curve of degree $2n$ in $\mathbb P^n$
 parametrized by
 \begin{align*}
 \nu: \mathbb P^1 & \longrightarrow \mathbb P^n  \,  \\
  (s:t) &\longmapsto [ s^{2n}: s^nt^n:s^{n-1}t^{n+1}:\cdots: s t^{2n-1}: t^{2n}].
\end{align*}
The curve $C$ is singular,   $p=[1:0:\cdots:0]$ is its unique singular point, and
the parametrization  $\nu$  is its desingularization.

 Any rational 1-form on $C$ writes $\omega=f(t)dt$  in the coordinate $(1:t)$, for a certain $f(t)\in \mathbb C(t)$.
 Assume that $\omega$ is abelian.
 Since $\omega_C$ coincides with the sheaf of holomorphic differentials on $C_{sm}=C\setminus \{p\}$, the differential
$\omega=f(t)dt$  must be holomorphic on $\mathbb C \setminus \{ 0 \}$  as well as at infinity. Therefore  $\omega=t^{-a}dt$ for a certain integer $a>1$.
It is  an instructive  exercise to show that
\begin{equation*}
H^0(C,\omega_C) \simeq \Big\langle
\frac{dt}{t^2} , \frac{dt}{t^3} , \ldots, \frac{dt}{t^n} ,\frac{dt}{t^{n+1}} ,\frac{dt}{t^{2n+2}}    \Big\rangle .
\end{equation*}
\end{example}

\subsubsection{Abelian differentials and traces}

The following proposition can be seen as a first evidence
 that  the concept of abelian differential
is the appropriate one  to extend Abel's addition Theorem
to singular projective curves. Note that, as for Proposition \ref{P:pp}, its  converse does not hold true.

\begin{prop}\label{P:aimpliesb}
Let $\omega$ be an abelian differential on a curve $X$.
If  $f:X\rightarrow Y$ is a ramified covering onto a smooth curve $Y$ then
$\mathrm{tr}_f(\omega)$ is  an abelian, that is a holomorphic, differential on $Y$.
\end{prop}
\begin{proof}
The proof follows the same lines of Proposition \ref{P:pp}'s proof with some extra
ingredients borrowed from the proof of Proposition \ref{P:restra}. The reader is invited to
fill in the details.
\end{proof}

Besides the concept of abelian differential, the main extra ingredient to generalize Abel's addition Theorem from smooth to arbitrary curves,
is the following characterization of abelian differentials in terms of their traces under  linear projections.

\begin{prop}\label{P:project}
 Let $X \subset (\mathbb C^n,0)$ be a germ of curve and $\omega$ be a meromorphic $1$-form
on $X$. The following assertions are equivalent
\begin{enumerate}
 \item[(a)] $\omega$ is abelian;
\item[(b)] the trace of $\omega$ at $p$, $\mathrm{tr}_p(\omega)$, is holomorphic
 for a generic linear projection ${p:(\mathbb C^n,0) \to (\mathbb C,0)}$.
\end{enumerate}
\end{prop}
\begin{proof}
From Proposition \ref{P:aimpliesb} it is clear that  (a) implies (b).
To prove that (b) implies (a)  suppose that $\omega$ is not holomorphic  at $0$.
If this is the case, then
there exists $f \in \mathcal O_{X,0}$ such that
\[
{\mathrm Res}_0 (f\omega) \neq 0 \, .
\]
By the additivity of the residue, the function $f$ can be replaced by the restriction to $X$ of a monomial function
\[
x^J= x_1^{j_1} \cdots x_n^{j_n} \in \mathcal O_{\mathbb C^n,0} = \mathbb C\{x_1, \ldots, x_n \}
\]
which  still satisfies ${\mathrm Res}_0 (x ^J \omega) \neq 0 \, . $

\smallskip

For $\epsilon = ( \epsilon_1, \ldots, \epsilon_n ) \in (\mathbb C,1)^n$, define
$p_{\epsilon} : (\mathbb C^n,0) \to (\mathbb C,0)$ as the linear projection
\[
p_{\epsilon} ( x_1, \ldots, x_n )  = \sum_{i=1}^n \epsilon_i x_i \, .
\]
Consider now  the monomial $t^{|J|} \in \mathcal O_{\mathbb C,0} = \mathbb C \{ t\}$, where $|J| = \sum j_i$.
Notice  that, for every $\epsilon \in (\mathbb C,1)^n$,
\[
t^{|J|} \mathrm{tr}_{p_{\epsilon}}( \omega ) = \mathrm{tr}_{p_{\epsilon}} \left( p_{\epsilon}^* \big(t^{|J|}  \big) \omega \right) \, .
\]
Consequently, Proposition \ref{P:restra} implies
\[
{\mathrm Res}_0 ( t^{|J|} \mathrm{tr}_{p_{\epsilon}}( \omega ) ) = {\mathrm Res}_0 (  p_{\epsilon}^* \big( t^{|J|} \big) \omega ) \, .
\]
But, using again the additivity of the residue,
\[
{\mathrm Res}_0 \Big( p_{\epsilon}^*\big( t^{|J|} \big) \omega \Big)  = \sum_{\substack{ K \in\mathbb N^n  \\|K| = |J|}} \binom{|K|}{K} \epsilon^{K} {\mathrm Res}_0 (    x^{K} \omega ) \, ,
\]
where $\epsilon^K=\epsilon_1^{k_1}\cdots \epsilon_n^{k_n}$ and
$\binom{|K|}{K}=
\binom{|K|}{k_1}\cdots\binom{|K|}{k_n}$.

\smallskip

Since ${\mathrm Res}_0 (    x^{J} \omega ) \neq 0 $,  the polynomial in the variables
$\epsilon_1, \ldots, \epsilon_m$ on the righthand side is not zero. Consequently, for a generic $\epsilon$, the
meromorphic $1$-form $ {\rm  tr}_{p_{\epsilon}}(
p_{\epsilon}^*(t^{|J|}) \omega )$ has a non-zero residue at the origin. This suffices to establish  that (b) implies (a).
\end{proof}
\index{Abelian differential|)}

\subsection{Abel's addition Theorem }

Having the concept of abelian differential as well as  Proposition \ref{P:project} at hand, there is no difficulty to adapt the proof of
Theorem \ref{T:abel2} to establish Abel's addition Theorem for arbitrary projective curves.

\begin{thm}[Abel's addition Theorem]
If $\omega$ is a meromorphic $1$-form on a  projective curve $C$ then $\mathrm{Tr}(\omega)$ is a meromorphic
$1$-form on $\check{\mathbb P}^n$. Moreover, $\omega$ is abelian if and only if its trace $\mathrm{Tr}(\omega)$ vanishes identically.
\end{thm}

As a by product, Theorem \ref{T:curvesvswebs} also generalizes to the following

\begin{thm}\label{T:curvesvswebs2}
If  $C$ is a  projective curve of degree $k$ and $H_0$ is a hyperplane intersecting it transversely,
then the space of abelian  $1$-forms on $C$
injects into the space of abelian relations of the dual web $\mathcal W_C(H_0)$.
\end{thm}

Using notation similar to the one of Section
\ref{S:ARAW}, the preceding result can be formulated as follows:  the  injective linear map
{\begin{align*}
 H^0(C,\omega_C) &\longrightarrow \big(\Omega^1(\check{\mathbb P}^n, H_0)\big)^k \\
\omega \;\;&\longmapsto (p_1^* \omega, \ldots, p_k^*\omega)
\end{align*}
factors through $\mathcal A(\mathcal W_C(H_0))$.

\smallskip

Of course there is also a corresponding version of Castelnuovo's bound for arbitrary reduced curves which intersects
a  generic hyperplane in points in general position.

\begin{thm}
Let   $C$ be a  reduced projective curve on $\mathbb P^n$ of degree $k$. If the dual web is generically smooth, for instance if $C$ is irreducible
and non-degenerate,
then $h^0(C,\omega_C) \le \pi(n,k)$.
\end{thm}

To ease further reference to projective curves with generically smooth dual web, these will be labeled \defi[$\mathcal W$-generic curves].
\index{Curve!$\mathcal W$-generic} Notice that according to Proposition \ref{P:gp1}, an irreducible curve is $\mathcal W$-generic if and only  if
it is non-degenerate.

\subsection{Abelian differentials for  curves on surfaces}
Let now $X$ be a  curve on a smooth compact connected surface $S$.
The purpose of this  section is to  describe the sheaf $\omega_X$ in terms of sheaves over $S$.
As usual, $K_S$ denotes the sheaf of holomorphic 2-forms on $S$ -- the \defi[canonical sheaf] \index{Canonical sheaf} of $S$ -- and $K_S(X)$
is used as  an abbreviation of $K_S\otimes \mathcal O_S(X)$.

\smallskip

If  $U$ is a sufficiently small    neighborhood of a point  $p \in X$ then
$X\cap U=\{ f=0\}$ for some  $f\in \mathcal O_S(U)$  generating the ideal $\mathcal I_{X}(U)$.
Notice that any  section $\eta\in \Gamma(U , K_S(X))$ can be written as $\eta={f}^{-1} h {dx}\wedge dy$, with $h \in \mathcal O_S(U)$.

If the coordinate functions $x$ and $y$ are not constant on any of the
irreducible components of $X$ then ${\mathrm Res}_X(\eta)$ --  the \index{Residue!along a curve} \defi[residue
  of $\eta$ along $X$]-- is, by definition,
 the restriction to $X$ of the meromorphic $1$-form
$(h/ \partial_y f) dx $. Explicitly,
\[
{\mathrm Res}_X(\eta) = \left( \frac{hdx } {\partial_y f} \right)\bigg|_{X}  .
\]
It is easy to verify that ${\mathrm Res}_X(\cdot)$  does not depend on the choice of $f$ nor on the choice of  local coordinates $x,y$.
In the literature, ${\mathrm Res}_X(\eta)$ also appears under the label of Poincar\'{e}'s, as well as Leray's, residue of $\eta$ along $X$.
\index{Poincar\'{e}'s residue} \index{Leray's residue}

Notice that for every $g  \in \mathcal O_S(U)$ and $\eta$ as above,  $${\mathrm Res}_X(g \eta)=g|_{X} \cdot {\mathrm Res}_X(\eta) .$$
Thus the map ${\mathrm Res}_X$ can be interpreted as a morphism of $\mathcal O_S$-modules from
$K_S(X)$ to $\mathcal M_X$. Of course, the structure of $\mathcal O_S$-module on the latter sheaf is the
one induced by the  inclusion of $X$ into $S$.

Clearly, the kernel of ${\mathrm Res}_X : K_S(X) \to \mathcal M_X$ coincides with the natural inclusion of
the canonical sheaf $K_S$ into $K_S(X)$. Therefore the sequence
$$  \xymatrix@R=0.54cm@C=1.1cm{
0\ar[r]  &K_S \ar@{->}[r] &  K_S(X) \ar@{->}[r]^{{\mathrm Res}_X\qquad } &
{\rm Im}\,{\mathrm Res}_X \subset \mathcal M_X\,
} $$
is exact.

\begin{prop}\label{P:bpv}
The image of ${\mathrm Res}_X$ is exactly $\omega_X$, the sheaf of abelian differentials on $X$.
Consequently  the sheaf $\omega_X$ is  coherent and the  adjunction formula\index{Adjunction formula}
\begin{equation*}
\big( K_S\otimes  \mathcal O_S(X)\big)\big|_X\simeq \omega_X
\end{equation*}
 holds true.
\end{prop}

The  proof of Proposition \ref{P:bpv} will make use of the following lemma.

\begin{lemma}\label{L:integral}
Let $X=\{ f=0\} $ be a reduced complex curve defined at a neighborhood of the origin of $\mathbb C^2$.
Suppose, as above, that the coordinate functions $x,y$ are not constant on any of the irreducible components of $X$.
If $\mathbb S$ is a
sufficiently small sphere centered at the origin, and transverse  to $X$, then  the identity
\[
\frac{1}{2 \pi i} \int_{\mathbb S \cap X} \frac{h dx}{\partial_y f} = \lim_{\epsilon \to 0} \int_{\mathbb S \cap \{ |f| = \epsilon \}} \frac{h dx \wedge dy}{f} \,
\]
is valid for any meromorphic function $h$ on $(\mathbb C^2,0)$.
\end{lemma}

We refer to \cite[page 52]{BPPV} for a proof.

\begin{proof}[Proof of Proposition \ref{P:bpv}]
Clearly ${\rm Im}\,{\mathrm Res}_X=\omega_X$  is a local statement. Moreover, for a smooth point  $p$ of $X$, there is
no difficult to see that both $\omega_{X,p}$ and $({\rm Im}\,{\mathrm Res}_X)_p$ are isomorphic to $\Omega^1_{X,p}$.
Let $p \in X$ be a singular point and take local coordinates $(x,y)$ centered at $p$ satisfying the assumption
of Lemma \ref{L:integral}.

To prove that $({\rm Im}\,{\mathrm Res}_X)_p$ is contained in $\omega_{X,p}$ notice that the former $\mathcal O_{S,p}$--module
is generated by $ (\partial_y f)^{-1} dx = {\mathrm Res}_X (f^{-1} dx\wedge dy)$. Lemma \ref{L:integral} implies that for every
 $h \in \mathcal O_{S,p}$
\[
{\mathrm Res}_0 \left( \frac{h dx}{\partial_y f}\right) = \frac{1}{2 \pi i} \int_{\mathbb S \cap X} \frac{h dx}{\partial_y f} =
\lim_{\epsilon \to 0} \int_{\mathbb S \cap \{ |f| = \epsilon \}} \frac{h dx \wedge dy}{f} \, .
\]
But, by Stoke's Theorem,
\[
\int_{\mathbb S \cap \{ |f| = \epsilon \}} \frac{h dx \wedge dy}{f} = \int_{\mathbb S \cap \{ |f| \ge \epsilon \}} d \left( \frac{h dx \wedge dy}{f} \right) = 0\, .
\]
Therefore $(\partial_y f)^{-1} dx \in \omega_{X,p}$ as wanted.

\medskip

Suppose now that the $1$-form
$\eta = h(\partial_y f)^{-1} dx$ is abelian for some meromorphic function $h$ on $X$.
Let  $n \in \mathbb N$ be the smallest integer for which the function  $x^n h$ is holomorphic at $p$, that is belongs to $\mathcal O_{X,p}$.
Let $h_n \in \mathcal O_{S,p}$
be a holomorphic function with restriction to $X$ equal to $x^n h$.

If $n=0$ then the relation $\eta={\rm Res}_X(h_0 f^{-1}dx\wedge dy)$ with $h_0 \in \mathcal O_{S,p}$ shows  that $\eta\in ({\rm Im\, Res}_X)_p$ as
wanted. Thus, from now on,  $n$ will be assumed  positive.

Since $\eta$ is abelian
\[
0 = {\mathrm Res}_0( g x^{n-1}  \eta ) ={\mathrm Res}_0 \left( g \frac{h_n}{x}   \frac{dx}{\partial_y f} \right)
\]
for every $g \in \mathcal O_{S,p}$. Applying Lemma \ref{L:integral}, and then  Stoke's Theorem,
one deduces that the right-hand side is, up to multiplication by $2 \pi i$,  equal to
\[
 \lim_{\epsilon \to 0 } \int_{\mathbb S \cap \{ |f|=\epsilon \}} \left( g \frac{h_n}{x}   \frac{dx\wedge dy}{ f} \right)
 =  - \lim_{\epsilon \to 0 } \int_{\mathbb S \cap \{ |x|=\epsilon \} }\left(g \frac{h_n}{x}   \frac{dx\wedge dy}{ f} \right) \, .
\]
Applying Lemma \ref{L:integral} again, but now to the curve $Y=\{x=0\}$, yields
\[
{\mathrm Res}_0 \left( \frac{g h_n}{f} dy \right) = 0 \,
\]
for every $g \in \mathcal O_{S,p}$. But this implies that $h_n/ f$ is a holomorphic function on $Y$. Therefore
\[
h_n = h_{n-1} x + a f
\]
where  $a, h_{n-1} \in  \mathcal O_{S,p}$ are holomorphic functions. Thus on $X$
\[
x^{n-1} h = h_{n-1}
\]
contradicting the minimality of $n$. The inclusion  $\omega_{X,p} \subset ({\rm Im}\,{\mathrm Res}_X)_p$ follows.
Therefore $\omega_{X} = ({\rm Im}\,{\mathrm Res}_X)$. The coherence and  adjunction formula follow at once
from the exact sequence
\[
0 \to K_S \to K_S(X) \to \omega_X \to 0 \, . \qedhere
\]
\end{proof}

\medskip

The preceding proof also shows  the following
\begin{cor}
\label{C:plane=gorenstein}
For a curve $X$ embedded in a smooth surface the sheaf   $\omega_X$ is locally free.
\end{cor}

Curves for  which $\omega_X$ is   locally free  are usually called \index{Gorenstein curves}
\index{Curve!Gorenstein} \defi[Gorenstein].
Corollary \ref{C:plane=gorenstein} can be succinctly rephrased as: {\it germs of planar curves are Gorenstein.}

This is no longer true for arbitrary singularities. The simplest example is perhaps the  germ of curve $X$  on
$(\mathbb C^3,0)$ having the three coordinate axis as irreducible components.

\smallskip

Corollary \ref{C:plane=gorenstein} and  the genus formula (\ref{E:genusformula}) imply the following
\begin{cor}
If $X$ is an  irreducible projective  curve embedded in a smooth compact surface $S$ then
\begin{equation}
\label{E:genusformulaplanecurve}
 g_a(X)=\frac{(K_S+X)\cdot X}{2}+1\, .
\end{equation}
\end{cor}\bigskip

\subsubsection{Abelian differentials on planar curves}

The results just presented for arbitrary smooth compact surfaces $S$ will be now specialized to the
 projective plane $\mathbb P^2$. Let $C\subset  \mathbb P^2$
be a reduced algebraic curve of degree $k$, and $\mathcal I_C$ its ideal sheaf.

\smallskip

Since $K_{\mathbb P^2}  = \mathcal O_{\mathbb P^2}(-3)$ and $\mathcal O_{\mathbb P^2}(C)= \mathcal O_{\mathbb P^2}(k)$,
the adjunction formula reads as
$ \omega_C= \mathcal O_{\mathbb P^2}(k-3)|_C \,. $

\smallskip

Because $\mathcal I_C(k-3)$ is isomorphic to $\mathcal O_{\mathbb P^2}(-3)$, both cohomology groups $H^0(\mathbb P^2, \mathcal I_C(k-3))$ and
$H^1(\mathbb P^2, \mathcal I_C(k-3))$ are trivial. Therefore, the restriction map
$$H^0(\mathbb P^2, \mathcal O_{\mathbb P^2}(k-3)) \longrightarrow H^0(C, \mathcal O_{ C}(k-3))$$
is an isomorphism. Combined with the adjunction formula, this yields
\[
H^0\big(\mathbb P^2,
K_{\mathbb P^2}(C)
\big) \simeq H^0\big(\mathbb P^2,
\mathcal O_{\mathbb P^2}(k-3)
\big) \simeq H^0(C,\omega_C) \, .
\]
The isomorphism from the leftmost  to the rightmost group is induced by ${\mathrm Res}_C$.

\medskip

The discussion above  is summarized, and made more explicit, in the following

\begin{cor}
\label{C:omegaPlaneCurve}
 Let $\{f(x,y)=0\}$ be a reduced equation for  $C$ in generic affine
  coordinates $x,y$ on $\mathbb P^2$.  Then
$$
H^0(C,\omega_C) \simeq
\bigg\langle
\frac{p(x,y)}{\partial_y
 f }
\,dx \, \Big|\,
\begin{tabular}{l}
$ p\in \mathbb C[x,y]$\\
$\deg p\leq k-3 $
\end{tabular}
%p\in \mathbb C[x,y]\,, \; \deg p\leq k-3  \,
\bigg\rangle \, .
$$
Consequently  $ g_a(C) = \frac{
(k-1)(k-2)}{2}$.
 \end{cor}

\section{Algebraic webs of maximal rank}\label{S:CC}

In view of Theorem \ref{T:curvesvswebs2}, it suffices to
consider  $\mathcal W$-generic curves $C \subset \mathbb P^n$    with $h^0(C,\omega_C)$
attaining Castelnuovo's bound $\pi(n,\deg (C) )$ in order to have
examples of webs of maximal rank. Classically, a degree $k$  irreducible non-degenerate curve $C \subset \mathbb P^n$ such that $g_a(C)=\pi(n,k)>0$ is called
a \defi[Castelnuovo curve].  \index{Castelnuovo curve}
Notice that the definition implies that a Castelnuovo curve has necessarily degree $k\ge n+1$, since otherwise $g_a(C)=0$ according to
Theorem \ref{T:castelnuovobound}.

The simplest examples of Castelnuovo curves are the irreducible planar curves of degree at least three. Since
the arithmetic genus of a reduced planar curve $C$ is $\pi(2,\deg(C))$,
such curves are certainly Castelnuovo.

In sharp contrast, when the dimension is at least three,
the Castelnuovo curves are the exception, rather than the rule.
Indeed, the analysis carried out in Section \ref{S:constraints} to control the geometry of the conormals of maximal rank webs was
originally developed by Castelnuovo to control the geometry of hyperplane sections of Castelnuovo curves.
Reformulating Proposition \ref{P:normal} in terms of curves, instead of webs, provides the following

\begin{prop}\label{P:n-1}
If $C \subset \mathbb P^n$ is a Castelnuovo curve of degree $k\ge 2n +1$ then
a generic hyperplane section of $C$ is contained in a rational normal curve.
\end{prop}

Indeed, Castelnuovo's Theory goes further and says that a Castelnuovo curve
$C \subset \mathbb P^n$ is contained in a surface $S$, which is cut out by $| I_C(2)|$  the linear system of quadrics
containing $C$,   which has  rational normal curves
as generic hyperplane sections. Because the degree of $S$ is the same as the one of a generic hyperplane
section, and  rational normal curves in $\mathbb P^{n-1}$ have degree $n-1$, the surface
$S$ is of minimal degree. Theorem \ref{T:VMD} applies, and  implies that
$S$ is either a plane, a Veronese surface in $\mathbb P^5$, or a rational normal scroll $S_{a,b}$ with $a+b=n-1$.

The proof of these results will not be presented here\begin{footnote}{The interested reader may consult,
for instance, \cite{harris}, \cite{ACGH}, or \cite{EHCAMG}.}\end{footnote}, but the rather easier determination of Castelnuovo curves in surfaces
of the above list will be sketched below.

\subsection{Curves on the Veronese surface}
\index{Veronese surface}

Let $S \subset \mathbb P^5$ be the Veronese surface. Recall that $S$ is the embedding
of $\mathbb P^2$ into $\mathbb P^5$ induced by the complete linear system
$|\mathcal O_{\mathbb P^2}(2)|$.  If $C$ is an irreducible curve contained in $S$ then,
its degree $k$ as a curve in $\mathbb P^5$ is twice its degree $d$ as a curve in $\mathbb P^2$.

On the one hand,
\[
h^0(C,\omega_C) = \frac{(d-1)(d-2)}{2} =  \frac{(k-2)(k-4)}{8}\, .
\]
On the other hand, Castelnuovo's bound predicts
\[
h^0(C,\omega_C) \le \left\{
\begin{array} {lcl}
 (e-1) ( 2e -1 ) & \text{ if } & k = 4 e  \\
 e( 2e -1 ) & \text{ if } & k = 4 e + 2   \,.
\end{array} \right.
\]
Therefore, both cases lead to Castelnuovo curves.

\subsection{Curves on rational normal scrolls }

Let $S$ be a rational normal scroll $S_{a,b}$ with $0 \le a \le b$ and $a+b=n-1$.
If $S$ is singular, that is  $a=0$, replace $S$ by its desingularization
 $\mathbb P \big( \mathcal O_{\mathbb P^1}\oplus \mathcal O_{\mathbb P^1}(b) \big)$ which will still
be denoted by $S$.

The Picard group of $S$ is the free rank two $\mathbb Z$-module generated  by $H$, the class of a hyperplane section,   and $L$,  the class
of a line of the ruling. Of course,
\[
H^2 =  n-1, \quad H\cdot L = 1\, \quad \text{and} \quad L^2 =0 \, .
\]

If one writes $K_S = \alpha H + \beta L$ then the coefficients $\alpha$ and $\beta$ can be
determined using the genus formula  (\ref{E:genusformulaplanecurve}).  Since both $H$ and $L$ are smooth rational curves,
one has
\begin{align*}
-2  &= 2g(H) -2 = {H^2 + K_S \cdot H } = (\alpha + 1)(n-1) + \beta \, , \\
\mbox{and} \quad -2  &= 2g(L) -2 = {L^2 + K_S \cdot L} =  \alpha \, .
\end{align*}
Therefore $\alpha = -2$ and $\beta = n-3$, that is, $K_S = -2 H + (n-3) L$.

\medskip

Let $\alpha$ and $\beta$ be new constants distinct from the ones above. Let also $C$
 be an irreducible curve  contained in $S$, numerically equivalent to  $\alpha H + \beta L$.
  Notice that $\deg(C) = C\cdot H = \alpha(n-1) + \beta$. The genus formula  (\ref{E:genusformulaplanecurve})  implies
\[
g_a(C) = \frac{C^2 + K_S\cdot C}{2} + 1 = \frac{1}{2} (\alpha-1)\Big(
(n-1)\alpha+2(\beta-1)
\Big).
\]

Suppose now that the degree of $C$ is equal to $k \ge n+1$ and write
\[
 k-1 = m(n-1) + \epsilon \, ,
\]
as in Remark \ref{R:closed}. If $C$ is Castelnuvo then
\[
g_a(C) = \binom{m}{2} ( n- 1 ) + m \epsilon \, .
\]
Thus the Castelnuovo's curves on $S$  provide solutions to
the following system of equations
\begin{align*}
\deg(C) &= m(n-1) + \epsilon + 1 = \alpha(n-1) + \beta \, , \\
g_a(C) &= \binom{m}{2} ( n- 1 ) + m \epsilon = \frac{1}{2} (\alpha-1)\big(
(n-1)\alpha+2(\beta-1)
\big) \, ,
\end{align*}
subject to the arithmetical constraints $m,n \ge0$, and $0 \le \epsilon \le n-2$; and the geometrical
constraint $\alpha = C \cdot L > 0$.

It can be shown that the only possible  solutions are
\[
(\alpha, \beta) = ( m +1 , - (n-1 -\epsilon) ) \,
\]
or, only when $\epsilon =0$, $(\alpha, \beta) = ( m  , 1 ).$

\smallskip

If $\mathcal L =\mathcal O_S(\alpha H + \beta L )$, with $(\alpha,\beta)$ being one of the solutions
above, then  Riemann-Roch's
Theorem for surfaces, see for instance \cite[Theorem 1.6, Chapter V]{hartshorne},\index{Riemann-Roch Theorem}  implies
\begin{align*}
\chi(\mathcal L) &= \frac{1}{2} (\alpha H + \beta L)(  (\alpha -2) H+ (\beta +n -3)L) + \chi(S) \\
&=  \binom{\alpha+1}{2} (n-1) + (\alpha+1)(\beta +1)\, .
\end{align*}

Notice that $K_S \otimes  \mathcal L^* = \mathcal O_S( (-2-\alpha) H +(n-3-\beta)L )$. Because $\alpha > 0$, the following inequality holds true
\[
 ( K_S \otimes \mathcal L^*  ) \cdot L < 0 \, .
\]
Thus $h^0(S,K_S \otimes \mathcal L^* )=0$. Serre's duality for surfaces implies the same for $h^2(S, \mathcal L)$,
that is ${h^2(S,\mathcal L) = 0}$.  Consequently $h^0(S,\mathcal L)\ge \chi(\mathcal L)$.
Moreover, if $\alpha>2$ then $h^0(\mathcal L) \ge \chi(\mathcal L) >0$.
\smallskip

Since $\alpha$ is $\lfloor\frac{k-1 }{n-1} \rfloor$ or $\lfloor\frac{k-1 }{n-1} \rfloor +1$,
the linear system $|\mathcal L|$ is non-empty when  $k \ge 2n$.

One can push this analysis further and show that  the linear system $| \mathcal L|$ is base-point free  whenever
$k\ge 2n$. As a consequence,  the generic element of $|\mathcal L|$
is an irreducible smooth curve according to Bertini's Theorem.  See \cite{harris} for details.

\medskip

The discussion above is summarized in the following statement.

\begin{prop}
For any integer $n \ge 3$, any pair $(a,b)$ of non-negative integers
summing up to $n-1$, and any $k \ge 2n$, there exist
Castelnuovo curves of degree  $k$ contained in $S_{a,b} \subset \mathbb P^n$.
\end{prop}

\subsection{Webs of maximal rank}

From all that have been said in the previous sections, the following result  follows
promptly.

\begin{prop} For any integers $n \ge 2$ and  $k \ge 2n$, there exist smooth $k$-webs
of maximal rank on $(\mathbb C^n,0)$. These are the algebraic webs of the form
 $\mathcal W_C(H_0)$, where $C$ is a degree $k$ Castelnuovo curve in $\mathbb P^n$ and $H_0$ a hyperplane intersecting it transversely.
\end{prop}

It remains to discuss smooth $k$-webs of maximal rank on $(\mathbb C^n,0)$ with $k < 2n$.

\smallskip

For $k \le n$ there is not much to say inasmuch
smooth $k$-webs on $(\mathbb C^n,0)$ have always rank zero and are equivalent to $\mathcal W(x_1, \ldots, x_k)$.

\smallskip

For  $k=n+1$, a web of maximal rank carries exactly one non-zero abelian relation because  $\pi(n,n+1)=1$.
Thus, if  $$u_1,\ldots,u_{n+1}:(\mathbb C^n,0)\rightarrow \mathbb C$$ are submersions defining  $\mathcal W$ then  its unique abelian relation takes
the form
$
f_1(u_1)+\cdots + f_{n+1}(u_{n+1})=0
$
for suitable holomorphic germs ${f_i:(\mathbb C,0) \to (\mathbb C,0)}$. It follows that $\mathcal W$ is equivalent to
the parallel $(n+1)$-web $\mathcal W(x_1, \ldots, x_n, x_1 +  \ldots + x_n).$

\medskip

For $k \in \{ n+2,\ldots,2n-1\}$, it is fairly simple to construct
smooth $k$-webs of maximal rank $\pi(n,k)=k-n$. It suffices to consider submersions $u_1,\ldots,u_k: (\mathbb C^n,0) \to  (\mathbb C,0)$ of the form:
\begin{align*}
\label{RA:d<2n}
u_i(x_1, \ldots, x_n) = & \, x_i  &&\text{ for } i \in \underline n  \\
\text{ and} \quad u_i(x_1, \ldots, x_n) =  & \,\sum_{j=1}^n u_i^{(j)}(x_j)    &&\text{ for } i \in \{ n+1, \ldots, k \} \, . \,
\end{align*}
Here   $u_i^{(j)}: (\mathbb C,0) \to (\mathbb C,0)$ are germs of submersions in one variable. It is clear
that for a generic choice of the submersions $u_i^{(j)}$, the $k$-web $\mathcal W= \mathcal W(u_1, \ldots, u_k)$
is smooth. Moreover, the definition of $u_i$ when $i > n$  can be interpreted as an abelian relation of $\mathcal W$. Therefore
  $\mathrm{rank}(\mathcal W)\geq k-n$. But $\pi(n,k)=k-n$,  Therefore $\mathcal W$ is of maximal rank.

\dd

The examples of $k$-webs of maximal rank with $k= n+1$ or $k \ge 2n$  are of different
nature than the ones presented for ${k \in \{ n+2, \ldots, 2n-1\}}$.
While the equivalence classes of the former examples belong to finite dimensional families, the ones of the latter
belong to infinite dimensional families, as can be easily verified. Although logically this could
be just a coincidence, the main result of this book says that this is not case at least when  $n\ge 3$ and $k \ge 2n$.
A precise statement will be given in Chapter \ref{Chapter:Trepreau}.

\begin{remark} If $C\subset \mathbb P^n$ is  a non-degenerate, but not $\mathcal W$-generic, curve of degree $k$ then it may  happen
that $h^0(\omega_C)>\pi(n,k)$.
According to Proposition \ref{P:gp2}, such a curve $C$ cannot be irreducible.
 There are works  (see \cite{ballico,hartshorne',tedeschi}) showing the existence of a function $\tilde \pi(n,k)$
 which bounds  the arithmetical genus of non-degenerate projective curves in $\mathbb P^n$ of degree $k$. Of course, $\tilde \pi(n,k)$  is greater
  than Castelnuovo's number ${\pi}(n,k)$.  Moreover,
the curves $C$ attaining this bound have been classified. They all have a plane curve among their irreducible components.
\end{remark}

\section{Webs and families of hypersurfaces}\label{S:beyond}

The construction of the dual web $\mathcal W_C$ of a projective curve $C$, as well as the definition  of the trace relative to the family of
hyperplanes, make use of the incidence variety $\mathcal I \subset \mathbb P^n \times \hat{\mathbb P}^n$. The interpretation of  $\mathcal I$
as the family of hyperplanes in $\mathbb P^n$, suggests the extension of both constructions to other families of hypersurfaces.

In this section, one such extension is described, and used in combination with Chern's bound for the rank of
smooth webs  to obtain bounds for the geometric genus of curves on abelian varieties.
The exposition is deliberately sketchy. A more detailed account will appear elsewhere.

\subsection{Dual webs with respect to a family}

Let $X$ and $T$ be  projective manifolds,  and $\pi_T: X \times T \to T$, $\pi_X : X \times T \to X$
be the natural projections. Consider $\mathscr X$, a \defi[family of hypersurfaces]  in $X$
parametrized by  $T$. By definition, $\mathscr X$ is an irreducible
subvariety of $X\times T$ for which $\mathscr X_H = \pi_T^{-1}(H)\cap \mathscr X$ is a hypersurface of $X\times \{ H \}$ for
every $H \in T$ generic enough. It will be convenient, as has been done with the family of hyperplanes on $\mathbb P^n$, to
think of $H$ as point of $T$ ($H \in T$),   as well as a hypersurface  in $X$ ($H \subset X$).

Let  $C\subset X$ be a reduced curve and $H_0\subset X$ be a hypersurface which belongs to the family $\mathscr X$ -- that is, $H_0\subset X$ is
equal to $\pi_X(\mathscr X_{H_0})$ with $H_0 \in T$ being the corresponding point.

If $H_0\subset X$ intersects $C$ transversely in $k$ distinct points then, as for the family of hyperplanes in $\mathbb P^n$, there are holomorphic maps
\[
p_i: (T,H_0)\to C \, ,  \quad \text{for } i \in \underline k \, ,
\]
implicitly defined by
\[
C \cap \pi_X(\mathscr X_{H}) = \sum_{i=1}^k p_i(H) \, .
\]
It may happen  that one of the points $p_i(H_0)$ is common to all hypersurfaces in the family. It also
may happen that for some pair $i,j \in \underline k $, the functions $p_i$ and $p_j$ define the same foliation.
But, if   the maps $p_i$ are non-constant and define  pairwise distinct foliations then
there is a naturally defined germ of  (eventually singular) $k$-web on $(T,H_0)$:   $\mathcal W(p_1,\ldots, p_k)$.
It will called the \defi[$\mathscr X$-dual web of $C$ at $H_0\in T$], and denoted by $\mathcal W_C^{\mathscr X}(H_0)$.

If it is possible to define the germ of $k$-web $\mathcal W_C^{\mathscr X}(H_0)$ at any generic $H_0\in T$
then there is no obstruction to define the global $k$-web $\mathcal W_C^{\mathscr X}$, the \defi[$\mathscr X$-dual web of $C$].

\subsection{Trace relative to families of hypersurfaces}

For a curve $C \subset X$ and a meromorphic $1$-form $\omega$ on $C$, it is possible to define
the \defi[$\mathscr X$-trace of $\omega$ relative to the family $\mathscr X$] succinctly by the formula
\[
\mathrm{Tr}_{\mathscr X} (\omega) = ({\pi_T})_* \, (\pi_X)^*(\omega) \, .
\]
To give a sense to this expression, it is necessary to assume that a generic hypersurface in the family intersects $C$ in
at most a finite number of points. In other words, no irreducible component of $C$ is contained in the
  generic hypersurface of $\mathscr X$. If this is the case, consider then $\Sigma_0$, one of  the irreducible components of  $\pi_X^{-1}(C)\cap \mathscr X$
 with dominant projection to $T$. Because a generic $H \in T$ intersects $C$ in finitely many points, $\Sigma_0$ has the same dimension
 as $T$.

 Let $\Sigma \to \Sigma_0$ be a resolution of singularities, and still denote by $\pi_X$, $\pi_T$ the compositions
 of the natural projections with the resolution morphism. Then $\pi_X^*\omega$ is a meromorphic $1$-form on $\Sigma$, and for a generic
 $H \in T$  there are local inverses for $\pi_T:\Sigma \to T$. Proceeding exactly as for the family of hyperplanes in $\mathbb P^n$,
one can define a meromorphic $1$-form $\eta_{\Sigma_0}$ at the complement of the critical values of $\pi_T:\Sigma \to T$.

 To extend $\eta_{\Sigma_0}$ through the critical value set -- or discriminant -- of $\pi_T$, observe that outside a codimension
 two subset, the discriminant is smooth;  the fibers over points in it are finite; and locally, at each of its pre-images,
 $\pi_T$ is conjugated to a map of the form $(x_1, \ldots, x_n) \mapsto (x_1^r,x_2,\ldots,x_n)$ for some  suitable positive integer $r$. Hence, as for
 the family of hyperplanes, $\eta_{\Sigma_0}$ can be extended through this set. But a meromorphic $1$-form defined
at the complement of a codimension two subset extends to the whole ambient space according to Hartog's extension theorem.

The $\mathscr X$-trace $\mathrm{Tr}_{\mathscr X} (\omega)$ is then defined as  the sum of  the meromorphic $1$-forms $\eta_{\Sigma_0}$ for $\Sigma_0$
ranging over  all the irreducible components of  $\pi_X^{-1}(C)\cap \mathscr X$ dominating  $T$.

\medskip

It is not hard to prove the  following weak analogue of  Abel's addition Theorem.

\begin{prop}\label{P:wabel}
Let $\mathscr X$ be family of hypersurfaces of $X$ over a smooth projective variety $T$.
If $C \subset X$ is a curve intersecting a generic hypersurface of the family at finitely many points, and if $\omega$ is a meromorphic $1$-form on $C$,  then
the $\mathscr X$-trace of $\omega$ is a meromorphic $1$-form on $T$.
Moreover, if the pull-back of $\omega$ to a desingularization of $C$ is holomorphic
then    its $\mathscr X$-trace is a holomorphic $1$-form on $T$.
\end{prop}
\smallskip

If $\mathscr X$ is the family of all hypersurfaces of degree $d$ of $\mathbb P^n$ then a strong analogue of Abel's Theorem is  valid:
the $\mathscr X$-trace of $\omega$ is holomorphic if and only $\omega$ is abelian.  For arbitrary families, the equivalence between
holomorphic $\mathscr X$-traces and abelian differentials is too much to be hoped for.

\subsection{Bounds for  rank and genus}

From Proposition \ref{P:wabel}, one can obtain lower bounds for the
rank of  $\mathscr X$-dual webs for curves in $X$.

\begin{prop}\label{P:Xdual}
Let $C \subset X$ be a curve with $\mathscr X$-degree $k$. If $\mathcal W_C^{\mathscr X}$ is
a $k$-web then
\[
\mathrm{rank}(\mathcal W_C^{\mathscr X}) \ge g(C) - h^0(T,\Omega^1_T) \, .
\]
\end{prop}
\begin{proof}
If $\overline C$ is a desingularization of $C$ then,
according to Proposition \ref{P:wabel},  the $\mathscr X$-trace of any holomorphic $1$-form
in $H^0(\overline C, \Omega^1_C)$ is a holomorphic $1$-form on $T$. Moreover, the map
$
\omega \longmapsto \mathrm{Tr}_{\mathscr X} ( \omega )
$
is a linear map from $H^0(\overline C, \Omega^1_{\overline C})$ to $H^0(T, \Omega^1_T)$ whose kernel can be identified with a linear subspace of $\mathcal A(\mathcal W_C^{\mathscr X}(H_0))$ for a generic $H_0 \in T$. Since, by definition  $g(C)=h^0(\Omega^1_{\overline{C}})$, the proposition follows.
\end{proof}

In analogy with  the standard case of families of hyperplanes,  Proposition \ref{P:Xdual} read
backwards provides a bound for the genus of curves $C\subset X$ as soon as the $k$-web $\mathcal W_C^{\mathscr X}$
is a generically smooth.

\begin{prop}\label{P:boundfamilies}
Let $C \subset X$ be a curve with $\mathscr X$-degree $k$. If $\mathcal W_C^{\mathscr X}$ is
a generically smooth $k$-web    then
\[
g(C) \le \pi(\dim T, k) + h^0(T, \Omega^1_T) \, .
\]
\end{prop}

\medskip

\subsection{Families of theta translates}

Let $A$ be an abelian variety of dimension $n$, and  $H_0 \subset A$ be an irreducible divisor in $A$.
Recall that $A$ acts on itself by translations, and  assume that  $H_0$ has finite isotropy group under this action.

Consider a second copy of $A$ and denote it by $\check A$.
Let  $\mathscr X$ be the family of translates of $H_0 \subset A$ by points of $\check{A}$, that is
\[
\mathscr X = \{ (x,y) \in A \times \check{A} \, | \, x-y \in H_0 \} \, .
\]
The natural projections from $\mathscr X$ to $A$ and $\check A$ will be denoted by $\pi$ and $\check \pi$ respectively.

\begin{lemma}\label{L:naodeg}
Let $A$,  $H_0$ and $\mathscr X$ be as above.
If  $C\subset A$ is an irreducible  curve  not contained in a translate of $H_0$  then
$\mathcal W_C^{\mathscr X}$ is generically smooth.
\end{lemma}
\begin{proof}
Let $H_0$ be a generic hypersurface in the family, intersecting $C$ transversely in $k$ points. At a neighborhood of
$H_0$, write ${C = C_1 \cup \cdots \cup C_k}$. For simplicity, assume that $k= H_0 \cdot C \ge n$.

To prove that $\mathcal W_C^{\mathscr X}$ is generically smooth, first notice that
\begin{align*}
P = (p_{1}, p_{2}, \ldots, p_{k}): (A,H_0) &\longrightarrow C^k  \\
H & \longmapsto (H \cap C_{1}, \ldots, H \cap C_{k} )
\end{align*}
has finite fibers. If not then the fiber over $P(H_0)$ has positive dimension. Consequently one of the
irreducible components of its
Zariski closure is an analytic subset $Z \subset \check{A}$ of positive dimension. Clearly for
every $H \in Z \subset \check{A}$ the intersection   $H \cap C \subset A$ contains $H_0 \cap C$.
Since $H_0$ has finite isotropy group $\pi(\check{\pi}^{-1}( Z))$ is equal to $A$. Therefore given an arbitrary point $q \in C$, there
exists $H_q \in Z$ containing $q$. If $C$ is not contained in $H_q$ then $H_q \cdot C > k$. But $H_q$ and $H_0$ are algebraically
equivalent, thus $k > H_q \cdot C= H_0 \cdot C = k$. This  contradiction shows that $P$ has finite fibers.

Using the  irreducibility of $C$, one can  show
that  the map $P_I = ( p_{i_1}, \ldots, p_{i_n} ) : (A, H_0) \rightarrow C^n$ also
has finite fibers for any subset ${I= \{ i_1, \ldots, i_n\}}$ of $\underline k$ having  cardinality $n$.
The reader is invited to fill in the details.

Since $P_I$ has finite fibers and $\mathcal W(p_{i_1}, \ldots, p_{i_n})$ is
an arbitrary \mbox{$n$-subweb} of $\mathcal W_C^{\mathscr X}$ at $H_0$, it follows that
the web $\mathcal W_C^{\mathscr X}$ is generically smooth.
\end{proof}

\medskip

This lemma together with Proposition \ref{P:boundfamilies} promptly implies the following

\begin{thm}\label{T:abelabel}
Let  $C \subset A$ be a  curve with $\mathscr X$-degree $k$. If $C$ is irreducible and is not contained in
any translate of $H_0$ then
\begin{equation}\label{E:bound1}
g(C) \le \pi(n, k) + n \, .
\end{equation}
\end{thm}

\medskip

The most natural example of a pair $(A,H_0)$ satisfying the above conditions
is  an irreducible principally polarized abelian variety $(A,\Theta)$.
There is a version of Castelnuovo's Theory in this context,
developed by Pareschi and Popa in  \cite{Popa}. In particular, they
obtain the following bound for the genus of  curves in $A$.

\begin{thm}\label{T:popa}
Let $(A,\Theta)$ be an irreducible principally polarized abelian variety of dimension $n$.
Let $C \subset A$ be a non-degenerate\begin{footnote}{Here, non-degenerate means that the curve is not contained in any abelian subvariety.
Although this differs  from our assumption on $C$, one of the first steps in the proof of this
result is to establish that the maps $P_I$ used in the proof of Lemma \ref{L:naodeg} contain
open subsets of $A^n$ in their images. Consequently, if $C$ is non-degenerate then the $\mathscr X$-dual web
is generically smooth.
}
\end{footnote}
irreducible of degree $k = C \cdot \Theta$ in $A$.
Let $m = \lfloor \frac{k-1}{n}\rfloor$, so that $k - 1 = mn + \epsilon$, with $0 \le \epsilon  < n$. Then
\begin{equation}\label{E:bound2}
g(C) \le \binom{m+1}{2}n + (m+1) \epsilon + 1.
\end{equation}
Moreover, the inequality is strict for $n \ge 3$ and $k \ge n + 2$.
\end{thm}

For a fixed  $n$,  the bounds (\ref{E:bound1}) and  (\ref{E:bound2})
are asymptotically equal to $\frac{k^2}{2(n-1)}$ and $\frac{k^2}{2n}$ respectively.
The bound provided by Theorem \ref{T:abelabel} is asymptotically worse than the one provided by Theorem \ref{T:popa}.
But when  $n\ge 3$ and for comparably small values of $k$, the former  is sharper  than the latter.
Indeed, it can be verified that the bound of  Theorem \ref{T:abelabel} is sharper than the one of Theorem \ref{T:popa},
if and only if  $k$ is between  $n+2$ and $2n^2- 3n$.

%%%%%%%%%%%%%%%%%%%%%%%%%%%%%%%%%%%%%%%
%  Introduction
%  label = Chapter:intro
%  First version by JVP
%  last modification: 3/nov/2008
%  Remarks:
%%%%%%%%%%%%%%%%%%%%%%%%%%%%%%%%%%%%%%%

\chapter{The ubiquitous algebraization tool}\label{Chapter:4}
\thispagestyle{empty}

The main result of this chapter is a converse of Abel's addition Theorem
 stated in Section \ref{S:AIthmGEN}. It  assures the
algebraicity of local datum satisfying Abel's addition Theorem. Its  first version was
established by Sophus Lie in the  context of  double translation surfaces.
 Lie's arguments consisted in
a tour-de-force analysis of an overdetermined  system of PDEs.
Later Poincar\'{e} introduced geometrical methods to handle the problem solved analytically by Lie. Poincar\'{e}'s approach was later revisited, and
 made more precise by Darboux, to whom the  approach presented in Section \ref{S:AIdim2}
can be traced back. By the way, those willing to take for granted the validity of the converse of Abel's  Theorem can
safely skip Section \ref{S:AIdim2}.

Blaschke's school reinterpreted  the converse of Abel's  Theorem in the language of web geometry to obtain
the algebraicity of germs of linear webs  admitting  complete abelian relations. This result turned out
to be an  ubiquitous tool  for the algebraization problem of germs of webs. This dual version
also provides a complete description of the space of abelian relations  of  algebraic webs.  This will be treated in Section \ref{S:AIdual}.

\medskip

Web geometry also owes   Poincar\'{e}
a method to algebraize  smooth, non necessarily linear,
$2n$-webs on $(\mathbb C^n,0)$ of maximal rank.
The strategy is based on the study  of
certain natural map from $(\mathbb C^2,0)$ to $\mathbb P (\mathcal A(\mathcal W))$,
here called  Poincar\'{e}'s map of $\mathcal W$. Closely related are the  canonical maps
of the web. Their definition  mimics
the one of canonical maps for projective curves.  All this will be made precise in Section \ref{S:AIdual*}

\medskip

Poincar\'{e}'s original motivation had not much to do with web geometry, but instead focused on
the relations between  double-translation hypersurfaces and
Theta divisors on Jacobian varieties of projective  curves. In Section \ref{S:doubletrans}
these relations will be reviewed. Modern proofs of some of
the theorems by Lie, Poincar\'{e} and Wirtinger on the subject are also
laid out.

\section{The converse of Abel's Theorem}\label{S:AIthmGEN}

\subsection{Statement}\label{S:AIStatement}

Let $H_0 \subset \mathbb P^n$, $n \ge 2$, be a hyperplane, and  $k\ge 3$ be an  integer.
For $i \in \underline k$, let  $C_i$  be a	
germ of  complex  curve in $\mathbb P^n$ that intersects $H_0$ transversely
at the point $p_i(H_0)$.
Assume  that the points $p_i(H_0)$ are pairwise distinct.
Let also $p_i: (\check{\mathbb P}^n,H_0)  \rightarrow C_i$ be the germs of holomorphic maps
characterized by  $H\cap C_i=\{p_i(H)\}$ for every  $i \in \underline k $ and every $H\in (\check{\mathbb P}^n,H_0)$.
For a picture, see Figure \ref{F:AbelInverseTheorem0} below.

\begin{figure}[ht]
\begin{center}
\psfrag{C1}[][]{$ C_1 $}
\psfrag{C2}[][]{$ C_2 $}
\psfrag{C3}[][]{$ C_3 $}
\psfrag{C4}[][]{$ C_4 $}
\psfrag{P}[][]{$ \mathbb P^{n} $}
\psfrag{H0}[][]{$  H_0 $}
\psfrag{p1}[][]{$p_1(H_0) $}
\psfrag{p2}[][]{$p_2(H_0) $}
\psfrag{p3}[][]{$p_3(H_0) $}
\psfrag{p4}[][]{$p_4(H_0) $}
\resizebox{2.3in}{2.3in}{\includegraphics{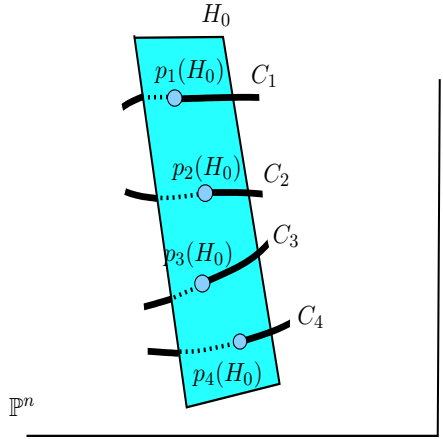} }
\end{center}
\caption[Germs of analytic curves]{Germs of analytic curves \ldots \label{F:AbelInverseTheorem0}}
\end{figure}

\smallskip

Using  the notation settled above, the converse of Abel's  Theorem can be phrased as follows.

\begin{thm}\label{T:AI}
For $i \in \underline k$, let $\omega_i$ be a germ of non-zero  holomorphic 1-form on $C_i$. If
\begin{equation}
 \label{E:RAlin}
p_1^*(\omega_1)+\cdots + p_d^*(\omega_d) \equiv 0
% \qquad \mbox{ on } \quad \check{U}\,.
\end{equation}
as a germ of $1$-form at $(\check{\mathbb P}^n, H_0)$ then there exist an
 algebraic curve $C\subset  \mathbb P^n$  of degree $k$,
 and an abelian differential $\omega\in H^0(C,\omega_C)$
 such that $C_i\subset C$ and $\omega|_{C_i} = \omega_i$ for all~$i \in \underline k$.
\end{thm}

\begin{figure}[ht]
\begin{center}
\psfrag{C}[][]{$ C $}
\psfrag{P}[][]{$ \mathbb P^{n} $}
\psfrag{H0}[][]{$  H_0 $}
\psfrag{p1}[][]{$p_1(H_0) $}
\psfrag{p2}[][]{$p_2(H_0) $}
\psfrag{p3}[][]{$p_3(H_0) $}
\psfrag{p4}[][]{$p_4(H_0) $}
\resizebox{2.3in}{2.3in}{\includegraphics{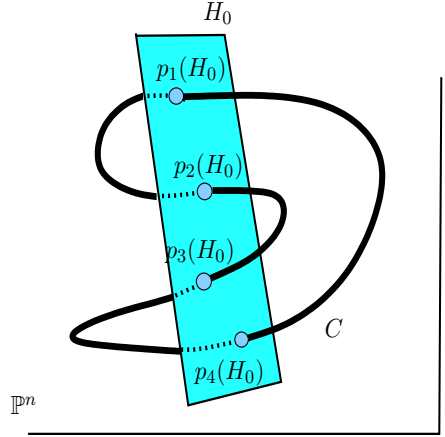} }
\end{center}
{\small \caption[\ldots globalize in the presence of an abelian relation]{\ldots globalize in the presence of a complete abelian relation.}}
\end{figure}%\vspace{0.2cm}

Starting from middle 1970's,  a number of generalizations of this remarkable result appeared in print.
It has been generalized from curves to higher dimensional varieties  in \cite{griffiths};
it has been shown in \cite{henkinpassare} that it suffices to have rational trace (\ref{E:RAlin}) to assure
the algebraicity of the data; and it has been considered in \cite{fabrethesis,Weimann} more general traces
in the place of the one  with respect to hyperplanes.

\medskip

The proof of Theorem \ref{T:AI} is postponed to Section \ref{S:AIdim2}. Assuming its validity,
a dual formulation in terms of webs is given and proved below.

\subsection{Dual formulation}\label{S:AIdual}

Let $H_0$ be a hyperplane in $\mathbb P^n$, $n\ge 2$.
In Section \ref{S:defalg} of Chapter \ref{Chapter:intro} it was shown the existence of
an equivalence between linear quasi-smooth $k$-webs on $(\check{\mathbb P}^n,H_0)$,
and $k$ germs of curves in $\mathbb P^n$ intersecting $H_0$ transversely in $k$ distinct points.
This equivalence implies the following variant of the converse of Abel's Theorem.

\begin{thm}
\label{T:ALGquasismoothW}
Let $\mathcal W$ be a linear quasi-smooth web on $(\check{\mathbb P}^n,H_0)$.
If it admits a complete abelian relation then it is algebraic. Accordingly, there exists a projective curve $C\subset \mathbb P^n$ such
that $\mathcal W$ coincides with the restriction of $\mathcal W_C(H_0)$.
Furthermore, the space of abelian relations of $\mathcal W$ is  naturally isomorphic to  $H^0(C,\omega_C)$.
\end{thm}
\begin{proof} For  reader's convenience, the proof starts by detailing the equivalence between germs of linear webs and  germs of curves.
Let $\mathcal W= \mathcal F_1 \boxtimes \cdots \boxtimes \mathcal F_k$
be a quasi-smooth linear $k$-web on $(\check{\mathbb P}^n,H_0)$. For  each of the foliations
$\mathcal F_i$, consider its Gauss map
\begin{align*}
\mathscr G_i : (\check{\mathbb P}^n, H_0) &\longrightarrow \;  \mathbb P^n \, \\
H\;  &\longmapsto  T_H \mathcal F_i \, .
\end{align*}

Since the foliation $\mathcal F_i$ has linear leaves,  the map $\mathscr G_i$ is constant along them. Notice
also that the restriction of $\mathscr G_i$ to a line transversal to $H_0$ is injective. These two facts together imply that $\mathscr G_i$ is a submersion defining $\mathcal F_i$ and
its image is a germ of smooth curve $C_{\mathcal F_i} \subset \mathbb P^n$ intersecting $H_0$ transversely.

Notice that $\mathscr G_i$ associates to a hyperplane $H \in (\check{\mathbb P}^n,H_0)$ the intersection of
$H \subset \mathbb P^n$ with $C_{\mathcal F_i}$. In other words $\mathscr G_i = p_i$ in the notation used in the converse of Abel's Theorem.

If $\eta_i$ is a germ of  closed $1$-form defining $\mathcal F_i$  then there exists a germ
of holomorphic $1$-form  $\omega_i$ in $C_{\mathcal F_i}$ such that $\eta_i = {\mathscr G_i}^* \omega_i$.
Therefore, if $\eta= (\eta_1, \ldots, \eta_k) \in \mathcal A(\mathcal W) $   then
there exist $1$-forms $\omega_i \in \Omega^1(C_{\mathcal F_i}, \mathscr G_i(H_0))$ satisfying
\[
\sum_{i=1}^k {\mathscr G_i}^*(\omega_i) = 0 \, .
\]

When $\eta$ is complete,  none of the $1$-forms $\omega_i$ vanishes identically. Thus the converse of Abel's Theorem
ensures the existence of a projective curve $C \subset \mathbb P^n$ containing all the curves $C_{\mathcal F_i}$, and of
an abelian differential $\omega \in H^0(C, \omega _C)$ which pull-backs through $\mathscr G_i$ to the $i$-th component
of the abelian relation $\eta$. It is then clear that $\mathcal W$ is the restriction  of $\mathcal W_C$ at $H_0$.
\end{proof}

\begin{cor}
\label{C:ALGsmoothW}
Let $\mathcal W$ be a linear smooth $k$-web on $(\mathbb C^n,0)$. If $\mathcal W$ has maximal rank then it  is algebraic.
\end{cor}
\begin{proof}
It suffices to show the existence of a complete  abelian  relation. It there is none then
there exists a $k'$-subweb $\mathcal W'$ of $\mathcal W$ with $k'< k$, and
 $ \mathrm{rank}(\mathcal W')= \mathrm{rank}(\mathcal W) = \pi(n,k)$. But $\mathrm{rank}(\mathcal W') \le \pi(n,k') < \pi(n,k)$.
 This contradiction proves the corollary.
\end{proof}

\begin{cor}
\label{C:ALGgeneralcase}
A smooth web of maximal rank  is algebraizable if and only if it is linearizable.
\end{cor}

The latter corollary  indicates
the general strategy for  the problem of algebraization of webs:
{\it in order to prove that a web of maximal rank is algebraizable, it suffices to show that it is linearizable}.
In fact a similar  strategy also applies to webs of higher codimension.
Most of the known algebraization results in web geometry are proved in this way. The simplest instance
of this approach  will be the subject of Section \ref{S:AIdual*}. A considerably more involved instance
will occupy  the whole Chapter \ref{Chapter:Trepreau}.

\section{Proof}\label{S:AIdim2}

 The  notation introduced in Section \ref{S:AIStatement} is valid throughout this section.

\subsection{Reduction to dimension two}

This section is devoted to prove that the converse of Abel's Theorem in dimension $n$ follows from the case of dimension  $n-1$ when $n>2$.

\smallskip

Assume $n>2$ and consider a generic point $p \in H_0 \subset \mathbb P^n$.
The linear projection  $\pi: \mathbb P^n \dashrightarrow \mathbb P^{n-1}$ with center at $p$ when
restricted to the germs of curves $C_i$, induces germs of biholomorphisms onto their images, which are denoted  $D_i$. Moreover,  the dual inclusion
$\check{\pi}: \check{\mathbb P}^{n-1} \to \check{\mathbb P}^n$ fits into the following commutative diagram :
$$
\xymatrix{
 C_i  \ar@{^{(}->}[r]  & \mathbb P^n \ar@{-->}^{\pi}[r] & \mathbb P^{n-1} & \ar@{_{(}->}[l] D_i  \\
& (\check{\mathbb P}^n,H_0)  \ar^{p_i}[ul]   & \ar_{\check{\pi}\quad }@{_{(}->}[l] \,\big(\check{\mathbb P}^{n-1},\pi(H_0)\big). \ar_{q_i = \pi \circ p_i \circ \check{\pi} }[ur]
}
$$

Since the restriction of $\pi$ to $C_i$ is a biholomorphism  onto $D_i$, the $1$-forms $\omega_i$ can
be thought as a $1$-form on $D_i$. Under this identification, it is clear that
\[
\sum_{i=1}^k q_i^* \omega_i = \check{\pi}^* \left( \sum_{i=1}^k p_i^* \omega_i \right) = 0 .
\]

The converse of Abel's Theorem in dimension  $n-1$ implies the existence of  an algebraic curve $D \subset \mathbb P^{n-1}$ containing all the curves $D_i$, and
of an abelian,  thus rational, $1$-form $\omega$ on $D$ such that $\omega|_{D_i} = \omega_i$.

Notice that $S = \overline{\pi^{-1}(D)}$ is the cone over $D$ with vertex at $p$. Notice also that it contains  the curves
$C_i$ and  has dimension two.

Let $p' \in \mathbb P^n$ be another generic point of $H_0$. The same argument as above implies the existence of another
surface $S'$ containing the curves $C_i$. It follows that the curves $C_i$ are contained in the intersection $S \cap S'$.
This suffices to ensure the existence of a projective curve $C$ in $\mathbb P^n$ containing all the curves $C_i$. Furthermore,
the pull-back by $\pi$ of $\omega$ from $D$ to $C$  is a rational $1$-form satisfying $\omega_|{C_i} = \omega$ for every $i \in \underline k$.

Thus, to establish Theorem \ref{T:AI} it suffices to consider the two-dimensional case. This will be done starting
from the next section.

\subsection{Preliminaries }

To keep in mind that the ambient space has dimension two, the hyperplanes in the statement of
Theorem \ref{T:AI} will be denoted by  $\ell_0$ and $\ell$,  instead of $H_0$ and $H$.

\begin{figure}[ht]
\begin{center}
\psfrag{C1}[][]{$ C_1 $}
\psfrag{C2}[][]{$ C_2 $}
\psfrag{Cd1}[][]{$ C_{k-1} $}
\psfrag{Cd}[][]{$ C_k $}
\psfrag{P}[][]{$ \mathbb P^{2} $}
\psfrag{L}[][]{$  \ell_0 $}
\psfrag{Lp}[][]{$  \ell $}
\psfrag{p1}[][]{$\scriptstyle{p_1(\ell_0)} $}
\psfrag{p2}[][]{$\scriptstyle{p_2(\ell_0)} $}
\psfrag{pd1}[][]{$\scriptstyle{p_{k-1}(\ell_0)} $}
\psfrag{pd}[][]{$\scriptstyle{p_{k}(\ell_0)} $}
\psfrag{p1p}[][]{$\scriptstyle{p_1(\ell)} $}
\psfrag{p2p}[][]{$\scriptstyle{p_2(\ell)} $}
\psfrag{pd1p}[][]{$\scriptstyle{p_{k-1}(\ell)} $}
\psfrag{pdp}[][]{$\scriptstyle{p_{k}(\ell)} $}
\resizebox{2.6in}{2.6in}{\includegraphics{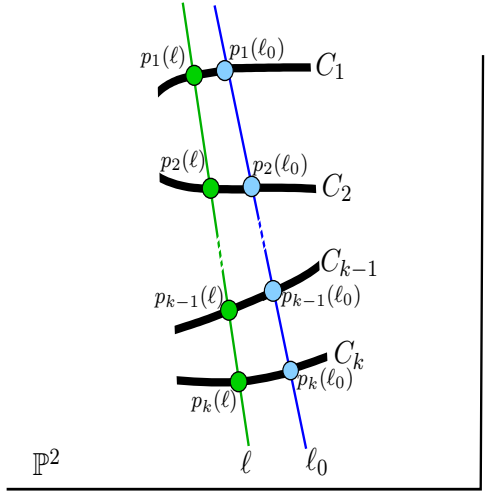} }
\end{center}
\caption{Geometry of the converse of Abel's  Theorem in $\mathbb P^2$. \label{F:AbelInverseTheorem}}
\end{figure}

Assume that the main hypothesis of the converse of Abel's  Theorem is satisfied: for $i \in \underline k$,
there are non-identically zero germs of   holomorphic  differentials $\omega_i\in \Omega^1_{C_i}$  such
that (\ref{E:RAlin}) holds.  \smallskip

Let $(x,y)$ be an affine system of coordinates on  an affine chart $\mathbb C^2\subset \mathbb P^2$ where
$\ell_0 \cap \mathbb C^2 = \{ x=0\}$, and none of the points $p_i(H_0)$ belong
to $\ell_\infty = \mathbb P^2 \setminus \mathbb C^2$.

Since a generic line in the projective plane admits
a unique  affine equation of the form $x=ay+b$,
the variables $a$ and $b$ can be considered as affine coordinates on
$(\check{\mathbb P}^2,\ell_0)$. If
$\ell_{a,b}$  denotes the line in ${\mathbb P}^2$ of affine equation
$x=ay+b$ then  $p_i(\ell_{a,b})$  can be written as $\big( x_i(a,b),y_i(a,b)\big)$,
where  $x_i,y_i: (\check{\mathbb P}^2, \ell_0) \rightarrow \mathbb C$ are
two germs of holomorphic functions satisfying $x_i(a,b)=a\,y_i(a,b)+b$ identically on $(\check{\mathbb P}^2, \ell_0)$.
It will be convenient to assume that for every  $i \in \underline k$, the  function $y_i$ is non constant. Of course, this holds true
for a generic choice of affine coordinates $(x,y)$  on  $\mathbb C^2 \subset \mathbb P^2$.

Let also $\eta_i \in \Omega^1(\check{\mathbb P}^2,\ell_0)$ be the pull-back of $\omega_i$ by $p_i$ ( $\eta_i= p_i^* \omega_i$ ), and
$u_i: (\check{\mathbb P}^2, \ell_0) \to (\mathbb C,0)$ be a primitive of $\eta_i$ with value at $\ell_0$ equal to zero, that is
\[
u_i ( a,b) = u_i(\ell_{a,b}) = \int_0^{(a,b)} \eta_i \, .
\]

\subsubsection{Differential identities}

One of the key ingredients of the proof of the converse of Abel's Theorem  here presented is the following observation. It was
 first made in this context by Darboux.

\begin{lemma}\label{L:shockW}
For every $i \in \underline k$,  the following differential equations are identically satisfied on $(\check{\mathbb P}^2,\ell_0)$:
\begin{equation}
\label{E:shockW}
 \frac{\partial y_i }{\partial a }= y_i
\frac{\partial y_i }{\partial b}\, , \qquad
\frac{\partial x_{i} }{\partial a }= y_i
\frac{\partial x_{i }}{\partial b}\,  \qquad \text{and} \qquad \frac{\partial u_{i} }{\partial a }= y_i
\frac{\partial u_{i }}{\partial b}.
\end{equation}
\end{lemma}
\begin{proof}
First notice that, for a fixed $i \in \underline k$, the functions $x_i, y_i$, and $u_i$ define the very same
foliation on $(\check{\mathbb P}^2, \ell_0)$. Consequently, the $1$-forms $dx_i$, $dy_i$, and $du_i$ are
all proportional. Thus, it suffices to prove the identity
\begin{equation}\label{E:fffacil}
 \frac{\partial y_i }{\partial a }= y_i
\frac{\partial y_i }{\partial b}
\end{equation}
to obtain the other two.

Since $C_i$ intersects $\ell_0=\{x=0\}$ transversely at $(0,y_0)$,
there exists a germ $g \in \mathcal O_{\mathbb C,0}$ for which
 $C_i = \{ y=g(x)\}$. Therefore
$y_i(a,b)=g(x_i(a,b))=g(a\, y_i(a,b)+b)$
for all $(a,b)\in (\check{\mathbb P}^2,\ell_0)$. Differentiation of this identity implies
\begin{align*}
 \frac{\partial y_i}{\partial a }
\Big(1-a\,g'(x_i)\Big)=& \,y_i \, g'(x_i) \\
\mbox{ and } \quad  \frac{\partial y_i}{\partial b }
\Big(1-a\,g'(x_i)\Big)=& \, g'(x_i).  \\
\end{align*}
The function $(a,b) \mapsto 1-a\,g'(x_i(a,b))$ does not vanish identically, otherwise the holomorphic function $g(x_i(a,b))$
would be equal to $\log a$ at a neighborhood of $(a,b)=(0,0)$.
Therefore $y_i$ verifies the  differential equation  (\ref{E:fffacil}). The lemma follows.
\end{proof}

Notice that the relations
\begin{equation}
 \label{E:sumti}
\sum_{i=1}^k \frac{\partial u_i}{\partial a }=
\sum_{i=1}^k y_i \frac{\partial u_i}{\partial b }
\equiv 0 \quad \mbox{ and } \qquad
\sum_{i=1}^k \frac{\partial u_i}{\partial b }
\equiv 0 \,
\end{equation}
follow immediately from  the hypothesis $\sum_i p_i^* \omega_i = \sum_i u_i=0$ combined with Lemma \ref{L:shockW}.
% identically on $\check{U}$.

\subsection{Lifting to the incidence variety}

Let $\mathcal I \subset \mathbb P^2 \times \check{\mathbb P}^2$ be the incidence variety, that is
\[
\mathcal I = \{ (p, \ell) \in  \mathbb P^2 \times \check{\mathbb P}^2 \, | \, p \in \ell \}  \, .
\]
As in Section \ref{S:awr} of Chapter \ref{Chapter:intro}, let  $\pi : \mathcal I \to \mathbb P^2$ and $\check \pi : \mathcal I \to \check{\mathbb P}^2$ be
the natural projections.

There is an open affine subset $V \subset \mathcal I$ isomorphic to the closed subvariety of $\mathbb C^2 \times \mathbb C^2$
defined by the equation $x = ay + b$, where $(x,y)$ and $(a,b)$ are, respectively, the  affine coordinates on  $\mathbb C^2 \subset \mathbb P^2$
and $(\check {\mathbb P}^2, \ell_0)  \subset \check{\mathbb P}^2$ used above.

\smallskip

Notice that $V$ is isomorphic to $\mathbb C^3$ and  $(y,a,b)$ is an affine coordinate system on it. Using these coordinates, define the
germ of meromorphic $2$-form $\check{\Psi}_0$ on $(\mathbb C^3, \{ ay + b =0\})$\begin{footnote}{Here and throughout in the proof
of the converse of Abel's Theorem,  the notation $(X,Y)$ means the germ of the variety $X$ at $Y$. One should think of  open subsets of $X$, arbitrarily small among
the ones  containing $Y$.
 }\end{footnote}
 \begin{equation}
\label{E:defPHI}
 \check{\Psi}_0(a,b,y)=\sum_{i=1}^d \frac{\eta_i(a,b)\wedge dy}{   y-y_i(a,b) } \, .
\end{equation}
Recall  that the $1$-forms $\eta_i\in \Omega^1(\check{\mathbb P}^n,\ell_0)$
were introduced in Section \ref{S:AIdual} as the ones  corresponding to the $1$-forms  $\omega_i$  via projective duality. Notice
that $\check{\Psi}_0$  is the restriction at $V$ of a germ of  meromorphic $2$-form $\Psi$ on $\check{\pi}^{-1}( \check{\mathbb P}^2, \ell_0)=\big(\mathcal I, \check{\pi}^{-1}(\ell_0)\big).$

\begin{remark}\rm
To understand the idea behind the definition of $\Psi$,  imagine that the local datum $\{ ( C_i, \omega_i) \}$ is indeed the germification
at $\ell_0$ of a global curve $C$ and an abelian differential $\omega \in H^0(C, \omega_C)$. In this case, there exists a meromorphic $2$-form $\Omega$
on $\mathbb P^2$ satisfying $Res_C \Omega = \omega$. Writing down the meromorphic $2$-form   $\pi^* \Omega$ in the coordinates
$(a,b,y)$, one ends up with an expression exactly like (\ref{E:defPHI}).
\end{remark}

\medskip

Recall that the $1$-form $\eta_i$ is equal to $du_i$, and that  $u_i$ verifies (\ref{E:shockW}).
Therefore
${ \eta_i=({\partial u_i}/{\partial b}) \big( y_ida+db\big) }$ for every $i \in \underline k$.
It is then easy to determine the expression ${\Psi}_0(x,y,a)$ of $\Psi$ in the coordinates $x,y,a$.
Using (\ref{E:sumti}) one obtains
\begin{equation*}
 {\Psi}_0(x,y,a)=\sum_{i=1}^k \frac{
\frac{\partial u_i}{\partial b}(a,x-ay)
}{   y-y_i(a,x-ay) }\, dx \wedge dy\, .
\end{equation*}

If $F$ is the germ of meromorphic function on $\check{\pi}^{-1} ( \check{\mathbb P}^2, \ell_0)$ which  in the
coordinates $(x,y,a)$ can be written as
\[
{F}_0(x,y,a)  = \sum_{i=1}^k
\frac{
\frac{\partial u_i}{\partial b}(a,x-ay)
}{   y-y_i(a,x-ay) }
\]
then
\begin{equation*}
 {\Psi}_0(x,y,a)=F_0(x,y,a)\,dx\wedge dy\, .
\end{equation*}

\subsection{Back to the projective plane}

The next step of the proof consists in showing that the $2$-form $\Psi$ defined above comes
from a $2$-form on the projective plane.

\begin{lemma} \label{L:desce}
There exists a germ of meromorphic function $f$ on $(\mathbb P^2, \ell_0)$  such that $F=\pi^*(f)=f\circ \pi$.
Consequently,  $\Psi = \pi^* \Omega$,  where $\Omega$ is the meromorphic $2$-form  $f(x,y) dx \wedge dy$ on $(\mathbb P^2, \ell_0)$.
\end{lemma}
\begin{proof}
It suffices to prove  that $\frac{\partial F_0}{\partial a}$ is identically zero. For that sake let
$\check{F}_0$ be the expression for $F$ in the coordinate system $(a,b,y)$, that is,
\begin{equation}\label{E:defF}
\check{F}_0(a,b,y) = F_0(ay+b,y,a) = \sum_{i=1}^k \frac{\frac{\partial u_i}{\partial b}(a,b)}{   y-y_i(a,b) } \, .
\end{equation}

Notice that
\begin{align*}
  y\,\check{F}_0= \sum_{i=1}^k \frac{y \frac{\partial u_i}{\partial b}}{   y-y_i}
=\sum_{i=1}^k
\frac{\partial u_i}{\partial b}+
\sum_{i=1}^k
\frac{y_i \frac{\partial u_i}{\partial b}}{   y-y_i} \, .
\end{align*}
Combining Equation (\ref{E:shockW}) with the hypothesis $\sum_{i=1}^k du_i=0$ yields
\begin{equation}
\label{E:yF}
y\,\check{F}_0=\sum_{i=1}^d \frac{\frac{\partial u_i}{\partial b}}{   y-y_i}\, .
\end{equation}

Differentiation of   (\ref{E:defF}) and (\ref{E:yF}) with respect to $a$ and $b$ respectively give, in their turn, the identities
\begin{align*}
 \frac{\partial \check{F}_0}{\partial a}= & \,  \sum_{i=1}^k
\frac{\frac{\partial^2 u_i}{\partial a\partial b}}{   y-y_i}
+\sum_{i=1}^k
\frac{y_i'\frac{\partial u_i}{\partial a} \frac{\partial u_i}{\partial b}  }{  ( y-y_i)^2}  \\
\mbox{ and } \quad
y\frac{\partial \check{F}_0}{\partial b}= & \,  \sum_{i=1}^k
\frac{\frac{\partial^2 u_i}{\partial b\partial a}}{   y-y_i}
+\sum_{i=1}^k
\frac{y_i'\frac{\partial u_i}{\partial a} \frac{\partial u_i}{\partial b}  }{  ( y-y_i)^2} \,
\end{align*}
where  $y_i'= dy_i/ du_i$.

Therefore $\check{F}_0$ satisfies the equation
\[
 \frac{\partial \check{F}_0}{\partial a}-y\frac{\partial \check{F}_0}{\partial b}=0 \, .
\]
To conclude notice that
\begin{align*}
 \frac{\partial {F}_0}{\partial a}(x,y,a)= & \,
\Big(\frac{\partial \check{F}_0}{\partial a}-y
\frac{\partial \check{F}_0}{\partial b}\Big)(a,x-ay,y) =  0\,.
\end{align*}
\end{proof}

\subsubsection{Recovering the curves}

Now, notice that the polar set of the  $2$-form $\Omega$ is nothing more than the union of the germs
  $C_i$ with $i \in \underline k$.

\begin{lemma}
If $\Omega$ is the $2$-form provided by Lemma \ref{L:desce} then
\[
( \Omega)_{\infty} = \bigcup_{i \in \underline k} C_i\, .
\]
\end{lemma}
\begin{proof}
Consider $\Psi= \pi^* \Omega$ in the coordinates $a,b,z=1/y$.
 Performing the change of variable $y=1/z$ in (\ref{E:defPHI}) gives
\begin{align*}
{\Psi}(a,b,z)=& \, -\sum_{i=1}^k \frac{
\frac{\partial u_i}{\partial b}
\big( y_ida+db\big)
}{  z\big( 1-z\,y_i\big) }
\wedge { dz}
\\
= & \, -   \sum_{r\geq 0} \sum_{i=1}^k \bigg[
\Big(\frac{\partial u_i}{\partial b}y_i^{r+1} z^{r-1}\Big) da+
\Big(\frac{\partial u_i}{\partial b}y_i^{r} z^{r-1}\Big) db
\bigg]\wedge dz \, .
\end{align*}

Lemma \ref{L:shockW} implies the vanishing of the coefficients of $z^{-1}da$ and $z^{-1}db$.
Thus ${\Psi}(a,b,z)$ is holomorphic at an open  neighborhood  of $\{z=0\}$.  Therefore
 $\Omega$ is holomorphic at a neighborhood of $ \ell_0 \cap \ell_\infty$.

For $(a,b) \in (\check {\mathbb P}^2, \ell_0)$, the restriction of $f$ at the line $\ell_{a,b} \subset (\mathbb P^2,\ell_0)$
 is $$f_{a,b}(y)=\check{F}_0(a,b,y)=\sum_{i=1}^k \frac{\frac{\partial u_i}{\partial b}(a,b)}{ y-y_i(a,b)} \, .$$

Recall that $du_i=\eta_i$. According to the hypotheses,
the partial derivative $\partial u_i/\partial b$  does not vanish identically on $C_i$ for  $i \in \underline k$.
Hence, for a generic $(a,b) \in (\check{\mathbb P}^2, \ell_0)$,
\begin{align*}
 \mathbb C^2 \cap (f_{a,b})_\infty =
 p_1(a,b)+\cdots + p_k(a,b) \, .
\end{align*}
This suffices to prove the lemma.
\end{proof}

\subsubsection{Recovering the $1$-forms}

It is also possible to extract from $\Omega$ the $1$-forms $\omega_1, \ldots, \omega_k$  with the help of Poincar\'{e}'s residue.

\begin{lemma}
For every $i \in \underline k$, one has
$ \mathrm{Res}_{C_i} \Omega = \omega_i .$
\end{lemma}
\begin{proof}
Fix $i \in \underline k $ and set  $p =  p_i(\ell_0)= (0,y_i(0,0)) \in C_i$.
 Let $q$ be the point in $\pi^{-1}( \ell_0)$  which in the  coordinate system $(a,b,y)$
is represented by $(0,0,y_i(0,0))$. If $$(D,q) = \{ (a,b,y) \in \pi^{-1}(({\mathbb P}^2, \ell_0))\, | \, a=0\}$$
then the restriction of $\pi$ to  $(D,q)$ is a germ of biholomorphism  $${\rho:(D,q) \to (\mathbb P^2, p)}.$$

In the coordinates $(b,y)$ on $D$,
\begin{itemize}
 \item[$a)$] the pull-back of $C_i$ to $D$ is $E_i=  \rho^{-1}(C_i) = \{ y-y_i(0,b)=0 \} \, ;$
\item[$b)$]  the $1$-form $\rho^*(\omega_i)$  coincides with   $\check{\pi}^*(\eta_i)|_{E_i}$; and
\item[$c)$]  the pull-back of $\Omega$ to $(D,q)$ by $\rho$ can be written as
$$
\rho^*(\Omega)=\sum_{i=1}^k \frac{\eta_i(0,b)\wedge dy}{   y-y_i(0,b) } \, .
$$
\end{itemize}

Item $a)$ implies that  $\frac{\eta_j\wedge dy}{   y-y_j }$ is holomorphic in a neighborhood of $E_i$ when $j\neq i$.
Item  $b)$ and $c)$, in their turn, imply
\begin{equation*}
 \mathrm{Res}_{E_i}\big( \rho^*\Omega \big)= \, \mathrm{Res}_{E_i}\Big(  \frac{\eta_i(0,b)}{  y- y_i(0,b) }\wedge dy\Big)
=  \eta_i (0,b) \Big|_{E_i}=
\rho^*\omega_i\, .
\end{equation*}
Since $\rho$ is an  isomorphism, it follows that
$$ \mathrm{Res}_{C_i}\big( \Omega \big)= \omega_i $$
for every $i \in \underline k$. The lemma is proved.
\end{proof}

\subsection{Globalizing to conclude}

At this point, to conclude the proof of Theorem \ref{T:AI}, it suffices to prove that the $2$-form $\Omega$ is the
restriction at $(\mathbb  P^2, \ell_0)$ of a rational $2$-form. Indeed, if this is the case then the  polar set
of $\Omega$ will be  a
projective curve $C$ containing the curves $C_i$,  and  its  residue along $C$ will be an abelian differential
$\omega \in H^0(C, \omega_C)$ according to Proposition \ref{P:bpv}.

The globalization of $\Omega$ follows from a particular case of a classical result of Remmert stated  below
as a lemma.

\begin{lemma}
Let $\ell \subset \mathbb P^2$ be a line. Any
germ of meromorphic function
  $g : (\mathbb P^2, \ell) \dashrightarrow \mathbb P^1$
extends to a rational function on $\mathbb P^2$.
\end{lemma}
\begin{proof}
Let $(x,y)$ be affine coordinates on $\mathbb C^2 \subset \mathbb P^2$.
Suppose that $\ell$ is the line at infinity. Fix an arbitrary representative of the germ $g$ defined in a neighborhood
$U$ of the line at infinity. Still denoted it by $g$. Notice that the restriction of $U$ to $\mathbb C^2$ contains the complement
of  a polydisc $\Delta$. Consider the Laurent expansion of $g$,
\[
g(x,y) = \sum_{i,j \in \mathbb Z^2} a_{i \,  j} x^i y^j \, .
\]
Since it converges at a neighborhood of infinity to a meromorphic functions  it suffices to show that
\[
\Gamma = \{ (i,j) \in \mathbb Z^2 \, | \, a_{i\, j} \neq 0 \}
\]
is contained in $(i_0,j_0) + \mathbb N^2$ for some $i_0,j_0 \in \mathbb Z$ in order to prove the lemma.

To prove this, rewrite  the Laurent series of  $g(x,y)$ as
\[
g(x,y) = \sum_{i \in \mathbb Z} b_i(y) x^i \, ,
\]
and consider the function $I: \mathbb C \to \mathbb Z \cup \{ - \infty \}$ defined by
\[
I(t) = \inf \{ i \in \mathbb Z \, | \, b_i(t) \neq 0 \} \, .
\]

If $|t| \gg 0$ then the function $g(x,t)$ is a global meromorphic function of $x$. Therefore
$I(t) \in \mathbb Z$. Since $\mathbb Z \cup \{ -\infty \}$ is a countable set, there exists
an integer $i_0$, and an uncountable set $\Sigma \subset \mathbb C$ for which the restriction
of $b_i$ to $\Sigma$ is zero whenever $i \le i_0$. Therefore, for $i \le i_0$, the functions
$b_i$ are indeed zero all over $\mathbb C$. In other words, ${\Gamma \subset ( i_0 + \mathbb N) \times \mathbb Z}$.

To conclude it suffices to change  the roles of $x$ and $y$ in the above argument, and remind that a global meromorphic
function on $\mathbb P^2$ is rational.
\end{proof}

\section{Algebraization of smooth $2n$-webs}\label{S:AIdual*}

As   the title of this chapter indictes, Theorem \ref{T:ALGquasismoothW} is an ubiquitous tool
when the algebraization of webs comes to mind.
It does not hurt to repeat that most of the known algebraization results
use the abelian relations in order to linearize the web and then
apply the converse of Abel's Theorem in its dual formulation. As promised,
the simplest instance where this strategy applies --
the case of   smooth $2n$-webs $\mathcal W$ on $(\mathbb C^{2n},0)$
of maximal rank --  is treated below.

\subsection{The Poincar\'{e} map}
\index{Poincar\'{e}'s map|(}
Let $\mathcal W= \mathcal F_1 \boxtimes \cdots \boxtimes \mathcal F_{2n} = \mathcal W(\omega_1, \ldots, \omega_{2n})$
be a smooth $2n$-web on $(\mathbb C^n,0)$.  Assume that   $\mathcal W$ has  maximal rank.
 Since $\pi(n,2n) = n+1$,   the space $\mathcal A(\mathcal W)$ is a complex vector space of dimension $n+1$.
In particular, according to  Corollary \ref{C:bbb},
\[
 \dim F^1 \!\mathcal A(\mathcal W) = 1  \, .
\]

If  $F^\bullet_x \mathcal A(\mathcal W)$  denotes the corresponding filtration of $\mathcal A(\mathcal W)$ centered\begin{footnote}{Here and throughout, the convention about germs made in Section \ref{S:gsw} is in use. If one wants to be
more precise, then $\mathcal W$ has to be thought as a web defined on an open subset $U$ of $\mathbb C^n$ containing
the origin  and $F^1_x \mathcal A ( \mathcal W)$ is the filtration of the germ of $\mathcal W$ at $x$. }\end{footnote} at $x$
then one still has  $\dim F^1_x \mathcal A(\mathcal W) = 1$.  \defi[Poincar\'{e}'s map] of $\mathcal W$ is  defined as \index{Poincar\'{e}'s map}
\begin{align*}
 P_{\mathcal W} : \,  (\mathbb C^n,0) & \longrightarrow
\mathbb P\mathcal A(W) \\
x &\longmapsto \left[ F^1_x \mathcal A(\mathcal W) \right] \,.  \nonumber
\end{align*}
It is a covariant of  $\mathcal W$: if $\varphi \in \mathrm{Diff}(\mathbb C^n,0)$ then
$$
P_{\varphi^* \mathcal W}= \varphi^*\big( P_{\mathcal W} \big)=P_{\mathcal W} \circ \varphi\, .
$$

\subsection{Linearization}
For every $i \in \underline{2n}$ and each $x \in (\mathbb C^n,0)$  consider the linear map
\begin{align*}
ev_i(x) : \mathcal A(\mathcal W) & \longrightarrow \Omega^1(\mathbb C^n,x)  \, \\
 (\eta_1, \ldots, \eta_{2n})& \longmapsto \eta_i(x) \, .
\end{align*}
The kernel of $ev_i(x)$ corresponds to the abelian relations
of $\mathcal W$ with $i$-component vanishing at $x$.

\begin{lemma}\label{L:67}
For every $i \in \underline{2n}$, the linear map $ev_i(x)$ has  rank  equal to one.
In particular, $x \mapsto \ker ev_i(x)$ is
a sub-bundle of corank one of the trivial bundle over $(\mathbb C^n,0)$ with fiber $\mathcal A(\mathcal W)$.
Moreover, for every subset $I \subset \underline {2n}$ of cardinality $n$ the following identity holds true
\[
\bigcap_{i\in I} \ker ev_i(x) = F^1_x \mathcal A(\mathcal W).
\]
\end{lemma}
\begin{proof}
Since $\eta_i$ defines the foliation $\mathcal F_i$, the rank of $ev_i(x)$ is at most one. By semi-continuity, if it is not
constant and equal to one then it must be equal to $0$ at $(\mathbb C^n,0)$. In other  words, the $i$-th component of every abelian relation
$\eta \in \mathcal A(\mathcal W)$ vanishes at the origin. Therefore
\[
\dim \frac{F^0 \mathcal A(\mathcal W) }{F^1 \mathcal A(\mathcal W)} \le (2n -1 ) - \ell^1(\mathcal W) = n-1
\]
according to the proof of Lemma \ref{L:besta}. But then, according to Corollary \ref{C:bbb},
$\mathrm{rank}(\mathcal W) \le (n-1) + 1 = n < \pi (n,2n)$  contradicting its  maximality.

To prove the  second part, notice that the smoothness of $\mathcal W$  ensures the linear independence of   $T_x \mathcal F_i$ with $i \in \underline{2n} \setminus I$.
\end{proof}

\begin{prop}\label{P:linear2n}
If $L$ is a leaf of $\mathcal W$ then $P_{\mathcal W}(L)$, its image  under  Poincar\'{e}'s map, is
contained in a hyperplane.
\end{prop}
\begin{proof}
At every point $x \in (\mathbb C^n,0)$, one has
\[
\bigcap_{i \in \underline{2n}} \ker ev_i(x) = F^1_x \mathcal A(\mathcal W) \, .
\]
In particular, $F^1_x \mathcal A(\mathcal W) \subset \ker ev_i(x)$ for every $i \in \underline {2n}$.
Notice that for every $x \in (\mathbb C^n,0)$,  $\ker ev_i(x) \subset \mathcal A(\mathcal W)$ is a hyperplane according to Lemma \ref{L:67}

Fix $i \in \underline{2n}$.
Suppose  $L$ is a leaf of $\mathcal F_i$ and let $u_i:(\mathbb C^n,0) \to \mathbb C$ be a submersion defining $\mathcal F_i$. Notice
that $L$ is a level hypersurface of $u_i$.  The $i$-component $\eta_i$ of an abelian relation $\eta \in\mathcal A(\mathcal W)$  is of the form $g(u_i) du_i$. Therefore if $x, y \in L $ are two distinct points of $L$, then $\eta_i$ vanishes at $x$ if and only it
vanishes at $y$. Thus $\ker ev_i(x) = \ker ev_i(y)$ for every $x,y \in L$. Hence, the  image of $L$ by $P_{\mathcal W}$ is contained
in the hyperplane $\mathbb P \ker ev_i(x) \subset \mathbb P \mathcal A(W)$  determined by any  $ x \in L$.
\end{proof}

\begin{prop}\label{P:bih2n}
Poincar\'{e}'s map $P_{\mathcal W} : (\mathbb C^n,0) \to \mathbb P \mathcal A (\mathcal W)$ is a germ of
biholomorphism.
\end{prop}
\begin{proof}
For $i \in \underline n$, let $L_i$ be the leaf of  $\mathcal F_i$ through zero. If $u_i:(\mathbb C^n,0) \to (\mathbb C,0)$ is a submersion
defining $\mathcal F_i$, then $L_i = u_i^{-1}(0)$.
Let also $C_i$ be the curve in $(\mathbb C^n,0)$ defined as
\[
C_i = \bigcap_{j \in \underline n \setminus \{ i \} } L_j \, .
\]
Because $\mathcal W$ is smooth, so is $C_i$.

If  $\gamma_i : (\mathbb C,0) \to C_i$ is a smooth parametrization of $C_i$ then
the image of $P_{\mathcal W} \circ \gamma_i$ is contained in the line $\ell_i$ of $\mathbb P \mathcal A(\mathcal W)$
described by the intersection
\[
\bigcap_{j \in \underline n \setminus \{ i \} } \mathbb P ( \ker ev_i(0) )\, ,
\]
according to Lemma \ref{L:67}.

Since the tangent spaces of the lines $\ell_1, \ldots, \ell_n$ at $P_{\mathcal W}(0)$ generate $T_{P_{\mathcal W}(0)} \mathbb P \mathcal A(\mathcal W)$,
it suffices to show that $\Gamma_i=P_{\mathcal W}\circ \gamma_i :(\mathbb C,0) \to \ell_i \subset \mathbb P^n $ has non-zero derivative at $0 \in \mathbb C$,
for every $i \in \underline n$. But, $\Gamma_i$ is a map between germs of  smooth curves, hence everything boils down to the injectivity of   $\Gamma_i$
for a fixed $i \in \underline n$.

If $\Gamma_i$ is not injective then there are pairs of distinct points $x,y \in C_i$ arbitrarily close to
zero such that $\ker ev_i(x)  = \ker ev_i(y)$. Hence, if the $i$-th component of an abelian relation of $\mathcal W$
vanishes  at $x$, then it also vanishes at $y$. It follows that the $i$-th component of the abelian relation
generating  $F^1_0 \mathcal A(W)$ vanishes
at the origin with multiplicity two. But this contradicts the equality  $\ell^2(\mathcal W) = 2n -1$ established in Proposition \ref{P:cast}.
Thus $\Gamma_i$ is injective for any $i \in \underline n$. Consequently, the differential of $P_{\mathcal W}$ at the origin
is invertible. The proposition follows.
\end{proof}

\subsection{Algebraization}

It is a simple matter to put the previous results together in order to
prove the following  algebraization result.

\begin{thm}
\label{T:alg2nWEB}
 A smooth $2n$-web on $(\mathbb C^n,0)$ of maximal rank is algebraizable. More precisely, its push-forward by its Poincar\'{e}'s map  is a web $\mathcal W_C$ where $C\subset  \mathbb P^n$  is a  $\mathcal W$-generic projective   of degree $2n$ and~genus~$n+1$.
\end{thm}
\begin{proof}
According to Proposition \ref{P:bih2n},  $P_{\mathcal W}$ is a germ of biholomorphism. Hence
 $({P_{\mathcal W}})_*(\mathcal W)$ is a smooth $2n$-web equivalent to $\mathcal W$.
In particular,  its rank is also maximal.
Proposition \ref{P:linear2n} implies that  $({P_{\mathcal W}})_*(\mathcal W)$  is a linear web.
To conclude apply Corollary \ref{C:ALGgeneralcase}.
\end{proof}

\subsection{Poincar\'{e}  map for planar webs}\label{S:PMPW}
\index{Poincar\'{e}'s map!for planar webs|(}

It is   possible to  generalize Poincar\'{e}'s map for smooth $k$-webs on $(\mathbb C^n,0)$
for all integers  $n$ and $k$ such that  $2n\leq k$. The idea is to consider the last  non-trivial piece  $F^{l} \mathcal A(\mathcal W)\neq 0$
of the filtration $F^{\bullet} \mathcal A( \mathcal W)$. To be more precise set, as in
Remark \ref{R:closed2},
$$\rho = \left\lfloor\frac{k-n-1}{n-1}\right\rfloor \quad \mbox{and} \quad \epsilon = (k-n-1)-\rho(n-1)\, .$$
In this notation,
the last non-trivial piece of $F^{\bullet} \mathcal A( \mathcal W)$
is  $F^{\rho+1} \mathcal A(\mathcal W)$.  Then  set $e=\dim F^{\rho+1} \mathcal A(\mathcal W)=\epsilon+1>0$ and define
Poincar\'{e}'s map of $\mathcal W$  as
\begin{align*}
\label{E:defPWgg}
 P_{\mathcal W} : \,  (\mathbb C^n,0) & \longrightarrow
{\rm Grass}\big(\mathcal A(W), e \big)  \\
x \; &\longmapsto  F^{\rho+1}_x \mathcal A(\mathcal W)  \,.
\end{align*}

When $e  =1$, that is  $n=2$ or $k \equiv 2 \mod (n-1)$, then one still gets a map
to a projective space with remarkable properties as shown in the next result for $n=2$. Nevertheless, it does not linearize the web
as when $k=2n$.

\begin{prop}\label{P:PW2}
If $\mathcal W$ is a smooth $k$-web of maximal rank on $(\mathbb C^2,0)$ then $\mathcal P_{\mathcal W}$ is an
immersion. Moreover if $S$ is the image of $\mathcal P_{\mathcal W}$ and $L$ is a leaf of $\mathcal W$ then the following assertions hold:
\begin{enumerate}
\item[(a)]  the image of $L$ by  $\mathcal P_{\mathcal W}(L)$ is contained in a
projective space of codimension $k-3$;
\item[(b)]  the union of the projective tangent spaces of $S$ along the points of the  image of $L$, that is
\[
\bigcup_{x \in \mathcal P_{\mathcal W}(L)}  T_x S \, ,
\]
is contained in a  projective space of codimension $k-4$;
\item[(c)] if $s \le k-4$ then the union of the projective osculating spaces of $S$ of order $s$ along the points of the  image of $L$, that is
\[
\bigcup_{x \in \mathcal P_{\mathcal W}(L)}   T^{(s)}_x S \, ,
\]
is contained in a  projective space of codimension $k-(3+s)$.
\end{enumerate}
\end{prop}
\begin{proof}
The proof is the natural generalization of the arguments used to prove Proposition \ref{P:linear2n} and Proposition \ref{P:bih2n}.
The reader is invited to fill in the details of the argument sketched below.

Instead of considering the evaluation morphism $ev_i^{(0)} (x)= ev_i(x):\mathcal A(\mathcal W) \to \mathbb C$ one
has to consider the higher order evaluation morphism  $ev_i^{(j)}(x):\mathcal A(\mathcal W) \to \mathbb C^{ j+1}$
which sends the $i$-th component of an abelian relation $\eta$ to its Taylor expansion truncated at order $j$. More precisely,
if $u_i$ is a submersion defining $\mathcal F_i$ and if the $i$-th component of $\eta$ is $\eta_i = f(u_i) du_i$
then
\[
ev_i^{(j)}(x) ( \eta)  = f(u_i(x)) + t f'(u_i(x)) + \cdots + t^j  f^{(j)} (u_i(x)) \,
\]
where $1,t, \ldots, t^j$ is thought as a basis of $\mathbb C^{j+1}$.

\smallskip

It is a simple matter to adapt the arguments used to prove Proposition \ref{P:bih2n}  in order to show that $P_{\mathcal W}$ is an immersion.

\smallskip

To prove item (a), notice that $ev_i^{(k-4)}(x)$ has maximal rank. Consequently $\dim \ker ev_i^{(k-4)}(x) = \pi(2,k) - (k-3)$.
Furthermore, if  $L$ is the leaf of $\mathcal F_i$  through $x$ then  its image $P_{\mathcal W}(L)$ is contained
in $\mathbb P \ker ev_i^{(k-4)}(x)$. Putted together, these two remarks imply item (a).

\smallskip

To prove items (b) and (c), the key point is
to notice the following. If $f:(\mathbb C^2,0) \to \mathcal A(\mathcal W)$  is such that
$f(x) \in \ker ev_i^{(k-4)}(x)$ for every $x \in (\mathbb C^2,0)$ then
the derivatives of $f$ at $x$ with order at most $s$ lie in $\ker ev_i^{(k-4-s)}(x)$.
\end{proof}

\medskip

The discussion of Poincar\'{e}'s map for planar webs of maximal rank just made leads naturally to the following
characterization of algebraizable planar webs.

\begin{thm}
If $k$ is a integer greater  than $4$ and $\mathcal W$ is a smooth planar $k$-web  of maximal rank then  the image
of Poincar\'{e}'s map $P_{\mathcal W}$ is contained in the $(k-3)$-th Veronese surface if and only if  $\mathcal W$ is algebraizable.
\end{thm}
\begin{proof}
Let $ S$ be the image of $P_{\mathcal W}$.
Suppose first that  $S$  is contained in the $(k-3)$-th Veronese surface and let $\nu= \nu_{k-3}: \mathbb P^2 \to S \subset \mathbb P \mathcal A(\mathcal W)$ be
the corresponding Veronese embedding. According to Proposition  \ref{P:PW2}, for every $i \in \underline k$
and every $x \in (\mathbb C^2,0)$, the image of $L$ ( the leaf of $\mathcal F_i$ through $x$ ) is contained in a hyperplane $H$ that
osculates $S$ up to order $k-3$ along $P_{\mathcal W}(L)$. But this means that $\nu^*H$ is a curve in $\mathbb P^2$ with an irreducible component
$C$ for which  every point is a singularity with algebraic multiplicity $k-3$. Since $\nu^* H$ has degree $(k-3)$, it follows
that $\nu^* H = C = (k-3) \ell $  for some line $\ell \subset \mathbb P^2$. Thus the composition $ \nu^{ -1} \circ P_{\mathcal W} $ linearizes
$\mathcal W$. Corollary \ref{C:ALGsmoothW} implies that $(\nu^{ -1} \circ P_{\mathcal W})_*\mathcal W$ is an algebraic web.

Conversely, if $\mathcal W= \mathcal W_C(H_0)$ is algebraic then it is a simple matter to show that
$P_{\mathcal W} : (\check{\mathbb P}^2,H_0) \to \mathbb P \mathcal A(\mathcal W)$ is the
germ at $H_0$ of the $(k-3)$-th Veronese embedding. For details see for instance \cite{PTese}.
\end{proof}

For $k=5$ the statement appears in \cite[page 255]{BB},
and the  proof there presented involves the study of  Poincar\'{e}-Blashcke's map of the web, which will be introduced in Chapter \ref{Chapter:Trepreau}.
The result  for arbitrary $k$ as stated above was proved in \cite{Henaut1994} using  arguments very similar to the ones
in the book just cited. The proof just  presented    uses slightly different arguments.

\index{Poincar\'{e}'s map!for planar webs|)}
\index{Poincar\'{e}'s map|)}

\subsection{Canonical maps}\label{S:canmap}

\index{Canonical maps!of a web|(}
For a $k$-web on $(\mathbb C^n,0)$ of rank $r  >0$, it is possible to mimic the construction of
canonical maps for projective curves as follows.  Notice that the evaluation morphisms $ev_i(x)$ are
linear functionals on $\mathcal A(\mathcal W)$. As such they can be thought as
points of $\mathcal A(\mathcal W)^*$. For every $i \in \underline k$ consider the germ
of meromorphic map
\begin{align*}
 \kappa_{{\mathcal W},i}: \,   (\mathbb C^n,0) & \dashrightarrow
 \mathbb P \mathcal A(W)^*\simeq \mathbb P^{r-1} \\
x & \longmapsto [ev_{i}(x)]\,.  \nonumber
\end{align*}
By definition, $\kappa_{\mathcal W,i}$ is the \defi[$i$-th canonical map of $\mathcal W$].

\index{Canonical map!of a curve}
Lemma \ref{L:67}, or more precisely its proof, shows that when $\mathcal W$ is smooth and of maximal
rank, $\kappa_{\mathcal W, i}$ is regular at zero. Moreover, if $\mathcal W= \mathcal W_C(H_0)$ is an
algebraic web on $(\check{\mathbb P}^n, H_0)$ dual to a projective curve $C$ then, after  identifying $H^0(C,\omega_C)$
with $\mathcal A(\mathcal W)$, one can put together any one of  the canonical maps of $\mathcal W$ with the canonical
map\begin{footnote}{Recall that for projective curve, the canonical map is defined as
\begin{align*}
 \kappa_C: \,   C & \dashrightarrow
 \mathbb P H^0(C, \omega_C)^*\simeq \mathbb P^{g_a(C)-1} \\
x & \longmapsto \mathbb P ( \ker \{ \omega \mapsto \omega(x) \} ) \, .
\end{align*}}\end{footnote} of $C$ in the following commutative diagram: 
\[
\xymatrix{
  &C_i  \ar@{^{(}->}[dr]   \\
 (\check{\mathbb P}^n,H_0) \ar^{p_i}[ur]  \ar@{-->}^{\kappa_{\mathcal W,i}} [dr]   & & C \ar@{-->}_{\kappa_C}[dl]  \\
  &\mathbb P \mathcal A(\mathcal W)^*.
}
\]

For smooth $2n$-webs on $(\mathbb C^n,0)$ of maximal rank, the Poincar\'{e}'s map and the canonical maps of
$\mathcal W$ are related through the formula
\[
P_{\mathcal W} (x) = \bigcap_{i \in \underline{2n}} \kappa_{\mathcal W,i} (x), 
\]
where the points $\kappa_{\mathcal W,i} (x)$ are interpreted as
 hyperplanes in $\mathbb P \mathcal A(\mathcal W)$.

Notice that the canonical maps, for no matter which $k$ and no matter which rank, always take
values in $\mathbb P \mathcal A(\mathcal W)^*$.
\index{Canonical maps!of a web|)}

\section{Double-translation hypersurfaces}\label{S:doubletrans}

By definition, a germ $S$ of  smooth  hypersurface at $(\mathbb C^{n+1},0)$ is a \defi[translation hypersurface] \index{Translation hypersurface}
if  it is non-degenerate and  admits a parametrization of the form
\begin{align*}
 \Phi\, : \; (\mathbb C^n,0)\;  & \longrightarrow \quad \;  (\mathbb C^{n+1},0) \\
 \big(x_1, \ldots, x_n \big)  & \longmapsto
\phi_{1}(x_1)+\cdots+ \phi_{n}(x_n)
\nonumber
\end{align*}
where $\phi_1,\ldots,\phi_{n+1}: (\mathbb C,0)\rightarrow (\mathbb C^n,0)$ are germs  of holomorphic maps.
Notice that $\Phi$ induces   naturally a  $n$-web  ${\mathcal W_{\Phi}= \Phi_* \mathcal W(x_1,\ldots, x_n)}$ on $S$.

To understand the logic behind the terminology, notice that
a surface $S$ in $\mathbb C^3$ is a  translation surface
if it can be generated by translating a curve  along another one (see the picture below).
\begin{figure}[ht]
\begin{center}
\resizebox{2.9in}{1.5in}{\includegraphics{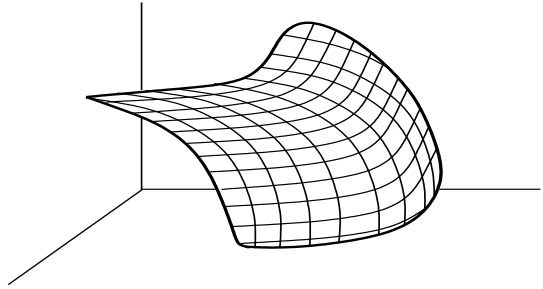} }
\end{center}
\caption{A translation surface in $\mathbb C^3$.}
\end{figure}

If
\begin{align*}
 \Psi\, : \; (\mathbb C^n,0)\;  & \longrightarrow \quad \;  (\mathbb C^{n+1},0) \\
\big(y_1, \ldots, y_n \big)  & \longmapsto
\psi_{1}(y_1)+\cdots+ \psi_{n}(y_n)
\nonumber
\end{align*}
 is another parametrization of $S$ as a translation hypersurface, then
 \index{Translation hypersurface!distinct parametrization}
 $\Psi$ is  \defi[distinct] from $\Phi$ if the superposition of the corresponding $n$-webs $\mathcal W_{\Psi}$ and
 $\mathcal W_{\Phi}$ is a quasi-smooth $2n$-web on $S$.

A  hypersurface $S$ is a \defi[double-translation hypersurface] \index{Double-translation hypersurface}
if it admits two distict parametrizations in the above sense. The quasi-smooth $2n$-web $\mathcal W_S=\mathcal W_{\Phi} \boxtimes \mathcal W_{\Psi}$
will be denoted by $\mathcal W_S$. Notice that there is a certain abuse of notation here since, a priori, a double-translation hypersurface
may admit more than two distinct parametrizations as a translation hypersurface. Implicit in the  notation $\mathcal W_S=\mathcal W_{\Phi} \boxtimes \mathcal W_{\Psi}$, is the fact that the
two distinct parametrizations of translation type are fixed in the definition of a double-translation hypersurface.      \smallskip

\begin{example}  \label{Ex:SC-algebraic}\rm
The   surface $S$ in $\mathbb C^3$ cut out by $4x+z^5-5zy^2=0$ is an example of an algebraic translation surface.
In fact,
\begin{equation*}
\label{E:doubletransSURFACE-example}
\Phi: (x_1,x_2)\longmapsto \Big( {x_1^{-5}}+ {x_2^{-5}},  x_1^{-2}+ {x_2^{-2}},  {x_1^{-1}}+ {x_2^{-1}} \Big).
\end{equation*}
and
\begin{equation*}
 \label{E:doubletransSURFACE-example2}
\Psi: (y_1,y_2)\longmapsto \Big(
 -{y_1^{-5}}- {y_2^{-5}},  -y_1^{-2}- {y_2^{-2}},  -{y_1^{-1}}- {y_2^{-1}}
 \Big)
\end{equation*}
are two  parametrizations of $S$.
Notice  that $\Phi(x_1,x_2)=\Psi(y_1,y_2)$ if and only if
$(x_1,x_2)=(-y_1,-y_2)$ or $y_1=x_1x_2\zeta_+$ and  $y_2=x_1x_2\zeta_-$ where $\zeta_\pm=\zeta_\pm(x_1,x_2)$
are the two complex roots of the polynomial $$(x_1^2+x_1x_2+x_2^2)\zeta^2+(x_1+x_2)\zeta+1=0\,. $$
If $p,q\in \mathbb C^2$ are points satisfying  $\Phi(p)=\Psi(q)$ but $p\neq -q$, then $\Phi: (\mathbb C^2,p)\rightarrow S$ and
$\Psi: (\mathbb C^2,q)\rightarrow S$ are   two distinct parametrizations of $S$  at $\Phi(p)=\Psi(q)$. Hence $S$ is a  double-translation
hypersurface with
$$
\mathcal W_S \simeq \mathcal W\Big(x_1,x_2,x_1x_2\zeta_{-}, x_1x_2\zeta_{+} \Big)\,.
$$
\end{example}

\subsection{Examples}
Example \ref{Ex:SC-algebraic} is a particular instance of the general construction presented below.

Let $C\subset \mathbb P^n$ be a reduced non-degenerate projective curve of degree $2n$.
Assume that $h^0(\omega_C) \geq n+1$. If $C$ is $\mathcal W$-generic then  $h^0(\omega_C)=n+1$,
 but otherwise $h^0(\omega_C)$ can be larger.

Let $H_0$ be a hyperplane intersecting $C$ transversely at $2n$  points. As usual, consider
 the maps ${p_1,\ldots,p_{2n}:(\check{\mathbb P}^n,H_0)\rightarrow C}$
  satisfying  $C\cdot H=\sum_{i=1}^{2n} p_i(H)$ for every ${H \in (\check{\mathbb P}^n,H_0)}$.
Since $C$ is non-degenerate, it is harmless to assume   that the maps  $p_i$ have been indexed in such a way that
$ p_1(H),\ldots,p_{n}(H)$ generate $H$ for every  $H  \in (\check{\mathbb P}^n,H_0)$.

If $\omega^1,\ldots\omega^{n+1}$ are linearly independent abelian differentials on $C$ then  for each $i \in \underline n$
consider the maps $\phi_i, \psi_i : (\check{\mathbb P}^n,H_0) \to \mathbb C^{n+1}$ defined as
\begin{align*}
\phi_i(H) = & \;\left( \int_{p_{i}(H_0)}^{ p_i(H)} \omega^1, \ldots, \int_{p_i(H_0)}^{ p_i(H)} \omega^{n+1} \right)  \\
  \text{ and } \quad \psi_i(H) =&   - \left( \int_{p_{i+n}(H_0)}^{ p_{i+n}(H)} \omega^1, \ldots, \int_{p_{i+n}(H_0)}^{ p_{i+n}(H)} \omega^{n+1} \right) \, .
\end{align*}
Notice that their sums, that is $\Phi = \sum_{i=1}^n \phi_i$ and $\Psi = \sum_{i=1}^n \psi_i$, satisfy
$\Phi(H) = \Psi(H)$  for  every $H \in (\check{\mathbb P}^n , H_0)$ according to Abel's addition Theorem. Furthermore,  they  parametrize a non-degenerate hypersurface  $S_C$, and $\mathcal W_{\Phi} \boxtimes \mathcal W_{\Psi}$
is a $2n$-web equivalent to $\mathcal W_C$. From all that have been said it is clear that $S_C$ is a double translation hypersurface.
By definition, it  is a \defi[double-translation hypersurface associated to $C$ at $H_0$]. The use of the indefinite article
{\it an} instead of the definite one {\it the} is due to the lack of uniqueness of the hypersurface for $C$ and $H_0$ fixed. It will
depend on the choice of the $1$-forms and of the points $p_i$ in general. When $C$ is irreducible, a monodromy  argument shows
it is possible to replace the {\it an} by a {\it the}.   \smallskip

\subsection{Abel-Jacobi map}

The double translation hypersurface associated to an irreducible
 non-degenerate projective curve $C \subset \mathbb P^n$ of degree $2n$ and  arithmetic genus $n+1$ admits
a more intrinsic description which has the advantage of being global. It is defined in terms of the   Abel-Jacobi map
of $C$. Although much of the discussion can be carried out in greater generality,  this will not be
done here. The interested reader can consult \cite{little}.

If $C_{sm}$ stands for the smooth part of $C$    then there is a linear map  from
$H_1(  C_{sm}, \mathbb Z)$ to $H^0(C,\omega_C)^*$ defined as
\[
\gamma \longmapsto \left( \omega \mapsto \int_\gamma \omega \right) \, .
\]
It can be shown that its image $\Gamma$, is a
discrete subgroup  of $H^0(C,\omega_C)^*$. Consequently, the  quotient
of  $H^0(C,\omega_C)^*$ by $\Gamma$ is a smooth  complex variety $J(C)$: the \defi[Jacobian] of
$C$. When $C$ is smooth then $J(C)$ is indeed  projective, as was shown by Riemann.

Once a point $p \in C_{sm}$ and a positive integer $k$ are fixed, one can consider the map
\begin{align*}
AJ_C^k : (C_{sm})^k & \longrightarrow J(C) \, \\
(x_1, \ldots, x_k ) & \longmapsto \left( \omega \mapsto \sum_{i=1}^k \int_{p}^{x_i} \omega \right) ,
\end{align*}
where the integrations are performed along paths included in  $C_{sm}$.
Notice that $AJ_C^k$ is well-defined since all  possible ambiguities disappear after  taking the quotient
by $\Gamma$. By definition, $AJ_C^k$ is the \defi[$k$-th Abel-Jacobi map] of $C$. \index{Abel-Jacobi map}

\medskip

 If $C$ is a smooth projective curve of genus $n+1$ then a classical theorem of Riemann (see \cite[Chap.I.\S5]{ACGH}) asserts that  the image
 of the  $n$-th Abel-Jacobi map is a translate of the  theta divisor $\Theta$  of $J(C)$, which  is, by definition,  the reduced divisor defined as the zero locus of {\it Riemann's theta function} $\theta$ of the jacobian $J(C)$\footnote{The Riemann's theta function $\theta_A$ of a polarized abelian variety $A=\mathbb C^g/\Delta$ with $\Delta=(I_g,Z)$ (where $Z \in M_g(\mathbb C) $ is such that $Z=\!{}^t\!Z$ and ${\rm Im}\,Z>0$) is defined by $\theta_A(z)=\sum_{m\in \mathbb Z^g} \exp\big( i\pi \langle m,Zm\rangle+2i\pi\langle m,z\rangle\big)$ for all $z\in \mathbb C^g$ (see \cite[Chap.I]{ACGH}).}.

\begin{prop}
If  $C \subset \mathbb P^n$ is a smooth non-degenerate curve of degree $2n$ and genus $n+1$
 then  the double-translation hypersurface $S_C$ associated to $C$ is nothing but the
lift to $\mathbb C^{n+1}$ of (a translate of) the theta divisor $\Theta\subset J(C)$.
\end{prop}
\begin{proof}
From Riemann's result referred  above it suffices to show that $S_C$ can be identified with the image of the $n$-th Abel-Jacobi map.

Since $C$ is non-generated  $P = (p_1, \ldots, p_n) : (\check { \mathbb P}^n , H_0) \to C^n$ is a germ of biholomorphism. Moreover,
$\Phi$ is, up to a suitable choice of affine coordinates, equal to  $AJ_C^n \circ P$ (or rather, its natural lift to $H^0(C,\Omega^1_C)^* \cong \mathbb C^{n+1}$).
The proposition follows.
\end{proof}

When $C$ is smooth, it is known that the lift of $\Theta$ in $\mathbb C^{n+1}$ is a transcendental hypersurface. When $C$ is singular, the hypersurface $S_C$  is not necessarily transcendent. Loosely speaking, the more $C$ is singular, the less $S_C$ is transcendent.
For instance, we let the reader verify that the rational double-translation surface of
Example \ref{Ex:SC-algebraic} is nothing but the surface $S_C$ associated to the plane singular rational quartic parametrized by $\mathbb P^1 \ni [s:t]\mapsto [s^4 : st^3 :t^4]\in \mathbb P^2$ (see also Example \ref{Ex:unicursalSINGULARcurve} in Chapter III).

\subsection{Classification}

\begin{thm}
 Let $S\subset \mathbb C^{n+1}$ be a non-degenerate double translation hypersurface   such that $\mathcal W_S$ is smooth.
 Then $S$ is the double-translation hypersurface associated to a $\mathcal W$-generic projective  curve in $\mathbb P ^n$
 of degree $2n$ and  arithmetic genus $n+1$.
\end{thm}
\begin{proof}
 Let $\Phi$ and $\Psi$ be   two distinct translation-type  parametrizations of a double-translation hypersurface $S\subset \mathbb C^{n+1}$.
 The coordinates $x_1,\ldots,x_n,y_1,\ldots,y_n$ in which $\Phi$ and $\Psi$ are expressed can be thought as  holomorphic functions on $S$ defining
 the web $\mathcal W_S$. Hence the identity
\[
 \sum_{i=1}^n \Phi(x_i(p)) +  \sum_{i=1}^n \Psi(y_i(p)) = 0 \in \mathbb C^{n+1}
\]
holds at a any point $p  \in S$.
Since $S$ is non-degenerate this equation provides $n+1= \pi(n,2n)$  linearly independent abelian relations
for $\mathcal W_S$. Theorem \ref{T:alg2nWEB} implies the result.
\end{proof}

When $\mathcal W_S$ is only assumed to be quasi-smooth, one has the following algebraization
result  which can be traced back to  Wirtinger \cite{wirtinger}.

\begin{thm}
 Let $S\subset (\mathbb C^{n+1},0)$ be a double-translation hypersurface. Assume  that its distiguished parametrizations $\Phi$ and $\Psi$ are  such that
\begin{quote}
\hspace{-0.5cm}$(\star)$ {\it  none of the  vectors
$\frac{d^2 \phi_i}{dx_i^2}(0), \frac{d^2 \psi_i}{dy_i^2}(0)$
is tangent to $S$ at the origin.}
\end{quote}
Then $S$  is the double-translation hypersurface associated to a non-degenerate curve $C\subset \mathbb P^n$ of degree $2n$ and such that $h^0(\omega_C)\geq n+1$.
\end{thm}

For a proof of the above result, the reader is redirected to  \cite{little} from where the formulation
above has been borrowed. There he will also find the following application to the Schottky Problem: characterize the
Jacobian of curves among the principally polarized abelian varieties. \index{Schottky problem}

\begin{thm} Let $(A,\Theta)$ be a principally polarized abelian variety of dimension $n$. Suppose that there exists a point
$p \in \Theta$ such that the germ of $\Theta$ at $p$ is a double-translation hypersurface satisfying $(\star)$. Then $(A,\Theta)$
is the canonically polarized Jacobian of a smooth nonhyperellitptic curve of genus $n$.
\end{thm}

For more information about the Schottky Problem the reader is urged to  consult  \cite{BeauvilleSchottky}, \cite[Appendix, Lecture IV]{Mumford} and the references therein.

%%%%%%%%%%%%%%%%%%%%%%%%%%%%%%%%%%%%%%%
%  Introduction
%  label = Chapter:intro
%  First version by JVP
%  last modification: 3/nov/2008
%  Remarks:
%%%%%%%%%%%%%%%%%%%%%%%%%%%%%%%%%%%%%%%

\chapter[Algebraization]{Algebraization of maximal rank webs}\label{Chapter:Trepreau}
\thispagestyle{empty}
This chapter is devoted to the  following result.

\begin{THM} \nonumber
Let $n\ge 3$ and  $k \ge 2n$ be integers.  If $\mathcal W$ is a  smooth $k$-web of maximal rank
on $(\mathbb C^n,0)$ then $\mathcal W$ is algebraizable.
\end{THM}

Its proof crowns the efforts spreaded  over at least three generations of mathematicians.  For $n=3$, the theorem
is due to Bol and is among the deepest results obtained by Blaschke's school.
For $n>3$, Chern and Griffiths provided a {\it proof }  in \cite{Jbr} which
later on \cite{Jbr2} revealed to be incomplete. The definitive version stated
above is due to Tr\'{e}preau \cite{Trepreau}. He not just
followed Chern-Griffiths's  general strategy, but also simplified it, to
prove the general case.

\medskip

While many of the concepts and ideas used in the proof --- Poincar\'{e}'s map and canonical maps ---
have already been  introduced in Chapter \ref{Chapter:4},
it  will be essential to consider also another map, here called Poincar\'{e}-Blaschke's map, canonically attached to
webs of maximal rank. This map was originally introduced by Blaschke in \cite{blaschke} to {\it prove} that
$5$-webs of maximal rank are algebraizable. Using  Blashcke's ideas introduced in this paper, Bol succeeded
to establish  the algebraization of smooth  $k$-webs of maximal rank on $(\mathbb C^3,0)$, for $k \ge 6$.
Ironically, not much latter \cite{bol36}, he  came up
with $6=\pi(2,5)$ linearly independent abelian relations for his $5$-web $\mathcal B_5$, showing in this
way that his source inspiration \cite{blaschke} was irremediably flawed.  Although wrong, Blaschke's paper
contained not just the germ of Bol's algebraization result for webs on $(\mathbb C^3,0)$, but also the germ
of Chern-Griffiths's   strategy.

 \medskip

The main novelty in Tr\'{e}preau's approach
is based on ingenious and involved, albeit elementary, computations. It is still considerably more technical than the other  results previously presented in this book.
Nevertheless the authors believe that the understanding  of    this praiseworthy theorem, as well as
of the ideas/techniques involved in its proof, will reward those who persevere  through this chapter.

\section{Tr\'{e}preau's Theorem}

Below, a more precise version of the theorem stated in the introduction of this chapter is formulated.
Then the heuristic behind its proof is explained.

Arguably, one could complain about the unfairness of the title of this  section. It would be perhaps more righteous
to call it Bol-Tr\'{e}preau's Theorem or even Blaschke-Bol-Chern-Grifitths-Tr\'{e}preau's Theorem. To justify the choice made
above, one could invoke the right to typographical beauty and/or the usual mathematical practice.

\subsection{Statement}

\begin{thm}\label{T:trepreau}
Let $\mathcal W$ be a smooth $k$-web on $(\mathbb C^n,0)$. If $n\ge 3$, $k \ge 2n$ and
\[
\dim \frac{\mathcal A(\mathcal W)}{F^2 \mathcal A(\mathcal W)} = 2k-3n + 1
\]
then $\mathcal W$  is  algebraizable.
\end{thm}

To explain  the heuristic behind the proof of  Theorem \ref{T:trepreau} some of the geometry of  Castelnuovo curves will be recalled
in the Section \ref{S:CC5}. Then in Section \ref{S:MR5} it will be discussed how one could infer corresponding geometrical
properties for webs of maximal rank using their spaces of abelian relations.

\subsection{The geometry of Castelnuovo curves}\label{S:CC5}

Let $C\subset \mathbb P^n$ be a Castelnuovo curve of degree $k\geq 2n$.
To avoid technicalities assume  $C$ is smooth.
By definition  $K_C=\omega_C$ and  $h^0(K_C)=\pi(n,k)=\pi$.
Recall from Chapter \ref{Chapter:4} that the canonical map of $C$  is \index{Canonical map!of a curve}
\begin{align}
\label{E:canmapC}
\varphi= \varphi_{|K_C|} :  C & \longrightarrow {\mathbb P} H^0(C,K_C)^*=   \mathbb P^{\pi-1}
 \\ x  & \longmapsto  \left[ \eta \mapsto \eta(x)  \right] \, .
\nonumber
\end{align}
Since $C$ is smooth of genus strictly greater than one, the canonical linear system  $|K_C|$ has no base point ({\it cf.} \cite[IV.5]{hartshorne}) thus $\varphi= \varphi_{|K_C|}$ is a birational morphism from $C$ onto its \defi[canonical model]  \index{Castelnuovo curve!canonical model}
 $C_{can}=\varphi(C)$.\footnote{In fact, one can prove that $C$ is not  hyperelliptic. Thus $|K_C|$ induces an isomorphism $C\simeq C_{can}$, see \cite[IV.5]{hartshorne}.}

Recall from Section \ref{S:CC} that the linear system $| I_C(2)|$ of quadrics containing $C$ cut out a
non-degenerate surface $S\subset \mathbb P^n$ of minimal degree $n-1$. Thus,
$S$ is a rational normal scroll, or the Veronese surface $v_2(\mathbb P^2) \subset \mathbb P^5$.

Still aiming at simplicity,    assume $S$ is  smooth. Because $S$   is rational, it does not have holomorphic differentials.
More precisely,   $h^0(S, \Omega^1_S) = h^0(S, \Omega^2_S) = 0$.
Consequently,
\begin{align*}
h^1(S,K_S)
=&\;
h^{1}(S,\mathcal O_S) &&  \mbox{by Serre duality  \cite[III.7]{hartshorne} }  \\
=&\;
h^{0}(S, \Omega^1_S) =0
 &&   \mbox{by
Hodge theory .}
\end{align*}

From the exact sequence
\begin{equation*}
0 \rightarrow K_S \rightarrow K_S\otimes \mathcal O_S(C) \rightarrow K_C \rightarrow 0 \,
\end{equation*}
one deduces  an isomorphism
\begin{equation*}
H^0\big(S,K_S\otimes \mathcal O_S(C)\big) \simeq
H^0\big(C,K_C)\, .
\end{equation*}
Thus  the  canonical map (\ref{E:canmapC}) extends to $S$: there is a rational map $\Phi: S\dashrightarrow \mathbb P^{\pi-1}$
fitting into the commutative diagram below.
\begin{equation*}
  \xymatrix@R=0.9cm@C=1.3cm{  C  \;    \ar@{^{(}->}[d]  \ar@{->}[r]^{\varphi} & \;{\mathbb P}^{\pi-1} \ar@{=}[d] \\
S \;      \ar@{-->}[r]^{\Phi\; \, } & \; {\mathbb P}^{\pi-1}. }
\end{equation*}

  Moreover, as it is explained in \cite{harris},  $X_C$,  the image of $\Phi$, is  a non-degenerate algebraic surface in $\mathbb P^{\pi-1}$.

   \bigskip

If $H$ is  a generic hyperplane in $\mathbb P^n$, then the  hyperplane section $C_H=S\cap H$ is irreducible, non-degenerate in $H$, and of degree $\deg C_H=\deg S=n-1$.
Hence $ C_H$ is a curve of minimal degree in $H$, that is a rational normal curve of degree $n-1$.
Now let $p_1, \ldots, p_n$ be  $n$ generic points on $S$.
They span a  hyperplane $H_p$ and the generic hyperplane is obtained in this way.
Thus $C_H=S\cap H_p$ is a rational normal curve of degree $n-1$ that contains the points $p_1,\ldots,p_n$ and is contained in $S$. 
%\begin{figure}[ht]
%\begin{center}
%\psfrag{S}[][]{$ S $}
%\psfrag{C}[][]{$ {\textcolor{red}{C }}$}
%\psfrag{CH}[][]{$ \textcolor{blue}{C_H} $}
%\psfrag{PD}[][]{$ \check{\mathbb P}^{n} $}
%\psfrag{P}[][]{$ {\mathbb P}^{n} $}
%\psfrag{H}[][]{$  H $}
%\psfrag{p1}[][]{\textcolor{red}{$\scriptscriptstyle{p_1(H)} $}}
%\psfrag{p2}[][]{\textcolor{red}{$\scriptscriptstyle{p_2(H)} $}}
%\psfrag{p3}[][]{\textcolor{red}{$\scriptscriptstyle{p_3(H)} $}}
%\psfrag{p4}[][]{\textcolor{red}{$\scriptscriptstyle{p_4(H)} $}}
%\resizebox{2.75in}{3.2in}{
%\includegraphics{TGV3.eps} }
%\end{center}
%\caption{The curve $C_H$.}
%\end{figure}
Therefore,  through $n$ general points of $S$ passes a rational normal curve of degree $n-1$.  Surfaces having this property are said to be \defi[$n$-covered by rational normal curves of degree $n-1$].
\index{Surfaces!$n$-covered}

It can be proved that the image under $\Phi$ of  $C_H$  is
a rational normal curve $\mathscr C_{H}$  in a projective subspace of $\mathbb PH^0(C,K_C)^* $
of dimension $k-n-1$. Consequently, the  surface $X_C={\rm Im}\,\Phi\subset \mathbb P^{\pi-1}$
is $n$-covered by rational normal curves of degree $k-n-1$.

The preceding facts will now be interpreted  in terms of the web $\mathcal W_C$ dual to the  curve $C$.
As usual, let $H_0\subset \mathbb P^n$ be a hyperplane transverse to $C$, and let $p_1,\ldots,p_{d}: (\check{\mathbb P}^n,H_0) \rightarrow C$
be the usual holomorphic maps describing the intersection of $H \in (\check{\mathbb P}^n,H_0)$ with $C$.
For $i \in \underline k$, set
\begin{equation*}
 \varphi_i= \varphi_{|K_C|}\circ p_i: (\check{\mathbb P}^n,H_0) \rightarrow \mathbb P H^0(C,K_C)^*\, .
\end{equation*}

If $H \in (\check{\mathbb P}^n,H_0)$ then the points $p_1(H), \ldots,p_d(H)$ span the hyperplane $H\subset \mathbb P^n$.
Thus they  belong to $C $, and hence to $S$.
Therefore each $p_i(H)$ belongs to the rational normal curve $C_H = C \cap H\subset S$.
Consequently the points $\varphi_i(H)$, $i \in \underline k$, belong to the
rational normal curve $\mathscr C_H \subset \mathbb P H^0(C,K_C)^*$.
Moreover, the curves $\mathscr  C_H$ for $H\in (\check{\mathbb P}^n,H_0)$ fill out a germ of surface along  $\mathscr C_{H_0}$.

\begin{figure}[H]
\begin{center}
%\tex{S}[][][][]{$ S $}
\psfrag{U}[][]{$ \check{U}  $}
\psfrag{S}[][]{$ S $}
\psfrag{SS}[][]{$ X_C $}
\psfrag{C}[][]{$ {\textcolor{red}{C }}$}
\psfrag{CC}[][]{$ {\textcolor{red}{C_{can} }}$}
\psfrag{CH}[][]{$ \textcolor{blue}{C_H} $}
\psfrag{CCHH}[][]{$ \textcolor{blue}{\mathscr C_H} $}
\psfrag{PD}[][]{$ \check{\mathbb P}^{n} $}
\psfrag{P}[][]{$ {\mathbb P}^{n} $}
\psfrag{PP}[][]{$ \scriptstyle{{\mathbb P}H^0(C,K_C)^*} $}
\psfrag{H}[][]{$  H $}
\psfrag{HH}[][]{$  \langle \textcolor{blue}{\mathscr C_H}  \rangle $}
\psfrag{PHI}[][]{$  \textcolor{red}{\varphi_{|K_C|}} $}
\psfrag{PSI}[][]{$  \Phi $}
\psfrag{H}[][]{$  H $}
\psfrag{p1}[][]{\textcolor{red}{$\scriptscriptstyle{p_1(H)} $}}
\psfrag{p2}[][]{\textcolor{red}{$\scriptscriptstyle{p_2(H)} $}}
\psfrag{p3}[][]{\textcolor{red}{$\scriptscriptstyle{p_3(H)} $}}
\psfrag{p4}[][]{\textcolor{red}{$\scriptscriptstyle{p_4(H)} $}}
\psfrag{ph1}[][]{\textcolor{red}{$\scriptscriptstyle{\varphi_1(H)} $}}
\psfrag{ph2}[][]{\textcolor{red}{$\scriptscriptstyle{\varphi_2(H)} $}}
\psfrag{ph3}[][]{\textcolor{red}{$\scriptscriptstyle{\varphi_3(H)} $}}
\psfrag{ph4}[][]{\textcolor{red}{$\scriptscriptstyle{\varphi_4(H)} $}}
\resizebox{4.5in}{3.5in}{
\includegraphics{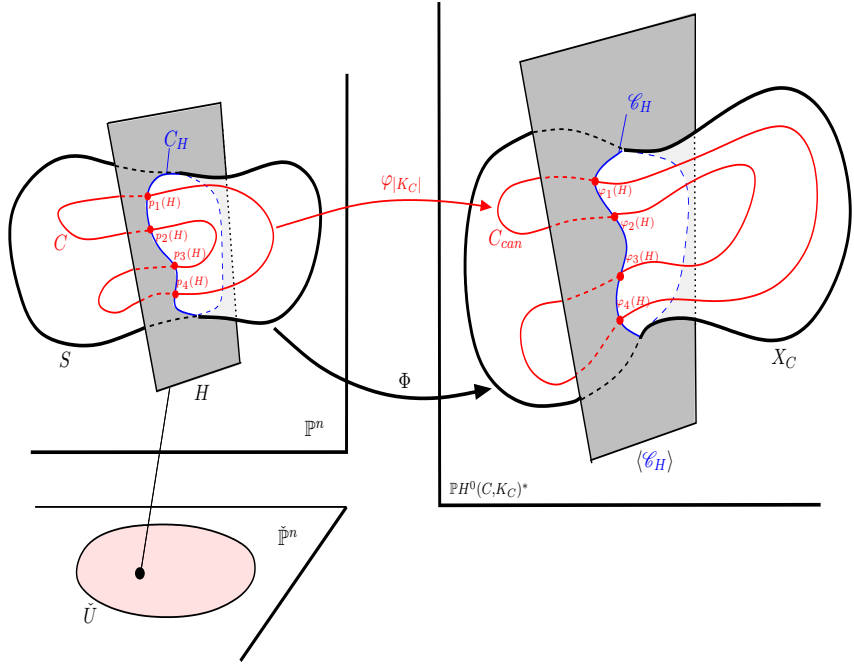} }
\end{center}
\caption[The curve $C_H$ and its image $\mathscr C_{H}$]{\!\! The curve $C_H$ and its image $\mathscr C_{H}$  under  $\Phi$. }
\end{figure}%\vspace{0.2cm}

\subsection{On the geometry of maximal rank webs}\label{S:MR5}
This section explains  how the constructions
presented in the preceding section extend to webs carrying sufficiently many abelian relations.
Since the case  $k=2n$ has  already been treated in Chapter \ref{Chapter:4}, it will be assumed that $k >  2n$.

\smallskip

Let $\mathcal W$ be a smooth $k$-web on $(\mathbb C^n,0)$.  To simplify the discussion, it will be assumed that  $\mathcal W$ has maximal rank,
even if the heuristic described  below will be implemented under the weaker hypothesis of Theorem \ref{T:trepreau}.

\smallskip

Using the hypothesis on the space of abelian relations of  $\mathcal W$, it can
be proved that the images $\kappa_{\mathcal W,1}(x), \ldots, \kappa_{\mathcal W,k}(x)$  of every $x \in (\mathbb C^n,0)$ by the canonical maps of
$\mathcal W$ generate a projective subspace $\mathbb P^{k-n-1}(x) \subset \mathbb P \mathcal A(\mathcal W)^*$ of dimension $k-n-1$, and lie on a unique
rational normal curve $\mathscr C(x) \subset \mathbb P^{k-n-1}(x)$.

The most delicate point, and the main novelty,  in Tr\'{e}preau's argument, is
his proof  that  the family of rational normal curves $\mathscr C_W=\{ \mathscr C(x) \, \}_{x\in (\mathbb C^n,0) }$ fills out a germ of non-degenerate, smooth surface  $X_{\mathcal W}\subset \mathbb P\mathcal A(W)^*$.

It is then  comparably simple  to show that any two curves $\mathscr C(x), \mathscr C(y)$ intersect in
exactly $n-1$ points. Therefore $\mathscr C(0)^2 > 0$,  and well known results about germs of surfaces
containing curves of positive self-intersection by Andreotti ( in the analytic category ) and Hartshorne
( in the algebraic category ) imply that $X_{\mathcal W}$ is contained in a projective surface  $S_{\mathcal W}$.
As the image under $\Phi$ of the surface of minimal degree in $\mathbb P^n$ containing  a Castelnuovo curve $C$ of degree $k$ (see previous section),
the surface $S_{\mathcal W}$ is a rational surface  $n$-covered by rational normal curves of degree $k-n-1$.

\begin{figure}[H]
\begin{center}
%\tex{S}[][][][]{$ S $}
\psfrag{U}[][]{$ { }  $}
\psfrag{u}[][]{$ {x}  $}
\psfrag{SS}[][]{$ S_{\mathcal W} $}
\psfrag{CCHH}[][]{$ \textcolor{blue}{\mathscr C(x)} $}
\psfrag{PD}[][]{$ \check{\mathbb P}^{n} $}
\psfrag{Cn}[][]{$ {\mathbb C}^{n} $}
\psfrag{PP}[][]{$ {{\mathbb P}\mathcal A(W)^*} $}
\psfrag{H}[][]{$  H $}
\psfrag{HH}[][]{$  \langle \textcolor{blue}{ \mathscr C(x) }\rangle $}
\psfrag{PHI}[][]{$  \textcolor{red}{\varphi_{|K_C|}} $}
\psfrag{PSI}[][]{$  \Phi $}
\psfrag{H}[][]{$  H $}
\psfrag{c1}[][]{\textcolor{red}{$C_{1} $}}
\psfrag{c2}[][]{\textcolor{red}{$  C_{2}$}}
\psfrag{c3}[][]{\textcolor{red}{$C_{3} $}}
\psfrag{c4}[][]{\textcolor{red}{$C_{4} $}}
%\psfrag{cw4}[][]{\textcolor{red}{${c_{W,4}}$}}

\psfrag{cw1}[][]{\textcolor{red}{\rotatebox{79}{$
\scriptstyle{{\kappa_{W,1}}}$}}}

\psfrag{cw2}[][]{\textcolor{red}{\rotatebox{79}{$ \scriptstyle{{\kappa_{\mathcal W,2}}}$}}}

\psfrag{cw3}[][]{\textcolor{red}{\rotatebox{76}{$ \scriptstyle{{\kappa_{\mathcal W,3}}}$}}}

\psfrag{cw4}[][]{\textcolor{red}{\rotatebox{65}{$ \scriptstyle{{\kappa_{\mathcal W,4}}}$}}}

%\psfrag{ph1}[][]{\textcolor{red}{$ \scriptscriptstyle{c_{W,1}(u)}$}}
%\psfrag{ph2}[][]{\textcolor{red}{$\scriptscriptstyle{c_{W,2}(u)} $}}
%\psfrag{ph3}[][]{\textcolor{red}{$\scriptscriptstyle{c_{W,3}(u)} $}}
%\psfrag{ph4}[][]{\textcolor{red}{$\scriptscriptstyle{c_{W,4}(u)} $}}
\resizebox{4in}{4in}{
\includegraphics{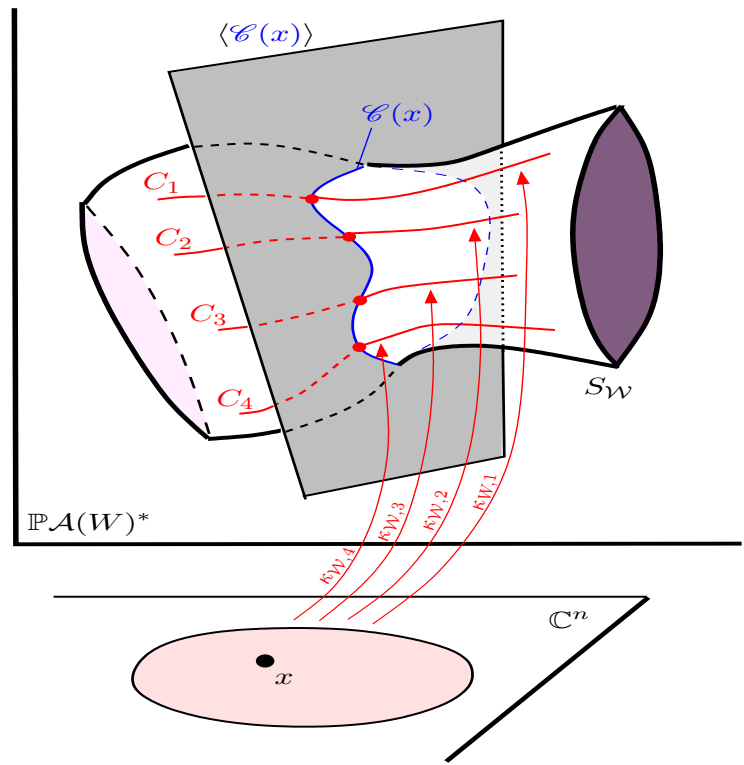} }
\end{center}
\caption[Geometry behind the proof]{Geometry behind the  proof of Tr\'{e}preau's Theorem.}
\end{figure}%\vspace{0.2cm}

At this point, one can apply an argument by Chern-Griffiths to  linearize $\mathcal W$.
Since $S_{\mathcal W}$ is rational, every in the family $\mathscr C_{\mathcal W}$ belongs to
one and only linear system $|\mathscr C|=\mathbb P H^0(S_{\mathcal W},\mathcal O_{S_{\mathcal W}}(\mathscr C))$, which turns out
to have dimension $n$. One then defines a map sending $x \in (\mathbb C^n,0)$ to the
point in $|\mathscr C|$ corresponding to the  rational normal curve $\mathscr C(x)$.
It is possible to prove that this map is a biholomorphism which linearizes $\mathcal W$.
Finally, the converse of Abel's Theorem presented in Chapter \ref{Chapter:4} allows to conclude that
$\mathcal W$ is algebraizable.

\section{Maps naturally attached to $\mathcal W$}

Until the end of this  chapter  $\mathcal W= \mathcal F_1 \boxtimes \cdots \boxtimes \mathcal F_k$ is a smooth $k$-web on $(\mathbb C^n,0)$ satisfying
$$ \dim \mathcal A(\mathcal W) / F^2\mathcal A(\mathcal W) = 2k-3n +1= \pi\, .$$
 According
to this hypothesis there exist $\pi$ linearly independent abelian relation  with classes in $\mathcal A(\mathcal W) / F^2\mathcal A(\mathcal W)$
generating the whole space. Fix, once and for all, $\pi$ abelian relations  $\eta^{(1)}, \ldots, \eta^{(\pi)}$
with this property, and let $\A2W \subset \mathcal A(\mathcal W)$ be the vector space generated by them.

\smallskip

Notice that the inclusion $\A2W \subset \mathcal A(\mathcal W)$  induces a linear projection
$\mathbb P \mathcal A(\mathcal W)^*  \dashrightarrow \mathbb P \A2W^*$. Notice also that the intersection of the
filtration $F^{\bullet} \mathcal A(\mathcal W)$ of $\mathcal A(\mathcal W)$
with $\A2W$ induces the  filtration
 $$F^{\bullet} \A2W = \A2W \cap F^{ \bullet } \mathcal A(\mathcal W). $$
From the choice of
$\A2W$ it is clear that
\[
 \quad \dim F^1 \!\A2W = k - 2n +1 \, \text{ and } \, \dim F^j\! \A2W = 0 \, \text{ for } j > 1 .
\]

\subsection{Canonical maps}

Fix $i \in \underline k$. If $\kappa_i = \kappa_{\mathcal W,i} : (\mathbb C^n,0) \to \mathbb P \mathcal A(\mathcal W)^*$ is the $i$-th canonical map
of $\mathcal W$, see Section \ref{S:canmap}, then its image does not intersect the center of the natural projection
 $\mathbb P \mathcal A(\mathcal W)^* \dashrightarrow\mathbb P \A2W^*$. Indeed,
as argued in  the proof of Lemma \ref{L:67}, the equality
${ \dim \A2W / F^1 \A2W = k -n}$ implies the existence of an abelian relation
in $\A2W$ with $i$-th component not vanishing at $0$. Thus
composing $\kappa_{i}$ with the linear projection $\mathbb P \mathcal A(\mathcal W)^*  \dashrightarrow \mathbb P \A2W^*$
defines a morphism  from $(\mathbb C^n,0)$ to $\mathbb P \A2W^*$. To keep the notation simple it will still
be denoted by $\kappa_{i}$,  and will  be called
the $i$-th canonical map of $\mathcal W$. More explicitly,  $\kappa_{i}$ is now the  map
\begin{align*}
 (\mathbb C^n,0) & \longrightarrow \; \mathbb P \A2W ^* \\
x & \longmapsto  \big[  ev_i(x) : \A2W \to \mathbb C\big].
\end{align*}

The  image of the $i$-th canonical map of $\mathcal W$
will be denoted by $C_i$ and will be called the $i$-th canonical curve of $\mathcal W$.

\subsection{Poincar\'{e}'s map}

Consider the natural   analogue of Poincar\'{e}'s map
\begin{align*}
 P_{\mathcal W} : \,  (\mathbb C^n,0) & \longrightarrow
 \mathrm{Grass}(\A2W, k-2n +1 )   \\
x &\longmapsto   F^1_x \A2W    \,.  \nonumber
\end{align*}
As in the case of canonical maps, it seems unjustifiable to change the terminology. Therefore the map $P_{\mathcal W}$
above will also be called the Poincar\'{e}'s map of $\mathcal W$.

For each $x \in (\mathbb C^n,0)$,
the projective subspace of dimension $k-n-1$ in $\mathbb P \A2W^*$   determined by $P_{\mathcal W}(x)$ through projective duality
will be denoted by $\mathbb P^{k-n-1}(x)$.

\subsection{Properties}
From the very definition of $F^1_x \A2W$ it follows  that
\[
 F^1_x \A2W = \bigcap_{i \in \underline k} \ker ev_i(x)
\]
where $ev_i(x)$ is considered as a linear form on $\A2W$.

This remark about subspaces of $\A2W$ translates into the following relation between the canonical maps and
Poincar\'{e}'s map: $\mathbb P^{k-n-1}(x)$ is the smallest projective space among the ones containing the set $\{ \kappa_i(x) \}_{i \in \underline k}$.
The lemma below shows that it is possible to replace $\underline k$, in the statement above, by any subset $B$ with at least $k-n$ elements.

\begin{lemma}\label{L:gera0}
For every $x \in (\mathbb C^n,0)$ and every subset $B \subset \underline k$ of cardinality $k-n$,
$\mathbb P^{k-n-1}(x)$ is the smallest linear subspace of $\mathbb P \A2W^*$ containing
$\{ \kappa_i(x)\}_{i \in B}$.
\end{lemma}
\begin{proof}
Notice that the smallest linear subspace of $\mathbb P \A2W^*$ containing $\{ \kappa_i(x)\}_{i \in B}$
is the dual of the intersection
\[
 I_x =  \bigcap_{i \in B} [\ker ev_i(x)] \subset \mathbb P \A2W \, .
\]
If  a non-trivial abelian relation, or rather its projectivization, is in $I_x$ then it has at most $k - (k-n) =n$
components not vanishing  at $x$. But the constant term of these  components are linearly independent
because $\mathcal W$ is a  smooth web on $(\mathbb C^n,0)$. Thus
\[
 I_x = \bigcap_{i \in B} [\ker ev_i(x)] = [ F^1_x \A2W ] \subset \mathbb P \A2W \, . \qedhere
\]
\end{proof}

\medskip

In order to express intrinsically the differential of $\kappa_i$ at a point $x \in (\mathbb C^n,0)$,
 observe that the  tangent space of $\mathbb P \A2W^*$ at the point $\kappa_i(x) = [ev_i(x)]$ is naturally isomorphic
to the quotient of $\A2W^*$ by the the $1$-dimensional subspace $\mathbb C ev_i(x)$. This quotient in its turn is isomorphic
to the dual of $\ker ev_i(x)$. Thus the differential of  $\kappa_i$ at $x$ can be written as  follows
\begin{align*}
d \kappa_i(x) :  T_x( \mathbb C^n,0) & \longrightarrow  \ker ev_i(x)^* \\
v &\longmapsto  \left\lbrace
v  \mapsto \displaystyle{\lim_{t \to 0} \frac{ev_i(x + tv ) - ev_i(x)}{t}} \,
 \right\rbrace \, .
\end{align*}
Therefore, inasmuch $\kappa_i(x)$ can be identified with abelian relations in  $\A2W$
with $i$-component  vanishing at $x$, the image of  its differential at $x$ can be identified with the abelian
relations in $\A2W$ with $i$-th component vanishing at $x$ with multiplicity  two.
In other words, if $ev_i^{(1)}(x) : \A2W \to \mathbb C^2$ is the  evaluation morphism of order one and $V_i(x) \subset \A2W$ is its kernel,
then the image of $d\kappa_i(x)$  is the quotient  of $\ker ( \A2W^* \to V_i(x)^* )$ by $\mathbb C ev_i(x)$.

\smallskip

\begin{lemma}\label{L:gera}
Let  $x \in (\mathbb C^n,0)$ and  $B \subset \underline k$ be a subset of cardinality
smaller than or equal to $k-2n +1 $. If  $Y \subset \mathbb P \A2W^*$ is the set
\[
\mathbb P^{k-n-1}(x) \cup \left ( \bigcup_{i \in B}  T_{\kappa_i(x)} C_i \right) \,
\]
then the smallest projective subspace of $\mathbb P \A2W^*$ containing $Y$ has
codimension equal to  ${(k- 2n +1) - \mathrm{card}(B)}$. In particular, none of the
canonical curves $C_i$ is tangent to $\mathbb P^{k-n-1}(x)$ at $\kappa_i(x)$.
\end{lemma}
\begin{proof}
The proof is similar to the proof of the previous lemma. As in the discussion preceding the statement,
let $V_i$ be the kernel of the evaluation morphism of order one  $ev_i^{(1)}(x) : \A2W \to \mathbb C^2$.

Since $\bigcap_{i \in \underline k} \ker ev_i(x) = F^1_x \A2W$, to prove the lemma it suffices to show
that the dimension of
\[
 A_x=  F^1_x \A2W \cap \left(\bigcap_{i \in B} V_i \right)  \, .
\]
is equal to  $a= k-2n + 1 - \mathrm{card}(B)$.
Since $\dim F^1_x \A2W = k -2n +1$ and $\dim F^1_x \A2W / (F^1_x \A2W \cap V_i(x)) \ge 1$, the number  $a$ is greater than or equal to $k-2n + 1 - \mathrm{card}(B)$.

Notice that the elements of $A_x$ are abelian relations with $i$-th components, for every $i \in B$, having  constant and linear terms at $x$
equal to zero. Therefore,
\[
a \le ( k - \mathrm{card}(B)) - \ell^2 \left( {\boxtimes}_{i \in \underline k \setminus B} \mathcal F_i \right)  + \dim F^2 \A2W \, .
\]
But $ \dim F^2 \A2W =0 $  and   Proposition \ref{P:cast} implies
\[
\ell^2 \left( {\boxtimes}_{i \in \underline k \setminus B} \mathcal F_i \right) \ge 2(n-1) +1 \, .
\]
Thus
\[
a \le (k - \mathrm{card}(B)  ) - 2(n-1) -1 = k -  2n +1 -\mathrm{card}(B) \, ,
\]
as wanted.
\end{proof}

\smallskip

\begin{remark}\label{R:kisub}
The proof above shows slightly more than what is stated. Indeed, it was proved that not just $C_i$ is smooth and not tangent  to
$\mathbb P^{k-n-1}(x)$ at $\kappa_i(x)$ but also that $\kappa_i$ is a submersion onto $C_i$. This fact will be used in the proof
of the next proposition and later on.
\end{remark}

Lemma \ref{L:gera} for subsets  $B \subset \underline k$ of cardinality one is,
 together with Lemma \ref{L:gera0}, the main ingredient in the proof of the next proposition.

\begin{prop}\label{P:pmii}
Poincar\'{e}'s map $P_{\mathcal W}$  is an immersion.
\end{prop}
\begin{proof} Let $\gamma: (\mathbb C,0)  \to (\mathbb C^n,0)$ be a holomorphic immersion.
Since $\mathcal W$ is smooth, $\gamma$ is tangent to at most $n-1$ of the foliations $\mathcal F_i$. Thus there exists
a set $B \subset \underline k$ of cardinality $k-n \le k- (n-1)$, such that  the composition $\kappa_i \circ \gamma$ is an immersion  for every $i \in B$.
Moreover, Lemma \ref{L:gera0} implies that for every $t\in (\mathbb C,0)$, the points $\{ (\kappa_i \circ \gamma)(t)\}_{i\in B}$ generate the projective subspace $\mathbb P^{k-n-1}(\gamma(t)) \subset \mathbb P \A2W ^*$ determined by $P_{\mathcal W}(\gamma(t))$.

\smallskip

If   $\mathrm{Grass}(\A2W, k-2n +1 )$ is identified with its  Pl\"{u}cker's embedding\begin{footnote}{ The Grassmannian $\mathrm{Grass}(V,r)$ is isomorphic
to  the projectivization of the image of the multilinear   map
\begin{align*}
\varphi :   V ^r  & \dashrightarrow  \bigwedge^r V \, \\
(v_1, \ldots, v_r) &\mapsto v_1 \wedge \cdots \wedge v_r  \, .
\end{align*}
The isomorphism is given of course by associating to any
$W \in \mathrm{Grass}(V,r)$ the point $[\varphi(w_1, \ldots, w_r)] \in \mathbb P \bigwedge^r V$ where
$w_1, \ldots, w_r$ is an arbitrary basis of $W$. Clearly, $[\varphi(w_1, \ldots, w_r)]$ does not depend on
 the basis chosen. The induced map $ \mathrm{Grass}(V,r) \to \mathbb P \left( \wedge^r V \right)$ is the so called \defi[Pl\"{u}cker embedding] of $\mathrm{Grass}(V,r)$.
}\end{footnote} of  $\mathrm{Grass}(\A2W^*, k-n)$ then one can write
\[
(\widehat{P_{\mathcal W}} \circ \gamma)(t) = \bigwedge_{i \in B} (\widehat{\kappa_i} \circ \gamma)(t) \, ,
\]
where the hats indicate   liftings to  $\bigwedge^{k-n} \A2W^*$ and $\A2W^*$  respectively.
Consequently, the identity
\[
(\widehat{P_{\mathcal W}} \circ \gamma)'(t)  \wedge ( \widehat{\kappa_j} \circ \gamma) (t) = (\widehat{\kappa_j} \circ \gamma)'(t) \wedge (\widehat{P_{\mathcal W}} \circ \gamma)(t) \,
\]
holds true for every $j \in B$.
Lemma \ref{L:gera} ( see also Remark \ref{R:kisub}~) ensures the non-vanishing of this latter expression. Since $\gamma$ is an arbitrary immersion,
the differential of $P_{\mathcal W}$ is injective at the origin. The proposition follows.
\end{proof}

\medskip

\begin{cor}
 \label{C:int=n-2}
For every  distinct pair of points $x,y \in (\mathbb C^n,0)$, the intersection  $\mathbb P^{k-n-1}(x) \cap \mathbb P^{k-n-1}(y)$
is a projective subspace of  $\mathbb P \A2W^*$ of  dimension  $n-2$.
\end{cor}
\begin{proof}
It suffices to prove the claim for $x=0$ and $y$ arbitrarily close to it.
Since $2(k-n-1) - (2k - 3n  )   = n-2$,
the claim is equivalent to the transversality of $\mathbb P^{k-n-1}(0)$ and $\mathbb P^{k-n-1}(y)$.
The reader is invited to verify that the lack of transversality between $\mathbb P^{k-n-1}(0)$ and $\mathbb P^{k-n-1}(y)$,
for $y$ arbitrarly close to $0$, would imply that the differential of $P_{\mathcal W}$ at the origin
is not injective. This contradiction implies the corollary.
\end{proof}

\section{Poincar\'{e}-Blaschke's map}
\label{S:defPBW}

In this section the Poincar\'{e}-Blaschke's map for webs satisfying the assumptions of Theorem \ref{T:trepreau} are defined, and it is proved
that they have rank two when in dimension at least three. Its content is considerably more technical, although rather elementary, than
the remaining of the book. The arguments herein follow very closely \cite{Trepreau}.

\subsubsection{Settling the notation}
Let ${u_1, \ldots, u_k : (\mathbb C^n,0) \to (\mathbb C,0)}$  be submersions defining the foliations $\mathcal F_1, \ldots, \mathcal F_k$.
Recall that Proposition \ref{P:normal} settles  the existence
of a coframe $\varpi = (\varpi_0, \ldots, \varpi_{n-1})$ for $\Omega^1(\mathbb C^n,0)$, and $k$ holomorphic functions
$\theta_1, \ldots, \theta_k$, such that the foliation $\mathcal F_i$ is induced by the $1$-form
\[
\omega_i = \sum_{q=0}^{n-1} (\theta_i)^{q} \varpi_q \, .
\]
Notice also the existence of holomorphic functions $h_1, \ldots, h_k$ satisfying $du_i = h_i \omega_i$ for every $i \in \underline k$.

\medskip

 Until the end of Section \ref{S:defPBW} the submersions $u_i$; the coframe $\varpi$; the functions $\theta_i$ and $h_i$; and
the $1$-forms $\omega_i$  will have the same meaning as above.

\medskip

For an arbitrary $1$-form $\alpha \in \Omega^1(\mathbb C^n,0)$ its $q$-th component in the coframe $\varpi$ will be
written as $\{ \alpha \}^q$. More precisely, the holomorphic functions $\{ \alpha \}^0, \ldots, \{\alpha\}^{n-1}$ are
implicitly defined by  the identity
\[
 \alpha = \sum_{q=0}^{n-1} \{ \alpha\}^q \varpi_i \, .
\]

\medskip

To write down the canonical maps $\kappa_i$ explicitly, identify
the  point $ (a_1, \ldots, a_{\pi}) \in \mathbb P^{\pi -1}$ with the hyperplane
$\{ a_1 \eta^{(1)} + \cdots + a_{\pi} \eta^{(\pi)}=0 \} $ in\begin{footnote}{Here the abelian relations
$\eta^{(j)}$ are thought as coordinate functions on $\A2W$, which is the same as thinking of them as elements of $\A2W^*$.}\end{footnote} $\A2W$. The $i$-th evaluation morphism
at a point $x \in (\mathbb C^n,0)$ is nothing more than
\[
(a_1, \ldots, a_{\pi})  \longmapsto  a_1 \eta^{(1)}_i(x) + \cdots + a_{\pi} \eta^{(\pi)}_i(x) \, .
\]
Notice  that the $1$-forms $\eta^{(j)}_i$ for $j \in \underline{\pi}$, are all proportional
to $\omega_i$.   Hence there are holomorphic functions $z^{(j)}_i$ such that
\[
\eta^{(j)}_i = z^{(j)}_i \omega_i \,
\]
 for every $i \in \underline k$  and every $ j \in \underline \pi $\, .
Therefore, for a fixed $i \in \underline k$, the map
\begin{align*}
Z_i: (\mathbb C^n,0) & \longrightarrow \mathbb C^{\pi} \\
x &\longmapsto ( z^{(1)}_i(x), \ldots, z^{(\pi)}_i(x))
\end{align*}
is a lift of $\kappa_i$ to $\mathbb C^{\pi}$. More precisely, the diagram
\[
\xymatrix{
   & \mathbb C^{\pi} \simeq \A2W^* \ar@{-->}[d] \\
(\mathbb C^n,0)  \ar^{Z_i}[ur] \ar^{\kappa_i}[r] & \mathbb P \A2W^*
}
\]
commutes.

\medskip

For further use, the translation of the conditions $\sum \eta_i =0 $ and $d\eta_i=0$ to conditions on the functions
$z^{(j)}_i$ is stated below as a lemma. The proof is immediate.
\begin{lemma}
\label{L:ARsystmALT}
If $z_1,\ldots,z_k:(\mathbb C^n,0) \to \mathbb C$ are  holomorphic functions on  $(\mathbb C^n,0)$
then $\eta = ( z_1 \omega_1, \ldots, z_k \omega_k)$  is an abelian relation of $\mathcal W$  if and only if
\begin{equation}
 \label{E:AReqNORM}
d( z_i \,\omega_i )=0 \qquad \mbox{ and  } \qquad  \sum_{i=1}^k z_i \, (\theta_i)^\sigma=0
\end{equation}
for every $i \in \underline k$ and every  $\sigma \in \{ 0,\ldots,n-1 \}$.
\end{lemma}

\subsection{Interpolation of the canonical maps}

For $i \in \underline k$ consider the polynomials $P_i \in \mathcal O{(\mathbb C^n,0)}[t]$ defined
through the formula
\[
 \quad P_i (t )  =\prod_{\substack{j=1\\ j \neq i }}^k  \big( t -  \theta_j \big)  \, .
\]

The canonical maps, or more precisely their lifts $Z_i$ defined above,  are interpolated
by the  map $Z_*$ defined below.
\begin{align*}
\label{E:Z(u,t)}
Z_*: (\mathbb C^n,0) \times \mathbb C &\longrightarrow \mathbb C^\pi \\
(x, t) & \longmapsto \sum_{i=1}^k   P_i(t) Z_i(x) \, .
\end{align*}
Indeed $Z_*(x,\theta_i(x))$ is proportional to $Z_i(x)$ since $${Z_*(x, \theta_i(x) ) = P_i(\theta_i(x)) Z_i(x)}.$$

Some properties of the map $Z_*$ are collected in the following lemma.
\begin{lemma}
\label{L:UtimesC}
The map $Z_*$ has the following properties:
\begin{itemize}
\item[(a)] for every $x\in (\mathbb C,0)$ and every $t \in \mathbb C$,  $Z_*(x,t)\neq 0$;
\item[(b)] the entries of $Z_*(x,t)$, seen as polynomials in $\mathcal O_{(\mathbb C^n,0)}[t]$, have
degree at most $k-n-1$;
\item[(c)]  the coefficient of $t^{k-n-1}$ in
$Z_*(x,t)$ is non-zero
and equal  to
$$
\sum_{i=1}^k\theta_i(x)^nZ_i(x)  \, .
$$
\end{itemize}
\end{lemma}
\begin{proof}
To prove item (a), let $(x_0,t_0) \in (\mathbb C^n,0) \times \mathbb C$ be a point where $Z_*$ vanishes.
If  $t_0=\theta_i(x_0)$ for some $i\in \underline k$, then clearly $Z(x_0,t_0)=P_i(\theta_i(x_0)) \, Z_i(x_0)\neq 0$.

\smallskip

Assume now that   $t_0$ belongs to $\mathbb C \setminus \{\theta_1(x_0),\ldots,\theta_k(x_0)\}$.
Because  ${\ell^1(\mathcal W)=k-n}$,   Lemma \ref{L:ARsystmALT} implies that every relation of the form
 $\sum_i c_i Z_i(x_0)=0$ is a linear combination of the relations
 $\sum_i (\theta_i(x_0))^\sigma Z_i(x_0)$ for $\sigma=0,\ldots,n-1$.

\smallskip

 If  $\sum_{i} (t-\theta_i(x_0))^{-1} Z_i(x_0)=0$, then there exist   $\mu_1,\ldots,\mu_n\in \mathbb C$  which satisfy
$$ \frac{1}{t_0-\theta_i(x_0)}=\sum_{\sigma=1}^{n} \mu_\sigma\,\theta_i(x_0)^\sigma
$$
for every $i \in \underline k$. But this is not possible, since $\theta_i(x_0) \neq \theta_j(x_0)$ whenever $i \neq j$.
Hence (a) holds true.

\medskip

To prove item (b), the dependence in $x$ will be dropped from the  notation in order to keep it simple. Let $P(t)  =\prod_{j=1  }^k  ( t -  \theta_j )$ and write
\begin{equation}\label{E:defsigma}
P_i(t) = \sum_{j=1}^{k-1} \sigma_j^{(i)} t^j   \quad \text{ and } \quad P(t) = \sum_{j=1}^{k} \sigma_j t^j.
\end{equation}
Comparing coefficients in the identities $P(t) = (t - \theta_i) P_i(t)$, one deduces that
\[
\sigma_{j+1}=\sigma_j^{(i)} -\theta_i\,\sigma_{k+1}^{(i)} \, .
\]
Consequently,
\begin{equation}
\label{sigmaka}
 \sigma_j^{(i)} =\sum_{s=0}^{k-j-1} (\theta_i)^s\,\sigma_{j+s+1} \; .
\end{equation}

From Equation (\ref{sigmaka}), it follows that
\begin{align}
Z_*(t)=& \, \sum_{j=0}^{k-1} \Big( \sum_{i=1}^k \sigma_j^{(i)} \,Z_i \Big)\, t^j \nonumber
\\=& \,\sum_{j=0}^{k-1} \Big(\sum_{i=1}^k \sum_{s=0}^{k-j-1} (\theta_i)^s \,\sigma_{j+s+1} \,Z_i \Big)\, t^j \nonumber  \\\label{Zstar}
=& \,\sum_{j=0}^{k-1}\Big(\sum_{s=0}^{k-j-1} \big( \sum_{i=1}^k(\theta_i)^s \,Z_i \big)\, \sigma_{j+s+1}\Big)\,  t^j \; .
\end{align}
According to Lemma  \ref{L:ARsystmALT},  $\sum_i Z_i (\theta_i)^s$  is identically zero for  any  $s \in \{ 0,\ldots,n-1 \}$.
Thus the  coefficient of $t^j$ in  (\ref{Zstar})
is identically zero for every  $j\geq k-n$. Item (b) follows.

\medskip

The coefficient of   $t^{k-n-1}$ in   $Z_*(t)$  is
$$ \sum_{s=0}^{n}
\Big( \sum_{i=1}^k(\theta_i)^s \,Z_i\Big)\, \sigma_{k - n + s}  =  \sigma_{k}
\sum_{i=1}^k(\theta_i)^n \,Z_i \,
$$
with the equality being obtained through the use of  Lemma \ref{L:ARsystmALT} exactly as above.
To conclude the proof of Item (c) it suffices to notice that $\sigma_{k} =1$ according to (\ref{E:defsigma}).
\end{proof}

 \subsection{Definition of Poincar\'{e}-Blaschke's map}

The Poincar\'{e}-Blaschke's map\footnote{Such  map was first introduced by Blaschke    extrapolating ideas of Poincar\'{e} \cite{poincare}.  In \cite{blaschke}, Blaschke constructs and studies the Poincar\'{e}-Blaschke map of a maximal rank planar 5-web.
He mistakenly asserted that its image lie in a surface of $\mathbb P^{5}$.} of $\mathcal W$,
\[ PB_{\mathcal W} : (\mathbb C^n,0) \times \mathbb P^1  \longrightarrow \mathbb P^{\pi-1} \, ,
\]
is defined as
\[
 \label{E:defPBWUtimesinfty}
PB_{\mathcal W} (x,t)  = \left\{ \begin{array}{lcl}
 \big[ Z_* (x,t) \big] & \text{for} & t \in \mathbb C  \\
 \big[  \sum_{i=1}^k\theta_i(x)^nZ_i(x)  \big] & \text{for} & t= \infty.
  \end{array}
 \right.
\]

Observe that  the definition of $Z_*$  does depend on the choice of:
the subspace  $\A2W \subset \mathcal A(\mathcal W)$; on the basis of $\A2W$;
and on the adapted   coframe $\varpi_0, \ldots, \varpi_{n-1}$. Nevertheless,
 modulo projective changes of coordinates on the target $\mathbb P^{\pi -1}$
and on the $\mathbb P^1$ factor of the source, $PB_{\mathcal W}$
only depends on the choice of the subspace $\A2W \subset \mathcal A(\mathcal W)$ as
the reader can easily verify.

\bigskip

Notice that Poincar\'{e}-Blaschke's map restricted at $\{ x \} \times \mathbb P^1$  parametrizes
a rational curve which interpolates  the canonical {\it points} $\kappa_1(x), \ldots, \kappa_k(x)$.
More precisely,

\begin{lemma}
\label{L:CRN0}
For $x \in (\mathbb C^n,0)$ fixed, the map $${\varphi_x : t\in {\mathbb P}^1\mapsto P\!B_{ \mathcal W}(x,t)\in \mathbb P^{\pi-1}} $$
is an isomorphism onto  a rational normal curve  ${\mathscr C}(x) \subset \mathbb P^{k-n-1}(x)$ of degree $k-n-1$.
\end{lemma}
\begin{proof}
Clearly the map under scrutiny parametrizes a rational curve $\mathscr C(x)$. Moreover  $\varphi_x(\theta_i(x))=[Z_i(x)]= \kappa_i(x)$ for every $i \in \underline k$, and
\begin{equation*}
 { \dim}\, \big\langle {\mathscr C}(x) \big\rangle\geq
{\dim}\, \Big\langle Z_1(x) ,\ldots, Z_k(x) \Big\rangle - 1 \, \, .
\end{equation*}
But the span of $Z_1(x), \ldots, Z_k(x)$ has dimension $\ell^1(\mathcal W) = k-n$.

Lemma \ref{L:UtimesC} tells that the map  $Z_*$ has degree  $k-n-1$ in $t$. Hence $\varphi_x$ parametrizes a non-degenerate rational curve in $\mathbb P^{k-n-1}(x)$
of degree at most $k-n-1$. It follows from Proposition \ref{P:ratdeg}   that $\mathscr C(x)$ is a rational normal curve of degree $k-n-1$.
\end{proof}

\subsection{The rank of Poincar\'{e}-Blaschke's map}

This section is devoted to the proof of the following proposition.

\begin{prop}\label{PBWrang}
If $n > 2$
then $P\!B_{\mathcal W}$ has rank $2$  at every  $(x,t) \in  (\mathbb C^n,0) \times {\mathbb P}^1$.
\end{prop}

\subsubsection{Lower bound for the rank}

It is not hard to show that the rank of $P\!B_{\mathcal W}$ is at least two as the result below shows.

\begin{lemma}\label{L:lower}
For every $(x,t) \in (\mathbb C^n,0) \times \mathbb P^1$, the rank of $PB_{\mathcal W}$ at $(x,t)$
is at least two. Moreover, if  $t = \theta_i(x)$ for some $x \in (\mathbb C^n,0)$, then
$PB_{\mathcal W}(x,t)$ has rank exactly two at $(x,t)$.
\end{lemma}
\begin{proof}
According to Lemma \ref{L:CRN0}, $P\!B_{\mathcal W}$ restricted to the line $\{x\} \times \mathbb P^1$ is
 an isomorphism onto the   rational normal  curve $\mathscr C(x) \subset \mathbb P^{k-n-1}(x)$.
In particular, the tangent line to $\mathscr C(x) $ at $PB_{\mathcal W}(x,t)$, namely
$$ \Big\langle PB_{\mathcal W}(x,t),\frac{\partial PB_{\mathcal W}}{\partial t}(x,t) \,
 \Big\rangle
$$
 is contained in $\mathbb P^{k-n-1}(x)$.

\medskip

The restriction of $P\!B_{\mathcal W}$ to the hypersurface $H_i=\{t= \theta_i(x)\}$ is a submersion
onto the $i$-th canonical curve of $C_i$.  Thus the image of the differential of ${PB_{\mathcal W}}|_{H_i}$
at $(x, \theta_i(x))$ is $T _{\kappa_i(x)} C_i$.  Lemma \ref{L:gera} implies
the rank of $PB_{\mathcal W}$ is exactly two  at any point of the hypersurface $H_i$.

\medskip

If $P(x,t) \neq 0$, that is if $(x,t) \notin \cup_{i\in \underline k} H_i$, then
one can deduce from Lemma \ref{L:gera} that the vectors
$$
\frac{ \partial}{ \partial x_j }\bigg(\sum_{i =1}^k P_i(x,t) Z_i(x)\bigg)
%\frac{ \partial \left(\sum_{i \in \underline k} P_i(x,t) Z_i(x)\right) }{ %\partial x_j }
\, \qquad \text{ with } j\in \underline n,
$$
 span the whole space $\mathbb C^{\pi}$. Details are
left to the reader.
\end{proof}

To prove that the rank of $P\!B_{\mathcal W}$ is at most two is considerably more delicate  as
the next few pages testify.

\subsubsection{The main technical point}

The next result is essential in the proof of Proposition \ref{PBWrang}.

\begin{prop}
\label{P:TECHmain}
There are germs of  holomorphic functions $M^p_r$  determined by the  coframe $\varpi$   such that
for any abelian relation $(z_1\omega_1,\ldots,z_k\omega_k)\in \mathcal A(\mathcal W)$,
for every $p \in \{ 0,\ldots,n-2 \} $, and  every $i \in \underline k$, the following identity  holds true
\begin{equation}
\label{E:rel-z1}
 \{ dz_i\}^{p+1}-\theta_i \{ dz_i\}^p-z_i\{ d\theta_i\}^p  =z_i \sum_{r=0}^{n-1} (\theta_i)^r M^p_{r} .
\end{equation}

Moreover, if $n>2$ then
\begin{equation}
\label{E:rel-theta1}
 \{ d\theta_i\}^{p+1}-\theta_i \{ d\theta_i\}^p=\sum_{\rho=0}^{n} (\theta_i)^\rho N^p_{\rho} \, ,
\end{equation}
 where $N^p_r$ are holomorphic functions also determined by $\varpi$.
\end{prop}

\begin{remark}
Let $\theta$ be any function on $(\mathbb C^n,0)$ and set $\omega_\theta=\sum_{q} \theta^{q}\varpi_q$.
If $\omega_{\theta}$ is integrable  then, after writing down
the coefficients of  $\varpi_p \wedge\varpi_p \wedge \varpi_r$ in
$\omega_{\theta} \wedge d \omega_{\theta}$  and imposing their vanishing, one deduces  relations of the form
\begin{equation}
\label{E:rel-theta0}
 \{ d\theta\}^{p+1}-\theta \{ d\theta\}^p=\sum_{\rho=0}^{n+p+1} \theta^\rho \mathcal N^p_{\rho}
\end{equation}
for $p=0,\ldots,n-2$,
where  $\mathcal N^p_{\rho}$  are certain holomorphic functions  that do not depend on $\theta$ but only on the adapted coframe $\varpi$. \smallskip

Similarly, one can prove that there are holomorphic functions $\mathcal M^p_\rho$ depending only on $\varpi$
such that if  $\omega_\theta$ is integrable, then any $z$  such that  $d(z\omega_\theta)=0$ necessarily verifies  the relations (for any $p=0,\ldots,n-2$):
\begin{equation}
\label{E:rel-z0}
 \{ dz\}^{p+1}-\theta \{ dz\}^p-z\{ d\theta\}^p  =z \sum_{\rho=0}^{n+p} \theta^\rho \mathcal M^p_{\rho}\, .
\end{equation}
%for $p=1,\ldots,n-1$.
%\smallskip

Relations (\ref{E:rel-theta0}) and (\ref{E:rel-z0}) are direct consequences of the integrability condition
and have nothing to do with webs and/or their abelian relations. Proposition \ref{P:TECHmain} improves
these relations by  lowing  down the upper limit of both summations.
\end{remark}

\subsubsection{Proof of the main technical point I -- Preliminaries }

For every  $x \in (\mathbb C^n,0)$ and every $i \in \underline k$, write the Taylor expansion of $u_i$ centered at the origin
as
$$  u_i(x)=\ell_i(x)+\frac{1}{2}q_i(x)+O_0(3)$$
where  $\ell_i$ (resp.   $q_i$) are linear  (resp. quadratic) forms.
Let  $\xi=(\xi_1 du_1, \ldots, \xi_k du_k)$ be an abelian relation of  $\mathcal W$.
 Since $d\xi_i\wedge du_i(0)=0$, the following identity holds
$$ \xi_i(x)=a_i+b_i \ell_i(x)+O_0(2) \qquad i \in \underline k
$$
for suitable complex numbers   $a_i,b_i$.
Looking at the order one jet at the origin of the relation  $\sum_i \xi_i du_i=0$, one deduces that
\begin{equation}
 \sum_{i=1}^k a_i \ell_i=0 \quad \mbox{ and } \quad
\sum_{i=1}^k a_i q_i+ \sum_{i=1}^k b_i (\ell_i)^2 = 0 .
\end{equation}

\medskip

Denote by  $\mathcal Q= \mathbb C_2[x_1, \ldots, x_n]$ the space
 of quadratic forms on  $\mathbb C^n$.  For   $Q=\sum_{i\leq j} Q^{ij}x_ix_j \in \mathcal Q$
 define the following differential operators
\begin{align*}
Q(\nabla F)=  \sum_{i\leq j} Q^{ij}\frac{\partial F}{\partial x_i}
\frac{\partial F}{\partial x_j} \quad \mbox{ and } \quad
Q_\partial(F)=   \sum_{i\leq j} Q^{ij}\frac{\partial^2 F}{\partial x_i\partial x_j }
\end{align*}
%\begin{align*}
%%F}{\partial u_i}
%\frac{\partial F}{\partial u_j} \\
%Q(\partial)F=  & \, \sum_{i\leq j} Q_{ij}\frac{\partial^2 %F}{\partial u_i\partial u_j }
%\end{align*}
where  $F$ is a germ of holomorphic function.

By hypothesis $F^0\!\mathcal A(W)/F^1\!\mathcal A(W) $ has dimension $k-n$.  Therefore
for every  $a=(a_1,\ldots,a_d)\in \mathbb C^d$, the following implication holds true
\begin{equation}
\label{E:condORDRE2}
  \sum_{i=1}^k a_i \ell_i= 0 \Longrightarrow
\sum_{i=1}^k a_i q_i \in {\rm Span}_{\mathbb C}\,  \big\langle (\ell_1)^2,\ldots,(\ell_k)^2 \big\rangle \,.
\end{equation}

To better understand this relation, notice that the smoothness of $\mathcal W$ implies that  $\ell_1,\ldots,\ell_n$ is a basis of $\mathbb C_1[x_1, \ldots, x_n]$.
 Thus for every  $i \in \underline k$, there is a decomposition  $ \ell_i= \sum_{j} l_{i }^j \ell_j$ with constants  $l_{i}^j$ uniquely determined.
Thus  (\ref{E:condORDRE2}) translates into to the more precise statement
\begin{equation}
 \label{E:qalpha}
q_i- \sum_{j=1}^n l_{i}^j q_j  \in {\rm Span}_{\mathbb C}\,  \big\langle (\ell_1)^2,\ldots,(\ell_k)^2 \big\rangle
\end{equation}
for any  $i \in \underline k$.

If  $G \in \mathcal Q$ and $\ell$ is a linear form then  $G_\partial(\ell^2)=2 G(\nabla \ell)$.   Suppose that
$G(\nabla \ell_i)=0$  for every $i \in \underline k$. Thus  $G_\partial(q)=0$ for every $q\in {\rm Span}_{\mathbb C}\,  \big\langle (\ell_1)^2,\ldots,(\ell_k)^2 \big\rangle$.  Using the relations (\ref{E:qalpha})  one deduces
$$
G_\partial(q_i)=\sum_{j=1}^n l_i^j G_\partial(q_i)
$$
for every  $i \in \underline k$.

Set  $M_G$ as the vector field $\sum_i G_\partial(q_i)\frac{\partial}{\partial x_i}$.
Since  $G_\partial(\ell)=0 $ for every linear form $\ell$, one deduces that
$$
G_\partial(u_i)(0)=\langle du_i(0),M_{G}\rangle
$$
for every  $i \in \underline k$.

By hypothesis, the implication  (\ref{E:condORDRE2}) holds true for every $x \in (\mathbb C^n,0)$.
The discussion above implies the following result.

\begin{lemma}
\label{L:condORDRE2}
Let  $\mathcal G=\sum_{i\leq j} \mathcal G ^{ij}(x) x_i x_j$ be a field of quadratic forms.
If  $\mathcal G(\nabla u_i)$ vanishes identically for every $i \in \underline k$ then there exits a vector field $X_{\mathcal G}$
such that
$$
\mathcal G_\partial(u_i)=\langle du_i ,X_{\mathcal G}\rangle $$
for every $i \in \underline k$.
\end{lemma}

\subsubsection{Proof of the main technical point I -- Conclusion }

Since $\varpi= (\varpi_0, \ldots, \varpi_{n-1})$  is a coframe on $(\mathbb C^n,0)$,
there are  basis change formulas  (for $j=1,\ldots,n$ and $q=0,\ldots,n-1$)
\begin{align*}
 du_j= \sum_{q=0}^{n-1} B^q_j \,\varpi_q \qquad \mbox{ and } \qquad
\varpi_q = \sum_{j=1}^n C^j_q\, du_j .
\end{align*}

For   $l,m \in \underline n$, set
$$  \quad u_{i, l}=\partial u_i/\partial x_l  \qquad \mbox{ and } \qquad
  u_{i, lm }=\partial^2 u_i/\partial x_l\partial x_m. $$

Thus $$du_i= h_i \sum_{q=0}^{n-1} (\theta_i)^q\varpi_q=\sum_{j=1}^n  u_{i, j} \,dx_j$$
and consequently  (for $p=0,\ldots,n-1$ and  $j=1,\ldots,n$)
\begin{equation}
 \label{E:fuckLRU}
h_i (\theta_i)^{p}=\sum_{j=1}^n B_j^p \,u_{i, j}
\qquad \mbox{ and } \qquad
u_{i, j}=h_i \sum_{p=0}^{n-1} C_p^j (\theta_i)^p.
\end{equation}

Consider four integers $p,p',q,q'$ in the interval $[0, n-1]$  which satisfy  $p+q=p'+q'$.
Because  $(\theta_i)^p(\theta_i)^q=
(\theta_i)^{p'}(\theta_i)^{q'}$, the equations (\ref{E:fuckLRU}) imply the relation
\begin{equation*}
 \label{}
\sum_{j,j'=1}^n \Big( B_j^pB_{j'}^q-B_j^{p'}B_{j'}^{q'}
\Big) u_{i, j}u_{i, j'}=0
\end{equation*}
holds true for every $i \in \underline k$.

Lemma \ref{L:condORDRE2} implies the existence of functions  $X^1,\ldots,X^n$, for which
\begin{equation*}
 \sum_{j,j'=1}^n \Big( B_j^pB_{j'}^q-B_j^{p'}B_{j'}^{q'}
\Big)u_{\alpha, jj'}=\sum_{l=1}^n X^l u_{i,l}
\end{equation*}
for every $i \in \underline k$.
It is important to observe that the functions  $X^1,\ldots,X^n$
do not depend on the function $u_i$ but only on the integers   $p,q,p',q'$ and, of course, on the coframe $\varpi$. \smallskip

Notice that
$$d(h_i (\theta_i)^p)=\sum_{j=1 }^n \partial_{x_j}( h_i (\theta_i)^p) dx_j=\sum_{q=0}^{n-1}  \Big(\sum_{j=1 }^n B_j^q \partial_{x_j}( h_i (\theta_i)^p) \Big) \varpi_q $$
and consequently
$\{ d(h_i (\theta_i)^p)\}^q= \sum_{j} B_j^q
\partial_{u_j}( h_i (\theta_i)^p)$.

Combining this last equation with
(\ref{E:fuckLRU}) one obtains
 \begin{align*}
 \{ d(h_i (\theta_i)^p)\}^q=  \sum_{j,j'=1}^n B_j^q
\partial_{x_j}( B_{j'}^p u_{i,j'})    
=  \,  \sum_{j,j'=1}^n B_j^q
\partial_{x_j}( B_{j'}^p) u_{i,j'}+
\sum_{j,j'=1}^n B_j^q
 B_{j'}^p u_{i,jj'}.
\end{align*}
%\begin{align*}
 %\{ d(h_i (\theta_i)^p)\}^q= & \,  \sum_{j,j'=1}^n B_j^q
%\partial_{x_j}( B_{j'}^p u_{i,j'})    \\
%= & \,  \sum_{j,j'=1}^n B_j^q
%\partial_{x_j}( B_{j'}^p) u_{i,j'}+
%\sum_{j,j'=1}^n B_j^q
 %B_{j'}^p u_{i,jj'}.
%\end{align*}

Thus, one can write
\begin{align}
\label{E:baz}
 \{ d(h_i (\theta_i)^q)\}^p
-\{ d(h_i (\theta_i)^{q'})\}^{p'}
=  \sum_{l=1}^n Y^l u_{i,l}
\end{align}
with
$$
Y^l=X^l+\sum_{j,j'=1}^n \big(  B_{j'}^p \partial_{x_j}( B_{l}^q)-
 B_{j'}^{p'} \partial_{x_j}( B_{l}^{q'}) \big) .$$

Once again one has to apply the relations (\ref{E:fuckLRU}). After setting
$M_r=\sum_l Y^l C_r^l $ for $r=0,\ldots,n-1$, it follows that
\begin{align*}
 \{ d(h_i (\theta_i)^q)\}^p
-\{ d(h_i (\theta_i)^{q'})\}^{p'}
= h_i \sum_{r=0}^{n-1} M_r \,(\theta_i)^r
\end{align*}
for no matter which  $i\in \underline k$.

Note that the functions  $M_r$  depend only on the integers  $p,q,p',q'$, but not on  $i$.
It suffices to take  $p'=p+1$, $q=1$ and  $q'=0$ to establish the  existence of functions  $M_r^p$
satisfying
\begin{align}
\label{E:LCDisCOOL}
 \{ d(h_i \theta_i)\}^p
-\{ dh_i\}^{p+1}
= h_i \sum_{r=0}^{n-1} M_r^p \,(\theta_i)^r.
\end{align}

Let $z=(z_1\omega_1,\ldots,z_k\omega_k)$ be an abelian relation of $\mathcal W$. The function $z_i$ is such that
$d(z_i \omega_i)=0$ then  $d( z_i/h_i)\wedge du_i=0$.
Therefore
$$
\sum_{p=0}^{n-1}  \Big(  {\{dz_i\}^p}-
 {z_i}\,{h_i}^{-1}\{dh_i\}^p
 \Big) \varpi_p \wedge \sum_{q=0}^{n-1} (\theta_i)^q \varpi_q=0 \, ,
$$
 which implies
\begin{equation}
\label{E:luc}
 \Big(\{dz_i\}^p-{z_i}{h_i}^{-1}\{dh_i\}^p
\Big) (\theta_i)^q-
\Big(\{dz_i\}^q-{z_i}{h_i}^{-1}\{dh_i\}^q
\Big) (\theta_i)^p=0
\end{equation}
for every  $p,q=0,\ldots,n-1$.  It suffices to set $q=p+1$
in (\ref{E:luc}) and combine it with (\ref{E:LCDisCOOL}) to obtain the relations  (\ref{E:rel-z1}) of  Proposition \ref{P:TECHmain}.
\smallskip

To obtain the relations (\ref{E:rel-theta1}), take $q=2$, $ q'=1$ and  $p'=p+1$ in  (\ref{E:baz}). Note
that this is possible only because $n$ is assumed to be at least $3$.  On the one hand, it follows from  (\ref{E:baz}) the existence of
holomorphic functions $L_0, \ldots, L_{n-1}$ which satisfy for every $i \in \underline k$ the following identity
\begin{align}
\label{E:subidon}
 \{ d(h_i (\theta_i)^2)\}^p
-\{ d(h_i \theta_i)\}^{p+1}
= h_i \sum_{r=0}^{n-1} L_r \,(\theta_i)^r.
\end{align}

On the other hand,
\begin{align*}
 \{ d(h_i (\theta_i)^2)\}^p
-\{ d(h_i \theta_i)\}^{p+1}= 
\theta_i  \Big(
\{ d(h_i \theta_i)\}^p-\{ dh_i\}^{p+1} \Big) 
+h_i \Big( \theta_i\{ d\theta_i\}^p
-\{ d\theta_i\}^{p+1}
\Big) .
\end{align*}
Plugging  these  formulae into  (\ref{E:subidon}) and using  (\ref{E:LCDisCOOL}),
one finally obtains  (\ref{E:rel-theta1}) and concludes in this way the proof of  Proposition \ref{P:TECHmain}.
\qed

 \begin{remark}
  Note that the condition $n\geq 3$ is only used at the very end of the proof to obtain
  the relations (\ref{E:subidon}) which imply rather straight-forwardly the relations (\ref{E:rel-theta1}).
  Except for these last lines, all the arguments above
  are valid in  dimension two.
 \end{remark}

\subsubsection{Upper bound for the rank I: technical lemmata}

To prove  that the rank of $PB_{\mathcal W}$ is two it suffices to show that
\[
\dim \mathrm{Span} \left\langle Z_*(x,t), \, \frac{\partial Z_*(x,t)}{\partial t}\, ,\frac{\partial Z_*(x,t)}{\partial x_1}, \, \ldots, \,
\frac{\partial Z_*(x,t)}{\partial x_n} \right\rangle \le 3 \, .
\]

Since the focus now, after Lemma \ref{L:lower}, is on points outside the hypersurfaces $\{ t = \theta_i(x) \}_{i \in \underline k}$, it is harmless
to replace $Z_*$ by
\begin{align*}
Z : ( \mathbb C^n,0) \times \mathbb C & \dashrightarrow \mathbb C^{\pi} \\
(x,t) & \longmapsto \sum_{i=1}^k \frac{Z_i(x)}{t - \theta_i(x)} \, .
\end{align*}
Notice that $Z_*(x,t) = P(x,t) \cdot Z(x,t)$. Hence $Z$ is also a lift of $PB_{\mathcal W}$.
Notice also that the map $Z$
has poles at the hypersurfaces $\{ t = \theta_i(x) \}_{i \in \underline k}$. That is why  $PB_{\mathcal W}$ was not defined as the projectivization of $Z$ from the beginning.

\smallskip
The following conventions will be used: the usual exterior derivative on $(\mathbb C^n,0) \times \mathbb P^1$
will be denoted by $\underline{d}$, while the exterior differential on $(\mathbb C^n,0)$ will denoted by $d$.
To clarify: $\underline{d}F=dF+(\frac{\partial F}{\partial t})dt$ for every germ of holomorphic function $F$  on  ${(\mathbb C^n,0) \times \mathbb P^1}$.

For further reference, observe that the $\mathbb C^\pi$-valued 1-form $dZ$ can written as
\begin{align}
\label{E:dZ}
 dZ=  \, \sum_{i=1}^k ({t-\theta_i})^{-1} dZ_i+
\sum_{i=1}^k ({t-\theta_i})^{-2} Z_i \,d\theta_i .
\end{align}

The following simple lemma will prove to be  useful later.

\begin{lemma}
\label{LEMclef1}
For $l=0,\ldots,n$ and $L=0,\ldots,n+1$,
respectively, the following identities hold true:
\begin{align}
 \sum_{i=1}^k \frac{ Z_i (\theta_i)^l }{t-\theta_i}  =   t^l\, Z 
\label{E:tlZ}
 \qquad \mbox{and }\qquad 
 \sum_{i=1}^k \frac{ Z_i\,(\theta_i)^L }{ {(t-\theta_i)}^{2} } =  L\,t^{L-1}\,Z-t^L\,
({\partial  Z}/{\partial t}).
\end{align}
%\begin{align}
 %\sum_{i=1}^k \frac{ Z_i (\theta_i)^l }{t-\theta_i}  &= \,  t^l\, Z \, ;
%\label{E:tlZ}\\
% \sum_{i=1}^k \frac{ Z_i\,(\theta_i)^L }{ {(t-\theta_i)}^{2} } &=\,  L\,t^{L-1}\,Z-t^L\,
%({\partial  Z}/{\partial t})
%\nonumber %\label{E:tlZ'}
%\; .
%\end{align}
\end{lemma}
\begin{proof}
Both identities are proved by induction.
Notice that they both are  trivially true  for $l=0$ and $L=0$.

Assume the first identity holds true for  $l < n $, and write
\begin{align*}
\sum_{i=1}^k { \frac{ Z_i\,(\theta_i)^{l+1}}{(t-\theta_i)}}=
\sum_{i=1}^k {\frac{ Z_i\,(\theta_i)^l \,\big((\theta_i-t)+t\big)}{(t-\theta_i)} } 
= 
-\sum_{i=1}^k  Z_i\,(\theta_i)^l +t\sum_{i=1}^k \frac{ Z_i\,(\theta_i)^l}{ (t-\theta_i)}. 
\end{align*}
Observe that   $\sum_{i}  Z_i\,(\theta_i)^l=0$ according to Equation (\ref{E:AReqNORM}).
Using this observation together with  the induction hypothesis, it follows that
 \[
 \sum_{i=1}^k  {(t-\theta_i)}^{-1} Z_i\,\theta_i^{l+1} =t^{l+1}\,Z
 \]
  as wanted.
\smallskip

Now, assume the second identity holds true for  $L\leq n$. Using the same trick as above, one obtains
\begin{align*}
\sum_{i=1}^k  \frac{Z_i\,(\theta_i)^{L+1}}{{(t-\theta_i)}^{2} } =\sum_{i=1}^k  \frac{Z_i\,(\theta_i)^L}{(t-\theta_i)}   +
t\sum_{i=1}^k { \frac{Z_i\,(\theta_i)^L}{(t-\theta_i)^2}}  \; .
\end{align*}
The first identity implies that the first summand of the righthand side is $t^L\,Z$. The induction hypothesis
implies that the second summand is $t \Big(L\,t^{L-1}\,Z-t^L\,({\partial  Z}/{\partial t})\Big)$.
 The lemma follows.
\end{proof}

\begin{lemma}\label{L:Ip}
If the identities (\ref{E:rel-z1}) and  (\ref{E:rel-theta1}) of Proposition \ref{P:TECHmain} hold true ( in particular if $n> 2$ ) then, for $p= 0 , \ldots, n-2$,
there are functions $F_p,G_p\in \mathcal O(\mathbb C^n,0)[t]$ such that
\begin{equation*}
 \{dZ\}^{p+1}-t\{dZ\}^p= F_p Z+ G_p ({\partial Z}/{\partial t} ) \, .
\end{equation*}
\end{lemma}
\begin{proof}
Let $p$ be fixed and notice that equation  (\ref{E:dZ}) implies
\begin{equation*}
\{d Z\}^p =\sum_{i=1}^k
 \frac{ \{d  Z_{i}\}^p}{t-\theta_i}
+  \sum_{i=1}^k \frac{  Z_i \,\{d\theta_i\}^p}{  {(t-\theta_i)}^2 }\;.
\end{equation*}

Decompose  $I_p= \{dZ\}^{p+1}-t\{dZ\}^p$ as  $  K_p+L_p$, where
\begin{align*}
 K_p= \sum_{i=1}^k  \, \frac{  \{d Z_i\}^{p+1} -
t \,\{d Z_i\}^{p} }{ t- \theta_i  }   \quad 
\mbox{ and } \quad
L_p= \sum_ {i=1}^k \,\,  \frac{  \big(\{d \theta_i\}^{p+1} -
t \,\{d \theta_i\}^{p} \big) \, Z_i  }{ (t-\theta_i)^{2}}  . \qquad
\end{align*}
Replacing $t$ by $(t-\theta_\alpha)+\theta_\alpha$ in the numerator  of   $K^p$  gives
\begin{align*}
 K_p=& \, \sum_{i=1}^k (t-\theta_i)^{-1} \, \Big( \{d Z_i\}^{p+1} -
\theta_i \,\{d Z_i\}^{p} \Big) +
\sum_{i=1}^k\{d Z_i\}^{p}. 
\end{align*}
According to (\ref{E:AReqNORM})  $\sum_i {Z_i}=0$, consequently
\begin{align}
\label{E:Kp}
 K_p= \,\sum_{i=1}^k  \, \frac{ \{d Z_i\}^{p+1} -
\theta_i \,\{d Z_i\}^{p}}{t-\theta_i}.
\end{align}

In exactly the same way, one  proves that
\begin{align*}
 L_p=&\sum_{i=1}^k  \frac{
\Big( \{d \theta_i\}^{p+1} -
\theta_i \,\{d \theta_i\}^{p} \Big) Z_i }{ (t-\theta_i)^{2}} \nonumber +
\sum_{i=1}^k \frac{ \,Z_i \,\{d \theta_i\}^{p}}{t-\theta_i} \; .
\end{align*}

\medskip

In what follows,
 $A \equiv B$ if and only if \label{Page:equiv}
$A-B$ is equal to  $F Z+ G({\partial Z}/{\partial t}) $ for suitable $F, G \in \mathcal O(\mathbb C^n,0)[t]$.
Notice that the lemma is equivalent to  $I^p\equiv 0$.

The outcome of Proposition \ref{P:TECHmain}, more specifically equation (\ref{E:rel-theta1}), implies
\begin{align*}
\sum_{i=1}^k \frac{\Big( \{d \theta_i\}^{p+1} -
\theta_i \,\{d \theta_i\}^{p} \Big) Z_i}{(t-\theta_i)^{2}}
=  \,
\sum_{i=1}^k
\frac{Z_i}{
(t-\theta_i)^{2}}\, \Big( \sum_{\rho=0}^{n}
(\theta_i)^\rho
 N^p_{\rho}  \Big) \, .
\end{align*}
Lemma \ref{LEMclef1} in its turn, implies
\[
 \sum_{i=1}^k
\frac{Z_i}{
(t-\theta_i)^{2}}\, \Big( \sum_{\rho=0}^{n}
(\theta_i)^\rho
 N^p_{\rho}  \Big) =   \sum_{\rho=0}^{n}    N^p_{\rho}   \Big( \sum_{i=1}^k
\frac{Z_i(\theta_i)^\rho}{(t-\theta_i)^{2}}
 \Big) \equiv 0.
\]
Therefore
\begin{align}
\label{Lpequiv}
L_p \equiv   \sum_{i=1}^k   \frac { Z_i \,\{d \theta_i\}^{p} }{t-\theta_i} .
\end{align}

Combining (\ref{E:Kp}) and (\ref{Lpequiv}), one obtains
\begin{align*}
I_p \equiv
\sum_{i=1}^k  \frac{
\{dZ_i\}^{p+1}-\theta_i \{dZ_i\}^{p}-Z_i \{d\theta_i\}^{p}}{t-\theta_i} \, .
\end{align*}
Equation  (\ref{E:rel-z1}) from Proposition \ref{P:TECHmain} together with equation (\ref{E:tlZ}) from Lemma \ref{LEMclef1}, allow
to conclude:
\begin{align*}
I_p \equiv & \,
\sum_{i=1}^k \frac{ \Big(
 \sum_{r=0}^{n-1} (\theta_i)^r M^p_{r}
\Big)  Z_i }{t-\theta_i} \\
\equiv & \,
\sum_{r=0}^{n-1} M^p_{\rho}
\Big(\sum_{i=1}^k \frac{ Z_i(\theta_i)^r } { t-\theta_i}
\Big)\equiv
\Big( \sum_{r=0}^{n-1} M^p_{r} t^r\Big) Z \equiv 0. \qedhere
\end{align*}
\end{proof}

\subsubsection{Upper bound for the rank II: conclusion}

\begin{prop}
\label{DIFFZ}
 There are 1-forms $\Omega$ and $\Gamma$ on $(\mathbb C^n,0) \times \mathbb P^1$ such that
\begin{equation}
\label{E:diffdZ}
 \underline{d} Z = \{dZ \}^0\left(  \sum_{p=0}^{n-1} t^p \varpi_p \right) \ + Z\Omega +(\partial Z/\partial t)\,(\Gamma+dt).
\end{equation}
\end{prop}

\begin{proof}
Lemma \ref{L:Ip} implies, for every $p=0,\ldots,n-1$,
$$
\{ dZ\}^p=t^{p} \{ dZ\}^0+\sum_{q=0}^{p}   I_p \,t^{p-q-1} \, .
$$
If $\Pi = \sum_{p=0}^{n-1} t^p \varpi_p$ then
\begin{align*}
dZ=\sum_{p=0}^{n-1} \{ dZ\}^p \varpi_p=& \, \sum_{p=0}^{n-1}
\Big(
t^{p}\,\{dZ\}^0+\sum_{q=0}^{p-1} I^{q}\, t^{p-1-q}
\Big) \,
 \varpi_p
 \\
=& \, \{dZ\}^0 \Pi
+   \sum_{p=0}^{n-1}
 \sum_{q=0}^{p-1} I^q\, t^{p-1-q}\,
 \varpi_p .
\end{align*}

Lemma  \ref{L:Ip} says that  $ I^q\equiv 0$ (~see the definition of $\equiv$ in page \pageref{Page:equiv}~) for every $q=0,\ldots,n-1$.  The existence of   two 1-forms
$\Omega$ and  $\Gamma$  satisfying
$$ dZ= \{dZ\}^0 \Pi+ Z\,\Omega+ ({\partial   Z}/{\partial t})\,\Gamma  $$
follows.
Moreover, the coefficients of $1$-forms $\Omega$ and  $\Gamma$ in the basis
$(\varpi_0,\ldots,\varpi_{n-1})$ are polynomials in $t$
with  holomorphic functions on $(\mathbb C^n,0)$ as coefficients. Since   $\underline{d}Z=dZ+({\partial   Z}/{\partial t})
\,dt$, the proposition follows.  \end{proof}

Proposition \ref{DIFFZ} clearly implies that $PB_{\mathcal W}$ has rank at most two at every $(x,t) \in (\mathbb C^n,0)\times \mathbb P^1$. But  Lemma \ref{L:lower}
says  it must be at least two. Hence Proposition \ref{PBWrang} follows.

\section{Poincar\'{e}-Blaschke's surface}

The \defi[Poincar\'{e}-Blaschke's surface] $X_{\mathcal W}$ of $\mathcal W$ is the image of
 its Poincar\'{e}-Blaschke's map $PB_{\mathcal W}$. That is,  $$X_{\mathcal W}={\rm Im}\, P\!B_{\mathcal W} \subset \mathbb P^{\pi-1}\,. $$
According to Proposition \ref{PBWrang}  it is a germ of smooth complex surface on $(\mathbb P^{\pi-1},\mathscr C(0))$.
 It is clearly non-degenerate. The remarks laid down at the end of  Section  \ref{S:defPBW}
 imply that $X_{\mathcal W}$ is canonically attached to the pair $(\mathcal W, \A2W)$ modulo
 projective transformations.  \smallskip

\subsection{Rational normal curves everywhere}

Notice that  $X_{\mathcal W}$ contains all the canonical curves  $C_i$ of $\mathcal W$.
It also contains a lot of rational  curves according to Lemma \ref{L:CRN0}: the
curves $\mathscr C(x)$.
Exploiting  the geometry of this family  of rational curves, it will be possible
to prove that $X_{\mathcal W}$ is the germification at $\mathscr C(0)$ of  a rational surface.  \smallskip

\begin{lemma}
\label{L:CRN}
The following assertions are verified:
\begin{enumerate}
\item[(a)] for every  $i \in \underline k $, the curve  ${\mathscr C}(x)$ intersects the canonical curve  $C_i$ transversely at $\kappa_i(x)$;
\item[(b)] for every  $x,y \in (\mathbb C^n,0)$,  the curve $\mathscr C(x)$ coincides with $\mathscr C(y)$ if and only if $x=y$;
\item[(c)] for every subset $J \subset \underline k$ of cardinality $n$, and every set  ${\mathcal P =\{ p_j \in C_j \, | j \in J\}}$,  there exists a unique
$x \in (\mathbb C^n,0)$ such that $\mathscr C(x)$ contains $\mathcal P$.
\end{enumerate}
\end{lemma}
\begin{proof}
The first assertion follows from the proof of Lemma \ref{L:lower}. It is also
 a direct consequence of the  expression  (\ref{E:diffdZ}) for  $\underline{d}Z$.

To prove the third assertion, notice that for every $j \in J$, $\kappa_j^{-1}(p_j)$ is
a leaf $L_j$ of $\mathcal F_j$. Since $\mathcal W$ is smooth $\cap_{j\in J} L_j$ must reduce to a point $x \in (\mathbb C^n,0)$
The assertion follows.  \smallskip

The second assertion follows immediately from the third.
\end{proof}

Lemma \ref{L:CRN} implies that the surface $X_{\mathcal W}$  is the union of rational normal curves
of degree $k-n-1$ belonging to the $n$-dimensional holomorphic family  $\{\mathscr C(x)\}_{x \in (\mathbb C^n,0)}$.

\subsection{Algebraization I : $X_{\mathcal W}$ is algebraic}

\begin{lemma}
\label{L:intn-1}
 For every  $x,y \in (\mathbb C^n,0)$, the intersection number   $\mathscr C(x) \cdot \mathscr C(y)$ is equal to $n-1$.
\end{lemma}
\begin{proof}
  Let   $B \subset \underline k$  be a subset of cardinality $n-1$. Suppose
  $y$ belongs to $\cap_{i\in B} L_i(x)$, $L_i(x)$ being the leaf of $\mathcal F_i$ through $x$, but    is not equal to $x$.
  After item (c) of Lemma \ref{L:CRN},
  the curves $\mathscr C(x)$ and $\mathscr C(y)$ are distinct. Anyway,  they share at least $n-1$ points
in common: $p_i=\kappa_i(x) =\kappa_i(y)$ for $i \in B$. Therefore $\mathscr C(x) \cdot \mathscr C(y) \ge n-1$.

If $\mathscr C(x) \cdot \mathscr C(y) \geq n$ then either  there exists $p\not \in \{p_i\}_{ i \in B}$ such that $p\in \mathscr C(x)\cap \mathscr C(y)$;
or there exists  $p \in \{p_i\}_{ i \in B}$ for which  $\mathscr C(x)$ and $\mathscr C({y})$ are tangent
at $p$.
But $n$ points, or $n-1$ points and one tangent, on a rational normal curve
span a projective subspace of dimension $n-1$. This contradicts Corollary
 \ref{C:int=n-2}.

 To conclude it suffices to observe that any two curves in the family $\{ \mathscr C(x)\}_{x \in (\mathbb C^n,0)}$ are
 homologous.
\end{proof}

\begin{prop}\label{P:bewildered}
Let  $(S_0,C)\subset \mathbb P^N$ be a germ of  smooth surface along a  connected projective  curve $C\subset S_0$. If $C^2>0$
then $S_0$ is contained in a projective surface $S\subset \mathbb P^N$.
\end{prop}
\begin{proof}
Let $\mathbb C(S_0)$ be the field of meromorphic functions on $S_0$.
To prove that $S_0$ is contained in a projective surface, it suffices to
show that the transcendence degree of $\mathbb C(S_0)$ over $\mathbb C$ is two.
Clearly,  since the  restriction  at $S_0$ of any two generic rational functions
on $\mathbb P^N$ are algebraically independent,
it suffices to assume that
the polar set of both does not contain $S_0$ and that their level sets are
generically transverse to obtain that $
{\mathrm{tr}_{\deg}[\mathbb C(S_0) : \mathbb C] \geq  2}$
.

The hypothesis  $C^2 > 0$ is equivalent to the ampleness of the normal
bundle $N_{C/S_0}$. Thus   \cite[Theorem 6.7]{hartshorneCD} implies that
${\mathrm{tr}_{\deg}[\mathbb C(S_0) : \mathbb C] \le 2}$.  \smallskip

Alternatively, it is also possible to apply a Theorem of Andreotti: since any representative of $S_0$ contains a curve of positive
self-intersection, it  also  contains a  pseudo-concave open subset. Hence,
 \cite[Th\'{e}or\`{e}me 6]{andreotti} implies $\mathrm{tr}_{\deg}[\mathbb C(S_0) : \mathbb C] \le 2$.
\end{proof}

The preceding proposition together with Lemma \ref{L:intn-1}   have the following consequence.
\begin{cor}
\label{C:SW}
Poincar\'{e}-Blaschke's surface  $X_{\mathcal W}$ is contained
in an irreducible non-degenerate projective  surface  $S_{\mathcal W}\subset \mathbb P^{\pi-1}$.
\end{cor}

\begin{remark}Notice, that nothing is said about the smoothness of  $S_{\mathcal W}$. A priori, it could even
happen that the germ of $S_{\mathcal W}$ along $\mathscr C(0)$ is singular. Of course, this would happen if  and only
if the germ of $S_{\mathcal W}$ along $\mathscr C(0)$ has other irreducible components besides $X_{\mathcal W}$.
The only thing clear
is that $S_{\mathcal W}$ contains the smooth surface $X_{\mathcal W}$.  \smallskip
\end{remark}

\begin{remark}Those bewildered with the use
of the results of Hartshorne or Andreotti in the
proof of Proposition \ref{P:bewildered}, might fell relieved
by knowing that Corollary \ref{C:SW} can be proved by rather elementary
means, which are sketched below.

Let $x_0$ be an arbitrary point of $\mathscr C(0)$ and consider the
subset $\mathfrak X$ of $\mathrm{Mor}_{k-n-1}(\mathbb P^1, \mathbb
P^{\pi -1})$\footnote{This is just the set of morphisms from $\mathbb
P^1$ to $\mathbb P^n$ of degree $k-n-1$ which can be naturally
identified with a Zariski open subset of $\mathbb P\left( \mathbb
C_{k-n-1}[s,t]^{\pi}\right)$. } consisting of morphisms $\phi$
which map $(\mathbb P^1,(0:1))$  to $(X_{\mathcal W},x_0)$.
Recall that the ring of formal power series in any number
variables is Noetherian.
If $\mathcal I \subset  \mathbb C [[x_1, \ldots, x_{\pi-1}]]$ is the ideal
defining $(X_{\mathcal W},x_0) $ then, expanding  formally $f(\phi(t:1))$ for
every defining equation $f \in \mathcal I$  one deduces that
$\mathfrak X$ is algebraic.

To conclude, one has just to prove that
 the natural projection from $\mathrm{Mor}_{k-n-1}(\mathbb P^1, \mathbb
P^{\pi -1})$   to $\mathbb P^{\pi -1}$ --- the evaluation morphism  --- sends
$\mathfrak X$ onto a surface $S_{\mathcal W}$ of $\mathbb P^{\pi -1}$
containing $X_{\mathcal W}$.
\end{remark}

\subsection{Algebraization II and conclusion}
Since the projective surface  $S_{\mathcal W}$ can be singular,
it will be replaced by one of its  desingularizations. It can  be assumed
that the chosen  desingularization contains an isomorphic copy of  $(X_{\mathcal W},\mathscr C(0))$.
Notice also that every desingularization of a singular projective surface is still projective.
To keep the notation simple, this desingularization will still be denoted
by $S_{\mathcal W}$.

\begin{prop}
There are no holomorphic $1$-forms on $S_{\mathcal W}$, that is  $$h^0(S_{\mathcal W},\Omega^1_{S_{\mathcal W}})=0.$$
\end{prop}
\begin{proof}
Let  $\xi$  be a holomorphic $1$-form on $S_{\mathcal W}$. If non-zero then it defines a foliation $\mathcal F_{\xi}$ on $S_{\mathcal W}$.
Since smooth rational curves have no holomorphic $1$-forms, the pull-back of $\xi$ to $\mathscr C(x)$ must vanish
for every $x \in (\mathbb C^n,0)$. Therefore these curves are invariant by $\mathcal F_{\xi}$. But a foliation
on a surface cannot have an $n$-dimensional family of pairwise distinct leaves. This contradiction shows that
$\xi$ is identically zero.
\end{proof}

Hodge theory implies $H^1(S_{\mathcal W},\mathcal O_{S_\mathcal W})$ is also trivial. Therefore from
the exponential sequence
\[
0 \to \mathbb Z \longrightarrow \mathcal O_S \longrightarrow \mathcal O_S^* \to 0
\]
one deduces that the Chern class morphism
\[
H^1(S_{\mathcal W}, \mathcal O_{S_{\mathcal W}}^*) \longrightarrow H^2(S_{\mathcal W}, \mathbb Z)
\]
is injective. Consequently, two projective curves in $S_{\mathcal W}$ are linearly equivalent
if and only if they are homologous.

\begin{remark}
The surface $S_{\mathcal W}$ is  rational. To see it,
take $n-1$ points in $\mathscr C(0)$, say $\kappa_1(0), \ldots, \kappa_{n-1}(0)$. If $L_i$ is a leaf of the foliation
$\mathcal F_i$  through the origin then $\cap_{i\in \underline{n-1}} L_i$ is a curve $C$ in $(\mathbb C^n,0)$. By construction, for every $x \in C$ the
curve $\mathscr C(x)$ intersects $\mathscr C(0)$ at $\kappa_1(0), \ldots, \kappa_{n-1}(0)$. Thus blowing up at these $n-1$ points,
 one obtains a surface $S$ containing a family parametrized by $(\mathbb C,0)$, of rational curves of self-intersection zero: the strict transforms of $\mathscr C(x)$ for $x \in C$.
 Any two of these curves are linearly equivalent, therefore there exists a non-constant holomorphic  map $F: S \dashrightarrow \mathbb P^1$ sending
 all of them to points. Thus $S$ is a rational fibration over $\mathbb P^1$,  hence a  rational surface.
 \end{remark}

\begin{thm}
 The web $\mathcal W$ is linearizable.
\end{thm}
\begin{proof}
Recall that any two curves in the family $\{\mathscr C(x)\}$ are homologous, and consequently  linearly equivalent. Thus
 they all
belong to the complete linear system $$|\mathscr C(0)|=\mathbb P H^0\big(S_{\mathcal W},\mathcal O_{S_{\mathcal W}}(\mathscr C(0) ) \big).$$

Tensoring the standard exact sequence
\[
0\to \mathcal O_{S_{\mathcal W}}(- \mathscr C(0)) \longrightarrow \mathcal O_{S_{\mathcal W}}
\longrightarrow \mathcal O_{\mathscr C(0)} \to 0
\]
by $\mathcal O_{S_{\mathcal W}}(\mathscr C(0))$, one obtains
\begin{equation}
\label{E:see}
0\to \mathcal O_{S_{\mathcal W}} \longrightarrow \mathcal O_{S_{\mathcal W}} ( \mathscr C(0))
\longrightarrow \mathcal O_{\mathscr C(0)}(\mathscr C(0)) \to 0 \, .
\end{equation}

Notice that
\begin{equation*}
%\label{E:ii}
 \deg \mathcal O_{\mathscr C(0)}(\mathscr C(0)) = \mathscr C(0)^2 = n-1\, .
\end{equation*}
Consequently $ \mathcal O_{\mathscr C(0)}(\mathscr C(0))) \cong \mathcal O_{\mathbb P^1}(n-1)$.

\medskip

Since   $H^1(S_{\mathcal W}, \mathcal O_{S_{\mathcal W}})\simeq H^0(S_{\mathcal W}, \Omega^1_{S_{\mathcal W}}) = 0$, it follows from
 (\ref{E:see}) that
\begin{align*}
\label{E:see2}
h^0(S_{\mathcal W}, \mathcal O_{\mathcal S_{\mathcal W}}(\mathscr C(0))  &= h^0(S_{\mathcal W}, \mathcal O_{\mathcal S_{\mathcal W}})
+ h^0(S_{\mathcal W}, \mathcal O_{\mathscr C(0)}(\mathscr C(0))) \\ &= 1 + n\, .
\end{align*}
Therefore  $|\mathscr C(0)|\simeq \mathbb P^{n}$.  According to Lemma \ref{L:CRN} item (b), the map
\begin{align*}
 \mathfrak C: \,(\mathbb C^n,0) & \longrightarrow |\mathscr C(0)|\simeq \mathbb P^{n} \\
x & \longmapsto \mathscr C(x)
\end{align*}
is  injective.  Moreover, there exists a factorization
$$  \xymatrix@R=0.9cm@C=0.9cm{
(\mathbb C^n,0)
 \ar[rr]^{\mathfrak C\qquad } \ar[rrd]_{P_{\mathcal W}} &&    |\mathscr C(0)| \simeq \mathbb P^{n}
\ar[d]^{\langle\; \rangle}
\\
 & & \mathrm{Grass}(\A2W^*, k-n)
} $$
where  $\langle\; \rangle$ is the map that associates to the curve $\mathscr C(x)$ its projective span $\langle \mathscr C(x) \rangle$.
Since $P_{\mathcal W}$ is an immersion  (Proposition \ref{P:pmii}), so is  $\mathfrak C$.
  The image by
$\mathfrak C$ of $L_i(x)$, the leaf of  $\mathcal F_i$ through $x \in (\mathbb C^n,0)$, is nothing more than the elements of  $|\mathscr C(0)| $ passing through  $\kappa_i(x) \in X_{\mathcal W} \subset S_{\mathcal W}$.
Because $|\mathscr C(0)| $ is a linear system, $\mathfrak C\big(L_i(x)\big)$ is a  hyperplane in $ \mathbb P^{n}$. Therefore
$\mathfrak C$ is a germ of biholomorphism which linearizes $\mathcal W$.
\end{proof}

It suffices to apply the algebraization  theorem for linear webs to conclude
that   $\mathcal W$ is algebraizable.

\chapter{Exceptional Webs}\label{Chapter:6}
\thispagestyle{empty}

Taking the risk of being anticlimactic, this last chapter is devoted to planar webs
of maximal rank. More specifically, a survey of the currently state of the art concerning
exceptional planar webs is presented.\smallskip 

In  Section \ref{S:lliinn}  a criterium
to  decide whether or not a given planar web is linearizable is presented. The criterium is phrased in terms of a projective connection adapted
to the web  under study. The lack of conciseness of the presentation carried out here is balanced by the geometric intuition
it may help to build. As corollaries, the existence of exceptional webs and an algebraization result for  planar webs, are obtained
in Section \ref{S:firstexceptional} and Section \ref{S:algdim2}  respectively.

\smallskip

Section \ref{S:iinnff} deals with planar webs with infinitesimal automorphisms following  \cite{MPP}.
Using the result on the structure of the space of abelian relations of a web $\mathcal W$ carrying
a transverse infinitesimal automorphism $v$ laid down earlier in Section \ref{S:infinitesimalggg} of Chapter \ref{Chapter:AR}, the rank
of the web obtained from  the superposition of $\mathcal W$ and the foliation $\mathcal F_v$ determined by $v$ , is
computed as a function of the rank of $\mathcal W$. As a corollary, the existence of
exceptional $k$-webs for arbitrary $k \ge 5$ is settled. Some place is taken to recall some
basic definitions from differential algebra, and make more precise a problem posed in \cite{MPP} that the authors
think has some interest.

\smallskip

Section \ref{S:Pantazi-Henaut} starts with  basic facts from Cartan-Goldsmith-Spencer theory on differential linear systems.
This theory is applied to the differential system which solutions are the abelian relations of a given planar web, in order to
obtain a computational criterium which decides whether or not  the web under study has maximal rank. The approach followed there is
a mix of the classical one  by Pantazi with the more recent by H\'{e}naut. The necessary criterium for the maximality of the rank
by  Mih\u{a}ileanu is briefly discussed without  proof.

\smallskip

In  Section \ref{S:CDQL} some recent classification results obtained by the authors are  stated, and the proof
of one of them is outlined.

\smallskip

Finally in Section \ref{S:Bestiarium} all the  exceptional webs known up-to-date, to the best of the authors
knowledge,  are collected.

\section{Criterium for linearization}\label{S:lliinn}

Throughout this section  $\mathcal W=\mathcal W(\omega_1, \ldots, \omega_k) $ is a germ of smooth
$k$-web on  $(\mathbb C^2,0)$. For $i\in \underline{k}$, let $v_i\in T(\mathbb C^2,0) $ be a germ of
 smooth vector field defining the same foliation as $\omega_i$. Thus $\omega_i(v_i)=0$ and $v_i(0)\neq 0$.

\subsection{Characterization of linear webs}

Recall that a web     $\mathcal W$ is \defi[linear]
\index{Web! linear} if all its leaves are contained in affine lines of  $\mathbb C^2$.
Recall also that  $\mathcal W$ is \defi[linearizable]  if it is equivalent to a linear web. \smallskip

It is rather simple to characterize the linear webs. It suffices to notice that a curve $C \subset (\mathbb C^2,0)$ is contained
in a line, if and only if for any parametrization $\gamma : (\mathbb C,0) \to C$  the following identity holds
true
\[
\gamma'(t) \wedge \gamma''(t) = 0 \, , \quad \text{ for every } t \in (\mathbb C,0) \, .
\]
Therefore, every orbit of a vector field $v$ will be linear, if and only if, the determinant
\[
\det \left( \begin{array}{cc}
              v(x)& v(y) \\
              v^2(x)&v^2(y)
            \end{array} \right)
\]
vanishes identically. Indeed,  if $\gamma(t)= (\gamma_1(t), \gamma_2(t)) $ satisfies $v(\gamma(t)) = \gamma'(t)$ then
\[
\det \left( \begin{array}{cc}
              v(x) & v(y) \\
              v^2(x)& v^2(y)
                          \end{array} \right) \big( \gamma(t) \big)  = \det \left( \begin{array}{cc}
              \gamma_1'(t) & \gamma_2'(t) \\
              \gamma_1''(t)& \gamma_2''(t)
            \end{array} \right) \, .
\]

Therefore   $\mathcal W$ is linear if and only if
\begin{equation*}
\det \left( \begin{array}{cc}
              v_i(x)& v_i(y) \\
              v_i^2(x)&v_i^2(y)
            \end{array} \right) = 0
\end{equation*}
for every  $i \in \underline k$.

\dd

To characterize linearizable webs one is naturally lead to less elementary considerations. Below, the approach laid down in Section \S 27 of Blaschke-Bol's book \cite{BB}
is presented in a modern language. For more recent references see \cite{Henaut1993}, \cite{AGL} and \cite{PIRIOLIN}. \smallskip

The inherent difficulty in obtaining an analytic criterium characterizing linearizable webs
comes from the fact that the usual method to treat linearization problems comes from differential geometry and
is well adapted to  deal with one-dimensional families of foliations on $(\mathbb C^2,0)$.
But a web is not builded from a continuous family of foliations but by a finite number of them.
It is the contrast between the finiteness and continuity that makes the linearization of webs a non-trivial question.
For instance, despite many efforts spreaded over time, Gronwall's conjecture is still unsettled:

\begin{conjecture}[Gronwall's conjecture]
If $\mathcal W$ and $\mathcal W'$ are  germs of linear $3$-webs on $(\mathbb C^2,0)$ with non-vanishing curvature and
$\varphi : (\mathbb C^2,0) \to (\mathbb C^2,0)$ is a germ of biholomorphism sending $\mathcal W$ to $\mathcal W'$ then
$\varphi$ is the germification of a projective automorphism of $\mathbb P^2$. In other words, a non-hexagonal $3$-web
admits at most one linearization.
\end{conjecture}

In sharp contrast, the equivalent statement for planar $k$-webs with $k\ge 4$, is true as will be explained below.
The  point is that it is possible  interpolate the defining foliations by a \emph{unique} second order differential equation cubic
in the first derivative when $k=  4$.
When $k=3$, although possible to interpolate as for $k = 4$, the lack of uniqueness adds an additional  layer of difficulty to the problem.

\subsection{Affine and projective connections}

A germ of holomorphic  \defi[affine connection]  \index{Affine connection} on $(\mathbb C^2,0)$ is a
 map
  $$\nabla: T (\mathbb C^2,0) \rightarrow \Omega^1(\mathbb C^2,0)  \otimes T(\mathbb C^2,0)   $$
   satisfying
   \begin{enumerate}
   \item[(1)]  $\nabla(\zeta+\zeta')=\nabla(\zeta)+\nabla(\zeta')$;
   \item[(2)] $\nabla(f\zeta)=f\nabla(\zeta)+df\otimes \zeta$;
    \end{enumerate}
for every $\zeta,\zeta'\in T(\mathbb C^2,0)$ and every   $f \in \mathcal O(\mathbb C^2,0) $.

\smallskip

Beware that the map $\nabla$ is not $\mathcal O(\mathbb C^2,0)$-linear. In particular, the image of  a vector field
$\zeta$ is a tensor which at a given point $p \in (\mathbb C^2,0)$ is   determined by the whole germ of $\zeta$ at $p$
and not only by its  value at $p$.

\medskip

If   $\chi=(\chi_0,\chi_1)$ is a holomorphic frame on  $(\mathbb C^2,0) $ and  $\omega=(\omega_0,\omega_1)$ is the dual coframe then
 $\nabla$ is completely determined by its  \defi[Christoffel's symbols] \index{Christoffel's symbols} $\Gamma^k_{ij}$ (relative to the frame $\chi$),
defined by the relations
$$ \qquad \qquad  \nabla(\chi_j)=\sum_{i,k=0}^1 \Gamma^k_{ij} \,\omega_i \otimes \chi_k  \qquad \mbox{ for }  j=0,1\,. $$

\smallskip

Although $\nabla$ when evaluated at a vector field $\zeta$  does depend on the germ of
$\zeta$ as explained above, if $C\subset (\mathbb C^2,0)$ is a curve and
$\zeta$ is a vector field  tangent to it then the pullback  of $\nabla(\zeta)$ at $C$ is completely determined by the restriction of $\zeta$ to $C$.
More precisely, $\nabla(\zeta)$ naturally determines a  section of $\Omega^1(C) \otimes T (\mathbb  C^2,0)$ which only depends
on the restriction of $\zeta$ to $C$.  Indeed, if
$\zeta,\zeta'$ and $\zeta''$ are vector fields such that
$\zeta - \zeta'= f \zeta''$ where   $f \in \mathcal O(\mathbb C^2,0)$
is a defining equation for $C$, then
\[
\nabla(\zeta) - \nabla(\zeta') = f\nabla(\zeta'')+df\otimes \zeta'' \,
\]
which clearly vanishes at $C$ when contracted with vector fields tangent to $C$.

\medskip

A smooth curve $C \subset (\mathbb C^2,0)$ is a \defi[geodesic] \index{Geodesic}of $\nabla$, if for every germ of vector field $\zeta \in T C \subset T(\mathbb C^2,0)$, the vector field
$\nabla_{\zeta} ( \zeta) :=\langle \nabla(\zeta),\zeta\rangle$ still belongs to $TC$. To wit, if  $\gamma:(\mathbb C,0) \rightarrow (\mathbb C^2,0)$
is a parametrization of $C$ then $C$ is a geodesic for $\nabla$ if and only if there exists
 a holomorphic 1-form   $\eta \in \Omega^1(\mathbb C,0)$ such that $\nabla(\gamma') =  \eta\otimes \gamma'$.
If  $\nabla(\gamma')$ vanishes identically  on  $(\mathbb C,0)$ then $\gamma$  is called a  \defi[geodesic parametrization] \index{Geodesic!parametrization} of
$C$.

\medskip

 Two affine connections on $(\mathbb C^2,0)$ are  \defi[projectively equivalent] \index{Affine connection!projectively equivalent}
 if they have exactly the same curves as geodesics. This defines a equivalence relation on the space of affine connections
 on $(\mathbb C^2,0)$. By  definition, a \defi[projective connection] \index{Projective connection} is an equivalence class of this equivalence relation. The
 class of an affine connection $\nabla$ will be denoted by $[\nabla]$.  A smooth curve  $C\subset (\mathbb C^2,0)$
 is a \defi[geodesic] of  $\Pi=[\nabla]$ if it is a geodesic of the affine connection  $\nabla$.  Of course, the definition does not depend on the
 representative $\nabla$ of $\Pi$.

\smallskip

Two projective connections are \defi[equivalent] if  there  exists a germ of biholomorphism  $\varphi$ sending the geodesics of
one into the geodesics of the other.
The  \defi[trivial  projective connection] $\Pi_0$ \index{Projective connection!trivial} is
the global projective connection on  $\mathbb P^2$ having as geodesics the projective lines. A projective connection   $\Pi$ is  \defi[flat] \index{Projective connection!flat} (or \defi[integrable]) \index{Projective connection!integrable} if it is equivalent to  $\Pi_0$.

\subsubsection{Projective connections and ordinary differential equations}

\begin{lemma}
\label{L:PI=coeffTHOMAS}
If  $\Pi$ is a projective  connection on $(\mathbb C^2,0)$ then there exists a unique affine connection  $\nabla$ on $(\mathbb C^2,0)$ such that
\begin{enumerate}
 \item the affine connection $\nabla$ is a representative of $\Pi$, that is  $[\nabla]=\Pi$;
\item the Christoffel's symbols $\Gamma_{ij}^k$ ($i,j,k=1,2$)  of  $\nabla$ relative to the  coframe  $(dx,dy)$ verify the relations
\begin{equation}
 \label{E:coeffTHOMAS}
\Gamma^k_{ij}=\Gamma^k_{ji} \qquad \mbox{ and  } \qquad  \Gamma^1_{1j}+\Gamma^2_{2j}=0
\end{equation}
for every $i,j,k=1,2$.
\end{enumerate}
\end{lemma}
\begin{proof}
Exercise for the reader.
\end{proof}

\smallskip

From now on, a projective connection $\Pi$, as well as its normalized representative $\nabla$
provided by the lemma above, will be fixed.

Let $\gamma: (\mathbb C,0)\rightarrow (\mathbb C^2,0)$ be a parametrization of a curve
 $C\subset (\mathbb C^2,0)$.  If one writes $\gamma(t)=(\gamma_1(t),\gamma_2(t))$ for  $t\in (\mathbb C,0)$ then
 it is a simple exercise to show the equivalence of the three assertions below:
\begin{enumerate}
 \item $C$ is a geodesic of  $\Pi$;
\item $C$ is a geodesic of  $\nabla$;
\item there exists a function $\varphi$ such that
\begin{align}
\label{E:secondORDER}
\frac{d^2 \gamma_k}{dt^2}
+\sum_{i,j=1}^2 \Gamma^k_{ij}\Big(\frac{d\gamma_i}{dt}\Big)\Big(\frac{d\gamma_j}{dt}\Big)
=\varphi \frac{d \gamma_k}{dt}\
\end{align}
for $k= 1, 2$.
\end{enumerate}

Under the additional hypothesis that  $C$ is tranverse to the vertical foliation $\{x=cst.\}$, the inequality  $\frac{d\gamma_1}{dt}(0)\neq 0$ holds true and,  modulo a change of variables, one can assume that  $\gamma_1(t)=t$. It is then
a simple matter to eliminate the function $\varphi$ in (\ref{E:secondORDER}) and deduce the following lemma which can be traced back to  Beltrami \cite{Beltrami}.

\begin{lemma}
The geodesics of  $\Pi$ transverse to  $\{x=cte.\}$ can be identified with the solutions of the  second order differential equation
\begin{equation}
   \label{E:equatABCD}
(\mathscr E_\Pi) \qquad\qquad  \frac{d^2y}{dx^2} = A \big(\frac{dy}{dx}\big)^3
+
B \big(\frac{dy}{dx}\big)^2
+
C\frac{dy}{dx}
+D\qquad\qquad
\end{equation}
where $A,B,C,D$ are expressed in function of the Christoffel's symbols  $\Gamma_{ij}^k$ of $\nabla$ as follows
\begin{equation}
 \label{E:ABCD0}
A=  \Gamma^1_{22} , \quad
B=   2\,\Gamma^1_{12}   -\Gamma^2_{22}, \quad
C= \Gamma^1_{11}   -2\,\Gamma^2_{12}, \quad
   D=-\Gamma^2_{11}.
\end{equation}
\end{lemma}

Combining equations  (\ref{E:coeffTHOMAS}) and (\ref{E:ABCD0}), it follows that the Christoffell's symbols  $\Gamma_{ij}^k$
of the normalized affine connection $\nabla$ can be expressed in terms of the functions  $A,B,C,D$. Taking into account
Lemma \ref{L:PI=coeffTHOMAS}, one deduces the following proposition.

\begin{prop}
\label{P:connexP=Ypp}
Once a coordinate system   $x,y$ is fixed, a projective connection can be identified with a second order differential equation of the form (\ref{E:equatABCD}).
\end{prop}

It is  evident, in  no matter which affine coordinate system  $x,y$ on $\mathbb C^2$, that the second order  differential equation  $(\mathscr E_{\Pi_0})$
associated to the trivial  projective connection  $\Pi_0$  is nothing more than  ${d^2y}/{dx^2}=0$.  Therefore, a projective connection is integrable if
and only if  the second order differential equation $(\mathscr E_{\Pi})$ in transformed into the {\it trivial} one, ${d^2Y}/{dX^2}=0$, through
 a point transformation  $(x,y)\mapsto (X(x,y),Y(x,y)).$ \medskip

The characterization of the second order differential equations equivalent to the   trivial
one stated below is due to  Liouville \cite{Liouville} and Tresse \cite{Tresse}.

\begin{thm}\label{T:L1L2}
The second order differential equation  $y''=f(x,y,y')$ is equivalent to the trivial equation ${d^2y}/{dx^2}=0$
through a point transformation
if and only if the function  $F=f(x,y,p)$ verifies
 $\frac{\partial^4 F}{\partial p^4}=0$ and
\[
D^2(F_{pp})-4 D(F_{yp})-3 F_y F_{pp}+6 F_{yy}+F_p\Big( 4F_{yp}-D(F_{pp})   \Big)=0
\]
where $D =
\frac{\partial }{\partial x}
+p
\frac{\partial }{\partial y}
+F
\frac{\partial }{\partial p}$.
\end{thm}

Once one restricts to equations of the form (\ref{E:equatABCD}), which are exactly the ones satisfying the first condition $\frac{\partial^4 F}{\partial p^4}=0$,
the second condition can be rephrased in more explicit terms.

 If
\begin{align*}
L_1= & \,
2\, C_{xy}-B_{xx}  -3\, D_{yy}-6\, DA_{x}-3\, AD_{x}    \\
&   + 3\, DB_{y} +3\,DB_{y} +CB_{x}-2\, CC_{y}\\
\mbox{ and } \quad  L_2= & \,
2\,\, B_{xy}-C_{yy}  -3\, A_{xx}+6\, AD_{y}+3\, DA_{y}    \\
&  - 3\, AC_{x} -3\,CA_{y} -BC_{y}+2\, BB_{x}
\end{align*}
then, according to \cite{bryant}, Liouville has shown that  the tensor
$$
L= \big(L_1 dx+L_2 dy \big)\otimes (dx\wedge dy)
$$
is invariant under point transformations and that the differential equation (\ref{E:equatABCD})
is equivalent to the trivial one if and only if $L$ is identically zero.

\subsection{Linearization of planar webs}
Let  $\mathcal W$ be a smooth $k$-web on $(\mathbb C^2,0)$. It is  \defi[compatible with the projective connection] \index{Web!compatible with projective connection} $\Pi$ if the
leaves of  $\mathcal W$ are geodesics of $\Pi$.  This is clearly a geometric property: if $\varphi$ is a biholomorphism then $\mathcal W$ is compatible with $\Pi=[\nabla]$ if and only if  $\varphi^*\mathcal W$ is compatible with  $\varphi^*\Pi=[\varphi^*\nabla]$.
\smallskip

\begin{lemma}
\label{L:LINtautology}
A web $\mathcal W$ is linear if and only if it is compatible with the trivial projective connection $\Pi_0$.
Consequently,  $\mathcal W$ is linearizable if and only if it is compatible with a flat projective connection.
\end{lemma}

It is on this tautology that the linearization criterium presented below is based.

\begin{prop}\label{P:616}
If   $\mathcal W$ is a smooth $4$-web on  $(\mathbb C^2,0)$ then
\begin{enumerate}
 \item[(a)] there is a unique projective connection $\Pi_{\mathcal W}$ compatible with  $\mathcal W$;
\item[(b)] the web  $\mathcal W$ is linearizable if and only if  $\Pi_{\mathcal W}$ is flat.
\end{enumerate}
\end{prop}
\begin{proof}
It is clear that  {(b)} follows from  {(a)} combined with  Lemma \ref{L:LINtautology}. \smallskip

To prove  {(a)}, it is harmless to assume the leaves of $\mathcal W$ transverse to the line $\{ x=0\}$. The foliation
defining  $\mathcal W$ are thus defined by vector fields  $v_i$ (for $i=1,2,3,4$) which can be written as $v_i=\frac{\partial}{\partial_x}+e_i\frac{\partial}{\partial_y}$ with $e_i\in \mathcal O(\mathbb C^2,0)$.  The orbits of the vector fields $v_i$ are solutions of the second order differential equation
\begin{equation}
   \label{E:equatABCDpp}
  y'' = A (y')^3+B (y')^2+Cy'+D
\end{equation}
if and only if
\begin{equation}\label{E:smlpo}
\qquad\qquad A (e_i)^3+B (e_i)^2+C\,e_i+D=v_i(e_i)
\end{equation}
holds true on $(\mathbb C^2,0)$ for $i=1,\ldots,4$.

Consider  (\ref{E:smlpo}) as a system of  linear equations in the variables $A,B,C,D$.
The determinant of the associated homogeneous linear system is the Vandermonde  determinant
\[
\det   \begin{bmatrix} 1 & e_1 & e_1^2 &  e_1^{3}\\ 1 & e_2 & e_2^2 &  e_2^{3}\\ 1 & e_3 & e_3^2 &  e_3^{3}\\  1 & e_4 & e_4^2 &  e_4^{3}\\ \end{bmatrix}
= \prod_{1\le i<j\le 4} (e_j-e_i).
\]
Since $\mathcal W$ is smooth,  $e_i(0)\neq e_j(0)$ when  $i\neq j$. Therefore the Vandermonde  determinant above is non-zero and consequently there exists a unique second order differential equation
 of the form (\ref{E:equatABCDpp}) admitting the orbits of  $v_i$ for   $i=1,\ldots,4$, as solutions.  Item (b) follows  from  Proposition \ref{P:connexP=Ypp}.
\end{proof}

\begin{cor}
If  $\mathcal W$ is a smooth  $k$-web $\mathcal W$  on  $(\mathbb C^2,0)$ with $k\geq 4$ then the following assertions are equivalent
\begin{enumerate}
\item   $\mathcal W$ is linearizable;
\item  there exists a flat projective connection $\Pi$ such that
$ \Pi=\Pi_{{\mathcal W}'}$  for every
 $4$-subweb $\mathcal W'$ of $\mathcal W$.
\end{enumerate}
\end{cor}

\medskip

\begin{remark}If $v_i=\frac{\partial}{\partial x}+e_i\frac{\partial}{\partial y}$ (for  $i \in \underline{k}$)
are vector fields defining a smooth  $k$-web $\mathcal W$ then  mimicking the proof of Proposition \ref{P:616} it is possible
to prove the existence of a unique differential equation of the form $y'' = F(x,y,y')$ satisfying the following conditions:
\begin{enumerate}
\item the $p$-degree  of $F(x,y,p)$ is at most $k-1$; and
\item the leaves of $\mathcal W$ are solutions of $y'' = F(x,y,y')$.
\end{enumerate}
The previous corollary   can be rephrased as follows: $\mathcal W$ is linearizable if and only if the degree of $F$ is at most three, and its coefficients
satisfy the  conditions $L_1=L_2=0$ of Theorem \ref{T:L1L2}. Since the determination of the differential equation is purely algebraic and can be carried over
rather easily this provides a nice computational test to decide whether or not a given web is linearizable. The draw-back is that, in contrast with the case
 $k=4$, the differential equation  obtained does not behaves nicely under arbitrary change of coordinates $(x,y) \mapsto (X(x,y), Y(x,y))$ as a simple
computation shows. Indeed, it can be verified that in the resulting equation $Y'' = G(X,Y,Y')$ the function $G(x,y,p)$ may be no longer  polynomial, but
only rational, in the variable $p$.
\end{remark}

\medskip

\begin{cor}\label{C:linutil}
Let  $\mathcal W$ be a smooth linearizable  $k$-web  on $(\mathbb C^2,0)$ with  $k\geq 4$. Modulo projective transformations
it admits a unique linearization: if  $\varphi, \psi$ are germs of biholomorphisms such that  $\varphi^*\mathcal W$ and $  \psi^*\mathcal W$
are linear webs then there exists a projective transformation   $g\in PGL_3(\mathbb C)$ for which  $\psi=g\circ \varphi$.
\end{cor}
\begin{proof}
 The hypothesis implies that  $  \mu=\varphi \circ \psi^{-1} $ verifies  $\mu^*\Pi_0=\Pi_0$. In other words, if $U\subset \mathbb P^2$ is an open subset where  $  \mu$
 is defined then for every line $\ell \subset \mathbb P^2$ intersecting $U$, there exists a line $\ell_\mu$ such that  $\mu(U\cap \ell)=\mu(U) \cap \ell_\mu$.
 The fundamental theorem of projective geometry implies that $\mu$ is the restriction at $U$ of a projective transformation.
\end{proof}\medskip

\subsection{First examples of exceptional webs}\label{S:firstexceptional}
\medskip
 Corollary \ref{C:linutil} is rather useful to prove that a given web is not linearizable.  Most of the examples
of exceptional webs known up-to-date are the superposition of an algebraic, in particular linear,  $k$-web with one or more non-linear foliations.
For example, Bol's $5$ web is the superposition of the  algebraic  $4$-web dual to four lines in general position and of
a non-linear foliation: the pencil of conics through the four dual points.  It follows from Corollary \ref{C:linutil}
that $\mathcal B_5$ is non-algebraizable.

On the other hand, Bol realized that the rank of $\mathcal B_5$ is six. If $\mathcal B_5$ is presented as the web
$$
\mathcal B_5 = \mathcal W \left( x, y , \, \frac{x}{y}
 \, , \,\frac{1-y}{1-x}  \,,
\,\frac{x(1-y) }{y(1-x)} \right)
$$
then the abelian relations coming from its $3$-subwebs generate a subspace of $\mathcal A(\mathcal B_5)$ of dimension $5$.
Moreover, Bol found one extra abelian relation, which can be written in integral form using the logarithm and Euler's dilogarithm\begin{footnote}
{Euler's dilogarithm is the function $\l{2}(z)= \sum_{n=0}^{\infty} z^n/n^2$. The series converges for $|z| < 1$ and has analytic continuations
along  all paths contained in $\mathbb C \setminus \{ 0, 1\}$. }\end{footnote}, which
is essentially  equivalent to Abel's functional equation for the dilogarithm. Explicitly,
\[
D_2(x) - D_2(y) - D_2\left( \frac{x}{y} \right) - D_2\left( \frac{1-y}{1-x} \right)  + D_2\left( \frac{x(1-y)}{y(1-x)} \right)  =0 \,
\]
where $D_2(z) = \l{2}(z) + \frac{1}{2} \log(z) \log(1-z) - \frac{\pi^2}{6}$.

\smallskip

Although $\mathcal B_5$ was the first exceptional web to appear in the literature, there are  simpler examples. For instance
 the $5$-webs   presented in Example \ref{E:exabel} of  Chapter 2,  are all the superposition of
four pencils of lines and one non-linear foliation, thus non-algebraizable.  Since they all have rank six, they are exceptional webs.
\medskip

\subsection{Algebraization of planar webs }\label{S:algdim2}
\medskip

Let now  $\mathcal W= \mathcal F_1 \boxtimes \cdots \boxtimes \mathcal F_k$ be a smooth  $k$-web  on $(\mathbb C^2,0)$. In contrast with the higher dimensional case, no hypothesis on $\mathcal A(\mathcal W)$ is needed to assume that  the foliation $\mathcal F_i$ is induced by  $\omega_i=\varpi_0+\theta_i\varpi_1$,  where  $\varpi=(\varpi_0,\varpi_1)$ is a coframe and  $\theta_1,\ldots,\theta_k$ are functions on  $(\mathbb C^2,0)$.

\begin{prop}
\label{L:trou}
The assertions below are equivalent:
\begin{enumerate}
 \item  $\mathcal W$ is compatible with a projective connection $\Pi$;
\item there exist functions  $N_{0},N_{1},N_{2},N_3$ such that
\begin{equation}
\label{E:troty}
 \{ d\theta_i\}^{1}-\theta_i \{ d\theta_i\}^0=\sum_{\rho=0}^{3} (\theta_i)^\rho N_{\rho}
\end{equation}
for every  $i \in \underline k$.
\end{enumerate}
\end{prop}
\begin{proof}
Let  $ \nabla$ be an affine connection representing $\Pi$, and let  $\Gamma_{ij}^k$ (with $i,j,k=0,1$)
be its  Christofel's symbols relative to the coframe  $\chi=(\chi_0,\chi_1)$ dual to the coframe $\varpi$.
For  $i \in \underline k $, set $v_i=\theta_i \chi_0-\chi_1$: it is a nowhere vanishing section of  $T\mathcal F_i$.
If  $i\in \underline{k}$ is fixed then the leaves of  $\mathcal F_i$ are geodesics of
 $\Pi$ if and only if there exists   $\zeta_i\in \Omega^1{(\mathbb C^2,0)}$ satisfying  $\nabla(v_i)=\zeta_i\otimes v_i $. More explicitly
$$d\theta_i\otimes \chi_0 +\theta_i\nabla(\chi_0)-
\nabla(\chi_1)
 =\zeta_i\otimes \big(  \theta_i \chi_0-\chi_1  \big) . $$
After decomposing this relation in the basis $\chi_p\otimes \omega_q$ (with $p,q=0,1$), it is a simple matter to deduce
the following four scalar equations:
\begin{align}
\label{E:trugudu}
\{d\theta_\alpha\}^0+\theta_i\Gamma_{00}^0-\Gamma_{01}^0= & \, \theta_i \{ \zeta_i\}^0   \nonumber \\
\{d\theta_i\}^1+\theta_i\Gamma_{10}^0-\Gamma_{11}^0= & \, \theta_i \{ \zeta_i\}^1   \\
\theta_i\Gamma_{00}^1-\Gamma_{01}^1= & \, - \{ \zeta_i\}^0 \nonumber \\
\theta_i\Gamma_{10}^1-\Gamma_{11}^1= & \, - \{ \zeta_i\}^1.\nonumber
\end{align}

Notice that the  last two equations determine  $\{ \zeta_i\}^0$ and  $\{ \zeta_i\}^1$. Plugging them into the
first two equations to deduce the following. If the leaves of  $\mathcal F_i$ are geodesics for  $\Pi$ then $\theta_i$ verifies
\begin{equation}
 \label{E:secondorder}
\{d\theta_i\}^1-\theta_i\{d\theta_i\}^0=
A + \theta_i B + (\theta_i)^2 C + (\theta_i)^3 D
\end{equation}
where
\begin{align}
\label{E:ABCD}
 A= & \, \Gamma^0_{11}         &&    C=  \Gamma^0_{00}   -\Gamma^1_{10} - \Gamma^1_{01} \\
B= &   \Gamma^1_{11}   -\Gamma^0_{10} - \Gamma^0_{01}     &&   D=\Gamma^1_{00}.
\nonumber     \end{align}

Hence the first assertion does imply the second.  \smallskip

Reciprocally, if the second assertion holds true, --  what  is clearly equivalent to the validity of (\ref{E:secondorder}) --
let $\nabla$ be the affine connection with Christofel's symbols  (in the coframe  $\varpi$) determined by
   (\ref{E:ABCD}) and the last two equations of (\ref{E:trugudu}). It is a simple matter to verify that
   the result is indeed an affine connection which represents a projective connection having the leaves
   of the web as geodesics.
\end{proof}

The next result, when $k=5$, is discussed in Section \S 30 of the book \cite{BB}.
In its most  general form, it has been obtained by H\'{e}naut in \cite{Henaut1994}.
He phrased it in a slightly different form and under the stronger assumption that  the rank of  $\mathcal W$ is  maximal, but his proof
still works under   the weaker assumption stated below.

\begin{thm}
Let $\mathcal W$ be a smooth $k$-web on $(\mathbb C^2,0)$. If
\begin{enumerate}
 \item
$\dim \,\mathcal A(\mathcal W)/
F^2 \mathcal A(\mathcal W)=2k-5$
; and
\item   $\mathcal W$ is compatible with a projective  connection
\end{enumerate}
then $\mathcal W$ is algebraizable.
\end{thm}
\begin{proof}
If  $k\geq 4$ and  $\dim \mathcal A(\mathcal W)/
F^2 \mathcal A(\mathcal W)=2k-5$ then it is possible to construct a  Poincar\'e-Blaschke map
 $ PB_{\mathcal W}: (\mathbb C^2,0)\times \mathbb P^1\rightarrow \mathbb P^{2k-6}$. If the rank of this map is two then
 the argument used at the end of the previous Chapter allows to conclude that  $\mathcal W$ is
 linearizable and, consequently, of maximal rank.

The key result to establish the bound on the rank of $PB_{\mathcal W}$ is  Lemma \ref{L:Ip}.
Its proof is based on the  relations (\ref{E:rel-z1}) and (\ref{E:rel-theta1}) of Proposition \ref{P:TECHmain}.
Recall that relations (\ref{E:rel-z1}) does hold true, in no matter  which dimension.
A careful reading of the proof of  Lemma \ref{L:Ip} reveals that to prove
 it, one just needs  to have  relations similar to
(\ref{E:rel-theta1}) but with the summation in the  right hand-side being allowed to range
from $0$ to $n+1$ instead of from $0$ to $n$.

But, according to Proposition \ref{L:trou}, the existence of such  relations
is equivalent to the compatibility of $\mathcal W$ with a projective connection.
Thus the proof of Tr\'{e}preau's Theorem presented in the previous Chapter works as well in dimension two when
$\mathcal W$ is compatible with a projective connection.
\end{proof}

\section{Infinitesimal automorphisms}\label{S:iinnff}

As already mentioned, exceptional planar webs and their abelian relations are still mysterious up-to-date.
For instance, even for a web  defined by rational submersions it is not known what kind of transcendency
its  abelian relations can have.  To formulate this question more precisely it is useful
to recall first some basic definitions of differential algebra.

\subsection{Basics on differential algebra}

Recall that a  \index{Differential field}\defi[differential field]\begin{footnote}{Here only differential fields over
$\mathbb C$ will be considered. Of course, it is possible to deal with more
general fields.}\end{footnote} is a pair $(\mathbb K, \Delta)$, where  $\mathbb K$ is field
containing $\mathbb C$, and $\Delta$ is a finite collection of $\mathbb C$-derivations
of $\mathbb K$ subject to the conditions
\begin{enumerate}
\item[(a).] any two derivations in $\Delta$ commute;
\item[(b).] the \defi[field of constants] of $\Delta$, \index{Field of constants} that is the intersection of the
kernels of the derivations in $\Delta$, is equal to $\mathbb C$.
\end{enumerate}

A \defi[differential extension] \index{Differential extension} of $(\mathbb K, \Delta)$ is a differential field
$(\mathbb K_0, \Delta_0)$ such that $\mathbb K_0$ is a field extension of $\mathbb K$, and
for every $\partial_0 \in \Delta_0$ there exists a unique $\partial \in \Delta$  satisfying
 $\partial = {\partial_0}|_{\mathbb K}$. A differential extension $(\mathbb K_0, \Delta_0)$  of $(\mathbb K, \Delta)$
is said to be \defi[primitive] \index{Differential extension!primitive} if there exists an element $h \in \mathbb K_0$ such that
$\mathbb K_0 = \mathbb K (h)$.

The simplest kind of  differential extension is when $\mathbb K_0$ is an algebraic field  extension of $\mathbb K$. These
are called  \defi[algebraic extensions].  \index{Differential extension!algebraic}

Another particularly simple kind of differential extension are the so called \defi[Liouvillian extensions]. \index{Differential extension!Liouvillian}
A differential field $(\mathbb K', \Delta')$ is a Liouvillian extension of $(\mathbb K, \Delta)$, if there exists
a finite sequence of differential extensions $$(\mathbb K , \Delta) = (\mathbb K_1, \Delta_1) \subset \cdots \subset (\mathbb K_r, \Delta_r) \subset (\mathbb K_{r+1}, \Delta_{r+1})=  (\mathbb K', \Delta')$$
such that for each $i \in \underline r$, $\mathbb K_{i+1} = \mathbb K_i(h_{i+1})$ for some  $h_{i+1} \in \mathbb K_{i+1}$
satisfying one of the following conditions
\begin{itemize}
\item[(a)] $h_{i+1}$ is algebraic over $\mathbb K_i$, or ;
\item[(b)] for every $\partial \in \Delta_{i+1}$, $\partial h_{i+1}$ belongs to $\mathbb K_i$, or;
\item[(c)] for every $\partial \in \Delta_{i+1}$, $\frac{\partial h_{i+1}}{h_{i+1}}$ belongs to $\mathbb K_i$.
\end{itemize}

If $(\mathbb K , \Delta )=  (\mathbb C((x,y)), \left\{ \partial_x, \partial_y \right\})$
is the differential field of germs of meromorphic functions at the origin of $\mathbb C^2$ endowed with
the natural derivations $\partial_x, \partial_y$, then a primitive Liouvillian extension over it is obtained
by taking  (a) a primitive  algebraic extension;  or (b) the integral of a   closed meromorphic $1$-forms; or (c) the
exponential of the integral of a closed meromorphic $1$-form.

If  $(\mathbb K, \left\{ \partial_x, \partial_y \right\})$ is
a differential subfield  of $(\mathbb C (( x,y )) , \left\{ \partial_x, \partial_y \right\})$ then, by definition,
a web $\mathcal W$ on $(\mathbb  C^2,0)$ is \defi[defined over $\mathbb K$] \index{Web!defined over $\mathbb K$} if there exists  a $k$-symmetric $1$-form
with coefficients in $\mathbb K$ defining $\mathcal W$. More explicitly, there exists
\[
\omega = \sum_{i+j = k} a_{ij}(x,y) dx^i dy^j  \in \mathrm{Sym}^k\Omega^1 (\mathbb C^2,0)
\]
such that $\mathcal W = \mathcal W(\omega)$ and $a_{ij} \in \mathbb K$ for every
pair $(i,j) \in \mathbb N^2$ satisfying $i+j=k$. Similarly, an abelian relation of a given smooth web on $(\mathbb C^2,0)$ is defined
over $\mathbb K$, if its components are $1$-forms with coefficients in $\mathbb K$. If $\mathcal W$ is a web defined over
$\mathbb K$ then the \defi[field of definition] \index{Abelian relation!field of definition} of its abelian relations is the differential extension
 of $\mathbb K$ generated
by  all  coefficients of all  components of all  abelian relations of $\mathcal W$.

\begin{problem}\label{Prob:1}
Let $\mathcal W$ be a germ of smooth $k$-web  on $(\mathbb C^2,0)$. If $\mathcal W$
is defined over $\mathbb K$, what can be said about the field of definition of
its abelian relations ? Is it a Liouvillian extension of $\mathbb K$ ?
\end{problem}

If $\mathcal W$ is a hexagonal $3$-web  on $(\mathbb C^2,0)$ then its unique abelian relation is
defined over a Liouvillian extension of its field of definition. An argument has already
been given in the proof of the implication (b) $\implies$ (c) of Theorem \ref{T:hexagonal}.

\subsubsection{An answer, but in a very particular case}

Due to the
apparent difficulty of Problem \ref{Prob:1} one could propose to low down the dimension of the ambient space
in order to obtain some progress. At first sight this seems to be pure non-sense, since it doesn't appear
to be reasonable to talk about webs on $(\mathbb C,0)$. A way out, is to
interpret less strictly the {\it lowing down} of the dimension. In the study of systems of  differential equations, the usual setup
where one is allowed to low down dimensions is when the system posses  infinitesimal symmetries.
Having this vague discourse in mind, it suffices to  look back at Chapter \ref{Chapter:AR} to
realize that it has been formally implemented in the proof of Proposition \ref{description}.
In particular, one obtains as a corollary the following

\begin{prop}\label{P:liouville1}
Let $ (\mathbb K, \left\{ \partial_x, \partial_y \right\})$ be a differential subfield  of $(\mathbb C (( x,y )) , \left\{ \partial_x, \partial_y \right\})$,
and let $\mathcal W$ be a germ of smooth $k$-web defined over $\mathbb K$.
If $\mathcal W$ admits a transverse infinitesimal automorphism, also defined over $\mathbb K$,  then there exists a
 Liouvillian extension of  $\mathbb K$ over which all the abelian relations of $\mathcal W$ are defined.
\end{prop}
\begin{proof}
It follows immediately from Proposition \ref{description}.
\end{proof}

\subsection{Variation of the rank}\label{S:rank}
Proposition \ref{description} also allows to compare the rank of a web $\mathcal W$ admitting a transverse infinitesimal
automorphism $v$, with the rank of the web $\mathcal W \boxtimes \mathcal F_v$ obtained by the superposition of
$\mathcal W$ and the foliation induced by $v$.

\begin{thm}\label{T:1}
Let ${\mathcal W}$ be a smooth $k$--web which admits a transverse infinitesimal automorphism $v$. Then
\[
\mathrm{rank}(\mathcal W \boxtimes \mathcal F_v) =\mathrm{rank}(\mathcal W) + (k -1)\, .
\]
In particular,  ${\mathcal W}$ is of maximal rank if and only if
$\mathcal W \boxtimes \mathcal F_{v}$ is of maximal rank.
\end{thm}
\begin{proof}
Let  $\mathcal W= \mathcal W(\omega_1, \ldots, \omega_k)=  \mathcal F_1 \boxtimes \cdots \boxtimes \mathcal F_k$.
Recall from Chapter \ref{Chapter:AR} that the canonical first integral
of $\mathcal  F_i$ relative to $v$  is
\[
u_i = \int \frac{\omega_i}{\omega_i(v)} \, .
\]
In particular, its differential is    $\eta_i =\frac{\omega_i}{\omega_i(v)}= du_i$.

Notice that when $j$ varies from $2$ to $k$, the following identities hold
\[
i_v(\eta_1 - \eta_j) =0 \qquad \text{ and } \qquad  L_v(\eta_1 - \eta_j) = 0 \, .
\]

Consequently there exists $g_j \in \mathbb C\{t\}$ for which
\begin{equation}\label{E:xxx}
du_1 - du_j - g_j(u_{k+1})du_{k+1} = 0 \, ,
\end{equation}
where $u_{k+1}$ is a first integral of $\mathcal F_v$.

Clearly these are abelian relations for the web $\mathcal W \boxtimes \mathcal F_v$.
If  $\mathcal A_0(\mathcal W \boxtimes \mathcal F_v)$ stands for the
maximal eigenspace of $L_v$ associated to the eigenvalue zero, then the abelian relations above
 span a vector subspace of it  which
will be denoted by $\mathcal V$.   Notice that  $\dim \mathcal V= k-1$.

Observe that $\mathcal V$ fits into  the   sequence
\[
0 \to {\mathcal V} \stackrel{i}{\longrightarrow}
{\mathcal A_0(\mathcal W \boxtimes \mathcal F_v)}\stackrel{L_v}{\longrightarrow}
 {\mathcal A_0(\mathcal W)}   .
\]
Notice that this sequence is exact. Indeed,   $ K= \ker \{ L_v : \mathcal A_0(\mathcal W \boxtimes \mathcal F_v) \to
 {\mathcal A_0(\mathcal W)} \}$ is generated by abelian relations of the form
$\sum_{i=1}^k c_i du_i + h(u_{k+1})du_{k+1} = 0$, where $c_i \in \mathbb C$
and $h \in \mathbb C\{t\}$.   Since $i_v du_i=1$ for each $i\in \underline k$,
 it follows that the constants
$c_i$   satisfy $\sum_{i=1}^k c_i =0$. This implies that
 the abelian relations in the kernel
of $L_v$ can be written as linear combinations of abelian relations
of the form (\ref{E:xxx}).  Therefore
\begin{equation}\label{E:yyy}
K =  \mathcal V
\end{equation}
and consequently $\ker L_v \subset \mathrm{Im } \, \, i$. The exactness
of the above sequence~follows~easily.

From general principles one can  deduce that the sequence
\[
0 \to \frac{\mathcal V}{\mathcal A_0(\mathcal W)\cap \mathcal V} \stackrel{i}{\longrightarrow}
\frac{\mathcal A_0(\mathcal W \boxtimes \mathcal F_v)}{\mathcal A_0(\mathcal W)} \stackrel{L_v}{\longrightarrow}
 \frac{\mathcal A_0(\mathcal W)}{L_v \mathcal A_0(\mathcal W)} \, ,
\]
is also exact. Thus to  prove the theorem it suffices to
verify that
\begin{enumerate}
\item[(a)] $\mathcal V$ is isomorphic to $
\frac{\mathcal V}{\mathcal A_0(\mathcal W)\cap \mathcal V} \oplus
  \frac{\mathcal A_0(\mathcal W)}{L_v \mathcal A_0(\mathcal W)} \, ;
$
\item[(b)] the morphism $L_v:\mathcal A_0(\mathcal W \boxtimes \mathcal F_v) \to
 \mathcal A_0(\mathcal W)$
 is surjective;
 \item[(c)] the vector spaces
 $\frac{\mathcal A_0(\mathcal W \boxtimes \mathcal F_v)}{\mathcal A_0(\mathcal W)}$ and
 $\frac{\mathcal A(\mathcal W \boxtimes \mathcal F_v)}{\mathcal A(\mathcal W)}$ are isomorphic.
\end{enumerate}

The key to  verify  (a) is the nilpotency of $L_v$  on $\mathcal A_0(\mathcal W)$. It
implies that $\frac{\mathcal A_0(\mathcal W)}{L_v \mathcal A_0(\mathcal W)}$ is isomorphic
to $\mathcal A_0(\mathcal W) \cap K$. Combining this
with (\ref{E:yyy}) assertion (a) follows.

To prove assertion (b) it suffices to construct a map $\Phi:\mathcal A_0(\mathcal W)
\to \mathcal A_0(\mathcal W \boxtimes \mathcal F_v)$ such that $ L_v \circ \Phi = \mathrm{Id}$.
Proposition \ref{description}  implies that $\mathcal A_0(\mathcal W)$ is spanned by abelian
relations of the form
$ \sum_{i=1}^k c_i u_i^r du_i =0  ,$
where $c_1, \ldots, c_k$ are complex numbers and $r$ is a non-negative integer.
Since
\[
   \sum_{i=1}^k c_i u_i^r du_i = \frac{1}{r+1}L_v \left( \sum_{i=1}^k c_i u_i^{r+1} du_i \right) =0
\]
there exists a unique function $h\in \mathbb C \{ t\}$ satisfying
$$  \sum_{i=1}^k c_i u_i^{r+1} du_i + h(u_{k+1}) du_{k+1} =0 \, .
$$
If one sets
\[
\Phi\left(\sum_{i=1}^k c_i u_i^r du_i \right)
= \frac{1}{r+1} \left(  \sum_{i=1}^k c_i u_i^{r+1} du_i + h(u_{k+1}) du_{k+1}  \right) \,
\]
then  $L_v \circ \Phi= \mathrm{Id}$ and assertion (b) follows.

To prove assertion (c), first notice that
\[
\mathcal A(\mathcal W \boxtimes \mathcal F_v)=\mathcal A_0(\mathcal W \boxtimes \mathcal F_v)  \oplus
  \mathcal A_*(\mathcal W \boxtimes \mathcal F_v) \,
\]
where $\mathcal A_*(\mathcal W \boxtimes \mathcal F_v)$ is the sum of eigenspaces corresponding to non-zero eigenvalues.
Of course $\mathcal A_*(\mathcal W \boxtimes \mathcal F_v)$ is invariant by $L_v$. Moreover  the equality
\[
L_v\big(\mathcal A_*(\mathcal W \boxtimes \mathcal F_v)\big) = \mathcal A_*(\mathcal W \boxtimes \mathcal F_v) \, ,
\]
holds true.
On the other hand $L_v$ {\it kills} the component of an abelian relation corresponding to the foliation $\mathcal F_v$.
In particular
\[
L_v\big(\mathcal A_*(\mathcal W \boxtimes \mathcal F_v)\big) \subset  \mathcal A_*(\mathcal W ).
\]
This is sufficient to show that $\mathcal A_*(\mathcal W \boxtimes \mathcal F_v) = \mathcal A_*(\mathcal W )$ and
deduce assertion (c).

Putting all together, it follows that
\[
\mathrm{rank}(\mathcal W \boxtimes \mathcal F_v) =\mathrm{rank}(\mathcal W) + (k -1)\, .
\]
To prove  the last claim of the theorem,   just  remark that the $(k+1)$-web $\mathcal W \boxtimes \mathcal F_v$ is of maximal rank if and only if
\[
 \mathrm{rank}({\mathcal W}\boxtimes \mathcal F_v)=\frac{k(k-1)}{2}=\frac{(k-1)(k-2)}{2}+(k-1)\, . \qedhere
\]
\end{proof}

This result was obtained in \cite{MPP} by David Mar\'{\i}n  together with the authors of this book.
As a corollary, it was then obtained  the existence of  exceptional planar $k$-webs for every $k \ge 5$, as explained
in the next section.

Of course, one can also deduce from Theorem \ref{T:1} the following analogue of Proposition \ref{P:liouville1}.

\begin{prop}\label{P:liouville2}
Let $ (\mathbb K, \left\{ \partial_x, \partial_y \right\})$ be a differential subfield  of $(\mathbb C (( x,y )) , \left\{ \partial_x, \partial_y \right\})$,
and let $\mathcal W$ be a germ of smooth $k$-web defined over $\mathbb K$.
If $\mathcal W$ admits a transverse infinitesimal automorphism, also defined over $\mathbb K$,  then there exists a
 Liouvillian extension of  $\mathbb K$ over which all the abelian relations of $\mathcal W\boxtimes \mathcal F_v$  are defined.
\end{prop}

It should not be very hard to drop the hypothesis, in Proposition \ref{P:liouville1} as well as in  Proposition \ref{P:liouville2}, on the field of definition of the infinitesimal automorphism of $\mathcal W$. More precisely, it should be true that the infinitesimal automorphisms of a web defined over $\mathbb K$ should
also be defined over $\mathbb K$, or at least over a Liouvillian extension of $\mathbb K$.

\subsection{Infinitely many families of   exceptional webs}

Let   $C$ be a degree $k$ curve in $\mathbb P^2$  invariant by a $\mathbb C^*$-action
$\varphi: \mathbb C^* \times \mathbb P^2 \to \mathbb P^2.$
Notice that  $\varphi$ induces a dual action
 $\check{\varphi}:\mathbb C^* \times
\check{\mathbb P}^2 \to \check{\mathbb P}^2$
which is a one-parameter group of  automorphisms of the dual $k$-web $\mathcal W_C$.
Consequently, the web $\mathcal W_C(\ell_0)$, the germification of $\mathcal W_C$ at
a generic point $\ell_0 \in \check{\mathbb P}^2$  admits an infinitesimal automorphism.

It is a simple matter to show that in a suitable projective coordinate system $[x:y:z]$, a plane curve $C$
invariant by a ${\mathbb C}^*$-action is cut out by an equation of the form
\begin{equation}\label{E:aluffi}
x^{\epsilon_1}\cdot y^{\epsilon_2}\cdot z^{\epsilon_3} \cdot \prod_{i=1}^n \big(  x^a + \lambda_i y^b z^{a-b}  \big) \,
\end{equation}
where $\epsilon_1, \epsilon_2, \epsilon_3 \in \{ 0,1 \}$,  $n,a,b \in \mathbb N$ are such that $n\ge1$, $a\ge 2$, $1\le b\le a/2$,
$\gcd(a,b)=1$
and the $\lambda_i$ are distinct non zero complex numbers. For a curve of this form  the $\mathbb C^*$-action in question is
\begin{align*}
\varphi :  \qquad  \mathbb C^* \times \mathbb P^2 \:  &\longrightarrow \:  \mathbb P^2 \\
 ( t, [x:y:z] ) &\longmapsto \big[t^{b(a-b)}x:t^{a(a-b)} y: t^{ab}z \big] \, . \nonumber
\end{align*}

Moreover once  $\epsilon_1,\epsilon_2,\epsilon_3,n,a,b$ are fixed, one can always choose $\lambda_1=1$ and in this
case the set of $n-1$ complex numbers $\{\lambda_2,\ldots, \lambda_n\}$ projectively characterizes the
curve $C$. In particular,
 there exists a $(d-1)$-dimensional family  of degree
 $2d$ (or $2d +1$)
 reduced plane curves all projectively
distinct and invariant by the same  $\mathbb C^*$-action: for a given $2d+ \delta$
with $\delta\in \{0,1\}$ set $a=2$, $b=1$, $\epsilon_1=\delta$ and $\epsilon_2=\epsilon_3=0$.

If $C$ is a reduced curve of the form (\ref{E:aluffi}) then $\mathcal W_C$ is  invariant
by the $\mathbb C^*$-action $\check \varphi$ dual to $\varphi$. Denote by $v$ the infinitesimal
generator of $\check \varphi$ and by $\mathcal F_v$ the corresponding foliation.

\begin{thm}\label{T:geral}
If $\deg C\ge 4$ then $\mathcal W_C \boxtimes \mathcal F_v$ is exceptional. Moreover if $C'$
is another curve invariant by $\varphi$ then $\mathcal W_C \boxtimes \mathcal F_v$  is analytically
equivalent to $\mathcal W_{C'} \boxtimes \mathcal F_X$ if and only if the curve $C$ is projectively equivalent to
$C'$.
\end{thm}
\begin{proof} Since $\mathcal W_C$ has maximal rank it follows from Theorem \ref{T:1}
that $\mathcal W_C \boxtimes \mathcal F_v$ is also of maximal rank. Suppose that its
localization  at a point $\ell_0 \in \check{\mathbb P}^2$ is algebraizable and
let $\psi:(\check{\mathbb P}^2,\ell_0) \to (\mathbb C^2,0)$ be a holomorphic algebraization. Since
both $\mathcal W_C$ and $\psi_*(\mathcal W_C)$ are linear webs of maximal rank
it follows from Corollary \ref{C:linutil}  that $\psi$ is the localization of
an automorphism of $\mathbb P^2$.  But the generic leaf of $\mathcal F_v$
is not contained in any line of $\mathbb P^2$ and consequently $\psi_*(\mathcal W \boxtimes
\mathcal F_v)$ is not linear. This concludes the proof of the theorem.
\end{proof}

\begin{remark}\rm
It is  not known if the families of examples above are {\it irreducible components of the space of exceptional
webs} in the sense
that the generic element does not admit a deformation as an exceptional
web that is not of the form $\mathcal W_C \boxtimes \mathcal F_v$. Due to the presence of infinitesimal automorphisms
one could imagine that they are indeed degenerations of some other
exceptional webs.
\end{remark}

\section{Pantazi-H\'{e}naut criterium}\label{S:Pantazi-Henaut}

If $\mathcal W$ is a $3$-web on $(\mathbb C^2,0)$ then $\mathcal W$ has maximal rank
if and only if its curvature $K(\mathcal W)$ vanishes identically. In this section, a generalization of this
 result to arbitrary planar webs will be presented.
The strategy   sketched below  can be traced back to
Pantazi \cite{Pantazi}. Recently, unaware of Pantazi's result,
H\'{e}naut \cite{Henaut2004} proved an essentially
equivalent result but formulated in more intrinsic terms.

Even more recently, Cavalier and Lehmann  proved that it
is possible to extend Pantazi-H\'{e}naut construction to {\it ordinary}
codimension one webs on $(\mathbb C^n,0)$ for arbitrary $n$. This text
will not deal with this generalization. For details see \cite{CL1}.

Below, after presenting the basics of the theory of linear differential systems,  the arguments of \cite{Pantazi} are
presented using  the modern formalism introduced in this context by  \cite{Henaut2004}.

\subsection{Linear differential  systems}

For a more detailed exposition the reader can consult \cite{Spencer} and  references therein.

\subsubsection{Jet spaces }\index{Jet spaces}

Let  $E$ be a rank $r$ vector  bundle over $(\mathbb C^n,0)$. Since the setup is local, $E$ is of course trivial.
Thus $x,p$ with $x=(x_1,\ldots,x_n)\in (\mathbb C^n,0)$ and  $p=(p^1,\ldots,p^r)\in \mathbb C^r$
constitute a coordinate system on the total space of  $E$.
A section $\xi$ of  $E$ will be identified with a map
 $(\xi^1,\ldots,\xi^r): (\mathbb C^n,0) \rightarrow \mathbb C^r$. \smallskip

The space  $J^{\ell}(E)$ of  $\ell$-jets of sections of  $E$  is the vector bundle over $(\mathbb C^n,0)$ with fiber
$J^{\ell}_x(E)$ over a point $x \in (\mathbb C^n,0)$
equal to the quotient of the space of  germs of sections of $E$ at $x$ by the subspace of germs vanishing at $x$ up to order $\ell+1$.
Given a section $\xi$ of $E$ then its $\ell$-jet at $x$ will be denoted by $j^{\ell}_x(\xi)$. In order to make sense of the following map
\[
j^{\ell} : E \longrightarrow J^{\ell}(E)
\]
one has to think of it not as a map of vector bundles ( derivations  are not $\mathcal O(\mathbb C^n,0)$-linear ) but as
a morphism of sheaves of $\mathbb C$-modules. On $J^{\ell}_x(E)$ there is natural system of  linear coordinates:
  $$p^s_\sigma(j^\ell_x(\xi))=\partial_\sigma (\xi^s)(x)=\frac{\partial^{|\sigma|} \xi^s}{\partial x^{\sigma}}(x)
$$ for  $s \in \underline r$ and  $\sigma=(\sigma_a)_{a=1}^n\in \mathbb N^n$ with  $|\sigma|=\sum_a \sigma_a\leq \ell$. \smallskip

For every  $\ell,\ell'$ with $\ell\geq \ell'$, there is a  natural morphism of vector bundles ( unlike $j^{\ell}$ this morphisms is
 $\mathcal O(\mathbb C^n,0)$-linear )
\begin{align*}
 \pi_{\ell}^{\ell'}\, : \;  {J}^\ell(E) & \longrightarrow
 {J}^{\ell'}(E)
\\
(x,j^{\ell}_x(\xi)) & \longmapsto
(x,j^{\ell'}_x(\xi))
\,. \end{align*}

By convention,  ${J}^q(E)$ is the $0$ vector bundle when  $q<0$.

\subsubsection{Linear differential systems}

A \defi[ linear differential system of order $q$ in  $r$ indeterminates] \index{Linear differential system} is defined as the kernel
$\boldsymbol{S} ={\rm Ker}\,\Phi \subset J^q(E)$ of a
\defi[ linear differential operator  of order $q$ in $r$ indeterminates], \index{Linear differential operator}
that is  a morphism of $\mathcal O(\mathbb C^n,0)$-modules
$$ {\Phi}\, : \, {J}^q(E) \longrightarrow F \, . $$
Explicitly, if $F=(\mathbb C^n,0)\times \mathbb C^m$ and
$\Phi=(\Phi^1,\ldots,\Phi^m)\, :  \,   {J}^q(E)
\longrightarrow F$ with
$$ \Phi^\kappa(x,p)= \sum_{s, |\sigma|\leq q}
A^{\kappa}_{s,\sigma}(x)\,p^s_{\sigma} \, , \qquad \kappa \in \underline{m}$$
then the sections of $\boldsymbol{S} ={\rm Ker}\,\Phi$
which are  also images of sections of $E$ through   $j^q: E \to J^q(E)$ correspond to solutions of the system of differential equations
$$ (\boldsymbol{S}) \qquad \qquad \sum_{s\, ,\, |\sigma|\leq q}
A^{\kappa}_{s,\sigma}\,\partial_\sigma (\varphi^s)= 0 \quad  \quad \kappa \in \underline m  $$
in the unknowns $\varphi^1,\ldots, \varphi^r: (\mathbb C^n,0) \rightarrow \mathbb C$.
Beware that in general $
\boldsymbol{S} \subset J^q(E)$ is not a vector subbundle of $ {J}^q(E)$: the rank of the fibers of $\boldsymbol{S}$ may vary from point to point.

If the $q$-jet of  $(\varphi_1,\ldots,\varphi_r) $ is in  $\boldsymbol{S}$ then its $(q+1)$-jet will be in
$\boldsymbol{S}^{(1)} \subset J^{q+1}(E)$, the homogeneous linear partial differential equation
of order $(q+1)$ deduced  from $\boldsymbol{S}$ through derivation with respect to  the free variables  $x_a$:
$$
(\boldsymbol{S}^{(1)}) \quad \begin{cases}
\sum_{s\, ,\, |\sigma|\leq q}
A^{\kappa}_{s,\sigma}\,\frac{ \partial^{|\sigma|} \varphi^s}{\partial
  x_\sigma}= 0 \hspace{3.3cm}\kappa \in \underline m;
\vspace{0.15cm}\\
\sum_{s\, ,\, |\sigma|\leq q}
\frac{\partial A^{\kappa}_{s,\sigma}}{\partial x_a} \,\frac{ \partial^{|\sigma|} \varphi^s}{\partial
  x_\sigma}+
A^{\kappa}_{s,\sigma}\,
\,\frac{ \partial^{|\sigma|+1} \varphi^s}{\partial
  x_{\sigma(a)}}=0
\hspace{0.6cm}
a \in \underline n
 \end{cases}  $$
with $\sigma(a)=(\overline{\sigma}_b)_{b=1}^n$
defined as follows:  $\overline{\sigma}_a=\sigma_a+1$ and   $\overline{\sigma}_b=\sigma_b$ if  $b\neq a$.

Of course, this operation can be iterated. Setting
$\boldsymbol{S}=\boldsymbol{S}^{(0)}$ and
$$
\boldsymbol{S}^{(\ell+1)}= \big(\boldsymbol{S}^{(\ell)}\big)^{(1)}
$$
for every  $\ell\geq 0$, one derives from $\boldsymbol{S}$   a family of linear systems of  partial differential equations  $
\boldsymbol{S}^{(\ell)}\subset J^{q+\ell}(E)$. By definition, $
\boldsymbol{S}^{(\ell)}$ is the \defi[$\ell$-th  prolongation] of $\boldsymbol{S}$.  It is not hard to construct
a morphism of vector bundles
$$ {\Phi}^{(\ell)}\, : \, {J}^{q+\ell}(E) \longrightarrow {J}^{\ell}(F) \, $$
such that
$\boldsymbol{S}^{(\ell)}= {\rm Ker}\, {\Phi}^{(\ell)}$. The morphism   $ {\Phi}^{(\ell)}$ is the  \defi[$\ell$-th prolongation]
of $ {\Phi}$.\medskip \index{Linear differential system!prolongation}

\subsubsection{Formal integrability}

A differential system $\boldsymbol{S}$ is  \defi[regular] \index{Linear differential system!regular} if
$\boldsymbol{S}^{(\ell)}$ is a vector subbundle of $J^{q+\ell}(E)$ for no matter which  $\ell\geq 0$.

Let  $\ell\geq 0$ be fixed. The restriction of the natural projection  $J^{q+\ell+1}(E) \to J^{q+\ell}(E)$
to $\boldsymbol{S}^{(\ell+1)}$ induces  a  natural morphism of \mbox{$\mathcal O(\mathbb C^n,0)$-modules}
$$  \xymatrix@R=0.3cm{
   \boldsymbol{S}^{(\ell+1)}    \ar[rr]^{\overline{\pi}_{q+\ell+1}^{q+\ell}}  &&  \;
  \boldsymbol{S}^{(\ell)}   \, .  }\vspace{0.1cm} $$

By definition,  $\boldsymbol{S}$ is \defi[formally  integrable] \index{Linear differential system!formally integrable}
if for every $\ell \ge 0$ the  morphism $\overline{\pi}_{q+\ell+1}^{q+\ell}$ is surjective. In less precise  terms, a system is formally integrable if given a $q$-jet in $\boldsymbol{S}$
then there exists a $(q+\ell)$-jet in $\boldsymbol{S^{(\ell)}} \subset J^{q+ \ell}(E)$ coinciding with the original one up to order $q$  for no matter which $\ell\ge 0$.
\smallskip

A  \defi[solution] \index{Linear differential system!solution} of  $\boldsymbol{S}$ is a section $\sigma$ of  $\sigma$  of $E$  such that $j^q(\sigma)$ is a section  of $\boldsymbol{S}$. Consequently,
$j^{q+\ell}(\sigma)$ is a section of $\boldsymbol{S}^{(\ell)}$ for every  $\ell\geq 0$. If  $Sol(\boldsymbol{S})$ denotes the space of solutions of
 $\boldsymbol{S}$, then the surjectivity of  $j^q: Sol(\boldsymbol{S}) \rightarrow \boldsymbol{S}$ on the fibers is a sufficient condition
 for the formal integrability of $\boldsymbol{S}$.
\smallskip

In the analytic category the formal integrability of a differential system is particularly
meaningful:  Cartan-K\"{a}hler theorem ensures the convergence of formal solutions.  In the case of linear differential systems,
 the proof of this result boils down to a simple application of the  method of majorants for formal power   series.  \medskip

For  $\ell\geq 0$, let  $f_1,\ldots,f_\ell\in \mathcal O{(\mathbb C^n,0)}$\begin{footnote}{Recall the convention about germs used through out. Here $(\mathbb C^n,0)$
must be thought as a small open subset containing the origin. }\end{footnote} be functions in the maximal ideal $\mathfrak{M}_x$ of $x$, that is  $f_1(x)=\cdots=f_\ell(x)=0$.
Set  $f=f_1 \cdots f_{\ell}$.
Since $ {\rm Sym}^\ell(T^*_x (\mathbb C^n,0) ) $ is generated, as a  $\mathbb C$-vector space, by the
\mbox{$\ell$-symmetric} forms  $(df_1\cdots df_{\ell})(x)$,
 one can define a linear map  $$ \varepsilon^{\ell}_x: {\rm Sym}^\ell(T^*_x (\mathbb C^n,0) )\otimes E \rightarrow
{J}^{\ell}_x(E)$$
 through the formula  $$
\varepsilon^{\ell}_x(df_1\cdot\cdots\cdot df_{\ell}\otimes e)=j^{\ell}( f_1\cdots f_\ell e)(x).$$ Because  $f\in (\mathfrak M_x)^\ell$,
 $\pi_{\ell}^{\ell-1}( j^{\ell}( f e)(x))=0$ for every section $e$ of $E$.  Varying  $x \in (\mathbb C^n,0)$, one deduces an
injective morphism of vector bundles
$$  \xymatrix@C=0.4cm@R=0.2cm{
   {\rm Sym}^\ell(T^*U)\otimes E \ar[rr]^{\quad \varepsilon^{\ell}} &&
 {J}^{\ell}(E) \\
\big(
df_1\cdot\cdots\cdot df_{\ell}\big)\otimes e  \;
\ar@{|->}[rr] &&\; j^{\ell}\big( f_1 \cdots
f_{\ell}\,e\big) \, }$$
which fits into the exact sequence
%\begin{equation}
%\xymatrix@C=0.4cm@R=0.2cm{
% 0
%\ar[r] & {\rm Sym}^{\ell+1}(T^*U)\otimes E \ar[rr] &&
%{J}^{\ell+1}(E)
%\ar[rr] &&
%{J}^{\ell}(E)\ar[r] & 0\, . }
%\end{equation}
\begin{equation}
\label{E:ssuuiittee}
0 \rightarrow {\rm Sym}^{\ell+1}(T^*U)\otimes E
\stackrel{\varepsilon^{\ell}}{\longrightarrow}  {J}^{\ell+1}(E)
\stackrel{{{\pi}_{\ell+1}^{\ell}}}{\longrightarrow}
 J^{\ell}(E)
\rightarrow 0 \, .
\end{equation}
 \medskip

For  $\ell\geq 0$, the  \defi[$\ell$-th symbol map] of  $\Phi: J^q(E)\rightarrow F$ is defined as the  composition $$\sigma^{(\ell)}(\Phi)= \Phi^{(\ell)}\circ \varepsilon ^\ell: {\rm Sym}^{q+\ell}(T^*U)\otimes E   \longrightarrow
 {J}^{\ell}(F)\,.$$
The $0$-th symbol map is, by convention,   $\sigma(\Phi)=
\sigma^{(0)}(\Phi)$.
By definition, $\mathfrak S^{(\ell)}={\rm Ker}\,
\sigma^{(\ell)}(\Phi)$  is the \defi[$\ell$-th symbol]  of the system
$\boldsymbol{S}$. It is completely determined by $\boldsymbol{S}$, that is it does not depend  on the presentation of $\boldsymbol{S}$ as the kernel of  $\Phi$.
From the exact sequence (\ref{E:ssuuiittee}), it follows that
$\mathfrak S^{(\ell)}$ fits into  the exact sequence of $\mathbb C$-sheaves
\begin{equation}
\label{E:suitexxx}
0 \rightarrow \mathfrak S^{(\ell)}
\stackrel{\varepsilon}{\longrightarrow}
 \boldsymbol{S}^{(\ell)}
\stackrel{{\pi}}{\longrightarrow}
\boldsymbol{S}^{(\ell-1)}
\end{equation}
for every $\ell \ge 0$, with the convention that $\boldsymbol{S}^{(-1)} = J^{q-1}(E)$.

Notice that  $\mathfrak S^{(\ell)}$ is not a vector bundle a priori: it can be naively thought as a family of vector spaces
 $ \{ \,
\mathfrak S^{(\ell)}(x) \,
\}_{ x \in (\mathbb C^n,0)}$   but the  dimension may vary with  $x$.\smallskip

It can be verified that   $\mathfrak S^{(\ell)}$ is completely determined by  $\mathfrak S$. In particular, if $\mathfrak S=0$ then
$\mathfrak S^{(\ell)}=0$ for every  $\ell \geq 0$.

\begin{thm}
\label{T:formalintegrability}
Let  $\boldsymbol{S}\subset J^q(E)$
be a linear differential system with $\mathfrak S_{\boldsymbol{S}}=0$. If the natural morphism  $
\boldsymbol{S}^{(1)}\rightarrow \boldsymbol{S}$ is surjective then  $\boldsymbol{S} $ is  regular and formally integrable.
\end{thm}
\begin{proof}
This is a particular case of a theorem by  Goldschmidt, see \cite[Theorem 1.5.1]{Spencer}. The proof is omitted.
\end{proof}

\subsubsection{Spencer operator and connections}\label{S:spencer}

The \defi[Spencer operator] \index{Spencer operator} (for $s,q\geq 0$)
$$
D: \Omega^s \otimes J^{q}(E) \rightarrow \Omega^{s+1} \otimes J^{q-1}(E)
$$
is characterized by the following properties: 
\begin{itemize}
\item[(1)] for section $\sigma$ of $E$,  one has $D(j^q(\sigma))=0$; 
\item[(2)] for every $\omega \in \Omega^j$ and  $\eta\in \Omega^* \otimes J^{q}(E)$, $D(\omega\wedge \eta)=d\omega\wedge \pi_{q}^{q-1}(\eta)+(-1)^j\omega\wedge D(\eta)$.
\end{itemize}

It can be verified that  $D\circ D=0$. Consequently, there is following complex of  $\mathbb C$-sheaves
\begin{equation}
\label{E:spencerC}
0 \rightarrow E
\stackrel{j^q}{\longrightarrow}
 J^q(E)
\stackrel{D}{\longrightarrow}
\Omega^1 \otimes J^{q-1}(E)\stackrel{D}{\longrightarrow} \cdots
\stackrel{D}{\longrightarrow}
\Omega^n \otimes J^{q-n}(E)
\stackrel{}{\rightarrow} 0
\end{equation}
from which one can extract the exact sequence
\begin{equation*}
\label{E:spencerCse}
0 \rightarrow E
\stackrel{j^q}{\longrightarrow}
 J^q(E)
\stackrel{D}{\longrightarrow}
\Omega^1 \otimes J^{q-1}(E).
\end{equation*}

To simplify, suppose the dimension is two. If
$\boldsymbol{S}\subset J^q(E)$
is  a linear differential system over $(\mathbb C^2,0)$ then the
 restriction of the  complex (\ref{E:spencerC}) to $\boldsymbol{S}$ and its prolongations   yields the \defi[first Spencer complex] \index{First Spencer complex} of $\boldsymbol{S}^{(\ell)}$:
\begin{equation*}
0 \to
Sol(\boldsymbol{S})
\stackrel{j^{r}}{\longrightarrow}
  \boldsymbol{S}^{(r)}
\stackrel{D}{\longrightarrow}
\Omega^1 \otimes
\boldsymbol{S}^{(r-1)}
\stackrel{D}{\longrightarrow}
\Omega^2 \otimes
\boldsymbol{S}^{(r-2)}
\to 0\,
\end{equation*}
where  $r = q+ \ell$.
Note that this is not a complex of $\mathcal O$-sheaves but only of $\mathbb C$-sheaves. It is clarifying
to notice the resemblance with the usual de Rham complex
$
0  \rightarrow  \mathbb C \stackrel{d}{\longrightarrow} \Omega^0 \stackrel{d}{\longrightarrow} \Omega_1 \stackrel{d}{\longrightarrow}
\Omega^2 \rightarrow  0 $.\medskip

There is the following commutative diagram
$$ \xymatrix%@C=0.8cm@R=0.8cm
{
&& &  & 0 \ar[d] \\
& & &
& \Omega^2\otimes
\mathfrak S_{\boldsymbol{S}}
 \ar[d] \\
0 \ar[r] & Sol(\boldsymbol{S}) \ar@{=}[d]\ar[r]^{j^{q+2}} &
\boldsymbol{S}^{(2)} \ar[r]^{D} \ar[d] & \Omega^1
\otimes\boldsymbol{S}^{(1)}
\ar[r]^{D} \ar[d]
 & \Omega^2 \otimes \boldsymbol{S}\ar[d]   \\
0 \ar[r] & Sol(\boldsymbol{S}) \ar[r]^{j^{q+1}} &
\boldsymbol{S}^{(1)}
 \ar[r]^{D} & \Omega^1 \otimes\boldsymbol{S} \ar[r]^{D}
 & \Omega^2 \otimes J^{q-1}(E)
 } $$
with columns and lines being complexes.
\smallskip

Suppose now that   $\mathfrak S_{\boldsymbol{S}}^{(1)}=0$ and that the natural morphism  ${\boldsymbol{S}^{(1)}\rightarrow \boldsymbol{S}}$ is  surjective.
 Together, these two conditions imply the isomorphism $
\boldsymbol{S}^{(1)}\simeq \boldsymbol{S}$.
Then one can define two operators  $\nabla$ et $\nabla'$ of $\mathbb C$-sheaves such that the diagram below commutes.
\begin{equation}
\label{D:grosdiag}
 \xymatrix%@C=0.8cm@R=0.8cm
 {
&& &  & 0 \ar[d] \\
& & &
& \Omega^2\otimes
\mathfrak S_{\boldsymbol{S}}
 \ar[d] \\
0 \ar[r] & Sol (\boldsymbol{S}) \ar@{=}[d]\ar[r]^{j^{q+2}} &
\boldsymbol{S}^{(2)} \ar[r]^{D} \ar[d] & \Omega^1
\otimes\boldsymbol{S}^{(1)} \ar@{.>}[ru]^{\nabla'}
\ar[r]^{D} \eq[d]
 & \Omega^2 \otimes \boldsymbol{S} \ar[d]  \\
0 \ar[r] & Sol(\boldsymbol{S}) \ar@{=}[d] \ar[r]^{j^{q+1}} &
\boldsymbol{S}^{(1)}\eq[d]\ar@{.>}[ru]^{\nabla}
 \ar[r]^{D} & \Omega^1 \otimes\boldsymbol{S}  \ar[r]^D
 &
\Omega^2\otimes J^{q-1}(E)
 \\
0 \ar[r] & Sol(\boldsymbol{S}) \ar[r]^{j^q} &
\boldsymbol{S}\, .
 &  &
 }
\end{equation}

It is possible to associate to this diagram  a  connection
\begin{align*}
 \nabla: \boldsymbol{S}^{(1)} \rightarrow \Omega^1 \otimes
\boldsymbol{S}^{(1)}
\end{align*}
with  corresponding curvature given by the   $\mathcal O$-linear operator
\begin{align*}
 K_{\boldsymbol{S}}=\nabla'\circ \nabla: \boldsymbol{S}^{(1)} \rightarrow \Omega^2 \otimes \mathfrak S_{\boldsymbol{S}}.
\end{align*}

Since the first line of (\ref{D:grosdiag}) is a complex, it is clear that the surjectivity of  $
\boldsymbol{S}^{(2)} \rightarrow
\boldsymbol{S}^{(1)}$ implies the vanishing of the curvature  $ K_{\boldsymbol{S}}$. From
 Theorem \ref{T:formalintegrability}
 applied to $\boldsymbol{S}^{(1)}$, one deduces that $ K_{\boldsymbol{S}}\equiv 0$ if  $\boldsymbol{S}\simeq \boldsymbol{S}^{(1)}$ is formally integrable. \medskip

Conversely, one can show that  $ K_{\boldsymbol{S}}\equiv 0$ implies that  $\boldsymbol{S}$ is regular and integrable. This non-trivial fact will
not be shown here. Those interested might consult \cite{Spencer}.

\begin{cor}
\label{C:K=0=integrability}
If $\mathfrak S_{\boldsymbol{S}}^{(1)}=0$ and  the natural morphism $
\boldsymbol{S}^{(1)}\rightarrow \boldsymbol{S}$ is  surjective then $\boldsymbol{S}$ is integrable if and only if $K_{\boldsymbol{S}}\equiv 0$.
\end{cor}

\subsection{The differential system  $\mathcal S_{\mathcal W}$}
The theory sketched above will now be applied to
the differential system which has as  solutions  the integral forms of  abelian relations of a smooth  planar web. \medskip

Let  $\mathcal W=\mathcal W(\omega_1, \ldots,  \omega_k) $ be a smooth $k$-web on $(\mathbb C^2,0)$.
It is harmless to assume that
$$
\omega_i=dx+\theta_i dy
$$
for  $i\in \underline{k}$ and suitable germs $\theta_1, \ldots, \theta_k$.

Let  $v_i=\partial_y-\theta_i \partial_x$  and consider the following differential system
$$ \boldsymbol S_\mathcal W \quad:\qquad \begin{cases}
 \;  \sum_{i=1}^k\varphi_i=0 \, , \\
\;
v_i\big( \varphi_i \big) =0 \quad  i\in \underline{k}
 \;. \end{cases} $$

The space of holomorphic solutions of $\boldsymbol S_{\mathcal W}$ will be denoted by  ${S}ol(\mathcal W)$.
Notice that $Sol(\mathcal W)$ fits into the exact sequence of vector spaces
\[
0 \rightarrow \mathcal A(\mathcal W) \stackrel{\int_0}{\longrightarrow}  Sol(\mathcal W) \longrightarrow \mathbb C^k \stackrel{\sum}{\longrightarrow} \mathbb C \to 0 \, ,
\]
where the arrow from $Sol(\mathcal W)$
to $\mathbb C^k$ is given by evaluation at the origin. In particular
the rank of $\mathcal W$ is maximal if and only if
 $ {S}ol(\mathcal W)$ has dimension
 dimension $k(k-1)/2$. \smallskip

Notice that $\boldsymbol S_{\mathcal W}$ is a differential linear system of first order and
as such can be defined as follows.  If $E$ (resp. $F$) is the trivial bundle of rank
 $k$ (resp. $k+3$) over  $(\mathbb C^2,0)$ then ${S}ol(\mathcal W)$
 can be identified with the sections of  $E$ with first jet belonging to the kernel of the
 map
$$ \Phi=(F_{00},F_{10},F_{01},G_1,\ldots,G_k)\, : \,
 {J}^1(E) \longrightarrow F$$
where
$$ \quad F_\sigma= \sum_{i=1}^k p^i_\sigma \qquad\quad \mbox{ and }
\qquad \quad  G_i=p^i_{01}-\theta_i\, p^i_{10}
 \;  $$
 for $\sigma$ such that $|\sigma|\leq 1$ and $i\in \underline{k}$.

The smoothness of $\mathcal W$ promptly implies that $\Phi$ has constant rank.
The subbundle  $ {\rm Ker}(\Phi) \subset {J}^1(E)$ will be identified with  $ {\boldsymbol S}_{\mathcal W}$ .

\subsubsection{Prolongations of $\boldsymbol S_{\mathcal W}$}
\label{S:diffprolongAW}

For every $\ell \ge 0$, denote by
$ {\boldsymbol S}_{\mathcal W}^{(\ell)} \subset  J^{1+\ell}(E)$ the  $\ell$-th  prolongation of  $ {\boldsymbol S}_{\mathcal W}$.
To write down  $ {\boldsymbol S}_{\mathcal W}^{(\ell)}$,
explicitly, set   $D^\tau=D^{\tau_1}_x \circ
D^{\tau_2}_y$ for  $\tau=(\tau_1,\tau_2) \in \mathbb N^2$, where  $D_x$
and $D_y$ are the total derivatives
$$ D_x=\frac{\partial}{\partial x} + \sum_{i \,, \,
    \sigma  } \, p^i_{\sigma(1)} \,\frac{\partial}{\partial
    p^i_\sigma}\,, \qquad
 D_y=\frac{\partial}{\partial y} + \sum_{i \,, \,
    \sigma  } \, p^i_{\sigma(2)} \,\frac{\partial}{\partial
    p^i_\sigma}\,.
 $$

Using the notation just introduced it is not hard to check
that $ {\boldsymbol S}_{\mathcal W}^{(\ell)}$ is defined by the linear differential equations
$$  D^\tau(F_\sigma)=0 \quad |\sigma|=0,1\qquad \mbox{ and  } \qquad
D^\tau(G_i)=0 \quad i\in\underline{k}\, $$
with $\tau $ satisfying  $|\tau|\leq \ell$. \smallskip

Notice that for every  $\tau , \sigma\in \mathbb  N^2$ and  $i\in\underline{k}$, the identities
$$ D^{\tau}(F_\sigma)=\sum_{i=1}^d p^i_{\sigma+\tau}
\qquad\mbox{and} \qquad
D^{\tau}\big( G_i\big)=p^i_{\tau(2)}-
D^{\tau}\big(
\theta_i\,p^i_{10}\big) \,
$$
hold true.
Since $D_x$ and  $D_y$ are derivations it follows that
\begin{align*}
D^{\tau}\big( G_i\big)=&\, p^i_{\tau(2)}- \sum_{ \alpha+\beta = \tau}
D^{\alpha}\big(\theta_i\big)D^{\beta}\big(p^i_{10}\big) \\
=&\, p^i_{\tau(2)}- \sum_{ \alpha+\beta = \tau}
D^{\alpha}\big(\theta_i\big)\, p^i_{\beta(1)}\\
=&\, p^i_{\tau(2)}-\theta_i\, p^i_{\tau(1)}+
\sum_{\kappa=1}^{ |\tau|-1}
L_\kappa^\tau(i)
\end{align*}
where
$$ L_\kappa^\tau(i) = \sum_{\substack{  |\beta|= \kappa   \\ \alpha+\beta = \tau}}D^{\alpha}\big(\theta_i\big)\, p^i_{\beta(1)}\,.$$

If $\ell\geq 0$ and  $(e_1,\ldots,e_k)$ is a basis of  $E$, then $ \mathfrak S^{(\ell)}$, the $\ell$-th symbol of ${\boldsymbol S}_{\mathcal W}$, is generated by the elements
$$\sum_{\substack{i=1,\ldots,k\\ |\sigma|=1+\ell}} \xi^i_{\sigma}\,\big(dx^{\sigma_1}\cdot dy^{\sigma_2}  \big)\otimes e_i \in
{\rm Sym}^{1+\ell}(T^* (\mathbb C^2,0) )\otimes E$$
subject  to the conditions
$$ (1)_{ab} \qquad \sum_{k=1}^d \xi^k_{ab}=0\quad  \qquad \mbox{ and  } \qquad\quad
(2)_{ab}^i \qquad \xi^i_{a,b+1}=\theta_i\,\xi^i_{a+1,b} $$
for every $i \in \underline k$  and every  $(a,b)\in \mathbb N^2$ satisfying  $a+b=1+\ell$.

The equations $(2)_{ab}^i$ are equivalent to the equations below.
$$\qquad (2')_{ab}^i \qquad \qquad \xi^i_{ab}=(\theta_i)^b\,\xi^i_{\ell+1,0}  . \qquad\qquad
%(\mbox{avec } a+b=r+1)
$$
Consequently  $  \mathfrak S^{(\ell)} $ is generated by elements of the form
$$ \sum_{i=1}^k    z^i \,\Big(
\sum_{a+b=1+\ell}(\theta_i)^b \big(dx^a\cdot dy^b \big) \Big) \otimes e_i  $$
where the $z_i$ are subject  to the conditions
$$ \qquad (1')_s\qquad    \sum_{i=1}^k z^i\, (\theta_i)^s=0 \;  $$
for  $ s \in \{ 0,\ldots,\ell+1\}$. These equations can be written is matrix  form:
$$ \big(V_\theta^{(\ell)}\big) \qquad \quad  \begin{pmatrix} 1 & \ldots & 1 \\
\theta_1 & \ldots&  \theta_k  \\
\vdots &  & \vdots \\
(\theta_1)^{\ell+1} & \dots & (\theta_k)^{\ell+1}
\end{pmatrix}\!\!
 \begin{pmatrix}
z^1 \\ \vdots  \\ z^k
\end{pmatrix}= 0. \qquad \quad
$$

Because $\mathcal W$ is smooth, the functions  $\theta_i$ have pairwise distinct values at
every point of $(\mathbb C^2,0)$. Thus $\big(V_\theta^{(\ell)}\big)$ is a linear system
of Van der Monde type, and there are only two possibilities.
\begin{enumerate}
 \item[--] If $\ell\geq k-2$ then  $\big(V_\theta^{(\ell)}\big)$ has no non-trivial solution;  In other terms $\mathfrak S^{(\ell)}=(0)$.
\item[--]
If  $\ell \in \{ 0,\ldots, k-3\}$ then  $\big(V_\theta^{(\ell)}\big)$ is a linear system of rank  $\ell+2$,
therefore  $\mathfrak S^{(\ell)}$ has dimension  $k-\ell-2$, and can be parametrized
explicitly  via the
   VanderMonde  matrix associated to functions $\theta_i$. If
$ {\mathcal
     M}_\theta=(\alpha_{ij})$ is the inverse of  ${\mathcal V}_\theta=((\theta_i)^{j-1})_{i,j=1}^k$ then
$$\mathfrak S^{(\ell)}=
\Big\langle
\sum_{i=1}^k    \alpha_{ik} \,\Big(\!\!\!
\sum_{a+b=1+\ell}\!\!\!
 (\theta_i)^b\,dx^ady^b \Big)\otimes \,e_i
\, \Big| \, k=\ell+3,\ldots,k
\Big\rangle
\, .  $$
\end{enumerate}

\subsection{Pantazi-H\'{e}naut criterium }

Set $\mathscr S_{\mathcal W}$ as the differential system ${\boldsymbol S}_{\mathcal W}^{(k-3)}$.
According to Section  \ref{S:diffprolongAW} it has  the following properties:
\begin{enumerate}
 \item It is a subbundle of $ J^{k-2}(E)$ of rank $k(k-2)/2$;
\item Its symbol $\sigma_{\mathcal W}$ is a sub-line bundle of ${\rm Sym}^{k-2}(T^*)\otimes E$;
\item The map $
\mathscr S_{\mathcal W}^{(1)} \to \mathscr S_{\mathcal W}$ is an isomorphism ( thus $\sigma_{\mathcal W}^{(1)}=0$ ).
\end{enumerate}

The discussion laid down in  Section \ref{S:spencer}
imply the existence of \defi[H\'enaut's connection] \index{H\'{e}naut's connection} of  $\mathcal W$
\begin{align*}
 \nabla_{\mathcal W}:
\mathscr S_{\mathcal W}^{(1)}
\rightarrow \Omega^1 \otimes
\mathscr S_{\mathcal W}^{(1)} \, .
\end{align*}
Its curvature is a  $\mathcal O$-linear operator
\begin{align*}
 \Theta_{\mathcal W} : \mathscr S_{\mathcal W}
^{(1)} \rightarrow \Omega^2 \otimes
\sigma_{\mathcal W}.
\end{align*}
called the \defi[Pantazi-H\'enaut curvature] \index{Pantazi-H\'{e}naut curvature} of the web  $\mathcal W$.

\begin{thm}
\label{T:PantaziHenaut}
The following assertions are equivalent.
\begin{enumerate}
 \item  $\mathcal W$ has maximal rank;
\item Pantazi-H\'{e}naut curvature
$\Theta_{\mathcal W}$ vanishes identically.
\end{enumerate}
\end{thm}

The above theorem can be understood as a wide generalization from $3$-webs to  arbitrary  $k$-webs of the equivalence between the items  {\bf (b)}  and  {\bf (c)}  of  Theorem \ref{T:hexagonal}.
\medskip

If the  functions $\theta_1,\ldots,\theta_k$ are explicit, one can easily construct an effective algorithm  ( see \cite[Appendice]{PTese} for an implementation
in \verb"MAPLE" ) which computes the curvature  $\Theta_{\mathcal W}$ as defined above. If instead of the functions $\theta_i$ one only knows an explicit expression
of a $k$-symmetric $1$-form defining $\mathcal W$, then the implicit approach implemented by H\'{e}naut  shows  the existence, but without exhibiting, of an algorithm
to compute $\Theta_{\mathcal W}$. Such algorithm is not as easy to derive as in the approach followed here. Albeit,
Ripoll spelled  out and  implemented  in \verb"MAPLE" the corresponding algorithm for $3$, $4$ and
$5$-webs presented in implicit form.

\medskip

An extensive study of H\'{e}naut's  connection $\nabla_{\mathcal W}$
remains to be done. The authors believe that a careful investigation
of its properties may lead to an answer of Problem \ref{Prob:1}. Casale's results \cite{Guy}
on the Galois-Malgrange groupoid of codimension one foliations seems to be
useful in this context.

\subsection{Mih\u{a}ileanu criterium}

Despite the relative easiness to implement the computation of ${\Theta}_{\mathcal W}$, it  seems
very difficult to interpret the corresponding matrix of $2$-forms. Nevertheless, in \cite{Mihaileanu}  Mih\u{a}ileanu,
based on  Pantazi's result and after ingenious  computations, shows that
$$
 K(\mathcal W) = \sum_{\mathcal W_3 < \mathcal W}
K(\mathcal W_3) \, ,
$$
the sum of curvatures of all $3$-subwebs of $\mathcal W$, appears as a linear combination of the coefficients of
Pantazi-H\'{e}naut curvature. From this result he obtained from
Theorem \ref{T:PantaziHenaut} the following necessary condition for rank maximality.

\begin{thm}[Mih\u{a}ileanu criterium]
If $\mathcal W$ is a planar $k$-web of maximal rank then $K(\mathcal W)$ vanishes identically.
\end{thm}

By definition, the $2$-form $K(\mathcal W)$ is the  \defi[curvature of $\mathcal W$]. \index{Web!curvature for $k$-web}
An intrinsic  interpretation of it has been recently
 provided in  \cite[Th\'{e}or\`{e}me 5.2]{RipollT} when $k\le 6$,
and  stated for arbitrary $k$ in \cite[p. 281]{Henaut2}, \cite{Ripoll}. According to them  the curvature $K(\mathcal W)$ is  nothing more than
the trace of Pantazi-H\'{e}naut curvature $\Theta_{\mathcal W}$.

\section{Classification of CDQL webs}\label{S:CDQL}

Although Pantazi-H\'{e}naut criterium together with the linearization criterium presented in Section \ref{S:lliinn}
  provide an algorithmic way to decide if a given planar web is
exceptional or not, the classification problem for these objects is wide open up-to-date.
The only classification results available so far consider the classification of ( very ) particular classes of webs.
Even worse, only two classes have been studied so far.  The first class are the germs of $5$-webs on
 $(\mathbb C^2,0)$ of the form $\mathcal W(x,y,x+y,x-y,u(x) + v(y))$  mentioned in Example \ref{E:exabel}
of Chapter \ref{Chapter:AR}. The second class of exceptional webs so far  classified are the CDQL\begin{footnote}
{
The acronym CDQL stands for {\bf C}ompletely {\bf D}ecomposable {\bf Q}uasi-{\bf{L}}inear.
}
\end{footnote} webs on
compact complex surfaces.

\subsection{Definition}

Linear webs are classically defined as the ones for which all the leaves  are open subsets of lines. Here we will adopt the following global definition.
A web $\mathcal W$ on a compact complex surface $S$ is \defi[linear] \index{Web!linear (global)}  if (a) the universal covering of $S$ is an open subset $\widetilde S$
of $\mathbb P^2$; (b) the group of deck transformations acts on $\widetilde S$ by automorphisms of $\mathbb P^2$; and (c) the pull-back of
$\mathcal W$ to  $\widetilde S$ is linear in the classical sense.

A \defi[CDQL  $(k+1)$-web] \index{Web!CDQL} on a compact complex surface $S$ is,
by definition, the superposition of $k$ linear foliations and one non-linear foliation.
It can be verified ( see \cite{CDQL} )  that the only compact complex surfaces carrying  CDQL $(k+1)$-webs when
$k \ge 2$ are  are the projective plane, the complex tori and the Hopf surfaces. Moreover
the only Hopf surfaces admitting  four distinct linear foliations  are the primary Hopf surfaces  $H_{\alpha}$,  $|\alpha| > 1$. Here $
  H_{\alpha}$ is the quotient of $\mathbb C^2 \setminus \{ 0\}$ by the map $(x,y) \mapsto (\alpha x, \alpha y)$.

The linear foliations on  complex tori are pencils of parallel lines on theirs universal coverings. The ones
on Hopf surfaces are either pencils of parallels lines or the pencil of lines through the origin of $\mathbb C^2$. In particular
a completely decomposable linear web on compact complex surface is algebraic\begin{footnote}{Beware that algebraic here means that they are locally
dual to plane curves. In the cases under scrutiny  they are dual to  products of lines.   }\end{footnote} on its universal covering.

\subsection{On the projective plane}\label{S:CDQLP2}
The classification of exceptional CDQL webs on the projective plane is summarized in the following result.

\begin{thm}\label{T:CDQL2} Up to projective
automorphisms,  there are exactly four infinite families  and thirteen  sporadic  exceptional CDQL webs on  ${\mathbb P}^2$.
\end{thm}

To describe the exceptional webs mentioned in Theorem \ref{T:CDQL2}, the notation of \cite{CDQL} will be adopted. If $\omega$ is a rational $k$-symmetric differential $1$-form then
$[\omega]$ denotes the $k$-web on $\mathbb P^2$  induced by it.

In  suitable affine coordinates $(x,y)
\in \mathbb C^2 \subset \mathbb P^2$, the four infinite families are
\begin{align*}
\mathcal A_I^k =& \;  \big[ (dx^k - dy^k)\big] \boxtimes \big[ d(xy)  \big] && \mbox{where }\,  k\ge 4 \,; \\
\mathcal A_{II}^k = &\;  \big[(dx^k - dy^k)\, (xdy - ydx) \big] \boxtimes \big[d(xy) \big] &&   \mbox{where }\, k\ge 3 \,;\\
 \mathcal A_{III}^k = &\;  \big[(dx^k - dy^k)\, dx \, dy\big] \boxtimes \big[ d(xy)  \big] && \mbox{where }\, k\ge 2\, ;\\
\mathcal A_{IV}^k  = & \; \big[ (dx^k - dy^k)\,dx \, dy \, (xdy - ydx) \big] \boxtimes \big[ d(xy) \big]     && \mbox{where }\,  k\ge 1.
\end{align*}

The diagram below shows how these webs relate to each other in terms of inclusions for a fixed $k$. Moreover if $k$ divides $k'$ then
$\mathcal A_I^k, \mathcal A_{II}^k,\mathcal A_{III}^k,\mathcal A_{IV}^k$ are subwebs of  $\mathcal A_I^{k'}, \mathcal A_{II}^{k'},
\mathcal A_{III}^{k'},\mathcal A_{IV}^{k'}$ respectively.
\begin{center}
 \begin{tikzpicture}[shorten >=1pt,->]
  \tikzstyle{vertex}=[circle,fill=gray!25,minimum size=18pt,inner sep=0pt]
  \foreach \name/\angle/\text in {P-1/180/I, P-2/90/II , P-3/270/III , P-4/0/IV    }
    \node[vertex,xshift=5cm,yshift=.5cm] (\name) at (\angle:1.3cm) {\footnotesize{$\mathcal A_{\text}^k$}};
\foreach \from/\to in {1/2,1/3,1/4,2/4,3/4}
    \draw (P-\from) -- (P-\to);
\end{tikzpicture}
\end{center}

All the webs above are invariant by the action $t\cdot(x,y) = (tx,ty)$ of $\mathbb
C^*$ on $\mathbb P^2$. Among the
thirteen sporadic examples of exceptional CDQL webs on the projective plane, seven (four $5$-webs, two
$6$-webs and one $7$-web) are also invariant by the same  $\mathbb
C^*$-action. They are:
\[
\begin{array}{lclcl}
\mathcal A_5^a &=& \big[dx\, dy\, (dx+dy)\, (xdy - ydx)  \big] &\boxtimes& \big[d\big(xy(x+y)\big) \big]\, ;
\vspace{0.10cm}
 \\
\mathcal A_5^b &=& \big[
dx\, dy\, (dx+dy)\, (xdy - ydx) \big] &\boxtimes&
\left[d\big(\frac{xy}{x+y}\big)  \right] ;\vspace{0.10cm}
\\
\mathcal A_5^c &=& \big[
dx\, dy\, (dx+dy)\, (xdy - ydx)\big] &\boxtimes&
\left[d\big(\frac{x^2 + xy + y^2}{xy\,(x+y)}\big)  \right] \, ; \vspace{0.10cm}
\\
\mathcal A_5^d &=& \big[  dx\, (dx^3 + dy^3) \big] &\boxtimes&
\left[d\big(x(x^3+y^3) \big)  \right] ;
\vspace{0.10cm} \\
\mathcal A_6^a &=& \big[
 dx\, (dx^3 + dy^3)\, (xdy - ydx) \big]
&\boxtimes& \left[d\big(x(x^3+y^3) \big)  \right] ; \vspace{0.10cm}\\
\mathcal A_6^b &=& \big[ dx\, dy\, (dx^3 + dy^3) \big] &\boxtimes& \left[d(x^3+y^3)\right] ; \vspace{0.10cm}\\
\mathcal A_7 &=& \big[  dx\, dy\, (dx^3 + dy^3)\, (xdy - ydx) \big] &\boxtimes&\left[d(x^3+y^3)  \right] .
\end{array}
\]

Four of the remaining six sporadic exceptional CDQL webs {(one $k$-web for each $k\in \{5,6,7,8\}$)}  share the same non-linear foliation
$\mathcal F$: the pencil of conics through four points in general
position.  They all have been previously known (see  \cite{Robert}).
\[
\begin{array}{lclcl}
\mathcal B_5 &=& \left[ dx \, dy  \, d\big(\frac{x}{1-y}
\big) \,  d\big(\frac{y}{1-x} \big) \right] &\boxtimes&
\left[ d
\big(\frac{xy }{(1-x)(1-y)}\big) \right];
\vspace{0.10cm}\\
\mathcal B_6 &=&
\mathcal B_5
 &\boxtimes& \big[ d \left(x + y  \right) \big] ; \vspace{0.10cm} \\
\mathcal B_7 &=&
\mathcal B_6
&\boxtimes&
\left[ d \big(\frac{x}{y} \big) \right];  \vspace{0.10cm}   \\
\mathcal B_8 &=&
\mathcal B_7
&\boxtimes&
\left[ d \big(\frac{1-x}{1-y} \big) \right] .
\end{array}
\]

The last two sporadic CDQL exceptional webs (the $5$-web $\mathcal H_5$ and the  $10$-web $\mathcal H_{10}$) of Theorem \ref{T:1} also
share the same non-linear foliation: the \defi[Hesse pencil of cubics]. \index{Hesse pencil} Recall that
this pencil is the one generated by a smooth cubic and its
Hessian and that it is unique up to automorphisms of $\mathbb P^2$. Explicitly (with   $\xi_3=\exp(2i\pi/3)$):
{\small{\begin{align*}
\mathcal H_5 =\; & \Big[ (dx^3 + dy^3 ) \, d\Big(\frac{x}{y}
\Big)  \Big] \boxtimes \Big[ d \Big(\frac{x^3 + y^3 + 1
}{xy}\Big) \Big] \vspace{0.15cm} \\
  \mathcal H_{10} = \; & \Big[  (dx^3 + dy^3 ) \prod_{k=0}^2
\left( d\Big(\frac{y-\xi_3^k }{x } \Big) \cdot
d\Big(\frac{x-\xi^k_3 }{y } \Big) \right)\Big] \boxtimes \Big[
d \Big(\frac{x^3 + y^3 + 1 }{xy}\Big) \Big]\, .
\end{align*}}}
The $10$-web $\mathcal H_{10}$ is better described synthetically: it is  the superposition of Hesse pencil of cubics and of the
nine pencil of lines with base points at the base points of Hesse pencil.
It shares a number of features with Bol's web $\mathcal B_5$. For instance, they
both have a huge group of birational automorphisms (the symmetric group ${\mathfrak S}_5$ for
$\mathcal B_5$ and Hesse's group $G_{216}$ for ${\mathcal H}_{10}$),  and their abelian relations
can be  expressed in terms of  logarithms and dilogarithms.

\medskip

Because they have parallel $4$-subwebs whose slopes have non real cross-ratio  the webs $\mathcal A_{III}^k$ for $k \ge 3$, $\mathcal A_{IV}^k$ for
$k \ge 3$, $\mathcal A_5^d, \mathcal A_6^a, \mathcal A_6^b$ and $\mathcal A_7$
do not admit  real models. The web $\mathcal H_{10}$ also does not admit a real model. There are number of ways to verify this fact.
One possibility  is to observe  that  the lines passing through two
of the nine  base points always contains a third and notice that this contradicts  Sylvester-Gallai Theorem \cite{Coxeter}:  for every finite set
of non collinear points in $\mathbb P^2_{\mathbb R}$ there exists a line containing exactly two points of the set.
All the other exceptional CDQL webs on the projective plane admit real models. For a sample see  Figure \ref{F:cdqlfig}.

\begin{figure}[ht]
\begin{center}
\begin{tabular}{ccc}
\includegraphics[width=2.9cm,height=2.9cm]{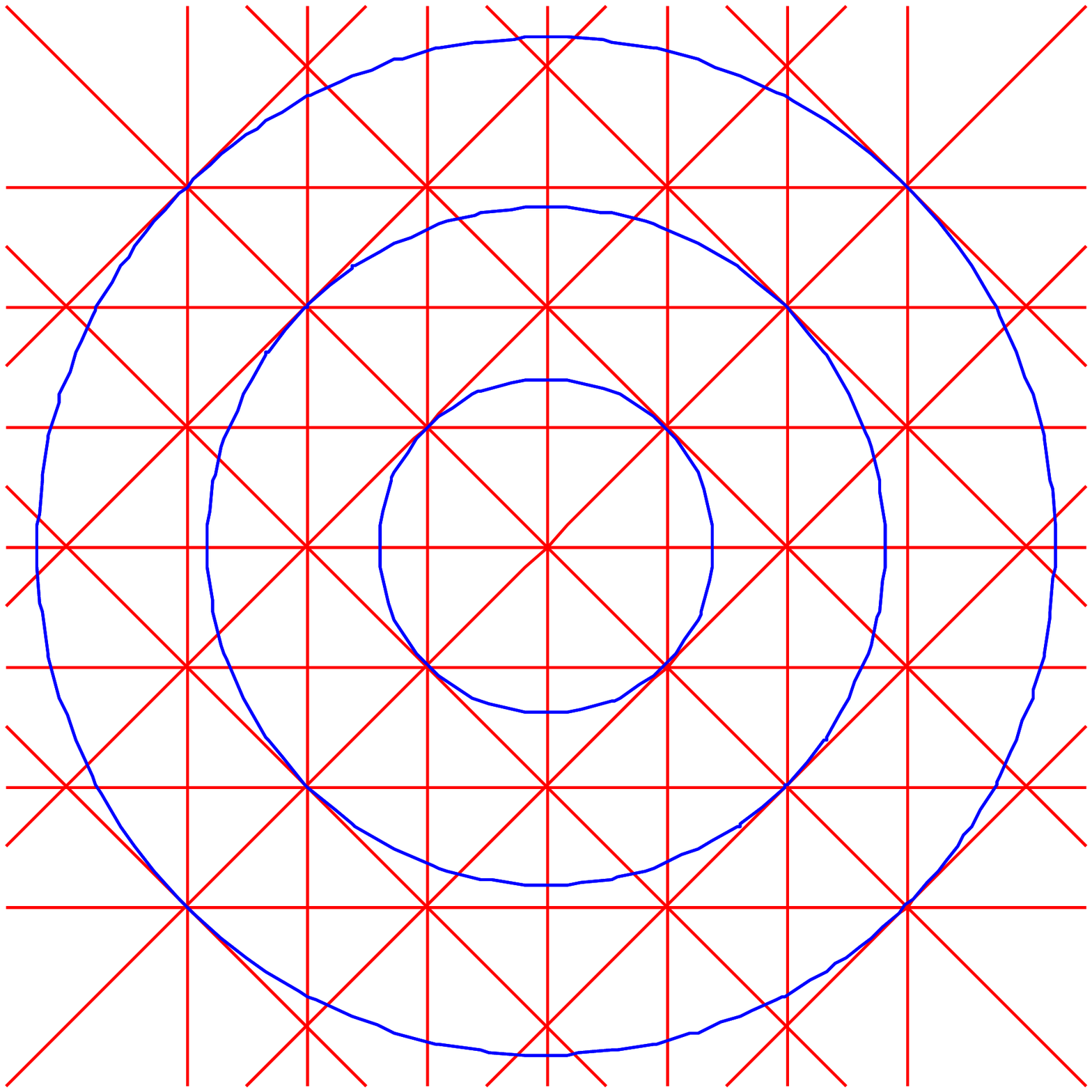} &  \includegraphics[width=2.9cm,height=2.9cm]{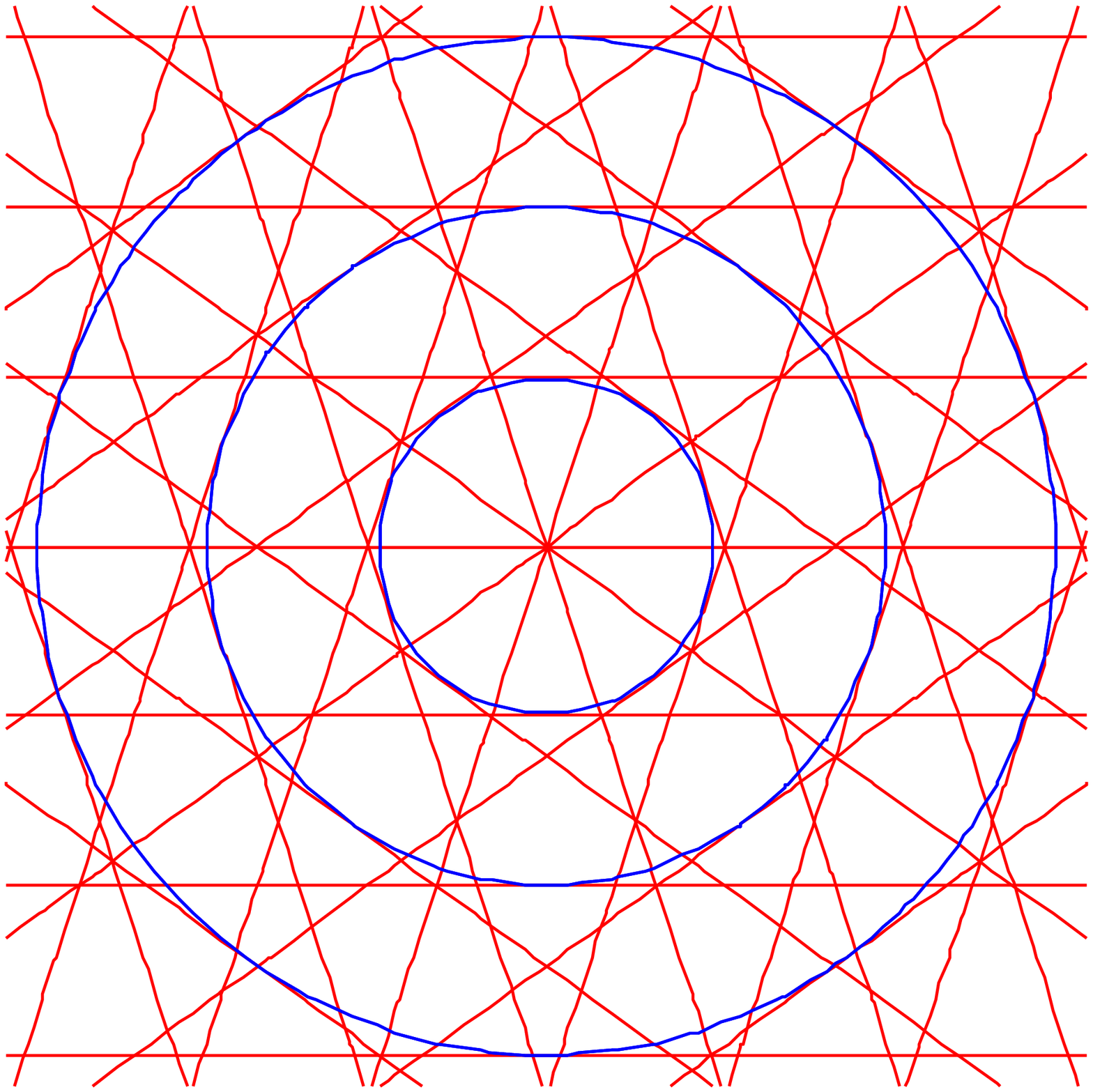} & \includegraphics[width=2.9cm,height=2.9cm]{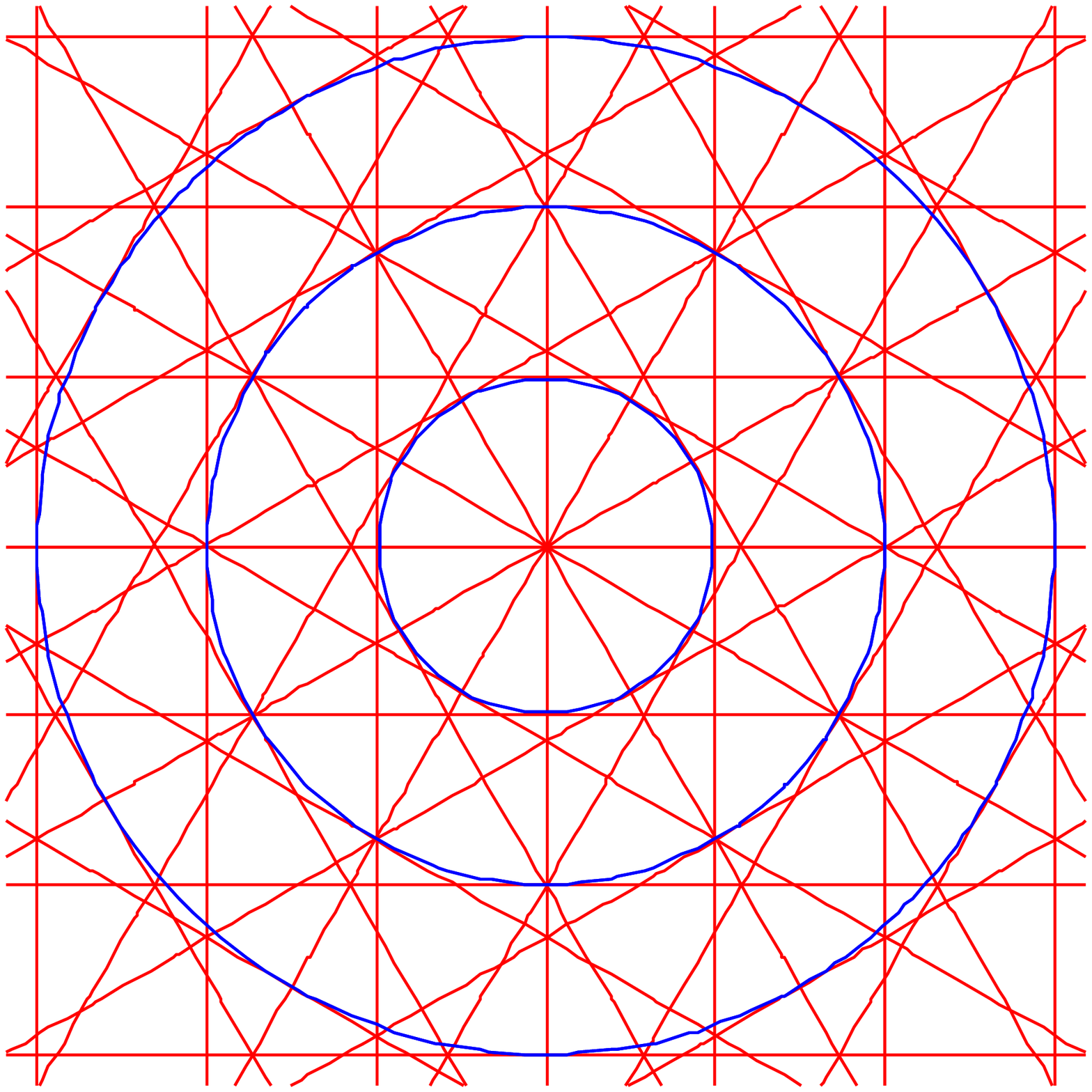} \\ 
 \includegraphics[width=2.9cm,height=2.9cm]{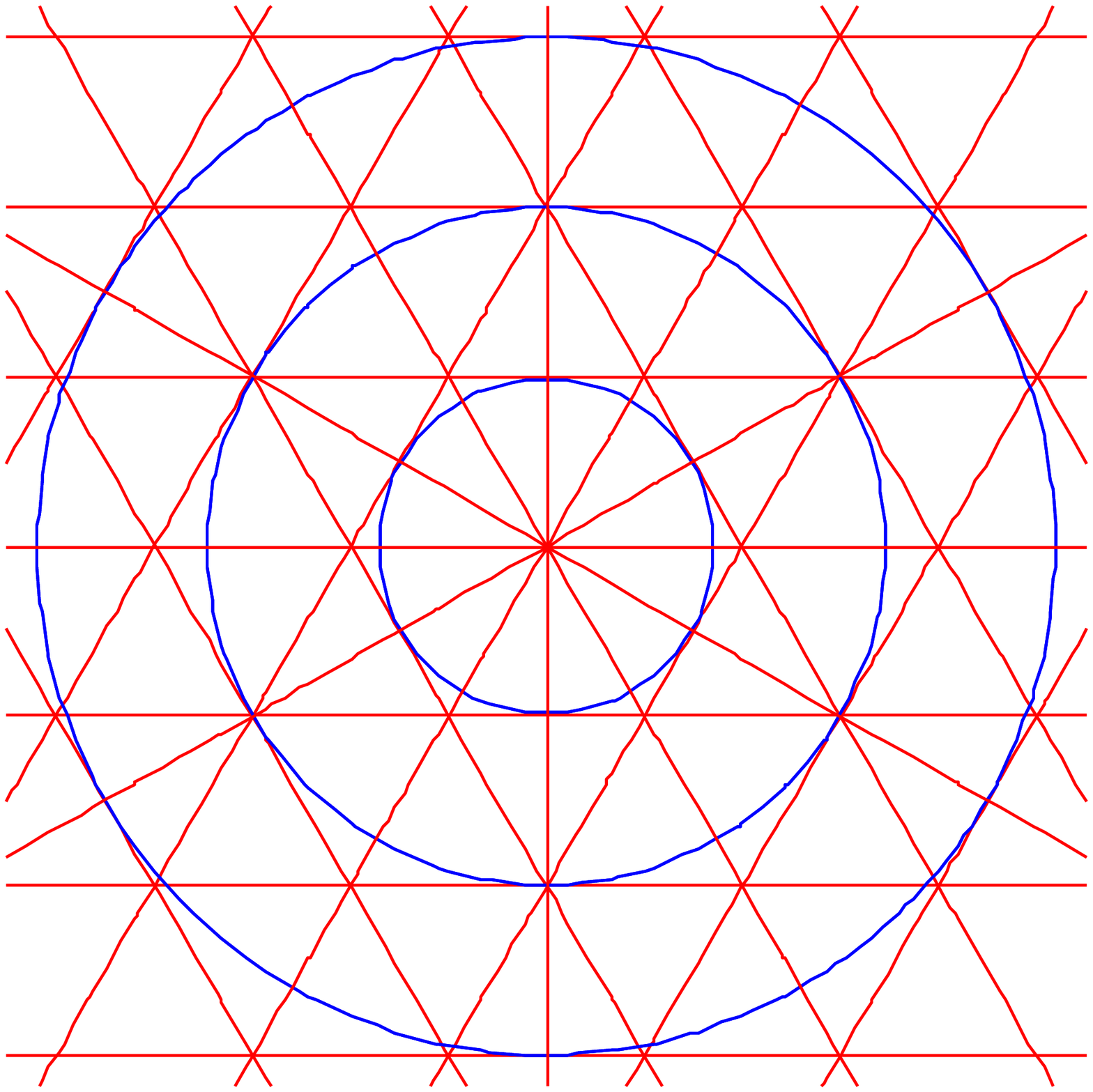} &   \includegraphics[width=2.9cm,height=2.9cm]{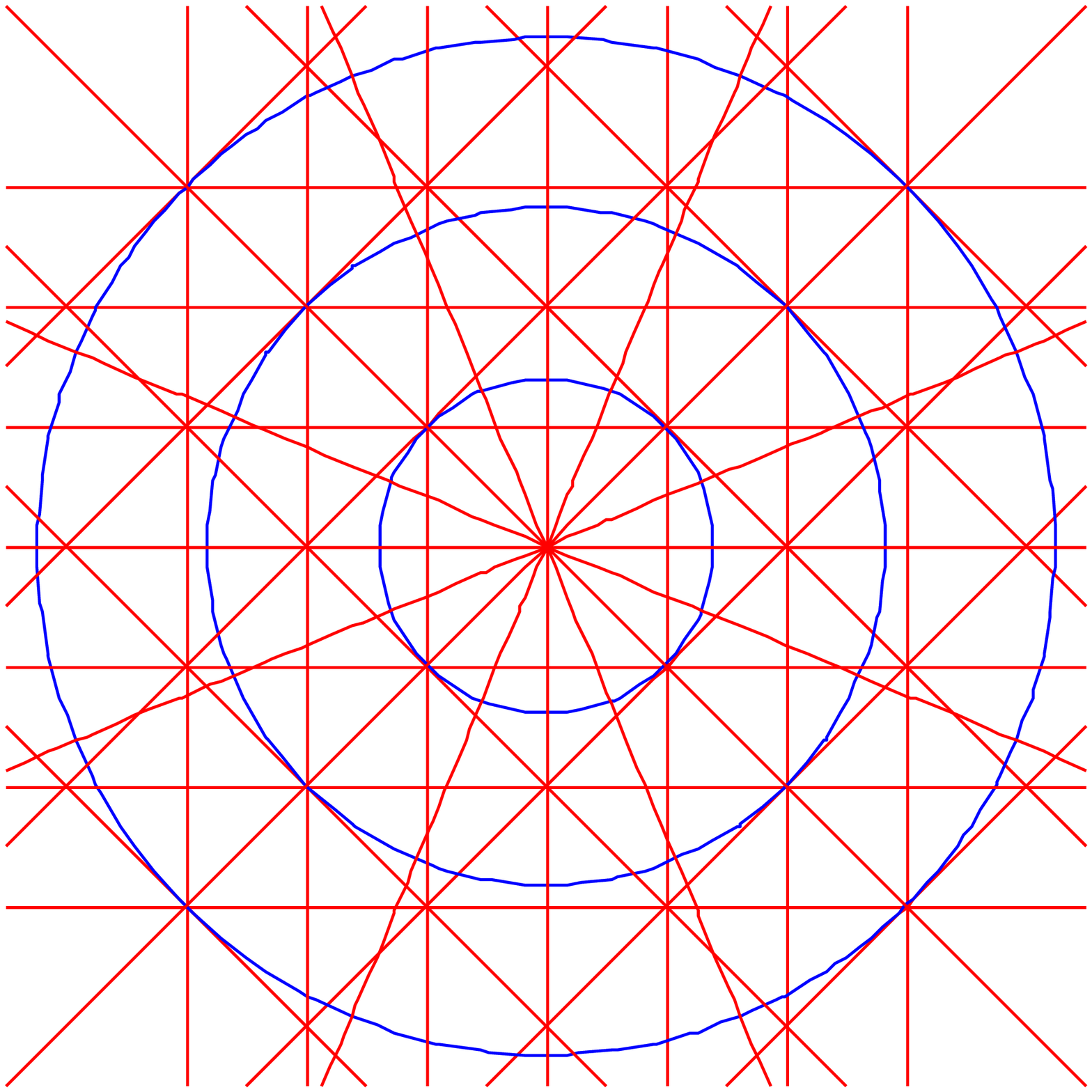} & \includegraphics[width=2.9cm,height=2.9cm]{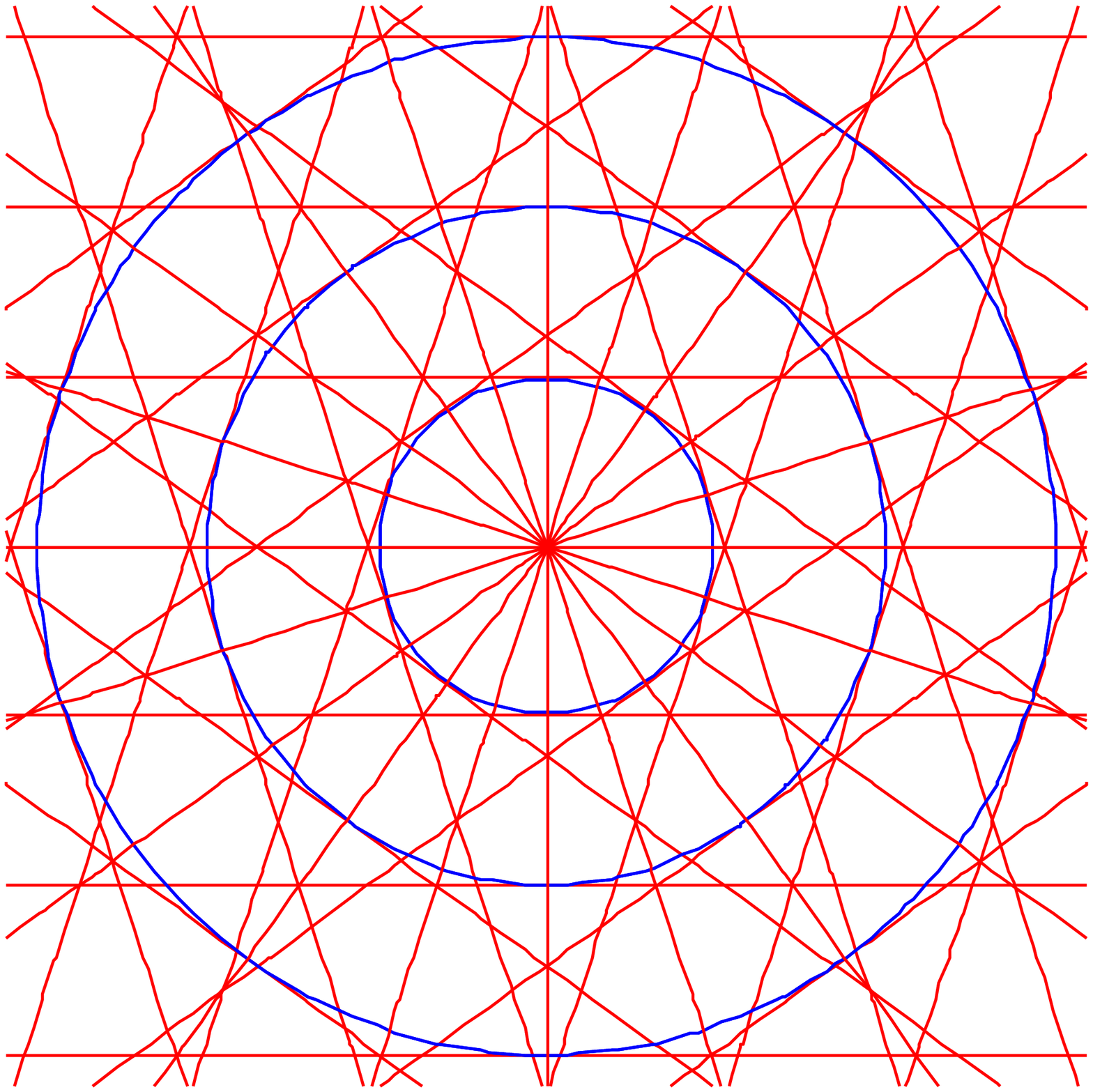} \\ 
 \includegraphics[width=2.9cm,height=2.9cm]{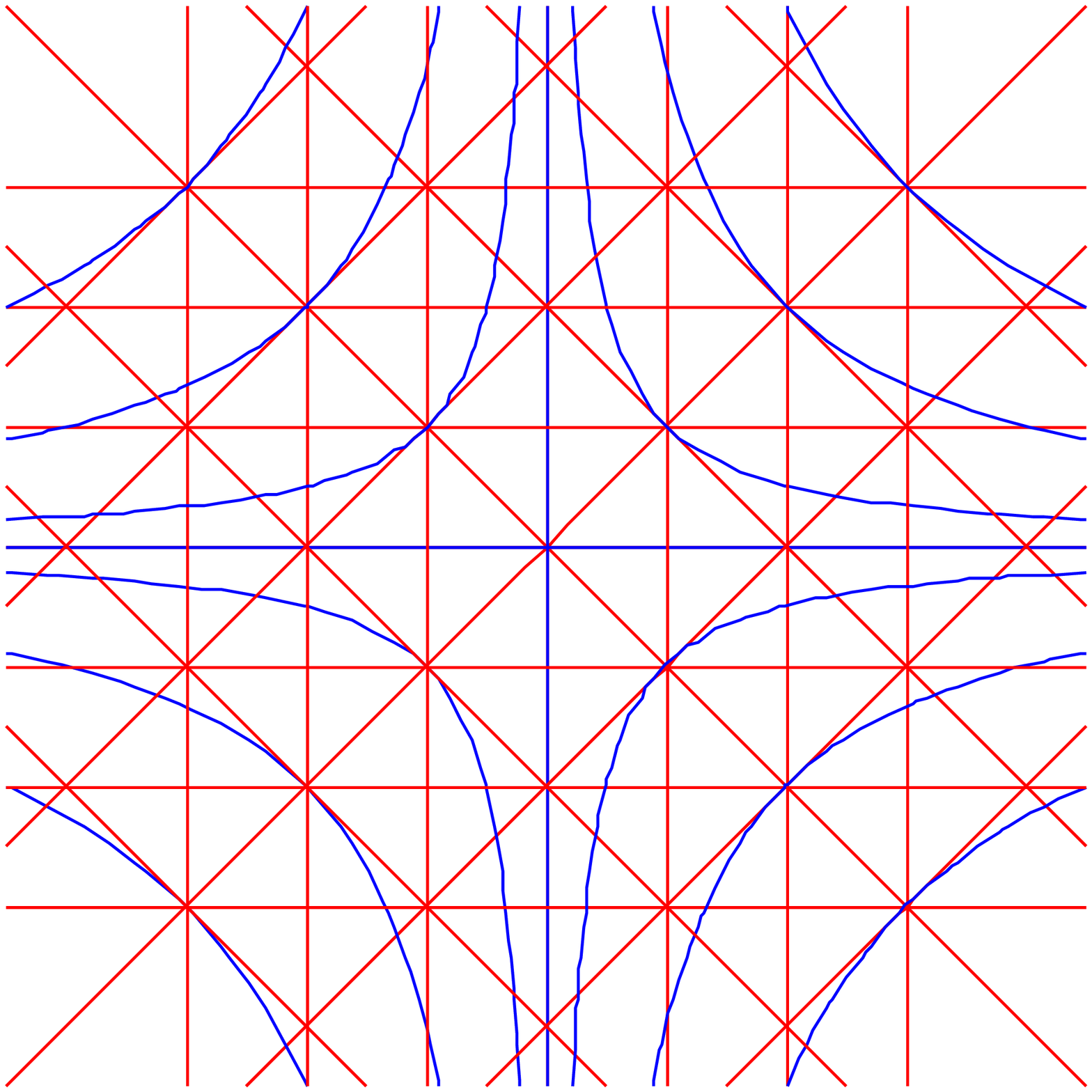}&\includegraphics[width=2.9cm,height=2.9cm]{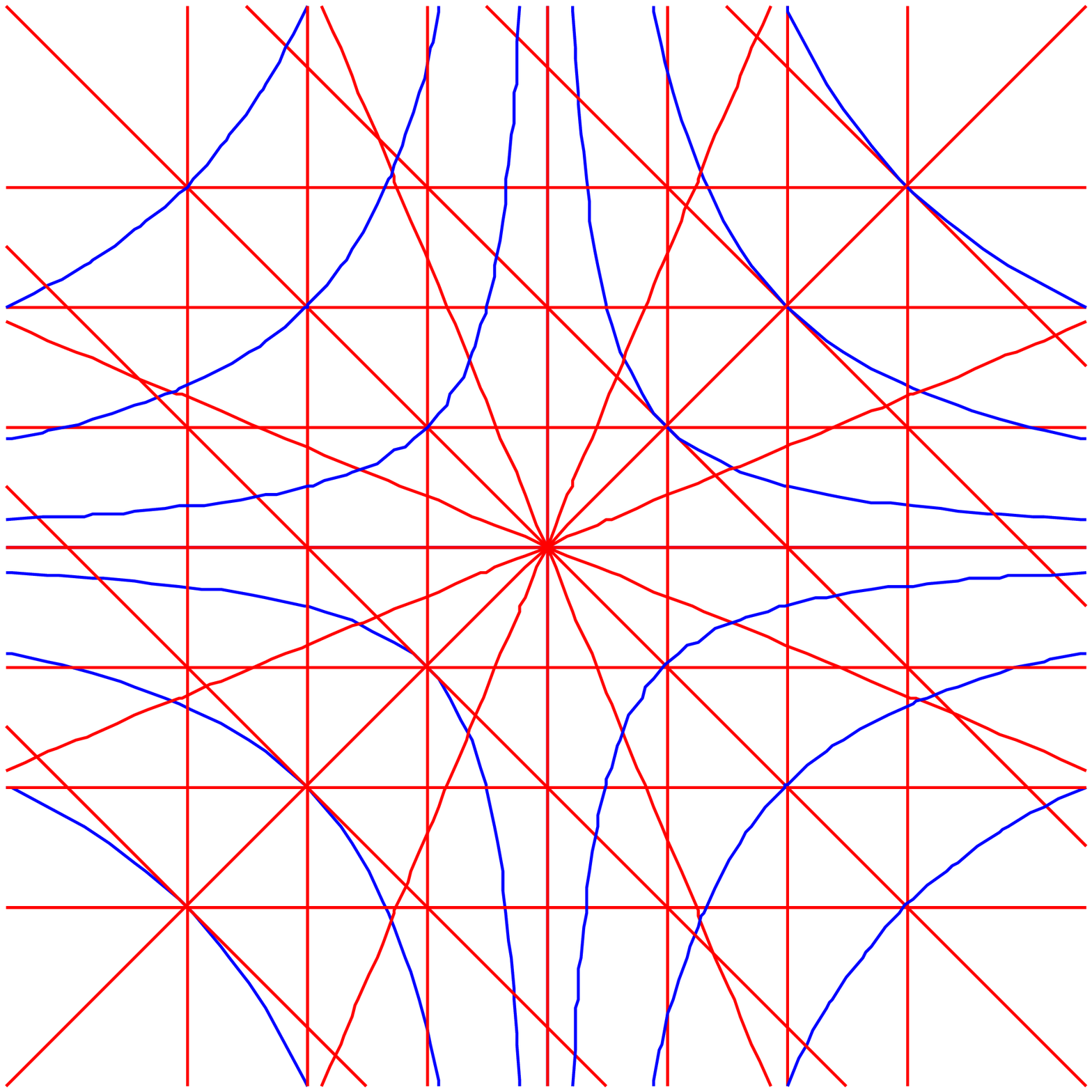}&
\includegraphics[width=2.9cm,height=2.9cm]{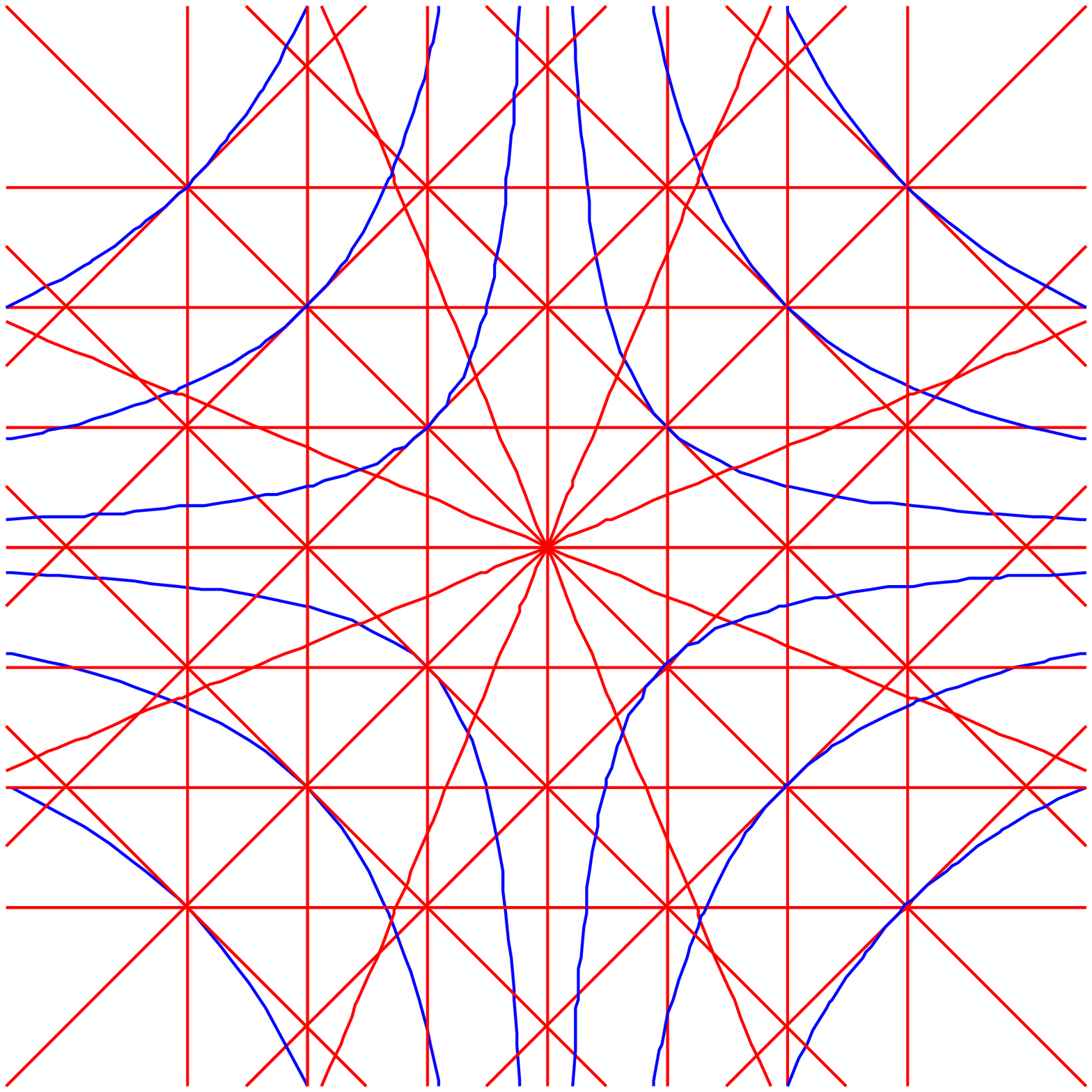} \\
\includegraphics[width=2.9cm,height=2.9cm]{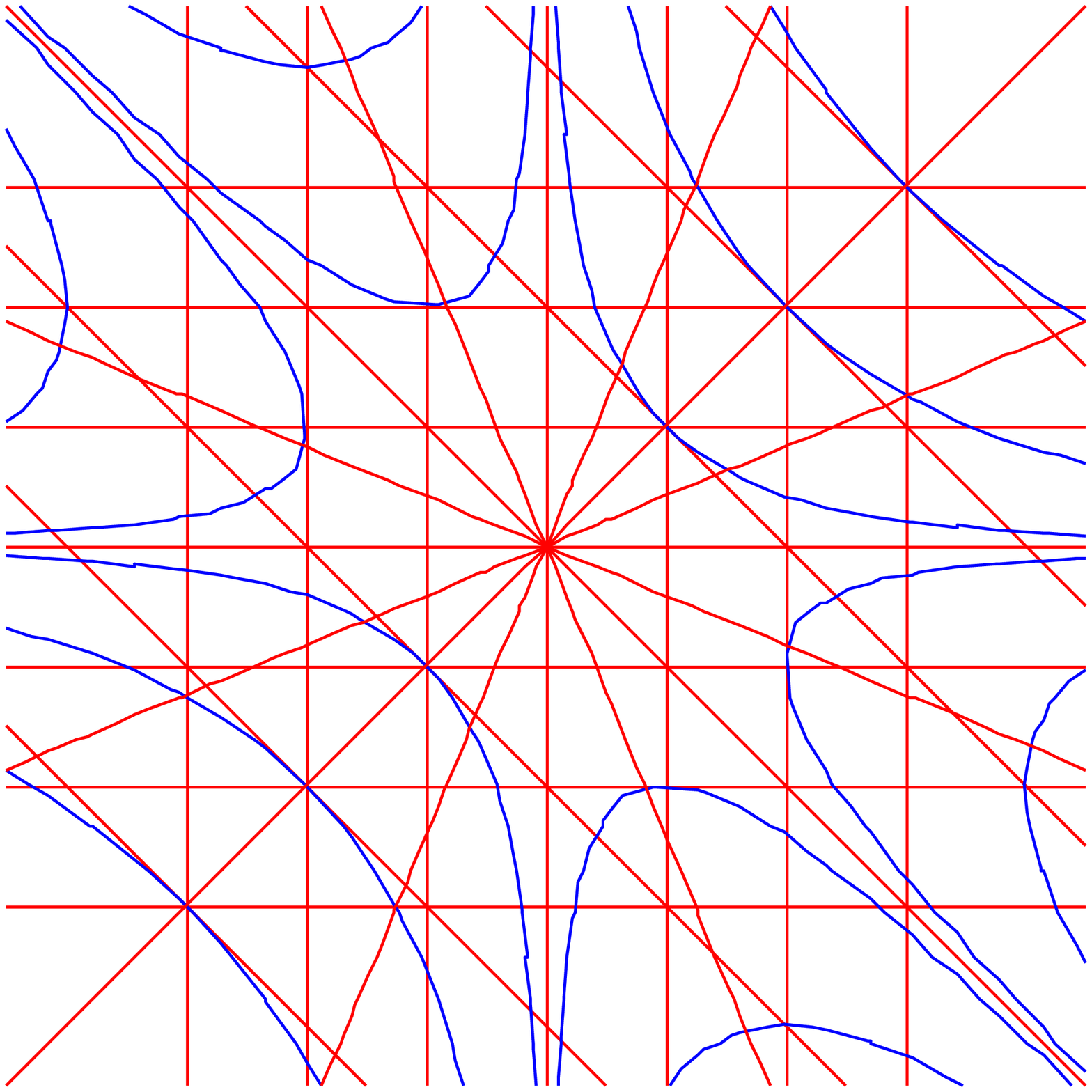}&\includegraphics[width=2.9cm,height=2.9cm]{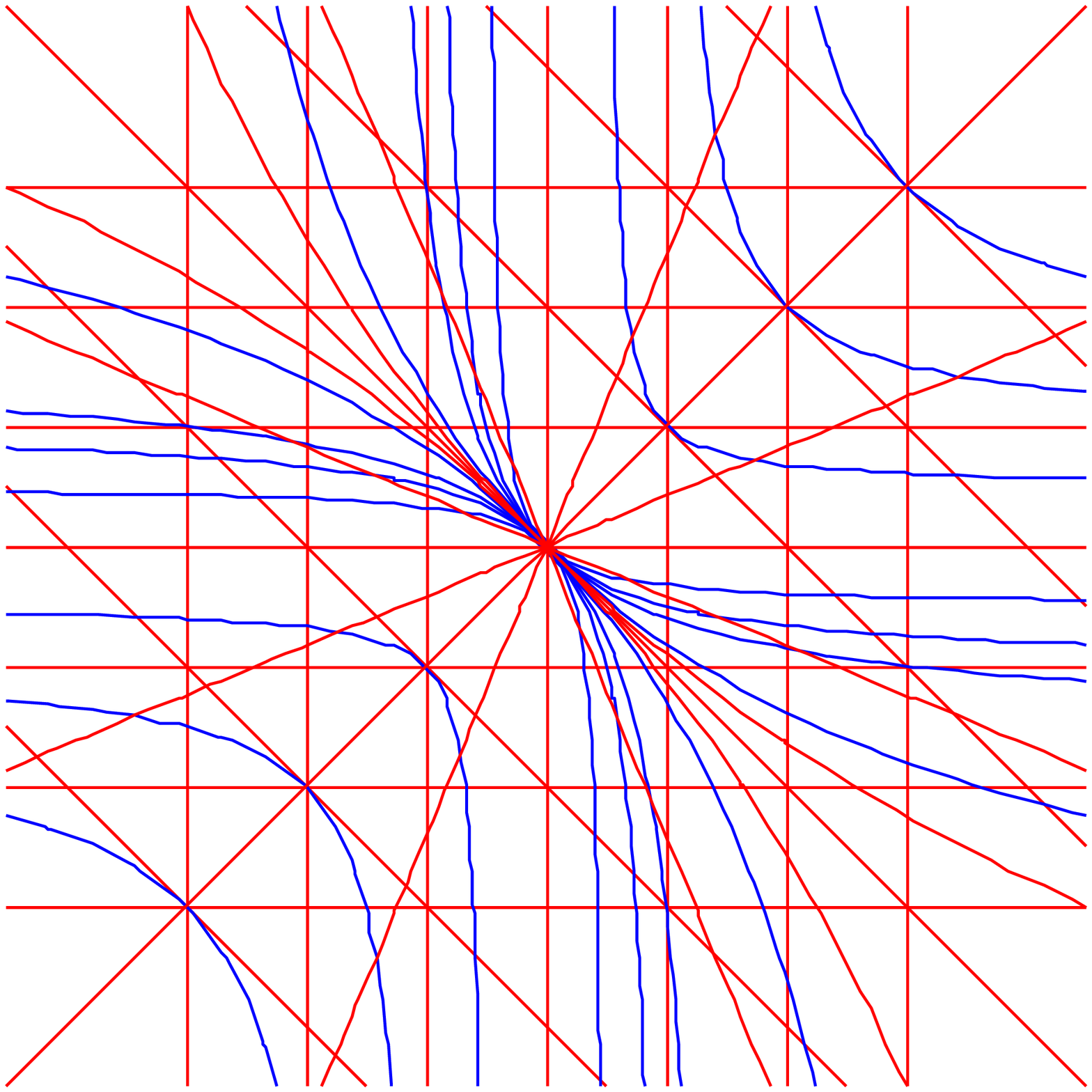}&\includegraphics[width=2.9cm,height=2.9cm]{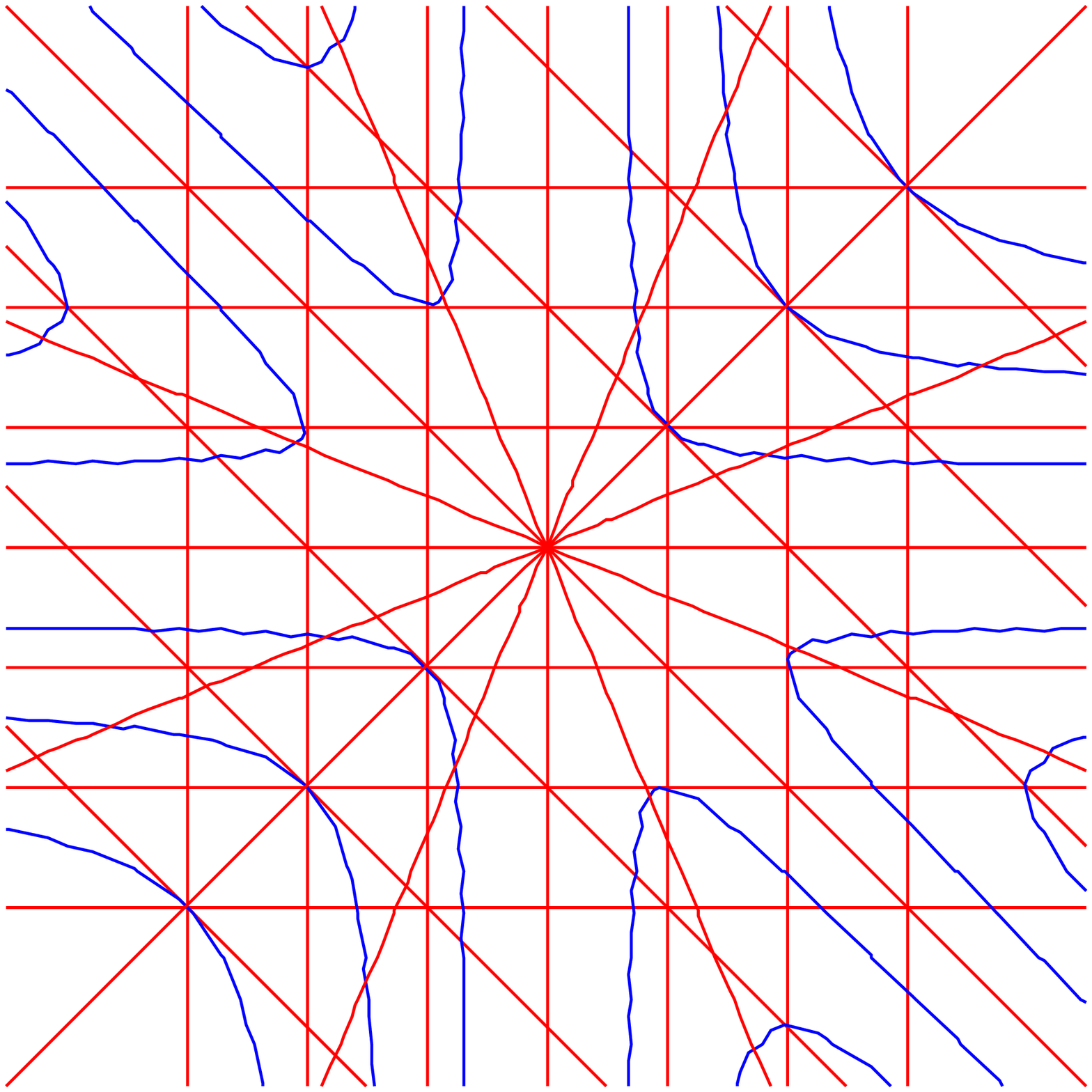}
\end{tabular} 
\caption[A sample of exceptional webs] {\label{F:cdqlfig}  A sample of the real models for exceptional CDQL on $\mathbb P^2$. In the first
and second rows the first three members of the infinite family $\mathcal A_I^k$ and  $\mathcal A_{II}^k$ respectively.
In the third row,  from left to right, $\mathcal A^2_{III}$, $\mathcal A^1_{IV}$ and $\mathcal A^2_{IV}$. In the  fourth row:
$\mathcal A_5^a,\mathcal A_5^b$ and $\mathcal A_5^c$.   }
\end{center}
\end{figure}

\subsection{On Hopf surfaces}

The classification of exceptional CDQL webs on Hopf surfaces reduces to the one on $\mathbb P^2$.
The result is far from interesting. One has just to remark that a foliation on a Hopf surface of type $H_{\alpha}$ when lifted to $\mathbb C^2 \setminus \{ 0 \}$
gives rise to an algebraic  foliation on $\mathbb C^2$ invariant by the $\mathbb C^*$-action $t\cdot(x,y) = (tx,ty)$, and then use the classification
of exceptional CDQL webs on $\mathbb P^2$ to obtain the following result.

\begin{cor}
Up to automorphisms, the only exceptional CDQL webs on Hopf surfaces are quotients   of  the restrictions of
the webs $\mathcal A^*_*$ to $\mathbb C^2 \setminus \{ 0 \}$ by the group of deck transformations.
\end{cor}

\subsection{On abelian surfaces}
The  CDQL webs on tori are  the superposition  of a
non-linear foliation with a product of foliations induced by global holomorphic $1$-forms.
Since \'{e}tale coverings between complex tori abound and  {because} the pull-back
of exceptional CDQL webs under these are still exceptional CDQL  webs, it is natural to
 extend the notion of isogenies between complex tori to isogenies between webs on tori.
Two webs $\mathcal W_1, \mathcal W_2$ on  complex tori $T_1,T_2$ are \defi[isogeneous] \index{Web!isogeneous}  if there exist  a  complex torus $T$
and \'etale morphisms $\pi_i: T \to T_i$ for $i=1,2$,
such that $\pi_1^* (\mathcal W_1) = \pi_2^* (\mathcal W_2)$.

{\begin{thm}\label{T:CDQLontori}
Up to isogenies, there are exactly
three sporadic exceptional CDQL $k$-webs (one for each $k \in \{ 5,6,7\}$) and one continuous family  of exceptional CDQL $5$-webs  on complex tori.
\end{thm}

 The elements of the continuous family are
\[
\mathcal E_{\tau} = \big[ \, dx\,  dy \, (dx^2 - dy^2) \, \big] \boxtimes \left[ d\left(\frac{\vartheta_1(x,\tau)  \vartheta_1(y,\tau) }{\vartheta_4(x,\tau)  \vartheta_4(y,\tau) }\right)^{\!\!2}\, \right] .
\]
respectively defined   on the square $(E_{\tau})^2$ of the elliptic curve
 $E_{\tau} = \mathbb C /  (\mathbb Z \oplus \mathbb Z \tau)$ for  arbitrary $\tau \in \mathbb H{=\{ z \in \mathbb C \, | \, \Im{\rm m}(z)>0\, \}}$.
The functions $\vartheta_i$ involved in the definition are the classical \defi[Jacobi theta functions], \index{Jacobi theta functions} whose definition
is now recalled.

For  $(\mu , \nu) \in \{ 0,1\} ^2$ and  $\tau \in \mathbb H$ fixed, let $\vartheta_{\mu,\mu}(\cdot, \tau)$ be the  entire function  on $\mathbb C$
\[
\vartheta_{\mu,\nu}(x,\tau) =  \sum_{n=-\infty}^{+\infty}  (-1)^{n\nu} \exp \left(  i\, \pi  \big( n + \frac{\mu}{2} \big)^2 \tau + 2\, i\, \pi  \big( n + \frac{\mu}{2} \big)x  \right)\, .
\]
These are usually called the \defi[theta functions with  characteristic]. \index{Theta function!with characteristic}
The Jacobi theta functions $\vartheta_i$  are nothing more than
\[
\vartheta_1 = -i\, \vartheta_{1,1} ,\quad \vartheta_2=\vartheta_{1,0}, \quad   \vartheta_3=\vartheta_{0,0} \quad \text{and} \quad   \vartheta_4=\vartheta_{{0},1}.
\]

The    webs $\mathcal E_{\tau}$ first appeared in  Buzano's work \cite{Buzano}   but their rank was not determined at that time. They were later
rediscovered\footnote{These are the $5$-webs mentioned in Example \ref{E:exabel}.}  in \cite{PT} where it is proved that they are all exceptional and  that
$\mathcal E_{\tau}$ is isogeneous to $\mathcal E_{\tau'}$ if and only if $\tau$ and $\tau'$
belong to the same orbit under the natural action  of the $\mathbb Z/2\mathbb Z$ extension of $\Gamma_0(2) \subset \mathrm{PSL}(2, \mathbb Z)$
generated by $\tau \mapsto -2 \tau^{-1}$. Thus the continuous  family of exceptional CDQL webs on tori is parameterized by a $\mathbb Z/2\mathbb Z$-quotient of the modular curve $X_0(2)$.

\smallskip

The sporadic  CDQL $7$-web $\mathcal E_7$ is strictly related to a particular element of the previous family. Indeed $\mathcal E_7$ is the $7$-web on  $(E_{1+i})^2$
\[
\mathcal E_7 = \big [dx^2 + dy^2\big]  \boxtimes {\mathcal E}_{1+i}\,.
\]

\smallskip

The sporadic  CDQL $5$-web $\mathcal E_5$ lives naturally on $(E_{\xi_3})^2$ and can be described as
the superposition of the linear web
\[
 \big[ dx\, dy\, (dx - dy) \, (dx + \xi_3^2\,  dy)  \big]
\]
and of the non-linear foliation
\[
 \left[ d\Big( \frac{\vartheta_1(x, \xi_3)\vartheta_1(y, \xi_3 )\vartheta_1(x-  y , \xi_3)\vartheta_1(x +  \xi_3^2 \, y, \xi_3) }{\vartheta_2(x, \xi_3)\vartheta_3(y, \xi_3 )\vartheta_4(x- y ,  \xi_3)\vartheta_3(x +  \xi_3^2 \,y, \xi_3)}\Big) \right] .
\]

\smallskip

The sporadic CDQL $6$-web $\mathcal E_6$ also lives in $(E_{\xi_3})^2$ and  is best described in terms of Weierstrass $\wp$-function.
\[
\mathcal E_6 = \big[\,  dx\,  dy \, (dx^3 + dy^3) \big] \boxtimes
\left[
{\wp(x,\xi_3)}^{-1}
 {dx} + {\wp(y,\xi_3)}^{-1}{dy}\right] .
\]

For a  more geometric description of these exceptional {\it elliptic webs} the reader is invited to  consult \cite[Section 4]{CDQL}.

\subsection{Outline of the proof}

In the remaining of this section, the proof of Theorem \ref{T:CDQL2} will be sketched. It makes use of Mih\u{a}ileanu's criterium
in an essential way. Its starting point  is the following trivial observation:  if    $K(\mathcal W)$, the curvature of $\mathcal W$,  is zero
then it must be, in particular,  a holomorphic $2$-form.

\subsubsection{Regularity of the curvature}

The tautology just spelled out, makes clear the relevance of obtaining criterium
to ensure the absence of poles of $K(\mathcal W)$.
The result obtained for that in \cite{CDQL}     is best stated in term of $\beta_{\mathcal F}( \mathcal W)$
--- the \defi[$\mathcal F$-barycenter] \index{Web!$\mathcal F$-barycenter} of a web $\mathcal W$.
If   $\mathcal W$ is a $k$-web and $\mathcal F$ is a foliation not contained in $\mathcal W$, both defined on a surface $S$ then
at a generic point $p \in S$ the tangents of $\mathcal F$ and $\mathcal W$ determine $k+1$ points in $\mathbb P (T_pS)$. The complement
of the point $[ T_p \mathcal F]$ in $\mathbb P (T_pS)$ is clearly isomorphic to $\mathbb C$, and any two distinct isomorphisms
differ by an affine map. The $\mathcal F$-barycenter of $\mathcal W$ is then defined as the  foliation on $S$ with tangent
at a generic point $p$ of $S$ determined by the barycenter of the directions determined by $\mathcal W$ at $p$ in the affine
structure on $\mathbb C \simeq \mathbb P (T_pS) \setminus [T_p \mathcal F]$ determined by $\mathcal F$.

\begin{thm}\label{TT:curvatura}
Let $\mathcal F$ be a foliation and  $\mathcal W= \mathcal F_1
\boxtimes \cdots \boxtimes \mathcal F_k$ be a
$k$-web, $k \ge 2$, both defined on the same domain $U \subset
\mathbb C^2$. Suppose that $C$ is an irreducible component of
$\mathrm{tang}(\mathcal F, \mathcal F_1)$ that
  is not contained in $\Delta(\mathcal W)$. The curvature
 $K(\mathcal F \boxtimes \mathcal W)$ is holomorphic over a generic point of $C$ if and only if the curve
  $C$ is $\mathcal F$-invariant or $\beta_{\mathcal F}(
\mathcal W')$-invariant, where $\mathcal W'=\mathcal F_2 \boxtimes
\cdots \boxtimes \mathcal F_k$.
\end{thm}

This result is the key tool used in \cite{CDQL} to achieve the classification. Having it at hand, the next step is to
the describe the $\mathcal L$-barycenters of completely decomposable linear (\defi[CDL]) \index{Web!CDL}  webs with respect to a linear foliation $\mathcal L$.

\subsubsection{$\mathcal L$-barycenters of CDL webs}

A linear foliation $\mathcal L$  on $\mathbb P^2$ is nothing more than a pencil of lines. Thus,  it is determined
by its unique singular point. So, let $p_0 \in \mathbb P^2$ be the
point determining  $\mathcal L$  and $\{ p_1, \ldots, p_k \} \subset \mathbb P^2$ be the
set of points determining a CDL $k$-web $\mathcal W$.
In order to  describe the  $\mathcal L$-barycenter of $\mathcal W$,
let  $\Pi: S \to \mathbb {\mathbb P}^2$ be the blow-up of $\mathbb P^2$ at  $p_0$; $E$ its exceptional divisor;  $\pi: S \to \mathbb P^1$
be the fibration on $S$ induced by the lines through $p_0$;  $\mathcal G$ be the foliation
$\Pi^* \beta_{\mathcal L}(\mathcal W)$; and  $\ell_i$ be the strict transform of the
line  $\overline{p_0p_i}$ under $\Pi$  for $i=1,\ldots,k$.

\begin{figure}[H]
\begin{center}
\resizebox{2.0in}{1.5in}{
\includegraphics{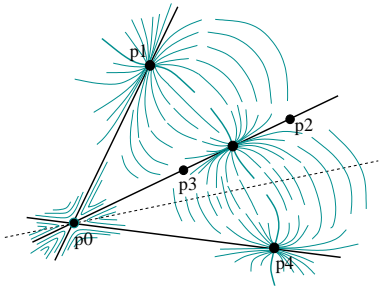}}
\end{center}
\caption{The $\mathcal L$-barycenter of a CDL web $\mathcal W$.}
\end{figure}%\vspace{0.2cm}

If  the points $\{p_0, \ldots, p_k\}$ are aligned then
$\beta_{\mathcal F}(\mathcal W)$ is also a pencil of lines with base point
at the $p_0$-barycenter of $\{ p_1, \ldots, p_k \}$. If instead the points $\{p_0, \ldots, p_k\}$
are not aligned then a simple computation shows that
the foliation $\mathcal G$ is a Riccati foliation with respect to $\pi$, that is, $\mathcal G$ has no tangencies
with the generic fiber of $\pi$. In fact, a much more precise description of  $\mathcal G$ can be obtained. In
\cite[Lemma 6.1]{CDQL} it is shown that $\mathcal G$ has the  following properties:
\begin{enumerate}
\item the exceptional divisor $E$ of $\Pi$ is $\mathcal G$-invariant;
\item the only $\mathcal G$-invariant fibers of $\pi$   are the lines $\ell_i$, for
$i \in \underline k$;
\item the singular set of $\mathcal G$ is contained in the union $ \bigcup_{i\in \underline k}\ell_i$;
\item over each line $\ell_i$, the foliation $\mathcal G$ has two singularities. One is a complex saddle
at  the intersection of $\ell_i$ with $E$, the other is a complex node at the $p_0$-barycenter
of  $\mathcal P_i=\{p_1, \ldots, p_k\} \cap \ell_i$. Moreover,  if $r_i$ is the  cardinality of $\mathcal P_i$ then
the quotient of eigenvalues of the saddle (resp. node) over $\ell_i$ is $-r_i/k$ (resp.  $r_i/k$);
\item the monodromy of $\mathcal G$ around $\ell_i$ is finite  of order $k/\gcd(k,r_i)$;
\item the  separatrices of $\beta_{\mathcal F}(\mathcal W)$ through $p_0$ are
 the lines $\overline{p_0p_i}$, $i\in \underline k$.
\end{enumerate}

It is interesting to notice that the generic leaf of $\beta_{\mathcal F}(\mathcal W)$ is
transcendental in general. Indeed, the cases when there are more algebraic
leaves than the obvious ones (the lines $\overline{p_0p_i}$)  are conveniently characterized by \cite[Proposition 6.1]{CDQL},
which says that
the foliation $\beta_{\mathcal F}(\mathcal W)$ has an algebraic leaf
distinct from the lines $\overline{p_0 p_i}$ if and only if all
its singularities distinct from $p_0$ are aligned. Moreover if this is the
case all its leaves are algebraic.

\subsubsection{The $\ell$-polar map and bounds for the degree of $\mathcal F$}
For a set $\mathcal P$  of $k$ distinct points in $\mathbb P^2$, let $\mathcal W(\mathcal P)$ be
the CDL $k$-web on the projective plane formed by the pencils of lines with base points at  the points of $\mathcal P$.

Once the description of the $\mathcal L$-barycenters of CDL webs have been laid down, the next step is to
use it to obtain constraints on the non-linear foliation $\mathcal F$ and on the position of the points $\mathcal P$
in case $\mathcal W(\mathcal P) \boxtimes \mathcal F$ has zero curvature.

It is not hard to show ( see \cite[Section 8]{CDQL} ) that when  the cardinality of $\mathcal P$ is at least $4$, either
(a)  there are three  aligned  points in ${\mathcal P}$; or (b)
${\mathcal P}$ is a set  of 4 points in general position and ${\mathcal F}$  is the pencil of conics
through them.

When in case (b) there is not much left to do, since $\mathcal F \boxtimes \mathcal W(\mathcal P)$ is nothing more than
Bol's $5$ web; in case (a) one is naturally lead to consider a line $\ell$ containing  $k_{\ell}$ points of $\mathcal P$, with $k_{\ell}  \ge
3$;  and the pencil $V = \{\mathrm{tang}(\mathcal F, \mathcal L_p)\}_{p \in \ell}$ of  polar curves of $\mathcal F$ centered at $\ell$.
It can be shown that $\ell$ is a fixed component of $V$ ( in other words $\ell$ is $\mathcal F$-invariant );  and the restriction
of $V - \ell$  to $\ell$ defines a non-constant rational map $f : \ell \simeq \mathbb P^1 \to \mathbb P^1$.
The map $f$ is
characterized by the   following equalities  between divisors on $\ell$
$$f^{-1}(p) = \Big(\mathrm{tang}(\mathcal F, \mathcal L_p) -
\ell\Big)\Big\vert_{ \ell} \, , $$
with  $p \in \ell$  arbitrary. The map $f$ is called the \defi[$\ell$-polar map of $\mathcal F$]. \index{$\ell$-polar map}

Once all these properties of $f$ are settled, it follows from a simple application of
Riemann-Hurwitz formula that  the degree of $\mathcal F$ is at most four.  Moreover, if
$\deg(\mathcal F) \geq  2$
then ${ k_{\ell}   \le 7 -\deg(\mathcal F) }$.

\subsubsection{The final steps}

At this point the proof has a two-fold ramification. In one branch one is lead to consider foliations
of degree one and put to a good use the acquired knowledge on the structure of the space of abelian relations of web admitting infinitesimal
automorphisms, see \cite[Section 9]{CDQL}. In the other branch, one first derive  from the structure
of the $\mathcal L$-barycenter of CDL webs, the normal forms  for the  $\ell$-polar map of $\mathcal F$  presented in  \cite[Table 1]{CDQL}. Then, the
proof goes by a case by case analysis, see \cite[Section 10]{CDQL}. While the arguments can be considered as elementary, they are too involved to be detailed here.

\section{Further examples}\label{S:Bestiarium}

In this last section, different examples of exceptional webs are listed. A part from the fact that they are all exceptional
webs, there is no general directrix. The reason behind this chaotic exposition is the lack of a general framework encompassing
all known the exceptional webs.

\subsection{Polylogarithmic webs}

It is well known that the  \defi[polylogarithms] \index{Polylogarithm}
\[
\l{n}(z)= \sum_{i=1}^{\infty} \frac{z^i}{i^n}
\]
 satisfy two variables functional equations at least when $n$ is small, see \cite{Lewin}. \smallskip

For instance,  Spence (1809) and Kummer (1840) have independently established some variants of the
following functional equation, nowadays called  \defi[Spence-Kummer equation],\index{Spence-Kummer equation}  satisfied by the  \mbox{trilogarithm} $ \l{3} $
\begin{align}
 2  \, \l{3} (x ) &   +   2\,  \l{3} (y )   -
    \l{3} \big( \, \frac{x}{y} \,\big)
+  2 \,  \l{3}\big(\,  \frac{1-x}{1-y}\, \big)    +   2\, \l{3} \big(\,
  \frac{x(1-y)}{y(1-x)} \,\big)
 \qquad \qquad \nonumber \\
  -  \l{3}  ( xy )   &   +     2\,
 \l{3} \big( \,\frac{x(y-1)}{{}(1-x)}           \,\big)
+   2   \, \l{3} \big(\,   \frac{{}(y-1)}{y(1-x)} \,\big)   -   \l{3}
  \big(\, \frac{x(1-y)^2}{y(1-x)^2} \,\big)
 \qquad      \nonumber \\
  &=2\, \l{3}(1)-\log(y)^2\log \big( \,\frac{1-y}{1-x}\,\big)
+\frac{\pi^2}{3}\log (y)+\frac{1}{3}\log(y)^3 \;
  \nonumber
\end{align}
when  $x, y $ are  real numbers subject to the constraint   $0<x<y<1$.

\medskip

Kummer proved that the tetralogarithm and the
pentalogarithm verify similar equations. If  $
      \zeta=1-x$ and $ \eta=1-y$ with   $0<x<y<1$, $x,y \in \mathbb R$ then the tetralogarithm
$\l{4} $ satisfies  the equation  $\mathcal K(4)$, written down  below:
\begin{align}
&\l{4}   \big(\,-\frac{x^2y\eta}{\zeta} \,   \big)+  \l{4}  \big(\,-\frac{y^2x\zeta}{\eta}\, \big) +
\l{4}  \big(\,\frac{x^2y}{\eta^2\zeta} \,   \big) +\l{4}  \big(\,\frac{y^2x}{\zeta^2\eta}\,
 \big) \nonumber \\
 & -6\, \l{4} (\,xy  \,)  -6\,\l{4}  \big( \,\frac{xy}{\eta \zeta} \,  \big)
-6\,\l{4}  \big(\,-\frac{xy}{\eta} \,   \big)
-6\, \l{4}  \big(\,-\frac{xy}{\zeta} \,   \big)
\nonumber \\
& -3\,\l{4} (\,x \eta \, )  -3\, \l{4} (\,y\zeta\,   ) -3\,\l{4} (\,\frac{x}{\eta} \,  )
-3 \,\l{4}  \big(\,\frac{y}{\zeta}\,    \big)   \nonumber \\
 & -3\, \l{4}  \big(\,-\frac{x \eta}{\zeta}\,    \big)
-3\,\l{4}  \big(\,-\frac{y\zeta}{\eta} \,   \big)
-3\,\l{4}  \big(\,  -\frac{x}{\eta \zeta} \,   \big)
-3 \, \l{4}  \big(\,  -\frac{y}{ \eta \zeta  }\,    \big)  \nonumber \\
 & +6\, \l{4} (\, x\,  )  +6\, \l{4} (y)  +6\, \l{4}
  \big(\,-\frac{x}{\zeta}\,   \big)
+6\, \l{4}  \big(\, -\frac{y}{\zeta} \,   \big)
  \nonumber \\
&= - \frac{3}{2} \log^2\,(\zeta) \log^2\,(\eta)\; .
\nonumber
\end{align}

The pentalogarithm
$\l{5} $ satisfies an equation of the same type, which will be referred as
${\cal K}(5)$. It involves  more than thirty terms  and  will not be written down to save space.
\medskip

All the known functional equations, in two variables, satisfied by the
classical      polylogarithms $\l{n} $ are of the form
\begin{equation}
\label{E:polylogGEN}
 \sum_{i=1}^N \, c_i \; \l{n} \big(U_i\big)\equiv
{\bf E}{\rm lem}_n
\end{equation}
where   $c_1,\ldots,c_N$ are integers; $U_1,..,U_N$ are rational functions; ${\bf E}{\rm lem}_n$ is of the form $ P (  \,\l{{k_1}}
\circ V_1, \ldots, \l{ {k_m}} \circ V_m)  $ with $P$ being a polynomial and   $V_1,\dots,V_m$ being rational functions;
and  $k_1,\ldots,k_m$ are integers satisfying   $1 \leq k_i <n$ for every  $i \in \underline m$. \smallskip

Of  relevance for web geometry are the webs defined by the functions $U_i$ appearing in   (\ref{E:polylogGEN}).
The presence of a non-vanishing  righthand side  ${\bf E}{\rm lem}_n $
is an apparent  obstruction to interpret  (\ref{E:polylogGEN}) as an abelian relation of the
web defined by the functions  $U_i$.  This difficulty can be bypassed
because the classical polylogarithms have univalued  ``cousins'', denoted by  ${\cal L}_n$, globally defined on
$\mathbb P^1$ which satisfy,  globally on  $\mathbb P^2$,  homogeneous analogues of every equation of the form
 (\ref{E:polylogGEN}) locally satisfied by  $\l{ n} $. \smallskip

For $n\geq 2$, the   \defi[$\mathbf n$-th modified polylogarithm] \index{Polylogarithm!modified} is the function
$$ \qquad \quad  {\cal L}_n(z ) = \Re_m \Big( \;  \sum_{k=0}^{n-1} \; \frac{ 2^k \,
    B_k}{k!} \, {\log|z|}^k \, \l{ {n-k}}(z) \Big)  $$
for  $z\in \mathbb C\setminus\{0,1\}$\footnote{In this definition, $\Re_m$ stands for the real part if  $n$ is odd otherwise
it is the imaginary part;  $B_k$ is  $k$-th Bernoulli number: $B_0=1$, $B_1=-1/2$, $B_2=1/6$, etc.}.
It can be shown that these functions are well defined real analytic functions on  $\mathbb C \setminus\{0,1\} $.
They can be continuously extended to the whole projective line  $\mathbb P^1$ by setting
${\cal L}_n(0)=0$, ${\cal  L}_n(\infty)=0$, and  ${\cal L}_n(1)=\zeta(n)$
    for  $n$ odd,  ${\cal L}_n(1)=0$ for $n$ even.

\bigskip

The existence of univalued versions of polylogarithms have been established by several
authors.
The ones introduced here have the peculiarity of satisfying
{\it clean} versions of the  functional equations for the
classical polylogarithms presend above. \medskip

\begin{thm}
\label{T:eqpolylog}
Let $n\ge 2$. The  following assertions
are equivalent:\vspace{0.2cm}
\begin{itemize}
\item[(a)] there exists a simply connected open subset of $\mathbb P^2$ where the functional
equation
$   \sum_{i=1}^N  c_i \,
\l{n} (U_i)={\bf E}{\rm lem}_n  $  holds true;
\item[(b)] the expression
$\sum_{i=1}^N  c_i \,
  \log|U_i|^{n-k}     {\cal L}_k (U_i)$  is constant on
$\mathbb P^2$  for  $k=2,\ldots,n$.
\end{itemize}
\end{thm}

For a proof of this result the reader is redirected to  \cite{oesterle92} and \cite[pages 45--46]{colmez}.
It implies that the web  $\mathcal W_{{\mathscr E}_n}$  associated to a polylogarithmic relation
${{\mathscr E}_n}$ as (\ref{E:polylogGEN}) admits {\it polylogarithmic abelian relations} and
hence are  susceptible of having high rank.\smallskip

For instance, according to  Theorem \ref{T:eqpolylog}, the function ${\cal L}_3$
verifies on  $\mathbb P^2$ the homogeneous version of
Spence-Kummer equation. For every  $x,y \in \mathbb R$, the following identity holds true.
\begin{align}
2\, {\cal L}_3 \, (x) &+
  2\, {\cal L}_3 \, (y)
 - \,{\cal L}_3 \, \big(\frac{x}{y} \big)+
  2\, {\cal L}_3 \, \big(\frac{1-y}{1-x} \big)\nonumber \\
&+2\,{\cal L}_3 \, \Big(\frac{x(1-y)}{y(1-x)} \Big) - \, {\cal L}_3 \, (xy)+2\,{\cal L}_3 \,  \Big(-\frac{x(1-y)}{1-x} \Big)
\nonumber \\
&+ 2\, {\cal L}_3 \,  \Big(-\frac{1-y}{y(1-x)} \Big)-\, {\cal L}_3 \,   \Big(\frac{x(1-y)^2}{y(1-x)^2} \Big)
= \, \frac{\zeta(3)}{2}\;.  \nonumber
\end{align}

This equation can be complexified, and the result after differentiation   gives rise to  an abelian relation
for the \defi[Spence-Kummer web] \index{Spence-Kummer web} $\mathcal W_{SK}$
defined as
$${\cal W}\Big( x
\,  ,     y \,,  xy\,  ,
\frac{x}{y}\, ,   \frac{1-x}{1-y}\, , \,  \frac{x(1-y)}{y(1-x)}\, ,
 \, \frac{x(1-y)}{{\,}(1-x)}  \, ,  \,
\frac{{\,}(1-y)}{y(1-x)}\, , \, \frac{x(1-y)^2}{y(1-x)^2} \Big).
 $$

This web seems to be to  Spence-Kummer equation for the  trilogarithm, what  Bol's web is to Abel's equation for the dilogarithm.
It was   H\'{e}naut in \cite{Henaut0}, that recognized it as a good candidate for being an exceptional  9-web.
This was later settled independently by the second author \cite{Piriopoly} and  Robert  \cite{Robert}.
It has to be emphasized that back then in  2001,
$\mathcal W_{SK}$ was the first  example of planar exceptional web to come to light after  Bol's exceptional 5-web.
Between the  appearance of the two examples a hiatus of more or less   70 years took place.

\bigskip

One might think that all the webs naturally associated to the equations
of the form (\ref{E:polylogGEN}) satisfied by the polylogarithms are all exceptional (see \cite[pages 196-197]{ABEL}).
Although these  webs are certainly of high rank, they are not necessarily of the highest rank.
For example, using  Mih\u{a}ileanu criterium, one can show by brute force
computation that the webs associated to Kummer equations
 $\mathcal K(4)$ and $\mathcal K(5)$
are not of maximal rank (for details, see \cite[Chap. VII]{PTese}).
\medskip

Nevertheless, it seems to exist a  (large ?) class of
global exceptional webs
with abelian relations expressed in terms
of a natural generalization of the classical polylogarithms: the iterated integrals of logarithmic $1$-forms on $\mathbb P^1$.
For instance all the elements of the family of  $10$-webs with parameters  $a,b\in \mathbb C\setminus\{0,1\}$
\begin{align*}
\mathcal W_{a,b}=\mathcal W\bigg(
x,&y,\frac{x}{y}, \frac{1-y}{1-x}, \frac{b-y}{a-x},\frac{x(1-y)}{y(1-x)},
\frac{x(b-y)}{y(a-x)},
\\
& \quad   \frac{(1-x)(b-y)}{(1-y)(a-x)}   ,   \frac{(bx-ay)(1-y)}{(y-x)(b-y)},
\frac{(bx-ay)(1-x)}{(y-x)(a-x)}
\bigg)
\end{align*}
 are exceptional webs.
Through a method proposed by  Robert in \cite{Robert}, it is possible to
determine $\mathcal A(\mathcal W_{a,b})$ for no matter which  $a$ and $b$ (see  \cite{piriorobert})
and deduce the maximality of the rank.

\subsection{A very simple series of exceptional webs}

It is hard to imagine  examples of webs simpler than the ones presented below.
\begin{align*}
\mathcal W_1=& \, {\mathcal W}\big( \, x \, , \,  y \, ,  \, x+y \, , \, x-y \,, \,  xy
\, \big) \nonumber \\
\mathcal W_2=&{\mathcal W}\big( \, x \, , \,  y \, ,  \, x+y \, , \, x-y \,, \,  xy
\, , \,  x/y \, \big)\nonumber \\
\mathcal W_3=&{\mathcal W}\big( \, x \, , \,  y \, ,  \, x+y \, , \, x-y \,, \,  x/y
\, , \,  x^2+y^2 \, \big)\nonumber \\
\mathcal W_4=&{\mathcal W}\big( \, x \, , \,  y \, ,  \, x+y \, , \, x-y \,, \,  xy
\, , \,  x^2+y^2 \, \big)\nonumber \\
\mathcal W_5=&{\mathcal W}\big( \, x \, , \,  y \, ,  \, x+y \, , \, x-y \,, \,  xy
\, , \,  x^2-y^2 \, \big)\nonumber \\
\mathcal W_6=&\mathcal W\big( \, x\, ,
\,y\, , \,x+y\, , \,x-y \, , \, xy \, , \,{x}/{y} \, , \, x^2-y^2
\, \big) \nonumber \\
\mathcal W_7=&\mathcal W\big( \, x\, ,
\,y\, , \,x+y\, , \,x-y \, , \, xy \, , \,{x}/{y} \, , \, x^2+y^2
\, \big) \nonumber \\
\mathcal W_8=&\mathcal W\big( \, x\, ,
\,y\, , \,x+y\, , \,x-y \, , \, xy \, , \,{x}^2-y^2 \, , \, x^2+y^2
\, \big) \nonumber \\
\mathcal W_9= &\mathcal W\big( \, x\, ,
\,y\, , \,x+y\, , \,x-y \, , \, xy \, , \,{x}/{y} \, , \, x^2-y^2
\, , \, x^2+y^2\,\big)
\end{align*}
It turns out that they are all exceptional as proved in
 \cite[Appendice]{PTese}.  Notice that the webs $\mathcal W_1$ and $\mathcal W_2$ above are nothing
 more than the webs $\mathcal A_{III}^2$ and $\mathcal A_{IV}^2$ from Section \ref{S:CDQLP2}.
 Moreover, $\mathcal W_3$ is equivalent to $\mathcal A_{II}^4$ under a linear change of coordinates.
 In the graph below the inclusions between them are schematically represented.
\begin{center}
 \begin{tikzpicture}[shorten >=1pt,->]
  \tikzstyle{vertex}=[circle,fill=black!25,minimum size=18pt,inner sep=0pt]

  \foreach \name/\angle/\text in {P-4/10/4, P-2/50/2 , P-1/90/1, P-3/130/3,
                                  P-5/170/5, P-6/210/6, P-8/250/8, P-9/290/9, P-7/330/7
                                   }
    \node[vertex,xshift=5cm,yshift=.5cm] (\name) at (\angle:2.5cm) {\footnotesize{$\mathcal W_{\text}$}};
\foreach \from/\to in {1/2,1/4,1/5,1/6,1/7,1/8,1/9,2/6,2/7,2/9,3/9,3/7,4/7,4/9,5/6,5/9,5/8,6/9,7/9,8/9}
    \draw (P-\from) -- (P-\to);
\end{tikzpicture}
\end{center}

\subsection{An exceptional 11-web}

Let  $\mathcal F_2$ be the degree two foliation on  $\check{\mathbb P}^2$ induced by the rational $1$-form
$
\check{y}(\check{y}-1) d\check{x} - \check{x}(\check{x}-1) d\check{y}\,.
$
It is nothing more than the pencil of conics $\frac{\check{x}(\check{y}-1)}{\check{y}(\check{x}-1)}={\rm cte}$.
Let  $C$ be the completely decomposable curve of degree nine in  $\check{\mathbb P}^2$ defined by the homogeneous polynomial
$$
 \check{x} \check{y} \check{z}(\check{x}-\check{z})(\check{y}-\check{z})(\check{x}-\check{y})(\check{x}+\check{y}) (\check{x}-\check{y} -\check{z})  (\check{x}-\check{y} +\check{z})=0\, .
$$
As can be seen above, $C$  is the reunion of six lines invariant by  $\mathcal F_2$ with
 three  extra lines synthetically described as the lines
 joining the three singular points of the fibers of the pencil:
 these latter are cut out by $\check{x}+\check{y}$, $ \check{x}-\check{y} -\check{z}$ and $\check{x}-\check{y} +\check{z}$.
\smallskip

The algebraic web  $\mathcal W_C$ is formed by nine pencil of lines.
If $\mathcal W_{\mathcal F_2}$ is the dual web
of $\mathcal F_2$, in the sense of Section \ref{subsection:Projective duality} of Chapter \ref{Chapter:intro}  then
$\mathcal W_{\mathcal F_2}\boxtimes \mathcal W_C
$
is an exceptional 11-web on $\mathbb P^2$.  After a two-fold ramified covering
it can be written as the completely decomposable web $\mathcal W=\mathcal W(F_1,\ldots,F_{11})$
where  $F_1,\ldots,F_{11}$ are the rational functions below:
\begin{align*}
F_1&=\frac{(x-1)y}{(y-1)x}  && F_2=\Big(\frac{y-x-1}{y-x+1}\Big)^2F_1
\\
F_3&=\frac{(y-1)y}{(x-1)x}
&& F_4=\frac{(y-x)y}{x-1} \\
F_5&=\frac{(x-1)y}{(y-1)x}   && F_6=
\frac{(x-y+1)y}{x}
\\
F_7&=\frac{x+y-1}{xy}
  && F_8=\frac{(y-x-1)x}{y}
\\ F_9
&=\Big(\frac{x-y+1}{y-x+1} \Big)F_1
&&
F_{10}=
\frac{y\,(x-1)(x-y+1)}{x\,(y^2-xy-x+1)} \\
  F_{11}&=
\frac{x\,(y-1)(y-x+1)}{y\,(x^2-xy-y+1)} \; .
   &&
\end{align*}
\smallskip

Using Abel's method, the abelian relations of  $\mathcal W$ can be explicitly determined.
As a by-product, it follows that not only $\mathcal W$ is exceptional, but also a certain number
of its subwebs. A partial list is provided by the following

\begin{prop}
The following ascending chain of subwebs of $\mathcal W$
$$\mathcal  W(F_1,\ldots,F_5)\subset
\mathcal W(F_1,\ldots,F_6)\subset  \ldots
\subset
\mathcal W(F_1,\ldots,F_{11}) = \mathcal W\;  $$
is formed by exceptional webs.
\end{prop}

It was  David Mar\'in together with the first author who guessed that this $11$-web was interesting
in what concerns its rank.
The  second author confirmed this intuition,  proving the proposition above using Abel's method.

\subsection{Terracini and Buzano's webs}

As explained in  Section \ref{S:PMPW} of Chapter \ref{Chapter:4},
there is a  germ of smooth surface $S_\mathcal W \subset \mathbb P^5$ attached to every exceptional $5$-web $\mathcal W$: the image of its Poincar\'{e}'s map.
Moreover, the geometry of $S_{\mathcal W}$ has rather special geometrical features as recalled below
 \begin{enumerate}
\item at a generic point, the second osculating space of $S_\mathcal W$ coincides with the whole $\mathbb P^5$;
\item the image of $\mathcal W$ by Poincar\'{e}'s map of $\mathcal W$ is   Segre's web of  $S_\mathcal W$;
\item  the union of the tangent planes of  $S_\mathcal W$ along one of the leaves
of its Segre's web is included in a hyperplane.
              \end{enumerate}

If   $S\subset \mathbb P^5$ is a germ of surface on $\mathbb P^5$ satisfying the above three conditions,
it is natural to ask if its Segre's web, as defined in Section \ref{S:Segre} of Chapter \ref{Chapter:intro}, is of maximal rank or not.
A positive answer would establish the equivalence between the
classification problem for  exceptional  5-webs with  a  problem  of projective differential geometry: the classification of surfaces subject
to the  constraints enumerated above.
It is the latter problem which motivated Terracini and subsequently  Buzano toward the results
recalled below. \medskip

A  surface   $S\subset \mathbb P^5$ will be called  \defi[exceptional surface]  \index{Exceptional surface}
if  it is not included in a Veronese surface,  its  Segre's $5$-web $\mathcal W_S$ is generically smooth, and if conditions {1.} and  {3.}
above are satisfied. Under this assumption, one proves the existence of five germs of
curve
$C_{S,i}\subset \check{\mathbb P}^5$  called  \defi[Poincar\'{e}-Blaschke's curves] \index{Poincar\'{e}-Blaschke's curves} of  $S$, satisfying
$$ S=\bigcap_{i=1}^5 {\big(C_{S,i}\big)}^* $$
where  $C^*\subset \mathbb P^5$ stands for the dual variety of a germ of curve
$C\subset \check{\mathbb P}^5$. In other words, $C^*$ is the subset of $\mathbb P^5$ corresponding
to the hyperplanes $H\in \check {\mathbb P}^5$ tangent to $C$.

\smallskip

In  \cite{Terracini37},  Terracini obtained a characterization of exceptional surfaces as
solutions of a certain non-linear differential system. Under additional simplifying hypotheses,
he succeeded to integrate explicitly the resulting system, and in this way proved the following result.

\begin{thm}
Up to projective automorphism, there are  exactly four exceptional surfaces $S \subset \mathbb P^5$ for which
three of its  Poincar\'{e}-Blaschke curves -- say $C_{S,i}$ for $i=1,2,3$ --
 are planar  and such that the three associated planes $\langle C_{S,i}\rangle \subset \check{\mathbb P}^5$
 have one point in common. One of these  surface is the  image of  Poincar\'{e}'s map of   Bol's web,
 and the other three are the image of Poincar\'{e}'s map of the following webs:
\begin{align}
\label{E:webTerracini}
 {\rm Terr}(b)= & \,\mathcal W\big(x,y,x+y,x-y,x^2-y^2\big)     \\
 {\rm Terr}(c)= & \, \mathcal W \Big( x , y, \frac{(x+y)^2}{1+y^2} , \frac{y\,(x^2y-2x-y)}{1+y^2} ,  \frac{x^2y-2x-y}{x^2+2xy-1}   \Big) \nonumber    \\
 {\rm Terr}(d)= &\,  \mathcal W \Big(x,y,x+y,\frac{x}{y},\frac{x}{y}(x+y)\Big)      \,    . \nonumber
\end{align}
\end{thm}

Using  Terracini's approach, Buzano  \cite{Buzano39} proved the following result.

\begin{thm}
Up to projective automorphism, there are exactly two exceptional surfaces for which three of its  Poincar\'{e}-Blaschke curves -- say  $C_{S,i}$ for $i=1,2,3$ --
are  planar and satisfy
\begin{enumerate}
\item[{\it (a).}] for every distinct $i,j\in \underline{3}$,  $\langle C_{S,i},C_{S,j}\rangle \subset \check{\mathbb P}^5$  is a hyperplane;
\item[{\it (b).}] the intersection
$\langle C_{S,1}\rangle \cap \langle C_{S,2}\rangle \cap \langle C_{S,3}\rangle$  is empty.                                             \end{enumerate}
They are the image of Poincar\'{e}'s map  of the following webs:
\begin{align}
\label{E:webBuzano}
 {\rm Buz}(a)= & \,\mathcal W\big(x,y,x+y,x-y,{\rm tanh}(x){\rm tanh}(y)\big)      \\
 {\rm Buz}(b)= & \, \mathcal W \big( x , y, x+y,x-y,e^x+e^y   \big)     \,    .  \nonumber
\end{align}
\end{thm}\medskip

It turns out that the webs  (\ref{E:webTerracini}) and   (\ref{E:webBuzano}) are all exceptional.
Curiously, this was not proved by Terracini nor by  Buzano. They focused  on the differential-geometric
problem. The exceptionality has been established just recently in   \cite{crasluc,PT}  (see also \cite{PTese}),
 using Abel's method.
\smallskip

Certain exceptional surfaces are transcendent, as for example the image of Poincar\'{e}'s map of Bol's $5$web, while other
are algebraic and even rational as the one associate to  ${\rm Terr}(b)$.
The image of  Poincar\'{e}`s map of ${\rm Terr}(b)$ can be described as the Zariski closure of the image of the  map
$$
(x,y)\longmapsto \Big[ 1:
x^3+y^3:x^3-y^3:x^2+y^2: x^2-y^2:(x^2-y^2)^2\Big]\,.
$$

A toy problem that might shed some light on the subject, consists in determining the linear systems
$\mathscr L\subset | \mathcal O_{\mathbb P^2}(q)|$, for small $q$,   of dimension  5 for which
the Zariski closure of the image of the associated rational map $\mathbb P^2 \dashrightarrow \mathbb P^5$
are exceptional surfaces. Notice that for  $q=4$, ${\rm Terr}(b)$ is an example, and that $4$ is the minimal $q$
which can happen. Indeed for $q=2$ one obtains a Veronese surface, and for $q=3$ the hyperplane containing
the tangent spaces of leaves of Segre's web would pull-back to a cubic containing an irreducible
component with multiplicity two. This implies that the pull-back of Segre's web to $\mathbb P^2$ is a linear, and consequently, algebraic web.

%\small

%\printindex[not]

%\clearpage

%\addcontentsline{toc}{chapter}{List of Figures}
%\thispagestyle{empty}
%\listoffigures

%\clearpage
%~
%\thispagestyle{empty}
%\newpage

\addcontentsline{toc}{chapter}{Index}
\small
\printindex

\addcontentsline{toc}{chapter}{List of Figures}
\listoffigures

\end{document}